\documentclass[preprint,12pt,authoryear,nopreprintline]{elsarticle}
\makeatletter
\let\robustsd@elsappendix\appendix
\renewcommand{\appendix}{%
  \robustsd@elsappendix
  \gdef\thetheorem{\@Alph\c@section.\arabic{theorem}}}
\makeatother
\usepackage{amsmath, amssymb, amsthm}
\usepackage{tikz}
\usepackage{graphicx}
\usepackage[UKenglish]{isodate}
\usepackage[font=small,labelfont=it]{caption}
\usepackage{subcaption}
\usepackage[hyperfootnotes=false]{hyperref}
\hypersetup{
  hidelinks,
  pdftitle={Proper-Score Observation-Driven Filters: Local Precision and Stationary Tracking Risk},
  pdfauthor={Giulia Livieri and Gianluca Palmari},
  pdfsubject={Observation-driven filtering under proper scoring rules},
  pdfkeywords={proper scoring rule, score-driven model, continuous-time limit, moving target, robust filtering}
}
\usepackage{multirow}
\usepackage{fancyvrb}
\usepackage[margin=1in]{geometry}
\usepackage{setspace}
\usepackage{float,lmodern}
\usepackage{mathtools}
\usepackage{booktabs}
\usepackage{array}
\usepackage{tabularx}
\usepackage{longtable}
\usepackage{pdflscape}
\usepackage{makecell}
\usepackage{placeins}
\newcolumntype{Y}{>{\centering\arraybackslash}X}
\usepackage{enumitem}
\mathtoolsset{showonlyrefs}
\graphicspath{{figures/}}

\DeclareMathOperator{\Var}{Var}
\DeclareMathOperator{\Cov}{Cov}

\makeatletter
\@ifclassloaded{elsarticle}
  {\newcommand{\appref}[1]{\ref{#1}}}
  {\newcommand{\appref}[1]{\ref{#1}}}
\makeatother

\renewcommand{\phi}{\varphi}

\let\emptyset\varnothing

\DeclareMathOperator*{\argmax}{arg\,max}
\DeclareMathOperator*{\argmin}{arg\,min}

\newcommand{\E}{\mathsf{E}}

\newcommand{\R}{\mathbb{R}}

\newcommand{\ud}{\mathrm{d}}

\def \E{\mathsf E}

\def \R{\mathbb{R}}

\def \ud{\mathrm{d}}

\newtheorem*{theorem*}{Theorem}
\newtheorem{theorem}{Theorem}[section]

 \newtheorem{proposition}[theorem]{Proposition}
\newtheorem{corollary}[theorem]{Corollary}
\newtheorem{remark}[theorem]{Remark}
\newtheorem{lemma}[theorem]{Lemma}
\theoremstyle{definition}
\newtheorem{example}[theorem]{Example}
\theoremstyle{plain}
\newtheorem{assumption}[theorem]{Assumption}
\newtheorem{definition}[theorem]{Definition}

\usepackage{xcolor}

\newcommand{\proofref}[1]{\par\noindent The proof is provided in \ref{#1}.}
\numberwithin{equation}{section}
\date{3 December 2026}

\begin{document}
\onehalfspacing\raggedbottom\begingroup\singlespacing
\begin{frontmatter}
\title{Proper-Score Observation-Driven Filters: Local Precision and Stationary Tracking Risk}
\author[lse]{Giulia Livieri}
\ead{g.livieri@lse.ac.uk}
\author[sns]{Gianluca Palmari\corref{cor1}}
\ead{gianluca.palmari@sns.it}
\cortext[cor1]{Corresponding author.}

\affiliation[lse]{organization={Department of Statistics, London School of Economics and Political Science},
  addressline={Columbia House, 69 Aldwych}, city={London}, postcode={WC2A 2AE}, country={United Kingdom}}
\affiliation[sns]{organization={Scuola Normale Superiore},
  addressline={Piazza dei Cavalieri 7}, city={Pisa}, postcode={56126}, country={Italy}}

\begin{abstract}
We study scalar observation-driven filters updated by derivatives of proper scoring rules and ask when local precision describes long-run tracking risk. For explicit scale filters with update growth exponent $\kappa_g$ and innovation moment index $\nu_\varepsilon$, we establish exact stationary moment boundaries under stated innovation and restoration conditions: $p=\nu_\varepsilon/\kappa_g$ below unit persistence and $p-1$ at unit persistence, where the latter case requires $p>1$. For a Gaussian log-variance filter with variance-standardised Student-$t_\nu$ innovations, matched persistence and independent Gaussian target shocks, stationary mean squared log-variance error is finite exactly when $\nu>4$ below unit persistence and $\nu>6$ at unit persistence. For $4<\nu\le6$ it diverges as persistence approaches one, although the invariant laws converge weakly. Every positive exponential moment of the positive error, and hence stationary reference-normalised Gaussian CRPS risk, is infinite for each finite $\nu>2$. For bounded, globally restoring drivers tracking a random walk, gain-uniform exponential bounds on the rescaled error yield convergence of stationary reference-scale risk to the local Ornstein--Uhlenbeck approximation and convergence of population risk-minimising gains. Under misspecification, rule-specific target shifts separate precision from latent-state accuracy and can reverse the next-order gain ranking. Population calculations and reproducible finite-horizon simulations illustrate these distinctions; an appendix reports descriptive market comparisons and a location-filter illustration.
\end{abstract}
\end{frontmatter}
\begin{center}3 December 2026\end{center}
\endgroup\onehalfspacing
\clearpage
\section{Introduction}\label{sec::Introduction}

Observation-driven filters track a time-varying parameter, such as the volatility of daily returns, by adjusting the current estimate after each new observation. In score-driven models the adjustment is proportional to the derivative of the log density \citep{creal2013gas,harvey2013dynamic}. Other proper scoring rules can take the place of the log density, such as the continuous ranked probability score (CRPS), a density-power score or a kernel score. Each choice gives a different filter. The filters differ in what they track and in how strongly they react to outliers. These filters form a scalar, parametric subclass of the PRADA framework of \citet{depunder2026prada}.

Much of the theory for such filters is local. When the sampling interval $h$ is small, the filtering error behaves over short horizons like an Ornstein--Uhlenbeck process \citep{beutner2026consistency}. The variance of this process measures how precisely the filter follows its target. We call this variance the filter's \emph{local precision}. A filter in use, however, runs indefinitely, so what matters in the long run is the stationary distribution of its error. This paper asks when local precision describes long-run tracking risk.

\paragraph{When local precision fails}
Our first result shows that the two can differ sharply. Consider the standard Gaussian log-variance filter applied to returns whose innovations follow a Student-$t_\nu$ distribution, standardised to unit variance. The latent log variance and the filter share the same autoregressive coefficient $\varphi$, and the latent process receives independent Gaussian shocks. Theorem~\ref{pur:thm:persistence} shows that the tracking error has a unique stationary distribution. Its absolute moment of order $r$ is finite exactly when
\[
r<\nu/2\quad\text{if } 0\le\varphi<1,\qquad\qquad r<\nu/2-1\quad\text{if } \varphi=1 .
\]
Hence, for $4<\nu\le6$, the stationary mean squared error is finite for every $\varphi<1$ but grows without bound as $\varphi$ approaches one, even though the stationary distributions themselves converge. At $\varphi=1$ it is infinite for every sampling interval, although the local Ornstein--Uhlenbeck transition has finite conditional variance for $\nu>4$ under the small-mesh scaling and initialisation stated below.

The mechanism is simple. A large squared return pushes the filtered variance far above its target. Once the filter is too high, every new observation looks small. The Gaussian-log update is then close to its lower bound, so the filter comes back down at a roughly constant speed. When $\varphi=1$, a jump of size $x$ therefore takes a number of observations proportional to $x$ to undo. Squared Student-$t_\nu$ innovations have tails of order $x^{-\nu/2}$, so these long excursions produce a stationary tail at least of order $x^{1-\nu/2}$. When $\varphi<1$, the autoregressive term removes a fixed fraction of any large error at each step, which shortens the excursions and gains one moment. The lower bound adapts a classical large-jump argument \citep{zachary2004veraverbeke}; we also prove matching upper bounds, so the moment boundary is exact.

\paragraph{When local precision succeeds}
Our second result gives conditions under which local precision does describe long-run risk. Suppose the update is bounded and pulls the filter back towards its target from any distance, and the target follows a random walk with Gaussian shocks. Proposition~\ref{trk:prop:selftune} shows that the stationary error, scaled by $h^{-1/4}$, has exponential moments that are bounded uniformly over compact sets of gains. Consequently, the scaled stationary mean squared error converges to the variance of the local Ornstein--Uhlenbeck limit. The gain that minimises stationary risk also converges to the gain that minimises local variance. The result concerns risk measured at a fixed reference scale. For the CRPS and density-power rules, the loss at the actual latent level carries an extra level-dependent weight, which the proposition does not cover.

\paragraph{What each filter tracks}
Our third group of results concerns targets. Under misspecification, different scoring rules track different quantities. In scale models, each rule targets the latent log variance plus a shift $\delta_{\mathsf S}(F)$ that depends on the rule and on the innovation distribution $F$. We compare each filter's local precision with an efficient benchmark. When the update is uncorrelated with the target shocks, this comparison reduces to the filter's \emph{tracking efficiency}: the squared correlation between its update and the true score (Proposition~\ref{trk:prop:eff}). In scale models, this efficiency equals the asymptotic relative efficiency of the corresponding static minimum-score estimator, once its target shift is removed \citep{godambe1960optimum,dawid2016minimum}.

Precision about a rule's own target is not the same as accuracy for the latent variance. Under Student-$t_3$ innovations, for example, a Student-$t_6$ CRPS filter is nearly efficient around its own target, yet that target is only about 63\% of the latent variance. Proposition~\ref{trk:prop:biasprec} separates the two effects: as $h\to0$ the squared target shift dominates, and the tracking variance, and hence efficiency, matters at the next order. Similarly, zero correlation between an even update and the target shock does not mean that the observation carries no information about that shock; Remark~\ref{rev:rem:leverage} states the stronger condition needed.

\paragraph{Relation to existing work}
Several building blocks are known, and we use them as such. The proper-score update and its one-step justification belong to the PRADA framework of \citet{depunder2026prada}, whose 20 November 2026 version establishes consistency with a possibly different estimation score (Theorem~2), asymptotic normality under functional correct specification (Theorem~3), and a long-run covariance formula under misspecification (Theorem~4). For differentiable proper scores and the scalar predictive families considered here, our derivative updates form a subclass of PRADA. Our mean-gradient and parameter-distance identities specialise \citet{gorgi2024optimality}. Whether an update reduces expected divergence is a separate question, with its own curvature and step-size conditions \citep{de2024kullback}. Quasi score-driven models \citep{blasques2023quasi} already separate the updating equation from the predictive density; proper scoring rules give one explicit way to choose the update within that design space.

The local tracking results build on the continuous-record theory of \citet{nelson1994filtering,nelson1996filtering} and on the generic-driver limits of \citet{beutner2026consistency}. The 2023 working-paper version of the latter already contains a functional Ornstein--Uhlenbeck limit (\citealp[Theorem~A.2]{beutner2023consistency}), and its gain--noise frontier is the benchmark for our efficiency results. We adapt these ideas to rule-specific targets using stopped martingale arguments. The stationary results are different in kind: they concern the stationary distribution of the nonlinear recursion at a fixed sampling interval. The integrated-tail mechanism behind Theorem~\ref{pur:thm:persistence} is classical \citep{zachary2004veraverbeke}; we obtain the exact moment cutoff for the matched nonlinear error recursion at and below unit persistence, together with weak convergence of the invariant laws and divergence of the indicated moments as persistence approaches one with the update and shock scales fixed.

Robustness can also be pursued without changing the scoring rule. Implicit score-driven filters \citep{lange2026implicit} and related stability results \citep{donker2025stability} change how the update is computed rather than which loss it uses; Section~\ref{subsec:tail_impulse} compares their response to a large observation with that of the explicit Gaussian-log update. Our moment boundary applies to the explicit update only. Stationary Gaussian approximations and first-order mean-bias expansions also arise in constant-step stochastic approximation \citep{dieuleveut2020bridging,chen2022sa,liwu2026moments}. The bounded, saturating mean fields in Proposition~\ref{trk:prop:selftune} do not satisfy the global strong-convexity, contraction or quadratic-confinement assumptions of the results compared here. We use bounded restoring updates and Gaussian target shocks to obtain exponential control of the rescaled invariant error, uniformly over compact gain sets. This control covers reference-score losses with exponential growth and supports convergence of risks and their minimising gains. Proposition~\ref{trk:prop:biasprec} applies the stationary drift-balance expansion to rule-specific target shifts.

The supporting results use established methods: minimum-score estimation and stochastic recurrence equations \citep{dawid2016minimum,straumann2006quasi,blasques2018feasible,blasques2022maximum}, and diffusion approximations \citep{buccheri2021continuous,wuhe2026continuous}. We do not claim new general estimation or invertibility theory.

\paragraph{Numerical illustration}
Section~\ref{sec::experiments} reports population targets and local efficiency constants, together with two controlled finite-horizon experiments. The experiments measure recovery from a specified large error and compare gains under a fully specified bounded-driver recursion. \ref{sec::empirical} reports descriptive fixed-fit market comparisons and a location-filter illustration. Table~\ref{tab:map} summarises the questions addressed and the results.

\begin{table}[ht]
\centering\footnotesize
\caption{Questions addressed and results.}\label{tab:map}
\begin{tabular}{>{\raggedright\arraybackslash}p{0.27\linewidth}>{\raggedright\arraybackslash}p{0.29\linewidth}>{\raggedright\arraybackslash}p{0.36\linewidth}}
\toprule
Question & Existing results & What this paper adds\\
\midrule
When does local precision describe long-run risk? & Local Ornstein--Uhlenbeck approximations; classical heavy-tail mechanisms & Exact stationary moment boundary; divergence near unit persistence (Theorem~\ref{pur:thm:persistence})\\
What does each filter track, and how precisely? & PRADA; quasi score-driven models; generic-driver local limits & Rule-specific targets; efficiency as a squared correlation; precision versus accuracy (Propositions~\ref{trk:prop:eff} and~\ref{trk:prop:biasprec})\\
Which gain minimises long-run risk? & Optimising the local Ornstein--Uhlenbeck variance & Stationary risk selects the locally optimal gain for bounded restoring updates (Proposition~\ref{trk:prop:selftune})\\
Can the parameters be estimated? & Stochastic recurrence equations; minimum-score estimation & Conditional inference; explicit contraction conditions for selected updates (Section~\ref{sec:static-estimation})\\
\bottomrule
\end{tabular}
\end{table}

\paragraph{Roadmap}
Section~\ref{section::introduction} defines the filters and their rule-specific targets. Section~\ref{trk:sec:remarks} establishes the stationary moment boundary for explicit scale filters, with the Gaussian log-variance filter as the leading case. Section~\ref{sec::local_consistency} develops the local tracking benchmark, and Section~\ref{trk:sec:stationarygain} shows when stationary risk and the risk-minimising gain converge to their local counterparts. Section~\ref{sec::experiments} gives a numerical illustration, and Section~\ref{sec::conclusion} concludes. \ref{etapp:proofs} proves the stationary results, and \ref{sec::empirical} presents the empirical illustrations. The Supplementary Material contains the estimation theory (Section~\ref{sec:static-estimation}), the continuous-time limits (Section~\ref{section::continuous_time}) and the remaining proofs, score formulas and deterministic illustrations.

\section{Scoring-rule-driven filters}\label{section::introduction}

\begingroup
\paragraph{Notation and terminology} Throughout, the following symbols are used.
\begin{itemize}
\item $\lambda$ is the filtered (working) state, and $\tilde\lambda$ the descriptive index of the true conditional law.
\item $\lambda^\star_{\mathsf S,t}$ is the rule-specific projection, or pseudo-true target, defined in~\eqref{eq::pseudotrueset}.
\item $\lambda^\dagger$ is an interior zero of the risk derivative. An identified interior minimiser is such a zero. Under correct specification and an injective model, strictly proper rules share the true parameter as their global target.
\item {A tilde marks the true conditional law, $\widetilde P_t$, its descriptive index, $\tilde\lambda_t$, and quantities computed under that law. The effective gain of Section~\ref{sec::local_consistency} is written $\varrho$, and the curvature-normalised driver of \ref{app::formulas} is written $g^{\circ}_{\mathsf S}$.}
\end{itemize}
A \emph{working family} is the parametric family of predictive laws. A \emph{scoring rule} $\mathsf S$ evaluates a predictive law at an observation, and its \emph{conditional risk}~\eqref{eq::conditional_scoring_risk} averages that loss under the true law. The \emph{scaling} $G_{\mathsf S}$ is a positive predictable, that is $\mathcal F_{t-1}$-measurable, factor that converts the negative score derivative $\psi_{\mathsf S}$ into the \emph{scaled update} or \emph{driver} $u_{\mathsf S}=G_{\mathsf S}\psi_{\mathsf S}$. We write the driver as $g_{\mathsf S}(z)$ when it depends on the observation only through the standardised residual $z$. Its \emph{tail response} is its behaviour as $|z|\to\infty$ (Section~\ref{subsec:tail_impulse}). A \emph{cell} is a combination of working density, scoring rule and scaling.

Four rules are used throughout: the logarithmic score, the density-power score with exponent $\beta$, the continuous ranked probability score (CRPS), and the maximum mean discrepancy (MMD) score with a radial-basis kernel. Their formulas, drivers and tuning constants are in \ref{app::formulas}.
\endgroup

Let $\{{y_t}\}_{t\in\mathbb Z}$ take values in a measurable space $(\mathcal Y,\mathcal B_{\mathcal Y})$, and let
$
\mathcal F_{t-1}=\sigma(y_{t-1},y_{t-2},\ldots)
$
denote the information available before observing $y_t$. The true conditional law of $y_t$ given $\mathcal F_{t-1}$ is denoted by $\widetilde P_t$. We assume that $\widetilde P_t$ is dominated by a fixed sigma-finite measure $\mu_{\mathcal Y}$ and write
$
\widetilde P_t(\mathrm dy)
=
\widetilde{\mathfrak p}_t(y)\,\mu_{\mathcal Y}(\mathrm dy).
$
We sometimes write $\widetilde P_t=\widetilde P_{\widetilde\lambda_t}$ for a descriptive state index; under misspecification, this law need not belong to the working model below. The main-text results are scalar, with $\widetilde\lambda_t\in\widetilde\Lambda\subset\mathbb R$; \appref{app::multivariate} gives the stopped vector extension.

In this fixed-frequency setup, conditional expectations are taken with respect to the true predictive law given observation history. Thus, for any $\mathcal F_{t-1}$-measurable random variable $Z$ and any integrable jointly measurable function $h$,
\[
\mathsf E_{t-1}\big[h(y_t,Z)\big]
=
\int_{\mathcal Y}
h(y,Z)\,\widetilde{\mathfrak p}_t(y)\,\mu_{\mathcal Y}(\mathrm dy),
\]
almost surely.
We specify a working conditional family
$
\mathcal P_\theta
=
\{P_{\lambda,\theta}:\lambda\in\Lambda\},
$
where
$
P_{\lambda,\theta}(\mathrm dy)
=
\mathfrak p(y\mid\lambda;\theta)\,\mu_{\mathcal Y}(\mathrm dy)
$, and $\theta\in\Theta$.
Here $\lambda$ is dynamic and $\theta$ static. The $\mathcal F_{t-1}$-measurable filter $\lambda_t=\lambda_t(y^{t-1},\theta,\lambda_1)$ is generated recursively. The family may be misspecified: scoring rules are evaluated at $P_{\lambda,\theta}$, but expectations are under $\widetilde P_t$.

We use negatively oriented proper scoring rules \citep[see, e.g.,][]{savage1954foundations,dawid2004probability,gneiting2007strictly,dawid2016minimum}. Let $\mathcal P_{\mathsf S}(\mathcal Y)$ be a class of probability measures on $(\mathcal Y,\mathcal B_{\mathcal Y})$. A scoring rule is a measurable map $\mathsf S:\mathcal P_{\mathsf S}(\mathcal Y)\times\mathcal Y\to\mathbb R\cup\{+\infty\}$, with smaller values indicating better predictive performance.

For $P,Q\in\mathcal P_{\mathsf S}(\mathcal Y)$, write
\(
\mathsf S(Q,P)=\mathsf E_Q[\mathsf S(P,Y)]
\)
whenever the expectation is well defined. The rule is proper if $\mathsf S(Q,Q)\leq\mathsf S(Q,P)$ for all $P,Q$, and strictly proper if equality implies $P=Q$. Its divergence is
\(
\mathsf D_{\mathsf S}(Q,P)=\mathsf S(Q,P)-\mathsf S(Q,Q)
\)
and is nonnegative whenever both expected losses are finite. Under strict propriety it vanishes exactly when $P=Q$. The logarithmic score $\mathsf S_{\log}(P_{\lambda,\theta},y)=-\log\mathfrak p(y\mid\lambda;\theta)$ induces the Kullback--Leibler divergence.

To avoid ambiguity, $\mathsf S(P_{\lambda,\theta},y)$ denotes realised loss, whereas $\mathsf S(Q,P)$ denotes expected loss under the reference law $Q$. Throughout, $\widetilde P_t$ is the true law, $P_{\lambda,\theta}$ the working law, $\lambda_t$ the filtered state, and $\theta$ the static parameter. Other notation is introduced when first used.

For fixed $\theta$, write $\ell_{\mathsf S}(y,\lambda;\theta)=\mathsf S(P_{\lambda,\theta},y)$ for the loss induced by $P_{\lambda,\theta}\in\mathcal P_\theta$.
When finite, the conditional scoring risk is
\begin{equation}\label{eq::conditional_scoring_risk}
\mathsf R_{\mathsf S,t}(\lambda;\theta)
=
\mathsf E_{t-1}
\left[
\mathsf S(P_{\lambda,\theta},y_t)
\right]
=
\int_{\mathcal Y}
\mathsf S(P_{\lambda,\theta},y)
\,\widetilde{\mathfrak p}_t(y)\,\mu_{\mathcal Y}(\mathrm dy).
\end{equation}
The risk varies with the working parameter $\lambda$, but integrates under the true law. We suppress fixed $\theta$ below. For a set $A$, $\operatorname{int}(A)$ denotes its interior.

\begingroup
\begin{assumption}[Conditional risk and measurable pseudo-true target]
\label{ass:measurable_target}
Let $\Lambda$ be a non-empty Borel subset of $\mathbb R$, and take an $\mathcal F_{t-1}$-measurable regular conditional probability kernel $\widetilde P_t(\omega,\mathrm dy)$. The map $(\omega,\lambda)\mapsto\mathsf R_{\mathsf S,t}(\omega,\lambda)$ is $\mathcal F_{t-1}\otimes\mathcal B(\Lambda)$-measurable and lower semicontinuous in $\lambda$ almost surely. Its minimum is finite and is attained at a unique point in $\operatorname{int}(\Lambda)$. This minimiser admits an $\mathcal F_{t-1}$-measurable version.
\end{assumption}
We assume a predictable unique target directly; we do not require every risk sublevel to be compact. Such global coercivity would exclude several examples on the log-scale domain $\mathbb R$: CRPS risk has a finite limit as the working scale tends to zero, the density-power risk tends to zero as the scale tends to infinity, and bounded-kernel MMD risk is bounded.

A useful sufficient measurability argument is available when conditional differentiation is valid, $\Lambda$ is an open interval and the conditional mean field is jointly measurable, continuous and strictly decreasing in $\lambda$, with a unique zero and opposite signs on either side. Its zero is
\[
 \lambda^\star_{\mathsf S,t}
 =\inf\{q\in\mathbb Q\cap\Lambda:
           \mathsf E_{t-1}[\psi_{\mathsf S}(y_t,q)]\le0\},
\]
so it is predictable by countable measurability. For nonmonotone mean fields, uniqueness, attainment and measurable selection require a separate argument or must remain explicit assumptions. A compact sublevel strictly above the minimum, separated from the boundary limits, can suffice; compactness of all sublevels is unnecessary.
\endgroup

\begingroup
Whenever $\widetilde P_t\in\mathcal P_{\mathsf S}(\mathcal Y)$, $\mathcal P_\theta\subseteq\mathcal P_{\mathsf S}(\mathcal Y)$, and the candidate and self-risks are finite, the definition of $\mathsf D_{\mathsf S}$ at $Q=\widetilde P_t$ gives
\begin{equation}\label{eq::criteria}
    \mathsf{R}_{\mathsf{S},t}(\lambda) = \mathsf{D}_{\mathsf{S}}(\widetilde{P}_t,P_{\lambda,\theta})+\mathsf{E}_{t-1}[\mathsf{S}(\widetilde{P}_t,y_t)],
\end{equation}
whose second term does not depend on $\lambda$. These domains are rule-specific:
\begin{itemize}
\item the logarithmic score needs $\widetilde P_t\ll P_{\lambda,\theta}$ and finite self-risk and cross-entropy;
\item the CRPS needs first moments under both laws;
\item the density-power rule needs both densities in $L^{1+\beta}(\mu_{\mathcal Y})$;
\item bounded-kernel MMD needs neither moments nor domination, and a characteristic kernel gives strict propriety.
\end{itemize}
Experiments outside these domains, including the Cauchy endpoint, are flagged where they occur.
\endgroup

Subject to these domain conditions, define the rule-specific projection
\begin{equation}\label{eq::pseudotrueset}
    \Lambda_{\mathsf{S},t}^{\star} = \argmin_{\lambda \in \Lambda} \mathsf{R}_{\mathsf{S},t}(\lambda).
\end{equation}
\begingroup
Its relation to the autoregressive filter path requires the tracking conditions in Section~\ref{sec::local_consistency}. Under Assumption~\ref{ass:measurable_target}, the minimiser is unique and has the predictable version denoted by $\lambda^\star_{\mathsf S,t}$. With several minimisers, existence of a measurable selection would require suitable selection conditions, and a chosen path could switch between separated minima. The tracking results below assume a uniquely identified target. Under correct specification, strict propriety and identification give $\Lambda^\star_{\mathsf S,t}(\theta)=\{\widetilde\lambda_t\}$. {Under misspecification, changing the scoring rule can change the working-family projection and hence the estimand.} First-order conditions below concern interior minimisers; boundary targets require variational inequalities.
\endgroup

We require some regularity conditions for defining the scoring-rule update.
{Assumption~\ref{ass:conditional_differentiation} in \ref{app::appendix_section_5} gives integrability and smoothness conditions under which the derivative of the conditional risk can be taken inside the conditional expectation.}

Define the innovation $\psi_{\mathsf S}(y,\lambda)=-\partial_\lambda\mathsf S(P_{\lambda,\theta},y)$. Assumption~\ref{ass:conditional_differentiation} permits interchange of differentiation and conditional expectation.
\begin{definition}[Scoring-rule-driven update]\label{def::scoring-rule-filter}
Suppose that Assumption~\ref{ass:conditional_differentiation} holds locally. Let $G_{\mathsf S,t}:\Omega\times\Lambda\to(0,\infty)$ be $\mathcal F_{t-1}\otimes\mathcal B(\Lambda)$-measurable and almost surely continuous in $\lambda$, and let $\alpha>0$ be a gain parameter. The scaled update is
\begin{equation}\label{eq::scoring_update}
    \lambda_{t+1}=\lambda_t+\alpha u_{\mathsf{S},t}(y_t,\lambda_t),\qquad
    u_{\mathsf{S},t}(y,\lambda)=G_{\mathsf{S},t}(\lambda)\psi_{\mathsf{S}}(y,\lambda).
\end{equation}
We either take $\Lambda=\mathbb R$, as in the log-variance application, or assume that the update map sends $\Lambda$ into $\Lambda$ almost surely. When the date and random environment are clear we suppress the subscript $t$ and write $G_{\mathsf S}(\lambda)$ and $u_{\mathsf S}(y,\lambda)$.
\end{definition}

When finite and positive, inverse conditional-risk curvature provides an oracle predictable scaling under misspecification; the finite-step filters instead use the feasible deterministic, cell-specific state scalings reported in \appref{app::formulas}.

The GAS update \citep[Definition~3]{blasques2015information} sets $\mathsf S=\mathsf S_{\rm log}$, so $\psi_{\mathsf S}(y,\lambda_t)=\partial_\lambda\log\mathfrak p(y\mid\lambda;\theta)_{|\lambda=\lambda_t}$; information scaling recovers the usual recursion. Other rules change the target, {the local curvature and noise of the update (Section~\ref{subsec:mean_noise_geometry})}, and tail response while retaining the observation-driven form. \appref{app::appendix_section_2} gives an example.

\begingroup
{As stated in the Introduction, these updates form the scalar, parametric subclass of the PRADA updates of \citet{depunder2026prada}.} With learning rate $\alpha G_{\mathsf S,t}$, \eqref{eq::scoring_update} is his canonical update, and with the autoregressive step of \ref{subsec:applied_ar_recursion} it is his PRADA(1,1) recursion. {Under the regularity, off-target and step-size conditions of that analysis, the expected score divergence decreases. We take that result as given; it is separate from the conditional-mean identity below.}
\endgroup

\begingroup
Under Assumption~\ref{ass:conditional_differentiation}, and when the moments involved are finite,
\begin{equation}\label{eq::gradient_update}
    \mathsf{E}_{t-1}[\lambda_{t+1}]=\lambda_t - \alpha G_{\mathsf{S}}(\lambda_t) \mathsf{R}_{\mathsf{S},t}^{'}(\lambda_t).
\end{equation}

The identity describes the expected update. The realised observation can give a different impression of its performance on a fresh draw. For example, take a unit-variance Gaussian location family, the true mean zero, and $Y,Y'$ independent standard normal. From the correct mean the log-score step produces $\lambda^+=\alpha Y$. Updating reduces the loss on $Y$ by $\alpha(2-\alpha)Y^2/2$ for $0<\alpha<2$, yet increases the expected loss on $Y'$ by $\alpha^2/2$. A decrease in expected predictive risk requires control of curvature, noise and gain \citep{de2024kullback}. Lemma~\ref{lem:onestep} instead concerns conditional-mean contraction and one-step parameter error under a monotone mean field, as in \citet{gorgi2024optimality}.

\endgroup

{The continuous-time and static-estimation sections use the autoregressive recursion $\lambda_{t+1}=\omega+\varphi\lambda_t+\alpha u_{\mathsf S,t}$, $|\varphi|<1$. An interior zero of its conditional mean increment, when it exists, is a composite centre $\lambda^{\mathrm{AR}}_{\mathsf S,t}$. The scoring-rule projection is such a zero precisely when $\omega=(1-\varphi)\lambda^\star_{\mathsf S,t}$ (\ref{subsec:applied_ar_recursion}).}

\subsection{\texorpdfstring{Local mean--noise {structure}}{Local mean--noise structure}}\label{subsec:mean_noise_geometry}

Fix $t$ and $\theta$, and suppose the model is correctly specified at an interior value $\lambda^\dagger$, so that $\widetilde P_t=P_{\lambda^\dagger,\theta}$. Write
$
\psi_{\mathsf S}(y,\lambda)
=
-\partial_\lambda \mathsf S(P_{\lambda,\theta},y).
$
Differentiation under conditional expectation gives
$
\mathsf E_{t-1}\left[\psi_{\mathsf S}(y_t,\lambda^\dagger)\right]
=
-\mathsf R'_{\mathsf S,t}(\lambda^\dagger)
=
0.
$
Define local curvature (sensitivity) and innovation second moment as
\begin{equation}\label{eq::distinction}
    J_{\mathsf S,t}(\lambda^\dagger)
    =
    \mathsf R''_{\mathsf S,t}(\lambda^\dagger),
    \qquad
    K_{\mathsf S,t}(\lambda^\dagger)
    =
    \mathsf E_{t-1}\!\left[
        \psi_{\mathsf S}^2(y_t,\lambda^\dagger)
    \right].
\end{equation}

At the centred target, $J$ measures curvature and $K$ is the conditional variance of the unscaled innovation \citep{dawid2016minimum}. A regular, correctly specified log likelihood with differentiable normalisation and common support satisfies the second Bartlett identity, $J=K$. Under misspecification the two quantities can differ even for log loss. Proposition~\ref{prop::local_mean} gives the conditional mean coefficient $1-\alpha G_{\mathsf S,t}J_{\mathsf S,t}$ and leading variance $\alpha^2G_{\mathsf S,t}^2K_{\mathsf S,t}$ at the target. The oracle choice $G=J^{-1}$ gives coefficient $1-\alpha$ and variance $\alpha^2 K/J^2$. \ref{app::formulas} specifies the feasible scalings used in the application.

\begingroup
The noise-to-curvature ratio $\kappa_{\mathsf S}=K_{\mathsf S}/J_{\mathsf S}^2$ reappears in Section~\ref{sec::local_consistency} as the constant that governs tracking efficiency.

{Since the filtered path depends on the rule only through the scaled derivative $\alpha G_{\mathsf S}\psi_{\mathsf S}$, rules are compared through $\kappa_{\mathsf S}$, which is invariant to rescaling (Remark~\ref{rem:dynamic_equivalence} in \ref{app::appendix_section_5}).}
\endgroup

\subsection{Tail behaviour and impulse propagation}\label{subsec:tail_impulse}
Robustness is governed by the scaled innovation $u_{\mathsf S}=G_{\mathsf S}\psi_{\mathsf S}$. Let $z_\lambda(y)$ denote the standardised residual and $y_\lambda(z)$ its inverse. Define $g_{\mathsf S}(\lambda,z)=u_{\mathsf S}(y_\lambda(z),\lambda)$.
At fixed $\lambda$, its tail envelope is\footnote{This parallels static robust inference: bounded estimating functions give bounded influence, while redescending functions give vanishing influence for extreme observations \citep{dawid2016minimum}.}:
\[\mathcal{T}_{\mathsf{S}}(r;\lambda)=\sup_{|z| \geq r}\lvert g_{\mathsf{S}}(\lambda,z)\rvert.\]
At state $\lambda$, the driver has \emph{bounded sensitivity} if $\mathcal T_{\mathsf S}(0;\lambda)<\infty$, and \emph{unbounded sensitivity} if $\mathcal T_{\mathsf S}(r;\lambda)=\infty$ for every $r$. A bounded driver is \emph{redescending} if $\mathcal T_{\mathsf S}(r;\lambda)\to0$ as $r\to\infty$, and \emph{tail-saturating} if $g_{\mathsf S}(\lambda,z)\to c_\pm(\lambda)\ne0$ as $z\to\pm\infty$. Thus, gross-outlier contributions vanish only for redescending drivers; tail-saturating drivers cap them at non-zero levels.
\begingroup
{For the logarithmic, CRPS and density-power scale filters below, the scalings of \appref{app::formulas} make the update depend on $(y,\lambda)$ only through $z=ye^{-\lambda/2}$, so $u_{\mathsf S}(y,\lambda)=g_{\mathsf S}(z)$. Fixed-data-bandwidth MMD is excluded from this residual-only statement: its driver also depends on the state through the effective bandwidth.} For $f(\lambda,y)=\omega+\varphi\lambda+\alpha g_{\mathsf S}(ye^{-\lambda/2})$,
\begin{equation}\label{eq::scale_derivative}
\partial_\lambda f(\lambda,y)=\varphi-\alpha\,m_{\mathsf S}(z),\qquad m_{\mathsf S}(z)=\tfrac12zg_{\mathsf S}'(z).
\end{equation}
The driver is \emph{monotone} if $m_{\mathsf S}\ge0$, that is, if $g_{\mathsf S}$ is non-decreasing in $|z|$. The Gaussian log score has $m_{\mathsf S}(z)=z^2$. The Gaussian and Student-$t$ CRPS and the Student-$t$ log score have bounded $m_{\mathsf S}\in[0,M_+]$. The density-power drivers have bounded $m_{\mathsf S}$ of both signs. Write $\operatorname{osc}(g)=\sup g-\inf g$.

\begin{proposition}[Outlier propagation]\label{prop::outlier}
Consider two paths with a common state at date $t$ and a common subsequent input sequence, differing only in the input at date $t$. In part~(i) the inputs are exogenous residuals; in part~(ii) they are raw observations.
\begin{enumerate}[label=(\roman*)]
\item \emph{Additive residual recursion.} If $\lambda_{t+1}=\omega+\varphi\lambda_t+\alpha g_{\mathsf S}(Z_t)$ with residuals that do not depend on the state, set $Z_t^{(B)}=B$ and $Z_t^{(0)}=z_0$. Then $\lambda^{(B)}_{t+1+k}-\lambda^{(0)}_{t+1+k}=\alpha\varphi^k\{g_{\mathsf S}(B)-g_{\mathsf S}(z_0)\}$ for every $k\ge0$.
\item

\emph{State-dependent recursion.} Let $y_t^{(B)}=B$ and $y_t^{(0)}=y_0$ now denote raw observations. Suppose the update scaling and other predictable inputs at date $t$ are common to both paths, and that all subsequent maps $F_s$ are common and satisfy $|F_s(\lambda,y)-F_s(\lambda',y)|\le q|\lambda-\lambda'|$ for a deterministic $0\le q<1$. Then, for every $k\ge0$,
\begin{align*}
 |\lambda^{(B)}_{t+1+k}-\lambda^{(0)}_{t+1+k}|
 &\le q^k\alpha\big|u_{\mathsf S,t}(B,\lambda_t)-u_{\mathsf S,t}(y_0,\lambda_t)\big|\\
 &\le q^k\alpha\operatorname{osc}_y u_{\mathsf S,t}(\cdot,\lambda_t),
\end{align*}
where $\operatorname{osc}_y u_{\mathsf S,t}(\cdot,\lambda)=\sup_y u_{\mathsf S,t}(y,\lambda)-\inf_y u_{\mathsf S,t}(y,\lambda)$ is the oscillation at a fixed state and environment. The sum of the absolute effects over $k\ge0$ is at most $\alpha\operatorname{osc}_y u_{\mathsf S,t}(\cdot,\lambda_t)/(1-q)$. In particular, if the deterministic bound
\[
 \Omega_{\mathsf S}=\sup_t\operatorname*{ess\,sup}_{\omega}\sup_{\lambda\in\Lambda}
 \operatorname{osc}_y u_{\mathsf S,t}(\omega;\cdot,\lambda)<\infty
\]
holds (with measurable suprema, or with any deterministic uniform upper bound in their place), the sum is at most $\alpha\Omega_{\mathsf S}/(1-q)$. For a deterministic residual-only driver with a residual transformation onto its full support this reduces to $\alpha\operatorname{osc}(g_{\mathsf S})/(1-q)$. A fixed-bandwidth MMD driver requires the state-dependent bound, not the unit-reference oscillation.

\item \emph{Primitive modulus.} If the driver is scale-invariant and monotone with $0\le m_{\mathsf S}\le M_+$, the recursion satisfies~(ii) with $q=\max\{|\varphi|,|\varphi-\alpha M_+|\}$. This is below one when $|\varphi|<1$ and $\alpha M_+<1+\varphi$, and equals $\varphi$ when $0\le\alpha M_+\le2\varphi$. For the Gaussian log score no such $q$ exists.
\end{enumerate}
\end{proposition}
\proofref{app:proof_outlier}
Part~(iii) supplies the modulus that part~(ii) assumes. {Part~(ii) is the standard contraction argument behind stochastic-recurrence results \citep{straumann2006quasi,blasques2022maximum}; part~(iii) reads the modulus off the scale derivative of the driver, the scalar scale-model instance of the Lipschitz condition of \citet{donker2025stability} for explicit gradient filters.} {Section~\ref{sec:static-estimation} uses the same derivative bound for invertibility on parameter domains where it holds uniformly,} and Table~\ref{tab::impulse} calibrates part~(i). {For a redescending driver, the gross-outlier update $g(B)$ vanishes as $|B|\to\infty$. The path perturbation in~(i), relative to its baseline, instead tends to $-\alpha\varphi^k g(z_0)$ and vanishes only if $g(z_0)=0$. Tail-saturating drivers retain a nonzero limiting gross-outlier update; their difference from the baseline remains bounded.}
\endgroup

\paragraph{Implicit versus explicit updates}
Bounded proper-score derivatives are not the only way to control extreme one-step responses. In the implicit Gaussian-log pure-integrator update (cf. \citealp{lange2026implicit}), hold the current target fixed at $x$, write the current error as $e=\lambda-x$ and define the correction $d$ by
\[
 d=a\{\exp(-e-d)\varepsilon^2-1\},\qquad a>0.
\]
Equivalently $(d+a)\exp(d+a)=a\exp(a-e)\varepsilon^2$, so
\[
 d=W\!\left(a\exp(a-e)\varepsilon^2\right)-a,
\]
where $W$ is the nonnegative real Lambert function, satisfying $W(z)e^{W(z)}=z$ for $z\ge0$. For fixed $e$ the correction is at least $-a$ and grows logarithmically in $\varepsilon^2$ as $|\varepsilon|\to\infty$, rather than quadratically in $\varepsilon$. This one-step calculation does not prove an invariant-moment classification for the implicit recursion. It shows that the score and the choice of explicit versus implicit update are distinct robustness mechanisms.

{Propriety} supplies the divergence representation and centres the update at a correctly specified target; strict propriety, together with identification, makes that target unique.

\section{Stationary tracking moments and persistence}\label{trk:sec:remarks}

We first state the exact finite-mesh stationary-moment boundary, before introducing the local tracking benchmark. The key distinction is between weak convergence on a finite fast-time window and convergence of expectations under invariant laws. The subsequent local approximation accommodates bounded and nonmonotone drivers without requiring global contraction, but invariant-risk comparisons require separate Lyapunov and moment arguments.

At a fixed observation frequency, persistence and the growth of the update can determine whether the long-run tracking loss exists. Let the latent log variance be $x_n$ and observe $Y_{n+1}=e^{x_n/2}\varepsilon_{n+1}$. Consider an explicit scale filter with a scale-invariant driver $g$, as in Section~\ref{subsec:tail_impulse}, and match the filter and latent intercept and persistence:
\[
 x_{n+1}=\omega+\varphi x_n+s\eta_{n+1},\qquad
 \lambda_{n+1}=\omega+\varphi\lambda_n+a\,g\bigl(Y_{n+1}e^{-\lambda_n/2}\bigr).
\]
Thus $E_n=\lambda_n-x_n$ is an autonomous Markov chain,
\begin{equation}\label{pur:eq:phi}
 E_{n+1}=\varphi E_n+a\,g\bigl(e^{-E_n/2}\varepsilon_{n+1}\bigr)-s\eta_{n+1}.
\end{equation}
Its law concerns the error at a fixed mesh, even when the latent level itself is a random walk. The Gaussian log-variance filter has $g(z)=z^2-1$. Two indices govern the stationary moments: the moment index of the innovations, $\nu_\varepsilon=\sup\{\beta>0:\E|\varepsilon_1|^\beta<\infty\}$, which equals $\nu$ for Student-$t_\nu$ innovations, and the growth exponent $\kappa_g$ of an unbounded driver.

\begin{theorem}[Stationary moments of explicit scale filters]\label{pur:thm:persistence}
Fix $a>0$, $s\ge0$ and $0\le\varphi\le1$. Let the pairs $(\varepsilon_n,\eta_n)$ be i.i.d.\ and independent of the initial state, with $\varepsilon_n$ independent of $\eta_n$ and $P(\varepsilon_n=0)=0$, and let $g$ be continuous and bounded below with $g(0)<0$. If $s>0$, let $\eta_n$ have mean zero and a density that is positive on $\mathbb R$. If $s=0$, let $\varepsilon_n$ have a density that is positive on $\mathbb R$, and let $g$ be even and continuously differentiable with $g'(z)>0$ for $z>0$.
\begin{enumerate}[label=(\roman*)]
\item \emph{Unbounded drivers.} Suppose that, for some $\kappa_g>0$,
\[
 0<\liminf_{|z|\to\infty}\frac{g(z)}{|z|^{\kappa_g}}\le\limsup_{|z|\to\infty}\frac{g(z)}{|z|^{\kappa_g}}<\infty,
\]
that $0<\nu_\varepsilon<\infty$ with $\E|\varepsilon_1|^{\nu_\varepsilon}=\infty$, and, if $s>0$, that $\E e^{\theta s|\eta_1|}<\infty$ for some $\theta>\nu_\varepsilon/2$. Put $p=\nu_\varepsilon/\kappa_g$, and assume $p>1$ if $\varphi=1$. Then \eqref{pur:eq:phi} has a unique invariant probability $\pi_\varphi$, and for $r>0$,
\[
 \int |e|^r\,\pi_\varphi(\mathrm de)<\infty
 \quad\Longleftrightarrow\quad
 \begin{cases}
 r<p & (0\le\varphi<1),\\
 r<p-1 & (\varphi=1).
 \end{cases}
\]
If $p>1$, then $\pi_\varphi\Rightarrow\pi_1$ as $\varphi\uparrow1$ with $a,s$ fixed; absolute moments of order $r<p-1$ converge to those of $\pi_1$, and those of order $p-1\le r<p$ diverge to infinity.
\item \emph{Bounded drivers.} Suppose instead that $g$ is bounded, that $\liminf_{e\to-\infty}\E[g(e^{-e/2}\varepsilon_1)]>0$, and, if $s>0$, that $\E e^{\theta|\eta_1|}<\infty$ for some $\theta>0$. Then, for every $0\le\varphi\le1$, \eqref{pur:eq:phi} has a unique invariant probability $\pi_\varphi$, and there is $k>0$ with $\sup_{0\le\varphi\le1}\int e^{k|e|}\,\pi_\varphi(\mathrm de)<\infty$. Moreover, $\pi_\varphi\Rightarrow\pi_1$ as $\varphi\uparrow1$, and all moments converge.
\end{enumerate}
\end{theorem}

\begin{corollary}[Gaussian log-variance filter under Student-$t$ innovations]\label{pur:cor:glog}
Let $g(z)=z^2-1$, let $\varepsilon_n$ be variance-standardised Student-$t_\nu$ with $\nu>2$, and let $\eta_n$ be standard normal. Then Theorem~\ref{pur:thm:persistence}(i) applies with $\kappa_g=2$, $\nu_\varepsilon=\nu$ and $p=\nu/2$: the stationary error has finite absolute moments exactly of orders $r<\nu/2$ when $0\le\varphi<1$, and $r<\nu/2-1$ when $\varphi=1$. In particular, for $4<\nu\le6$ the stationary mean squared error is finite at every $\varphi<1$ and diverges as $\varphi\uparrow1$; for $\nu>6$ it converges to a finite limit.
\end{corollary}

\begin{corollary}[Loss-dependent stationary risk]
\label{cor:stationary-loss-dependent-risk}
In the setting of Corollary~\ref{pur:cor:glog}, consider the matched Gaussian-log error recursion
\[
 E_{n+1}=\varphi E_n+a\{e^{-E_n}\varepsilon_{n+1}^2-1\}-s\eta_{n+1},
\]
where $a>0$, $s\geq0$, $0\leq\varphi\leq1$, and the innovation pairs are
i.i.d. and independent of the initial state. Let $\varepsilon_n$ have the
variance-standardised Student-$t_\nu$ law $F_\nu$, with finite $\nu>2$,
and let $\eta_n$ be standard Gaussian and independent of $\varepsilon_n$.
Let $\pi_\varphi$ be the invariant error law. For every $b,c,r>0$,
\begin{equation}
 \mathbb E_{\pi_\varphi}e^{cE^+}=\infty,
 \qquad
 \mathbb E_{\pi_\varphi}|e^{bE}-1|^r=\infty.
 \label{eq:stationary-relative-level-risk}
\end{equation}
In particular, every positive-order moment of relative variance error
($b=1$) and relative volatility error ($b=1/2$) is infinite. Moreover,
\begin{equation}
 \mathbb E_{\pi_\varphi\otimes F_\nu}
 \bigl[\operatorname{CRPS}(N(0,e^E),\varepsilon)\bigr]=\infty.
 \label{eq:stationary-reference-crps-infinite}
\end{equation}
This is the reference-normalised risk: for
$Y_{n+1}=e^{x_n/2}\varepsilon_{n+1}$ and $E_n=\lambda_n-x_n$,
\[
 e^{-x_n/2}\operatorname{CRPS}(N(0,e^{\lambda_n}),Y_{n+1})
 =\operatorname{CRPS}(N(0,e^{E_n}),\varepsilon_{n+1}).
\]
These conclusions also hold in the regimes with finite log-variance MSE.
\end{corollary}

\begin{proof}
The lower bounds in the proof of Theorem~\ref{pur:thm:persistence} give
$\mathbb E_{\pi_\varphi}(E^+)^q=\infty$ for $q\geq\nu/2$ when
$\varphi<1$, and for $q\geq\nu/2-1$ when $\varphi=1$. A finite positive
exponential moment would imply finiteness of every positive polynomial
moment, proving the first assertion in
\eqref{eq:stationary-relative-level-risk}. It also follows that
$\mathbb E_{\pi_\varphi}e^{cE}=\infty$ for every $c>0$. On
$E\geq(\log2)/b$,
$|e^{bE}-1|^r\geq2^{-r}e^{brE}$, proving the second assertion.

For independent standard Gaussian variables $Z,Z'$, the CRPS identity gives
\[
 \operatorname{CRPS}(N(0,\sigma^2),y)
 =\mathbb E|\sigma Z-y|-\tfrac12\mathbb E|\sigma Z-\sigma Z'|
 =\mathbb E|\sigma Z-y|-\frac{\sigma}{\sqrt\pi}.
\]
The first term is minimised at the Gaussian median $y=0$. Hence, for every
$\sigma>0$ and $y\in\mathbb R$,
\begin{equation}
 \operatorname{CRPS}(N(0,\sigma^2),y)
 \geq c_{\mathrm{CRPS}}\sigma,
 \qquad c_{\mathrm{CRPS}}:=\frac{\sqrt2-1}{\sqrt\pi}>0.
 \label{eq:gaussian-crps-pointwise-bound}
\end{equation}
Taking $\sigma=e^{E/2}$ proves
\eqref{eq:stationary-reference-crps-infinite}. The normalisation follows
from the scale homogeneity of CRPS. Since $\nu>2$, the risk at any fixed
reference variance is finite, so subtracting its finite minimum also
leaves infinite stationary excess risk.
\end{proof}

\begin{remark}[Other drivers and the target shocks]\label{pur:rem:drivers}
The moment boundary depends on the driver only through its growth exponent. A Laplace working density has log-score driver proportional to $\sqrt2|z|-1$, so $\kappa_g=1$. Under Student-$t_\nu$ innovations and the positive-target-noise assumptions of part~(i), its stationary mean squared error at unit persistence is finite exactly when $\nu>3$, compared with $\nu>6$ for the Gaussian log score. When $s>0$ and the target shocks satisfy the assumptions of part~(ii), the bounded Student-$t$ log, CRPS and density-power drivers with the scalings of \ref{app::formulas} fall under that part whenever their mean field is globally restoring in the sense of Proposition~\ref{trk:prop:selftune}. When $s=0$, the theorem also requires the stated positive innovation density and the derivative condition $g'(z)>0$ for $z>0$; its present proof does not automatically cover the nonmonotone density-power drivers in that case. Part~(ii) supplies the fixed-interval exponential control strengthened by that proposition. The MMD score is excluded. With a fixed kernel bandwidth it is proper but not scale-invariant: its update depends on the observation and the working scale separately, so \eqref{pur:eq:phi} is not autonomous. A bandwidth proportional to the working scale restores scale invariance but destroys propriety: for a Gaussian working family and kernel, the score then decreases in the working variance for every non-zero observation, so its risk has no interior minimiser. The CRPS is itself an energy/kernel score generated by the conditionally negative definite distance $d(x,y)=|x-y|$, with no bandwidth. It is scale-equivariant, so the theory includes this kernel-score construction. The exponential-moment condition on the target shocks in part~(i) is close to necessary: a negative error of size $y$ is followed by an update of order $e^{\kappa_gy/2}$, so the positive tail of the error inherits the tail of $e^{\kappa_gs\eta_1/2}$. Only this right tail of $\eta_1$ needs exponential control; for the left tail, finite moments of order $\max\{2,p\}$ suffice, and the theorem states a two-sided condition for brevity.
\end{remark}

\paragraph{Proof mechanism} When the error is large and positive, the standardised return is small and the driver is close to $g(0)<0$. At unit persistence this gives an asymptotically constant negative mean drift. When the error is moderate, a large innovation produces an update contribution comparable to $|\varepsilon|^{\kappa_g}$. Under part~(i), this random variable has moment index $p$ and an infinite $p$th moment; no power-law tail asymptotic is assumed. At unit persistence, a jump of size $x$ can sustain an excursion for a number of steps proportional to $x$. Counting disjoint excursions gives $\pi_1((x,\infty))\ge c\,x\,P(|\varepsilon_1|^{\kappa_g}>Cx)$ (Proposition~\ref{rev:prop:stationary-tail}), which yields the loss of one moment. For $\varphi<1$, the autoregressive term removes a fixed fraction of a large error, while the one-step lower bound gives the moment obstruction at $p$. Matching upper moment bounds use $V(e)=1+(e_+)^q+e^{-ke}$ with $q<p$, $k>q\kappa_g/2$ and, at unit persistence, $q>1$. The exponential term controls negative errors, from which the next update can be of order $e^{-\kappa_ge/2}$. Bounded drivers cap the observation-driven increment, and $\cosh(ke)$ gives an exponential moment of some positive order. The proof is in \ref{pur:sec:proof}.

\paragraph{Implication for squared log-variance loss} For the Gaussian log-variance filter (Corollary~\ref{pur:cor:glog}), for $4<\nu\le6$ the stationary MSE is finite at every fixed $\varphi<1$ and becomes arbitrarily large near one. For $\nu>6$ it converges to a finite limit. Under the local initialisation and target assumptions of Lemma~\ref{trk:lem:glog}, the unit-persistence filter with $a=\rho\sqrt h$ and $s=\sigma\sqrt h$ still has a finite-variance Gaussian limit on a finite local window for every $\nu>4$. The two limiting operations answer different questions. If the latent coefficients are $(\omega_x,\varphi_x)$, subtraction adds $\omega-\omega_x+(\varphi-\varphi_x)x_n$ to \eqref{pur:eq:phi}; applying the theorem to an individual fitted series would require a model for its latent dynamics.

\begin{center}
\small
\captionsetup{hypcap=false}\captionof{table}{Finite local update variance and finite stationary mean squared error of the matched Gaussian log-variance filter under variance-standardised Student-$t_\nu$ innovations (Corollary~\ref{pur:cor:glog}).}\label{pur:tab:moments}
\begin{tabular}{lccc}
\toprule
Degrees of freedom & Local update & Stationary MSE & Stationary MSE\\
 & variance finite & $0\le\varphi<1$ & $\varphi=1$\\
\midrule
$2<\nu\le4$ & No & No & No\\
$4<\nu\le6$ & Yes & Yes & No\\
$\nu>6$ & Yes & Yes & Yes\\
\bottomrule
\end{tabular}
\end{center}
Finite update variance is required by the stated square-integrable local theorem; its other hypotheses, including initialisation and martingale conditions, remain necessary. The stationary claims refer to matched persistence and intercept, not to an arbitrary fitted filter.

\begin{remark}[Local moments and finite-step stationary risk]\label{rev:rem:heavy-risk}
When $\E\varepsilon^4=\infty$, the Gaussian-log update variance is infinite and Theorem~\ref{th::local-ou} does not apply (Proposition~\ref{prop::heavy-tail-diagnostic}). The zero-efficiency entry for the Gaussian log score in Table~\ref{trk:tab:eff} is a limiting convention, not a regular square-integrable calculation. Bounded drivers remove this particular observation-moment obstruction, subject to the remaining tracking assumptions.

Proposition~\ref{rev:prop:stationary-tail} also gives the unit-persistence lower bound for independent centred target shocks with finite variance, conditional on an invariant law. The exponential-moment condition on the target shocks in Theorem~\ref{pur:thm:persistence}(i) supplies the control of the negative error tail that the matching finite-moment results require (Remark~\ref{pur:rem:drivers}).
\end{remark}

\section{Local tracking benchmark}\label{sec::local_consistency}

\paragraph{Local benchmark}
We now define the moving scoring-risk projection (Subsection~\ref{subsec:local_target}) and derive the local Ornstein--Uhlenbeck approximation on the $h^{-1/2}$ step scale (Subsection~\ref{subsec:local_ou}). Its variance combines update noise, target motion and their covariance. The survival-event bound in Theorem~\ref{th::local-tracking} and the local weak limit are distinct from the invariant results. Section~\ref{trk:sec:stationarygain} provides the complementary positive stationary-risk result for bounded restoring drivers.

The local results below do not require the latent process to evolve independently of the filter. For a concrete feedback model, use a bounded Student-$t_6$ log driver with independent Gaussian target shocks, and add $hd\tanh(\lambda_{kh}^{(h)})$ to the increment of the latent log variance, where $d\ne0$. Future latent motion now depends on the current filter. The additional term has magnitude at most $h|d|$, so its accumulated contribution on $O(h^{-1/2})$ steps is $O(h^{1/4})$ after normalisation. The frozen restoring slope and bracket therefore give the same local limit under the stated initial conditions. The stationary gain theorem uses driftless latent motion.

\subsection{The moving pseudo-true target}\label{subsec:local_target}

For each mesh $h$, distinguish the information used by the recursion from the
information used to analyse the moving data-generating law. Let
$\mathcal F_{kh}^{\mathrm{obs},(h)}$ be generated by the known initial filter
state, any available observed presample information, and observations through
time $kh$. Let
$(\mathcal G_{kh}^{(h)})_{k\ge0}$ be a filtration containing
$\mathcal F_{kh}^{\mathrm{obs},(h)}$ at each time and possibly the moving-law
index. Assume the observation at $(k+1)h$ is
$\mathcal G_{(k+1)h}^{(h)}$-measurable, the filter state at $kh$ is
$\mathcal G_{kh}^{(h)}$-measurable, and the index
$\tilde\lambda_{kh}^{(h)}$ is
$\mathcal G_{kh}^{(h)}$-measurable, that
$x\mapsto\widetilde P_x^{(h)}$ is a probability kernel, and that, for every
measurable $A$,
\[
 \mathsf P\bigl(Y_{(k+1)h}^{(h)}\in A\mid\mathcal G_{kh}^{(h)}\bigr)
 =\widetilde P_{\tilde\lambda_{kh}^{(h)}}^{(h)}(A)
 \quad\text{almost surely}.
\]
Throughout this section and its scalar proofs, $\mathsf E_{kh}$,
$\mathsf{Var}_{kh}$, and $\mathsf{Cov}_{kh}$ mean conditional expectation,
variance, and covariance given $\mathcal G_{kh}^{(h)}$. A filter initialised from known
information with scaling jointly measurable with respect to
$\mathcal F_{kh}^{\mathrm{obs},(h)}\otimes\mathcal B(\Lambda)$ uses only observed
information; scaling available only through
$\mathcal G_{kh}^{(h)}$ defines an oracle recursion. The projection below is
an analysis target and need not be observed by the implemented filter.
If the moving-law index is hidden from observed information, this target can
differ from the scoring-risk projection of the law conditional on
$\mathcal F_{kh}^{\mathrm{obs},(h)}$; the results below track the
$\mathcal G_{kh}^{(h)}$-conditional target.

\enlargethispage{3pt}
Suppressing the static parameter $\theta$ from the notation, let $\widetilde P_{\tilde\lambda}^{(h)}$ and $P_\lambda^{(h)}$ denote the true and working local laws. When finite, define the local risk by
\(
\widetilde{\mathsf R}_{\mathsf S}^{(h)}(\lambda;\tilde\lambda)
=
\mathsf E^{\widetilde P_{\tilde\lambda}^{(h)}}
\left[
\mathsf S(P_{\lambda}^{(h)},Y)
\right]
\).

When the minimiser is unique, define the rule-specific pseudo-true parameter at $\tilde\lambda$ by
\begin{equation}\label{eq::minimiser_pseudo_true}
    \lambda^{\star}_{\mathsf{S},h}(\tilde\lambda) = \argmin_{\lambda \in \Lambda} \widetilde{\mathsf{R}}_{\mathsf{S}}^{(h)}(\lambda;\tilde\lambda)\,.
\end{equation}
Whenever this selected projection is used along a {row of the data-generating triangular array}, assume
$x\mapsto\lambda^\star_{\mathsf S,h}(x)$ is Borel measurable. Along the
latent state, define the pseudo-true path by
\begin{equation}\label{eq::target}
    m_{k h}^{(h)}=\lambda^{\star}_{\mathsf{S},h}(\tilde\lambda^{(h)}_{k h}).
\end{equation}
Since the moving-law index is $\mathcal G_{kh}^{(h)}$-measurable, so is
$m_{kh}^{(h)}$.
Distinct proper rules may project the same true law onto different points of one working family; the rule therefore {defines} the tracked target.

When the global argmin remains in a local neighbourhood and the risk is sufficiently smooth with positive curvature, Lemma~\ref{lem:pseudotrue_implicit} gives the target sensitivity
\(
 \partial_{\tilde\lambda}\lambda^\star_{\mathsf S,h}
 =-\frac{\partial_{\lambda\tilde\lambda}^2\widetilde{\mathsf R}_{\mathsf S}^{(h)}}
 {\partial_{\lambda\lambda}^2\widetilde{\mathsf R}_{\mathsf S}^{(h)}}
\)
evaluated at $\lambda=\lambda^\star_{\mathsf S,h}(\tilde\lambda)$. Its uniform-in-$h$ conditions provide one route to the smooth-target bounds in Proposition~\ref{prop::latent_space_recursion} when the relevant regions $D_h$ lie in the lemma's index neighbourhood; the bounds can also be verified directly.

The pseudo-true path need not equal the latent state. Define its local
conditional risk by
\begin{equation}\label{eq::risk}
\widetilde{\mathsf R}_{\mathsf S,kh}^{(h)}(\lambda)
=
\widetilde{\mathsf R}_{\mathsf S}^{(h)}
\bigl(\lambda;\tilde\lambda_{kh}^{(h)}\bigr)
=
\mathsf E^{\widetilde P_{\tilde\lambda_{kh}^{(h)}}^{(h)}}
\left[
\mathsf S(P_{\lambda}^{(h)},Y_{(k+1)h})
\right].
\end{equation}

Allow $\psi_{\mathsf S}$, $G_{\mathsf S}$, and $u_{\mathsf S}$ to depend on $h$. Derivatives are in working parameter $\lambda$, while expectations use $\widetilde P_{\tilde\lambda}^{(h)}$; this distinction yields the recentring result.

{Proposition~\ref{prop::proposition_recentering} in \ref{app:moved_section5_more} shows that the scaled update is conditionally centred at the pseudo-true target, $\mathsf E_{kh}[u^{(h)}_{\mathsf S}(y_{(k+1)h},m^{(h)}_{kh})]=0$.}
Thus misspecification may shift the target away from the latent state, while the conditional mean update given $\mathcal G_{kh}^{(h)}$ remains zero at the pseudo-true target. We therefore study the $\sqrt h$-gain recursion
\begin{equation}\label{eq::dynamics}
    \lambda_{(k+1)h}^{(h)} = \lambda_{k h}^{(h)} + \rho \sqrt{h} u_{\mathsf{S}}^{(h)}(y_{(k+1)h}, \lambda_{k h}^{(h)}),
\end{equation}
where $\rho>0$ is fixed. The displayed recursion takes $\alpha_h=\rho\sqrt h$.

\begingroup
{Write $m_{kh}^{(h)}=\lambda_{\mathsf S,h}^{\star}(\tilde\lambda_{kh}^{(h)})$ for the scoring-risk projection defined in~\eqref{eq::minimiser_pseudo_true}. The autoregressive comparison in \ref{subsec:applied_ar_recursion} concerns the zero of a composite mean field, which can differ from this projection.} Recursion~\eqref{eq::dynamics} tracks the scoring-risk projection $m_{kh}^{(h)}$, not the autoregressive composite centre. Predictable drift terms of order $h$, with locally bounded coefficients, do not change the leading local Ornstein--Uhlenbeck limit: their accumulated contribution vanishes after the $h^{-1/4}$ error rescaling over $O(h^{-1/2})$ steps. They can nevertheless change the next-order mean error and the bias--precision expansion; see Remark~\ref{rev:rem:target-drift}. If autoregressive pull enters at order $\sqrt h$, the relevant centre is instead the zero of the composite field. These statements concern the declared in-fill scaling and do not directly approximate the fixed-$\varphi$ autoregressive recursion.
\endgroup

Since $m_{kh}^{(h)}$ moves, the error $e_{kh}^{(h)}=\lambda_{kh}^{(h)}-m_{kh}^{(h)}$ obeys
\begin{equation}\label{eq::recursion_tracking_error}
    e_{(k+1)h}^{(h)} = e_{k h}^{(h)} + \rho \sqrt{h} u_{\mathsf{S}}^{(h)}(y_{(k+1)h}, \lambda_{k h}^{(h)}) - (m_{(k+1) h}^{(h)}-m_{k h}^{(h)}).
\end{equation}

This decomposition separates filter innovation from movement of the pseudo-true target. Proposition~\ref{prop::latent_space_recursion} gives a primitive local expansion for the latter.

\begin{proposition}[Pseudo-true target expansion]\label{prop::latent_space_recursion}
    Suppose that the latent state follows
    \begin{equation}\label{eq::latent_space_recursion}
        \tilde\lambda_{(k+1)h}^{(h)}=\tilde\lambda_{k h}^{(h)} + h b_h(\tilde\lambda_{k h}^{(h)}) + \sqrt{h} \sigma_h(\tilde\lambda_{k h}^{(h)}) \eta_{(k+1) h},
    \end{equation}
where $b_h(\tilde\lambda_{k h}^{(h)})$ and $\sigma_h(\tilde\lambda_{k h}^{(h)})$ are $\mathcal G_{kh}^{(h)}$-measurable and, conditional on the analysis information, the latent innovation satisfies
\begin{equation*}
    \mathsf{E}_{k h}[\eta_{(k+1)h}]=0,\qquad \mathsf{E}_{k h}[\eta^{4}_{(k+1)h}]\leq C_{\eta}<\infty.
\end{equation*}
Let $m_{k h}^{(h)}$ be as in \eqref{eq::target}. Suppose there exist $h_0>0$ and deterministic convex regions $D_h$ containing every line segment between
$\tilde\lambda_{kh}^{(h)}$ and $\tilde\lambda_{(k+1)h}^{(h)}$ such that
\[
 \sup_{0<h\le h_0}\sup_{x\in D_h}
 \left\{|b_h(x)|+|\sigma_h(x)|
 +|\partial_{\tilde\lambda}\lambda^\star_{\mathsf S,h}(x)|
 +|\partial^2_{\tilde\lambda\tilde\lambda}\lambda^\star_{\mathsf S,h}(x)|\right\}<\infty.
\]
Define the $\mathcal G_{kh}^{(h)}$-measurable local innovation loading
\begin{equation}\label{eq::bar_Gamma}
    \bar{\Gamma}_{k h}=\partial_{\tilde\lambda} \lambda^{\star}_{\mathsf{S},h}(\tilde\lambda_{k h}^{(h)}) \sigma_h(\tilde\lambda_{k h}^{(h)}),
\end{equation}
which combines the sensitivity of the pseudo-true target to the latent state with the local volatility of the latent state. Then
\begin{equation}\label{eq::increments_m}
    m_{(k+1) h}^{(h)}-m_{k h}^{(h)} = \sqrt{h}\,\bar{\Gamma}_{k h} \eta_{(k+1) h} + r_{(k+1) h}^{(h)},
\end{equation}
where, for a finite constant $C$ independent of $h$ and $k$,
$\mathsf{E}_{k h}[r^2_{(k+1)h}] \leq C h^2$.

\end{proposition}
\proofref{app:proof_latent_recursion}
{Write $\bar u_{kh}(\lambda)=\mathsf E_{kh}[u^{(h)}_{\mathsf S}(y_{(k+1)h},\lambda)]$ for the conditional mean field of the update.}

\begingroup
Under the target-motion, local mean-reversion and moment conditions of Assumption~\ref{ass:local-tracking}, Theorem~\ref{th::local-tracking} gives geometric forgetting and an $O(h^{1/2})$ bound for the squared-error contribution on the survival event at positive times. Its separate exit-contribution condition yields the unstopped $O_{L^2}(h^{1/4})$ rate. The next subsection studies weak limits on this spatial scale.
\endgroup

\subsection{Local distributional approximation and variance}\label{subsec:local_ou}
\begingroup
For $T\in(0,T_0)$, set $k_T=\lfloor T/h\rfloor$, $n_h(\tau)=\lfloor\tau/\sqrt h\rfloor$ and $Z_{\mathsf S,h}(\tau)=h^{-1/4}e^{(h)}_{(k_T+n_h(\tau))h}$. The $h^{-1/2}$ time scale reveals contraction $1-\rho b_{\mathsf S}(T)\sqrt h$. The local theorem assumes a weak limit for the initial normalised error $Z_{\mathsf S,h}(0)$; the survival bound alone neither identifies this law nor makes it the invariant law of the limiting Ornstein--Uhlenbeck process.
\endgroup

The rescaled error has composite innovation
\begin{equation}
\label{eq:composite_innovation}
     \tilde\xi_{(k+1)h}^{(h)} =
    \rho u_{\mathsf{S}}^{(h)}
    (y_{(k+1)h},m_{kh}^{(h)})
    -
    \bar{\Gamma}_{kh}\eta_{(k+1)h}
\end{equation}
with conditional variance $q^2_{\mathsf S,h,kh}=\mathsf E_{kh}[(\tilde\xi_{(k+1)h}^{(h)})^2]$. It combines update noise and target motion. The next assumption freezes local coefficients at $T$, controls off-target driver evaluation, and imposes martingale central-limit conditions.

{Assumption~\ref{ass::local-ou} in \ref{app:moved_section5_more} lists the local conditions on a window $[T,T+M\sqrt h]$: a tight initial condition, a frozen drift $b_{\mathsf S}(T)$, linearisation of the mean field, Lipschitz evaluation noise, and a martingale limit for the composite innovation with variance $q_{\mathsf S}(T)^2$.}

Items~(i) and (v) give the joint FCLT in Appendix
Corollary~\ref{cor:primitive_joint_fclt}; items~(ii)--(iv) linearise the mean
field and control off-target noise for the stopped Ornstein--Uhlenbeck limit.

\begin{theorem}[Stopped local Ornstein--Uhlenbeck limit]
\label{th::local-ou}
Suppose Assumption~\ref{ass:local-tracking}(i) and
Assumption~\ref{ass::local-ou} hold. Let $Z_{\mathsf S}$ be the unique solution of
\[
    dZ_{\mathsf S}(\tau)
    =
    -\rho b_{\mathsf S}(T)Z_{\mathsf S}(\tau)\,d\tau
    +
    q_{\mathsf S}(T)\,dW_{\tau},
    \qquad
    Z_{\mathsf S}(0)=Z_0,
\]
where $W$ is a standard Brownian motion independent of $Z_0$.
\begin{samepage}
For $x\in D([0,M])$ and $K>0$, define
\[
 \sigma_K(x)=\inf\{\tau\in[0,M]:|x(\tau)|\ge K\}\wedge M,
 \qquad
 \Phi_K(x)=x(\,\cdot\wedge\sigma_K(x)).
\]
\end{samepage}
Set $Z_{\mathsf S,h}^K=\Phi_K(Z_{\mathsf S,h})$ and
$Z_{\mathsf S}^K=\Phi_K(Z_{\mathsf S})$. If $K$ is a continuity radius, meaning that
$\Phi_K$ is continuous at $Z_{\mathsf S}$ almost surely, then
\(
    Z_{\mathsf S,h}^K\Rightarrow Z_{\mathsf S}^K
\),
weakly in $D([0,M])$. If $Z_0=z$ is deterministic, the unstopped frozen OU marginal is
\begin{equation}
    Z_{\mathsf S}(\tau)
    \overset{\rm d}{\sim}
    \mathsf{Normal}\left(
    e^{-\rho b_{\mathsf S}(T)\tau}z,\,
    \frac{q_{\mathsf S}^{2}(T)}
    {2\rho b_{\mathsf S}(T)}
    (1-e^{-2\rho b_{\mathsf S}(T)\tau})
    \right).
\end{equation}

For random $Z_0$ the conditional transition given $Z_0$ has the displayed Gaussian variance. Its unconditional law is the convolution of $e^{-\rho b_{\mathsf S}(T)\tau}Z_0$ with the independent Gaussian innovation. In the matched Gaussian-log example, $\Var(\varepsilon^2-1)=2(\nu-1)/(\nu-4)$ for $\nu>4$.

The invariant variance of the frozen limiting OU is
\begin{equation}
    V_{\mathsf S}^{\star}(T)
    =
    \frac{q_{\mathsf S}^{2}(T)}
    {2\rho b_{\mathsf S}(T)} .
    \label{eq::stationary_local_variance}
\end{equation}
\end{theorem}
\proofref{app:proof_local_ou}

The approximation holds on each fixed tube. Comparing the filter with a global linear process also controls its probability of leaving successively larger tubes on the same finite time window.

\begin{corollary}[Removal of localisation]\label{cor:unstopped_local_ou}
Under the conditions of Theorem~\ref{th::local-ou}, on each prescribed finite window $[0,M]$,
\[
 \lim_{K\to\infty}\limsup_{h\downarrow0}
  \mathsf P\!\left(\sup_{\tau\le M}|Z_{\mathsf S,h}(\tau)|\ge K\right)=0,\qquad
  Z_{\mathsf S,h}\Rightarrow Z_{\mathsf S}\ \text{in }D([0,M]).
\]
A conclusion for second moments also requires control of the magnitude of rare errors.
\end{corollary}

\proofref{app:proof_unstopped_local_ou}

\begingroup
Theorems~\ref{th::local-tracking} and~\ref{th::local-ou} use the same restoring-field and composite-noise mechanism as the tracking results of \citet{beutner2026consistency}. The reference projection here is induced by the scoring rule, and the local assumptions are verified directly at that projection. The use of stopped martingale arguments permits a non-Markov formulation, but does not make the resulting variance formula a new generic diffusion or filtering identity. For misspecified scale models, the useful comparison combines the rule-specific target shift with its local precision, rather than comparing driver variances at different targets without accounting for their units.
\endgroup

\begingroup
Equation~\eqref{eq::stationary_local_variance} gives the invariant variance of the frozen OU limit. Theorem~\ref{pur:thm:persistence} identifies when the original Gaussian-log error can have infinite stationary MSE. For bounded restoring drivers, Proposition~\ref{trk:prop:selftune} supplies sufficient conditions for convergence of scaled invariant MSE.
\endgroup
Suppose
\begin{equation}
    \bar u_{kh}(\lambda)
    =
    -
    G_{\mathsf S}^{(h)}(\lambda)
    \partial_{\lambda}
    \widetilde{\mathsf R}_{\mathsf S,kh}^{(h)}(\lambda).
    \label{eq::mean_field_scoring_risk}
\end{equation}
Differentiating~\eqref{eq::mean_field_scoring_risk} at $m_{kh}^{(h)}$, the derivative of the scaling term drops out because the first-order condition $\partial_{\lambda}\widetilde{\mathsf R}_{\mathsf S,kh}^{(h)}(m_{kh}^{(h)})=0$ holds at the interior pseudo-true point. Hence
\begin{equation}
    b_{\mathsf S,h,kh}
    =
    -\partial_{\lambda}\bar u_{kh}(m_{kh}^{(h)})
    =
    G_{\mathsf S}^{(h)}(m_{kh}^{(h)})
    \partial_{\lambda\lambda}^{2}
    \widetilde{\mathsf R}_{\mathsf S,kh}^{(h)}
    (m_{kh}^{(h)}).
    \label{eq::local_curvature_interpretation}
\end{equation}
Thus $b_{\mathsf S}(T)$ is scaled risk curvature: larger values correct deviations faster. The coefficient $q_{\mathsf S}^2(T)$ is the variance of the composite innovation, not of the update alone.

At the pseudo-true point, both components are centred, so
\begin{align}
    q_{\mathsf S,h,kh}^{2}
    &=
    \rho^2
    \mathsf{Var}_{kh}
    \left[
        u_{\mathsf S}^{(h)}
        (y_{(k+1)h},m_{kh}^{(h)})
    \right]
    \nonumber\\
    &\quad
    +
    \bar \Gamma_{kh}^{2}
    \mathsf{Var}_{kh}
    \left[
        \eta_{(k+1)h}
    \right]
    -
    2\rho\bar\Gamma_{kh}
    \mathsf{Cov}_{kh}
    \left[
        u_{\mathsf S}^{(h)}
        (y_{(k+1)h},m_{kh}^{(h)}),
        \eta_{(k+1)h}
    \right].
    \label{eq::q_variance_decomposition}
\end{align}

The terms are update noise, target motion, and their conditional covariance. Hence the rule affects the target path and, conditional on it, curvature, driver variance, and update--target covariance. Together with~\eqref{eq::stationary_local_variance}, this decomposition summarises the stability--noise trade-off: curvature stabilises, while noisy or poorly aligned updates raise tracking variability.

{The deterministic-target special case, in which the frozen benchmark depends on the rule only through $\kappa_{\mathsf S}=K_{\mathsf S}/J_{\mathsf S}^2$, is treated in \ref{app:moved_section5_more}; Section~\ref{trk:sec:frozen} gives the general frontier.}
{When the innovation fourth moment is infinite, the Gaussian-log update has infinite variance and the square-integrable local OU theorem does not apply (Proposition~\ref{prop::heavy-tail-diagnostic} in \ref{subsec:local_moment_boundary}). Other normalisations require a separate argument.}

\subsection{\texorpdfstring{{Gain--noise frontier and tracking efficiency}}{Gain--noise frontier and tracking efficiency}}\label{trk:sec:frozen}
\begingroup
The frozen variance and its gain optimisation follow the minimum-variance analysis of \citet{beutner2026consistency}. Proposition~\ref{trk:prop:eff} expresses the information inequality in the units of a rule-specific projection; Proposition~\ref{trk:prop:matrix} gives the corresponding matrix calculation. These are comparisons within the declared driver class. When the driver has zero covariance with the target shock, identifying the resulting information benchmark with an unrestricted filtering optimum requires a stronger condition excluding predictable target-shock information in the observation; see Remark~\ref{rev:rem:leverage}.
\endgroup

\begingroup
Fix $T\in(0,T_0)$ and $M<\infty$, let $k_T=\lfloor T/h\rfloor$, and let $\mathcal W_h=\{k_T,\dots,k_T+\lfloor M/\sqrt h\rfloor\}$ be the local window of Assumption~\ref{ass::local-ou}. The filter is \eqref{eq::dynamics}, the target is $m^{(h)}_{kh}=\lambda^\star_{\mathsf S,h}(\tilde\lambda^{(h)}_{kh})$, the tracking error is $e^{(h)}_{kh}=\lambda^{(h)}_{kh}-m^{(h)}_{kh}$, and target motion satisfies \eqref{eq::increments_m}. Write $C_{\mathsf S,h,kh}=\Cov_{kh}[\psi_{\mathsf S}(y^{(h)}_{(k+1)h},m^{(h)}_{kh}),\eta_{(k+1)h}]$ and define the characteristics of the \emph{scaled} update at the target:
\[
b_{\mathsf S,h,kh}=G_{\mathsf S}(m^{(h)}_{kh})J_{\mathsf S,h,kh},\qquad
\Xi_{\mathsf S,h,kh}=G_{\mathsf S}(m^{(h)}_{kh})^2K_{\mathsf S,h,kh},\qquad
\chi_{\mathsf S,h,kh}=G_{\mathsf S}(m^{(h)}_{kh})C_{\mathsf S,h,kh}.
\]
Thus $b_{\mathsf S,h,kh}$ is the frozen drift of \eqref{eq::local_curvature_interpretation}, $\Xi_{\mathsf S,h,kh}=\Var_{kh}[u_{\mathsf S}(y^{(h)}_{(k+1)h},m^{(h)}_{kh})]$, and $\chi_{\mathsf S,h,kh}$ is the conditional covariance of the same update with $\eta_{(k+1)h}$, because $G_{\mathsf S}(m^{(h)}_{kh})$ is $\mathcal G^{(h)}_{kh}$-measurable. Let $s^2_{\eta,h,kh}=\Var_{kh}(\eta_{(k+1)h})$.

\begin{assumption}[Frozen characteristics]\label{trk:ass:frozen}
There are constants $b_T,\Xi_T\in(0,\infty)$ and $\bar\Gamma_T,\chi_T\in\R$ such that, uniformly over $k\in\mathcal W_h$ and in probability,
\[
b_{\mathsf S,h,kh}\to b_T,\quad \Xi_{\mathsf S,h,kh}\to\Xi_T,\quad \bar\Gamma_{kh}\to\bar\Gamma_T,\quad s^2_{\eta,h,kh}\to1,\quad \chi_{\mathsf S,h,kh}\to\chi_T .
\]
\end{assumption}

\begingroup
Only scaled quantities need to converge. In the scale model of Section~\ref{trk:sec:scale}, the raw sensitivity, variance and cross-covariance carry the factors $A(\log h+x)$, $A(\log h+x)^2$ and $A(\log h+x)$, respectively; these cancel from the scaled characteristics and from $K/J^2$ and $C/J$. The unit target-shock variance is a normalisation. {Assumptions~\ref{ass::local-ou} and~\ref{trk:ass:frozen} use finite deterministic limiting coefficients.} Random coefficients depending on the latent state at $T$ require a conditional or stable convergence theorem with joint convergence relative to the freezing information; that extension is not claimed here. The constant-loading scale examples satisfy the deterministic formulation.
\endgroup

{Lemma~\ref{trk:lem:ident} in \ref{app:moved_section5_more} shows that these limits identify the frozen coefficients: $b_{\mathsf S}(T)=b_T$ and $q_{\mathsf S}(T)^2=\rho^2\Xi_T+\bar\Gamma_T^2-2\rho\bar\Gamma_T\chi_T$.}

Define the \emph{effective gain} and the two \emph{rule constants}
\[
{\varrho}=\rho\,b_T,\qquad \kappa_{\mathsf S}(T)=\frac{\Xi_T}{b_T^2}=\lim\frac{K_{\mathsf S,h,kh}}{J_{\mathsf S,h,kh}^2},\qquad \gamma_{\mathsf S}(T)=\frac{\chi_T}{b_T}=\lim\frac{C_{\mathsf S,h,kh}}{J_{\mathsf S,h,kh}}.
\]
Both rule constants are invariant to positive rescalings of the score derivative and to the scaling. $\kappa_{\mathsf S}(T)$ is the constant of the deterministic-target case in Section~\ref{subsec:local_ou}. $\gamma_{\mathsf S}(T)$ measures the alignment of the update with the target shock. Call $V^\star_{\mathsf S}(T)$ in \eqref{eq::stationary_local_variance} the \emph{frozen tracking risk}. It is the invariant variance of the frozen limit in Theorem~\ref{th::local-ou}, in units of the scaled error $h^{-1/4}e^{(h)}_{kh}$.

\begin{proposition}[Gain--noise frontier for a given rule]\label{trk:prop:frontier}
Under Assumption~\ref{trk:ass:frozen}, define the frozen OU risk algebraically by \eqref{eq::stationary_local_variance} using the limiting drift and innovation variance. Its interpretation as the limit of a local error process requires the hypotheses of Theorem~\ref{th::local-ou}.
\begin{enumerate}[label=(\roman*)]
\item The frozen tracking risk is
\begin{equation}\label{trk:eq:frontier}
V^\star_{\mathsf S}(T)=\frac{{\varrho}^{\,2}\kappa_{\mathsf S}(T)-2{\varrho}\,\bar\Gamma_T\gamma_{\mathsf S}(T)+\bar\Gamma_T^2}{2{\varrho}}.
\end{equation}
For a given target loading $\bar\Gamma_T$, a change of positive driver scaling affects this formula only through $\varrho$, and the remaining rule-specific quantities are $\kappa_{\mathsf S}(T)$ and $\gamma_{\mathsf S}(T)$. Comparing rules that induce different projections also requires their respective loadings $\bar\Gamma_T=D_T\sigma_T$, when the differentiability conditions below hold.
\item $\gamma_{\mathsf S}(T)^2\le\kappa_{\mathsf S}(T)$, so the numerator of \eqref{trk:eq:frontier} is nonnegative for every ${\varrho}>0$.
\item If $\bar\Gamma_T\neq0$, the map ${\varrho}\mapsto V^\star_{\mathsf S}(T)$ is strictly convex on $(0,\infty)$, with unique minimiser ${\varrho}^{\,\star}_{\mathsf S}(T)=|\bar\Gamma_T|\,\kappa_{\mathsf S}(T)^{-1/2}$ and minimum
\begin{equation}\label{trk:eq:vmin}
V^\star_{\mathsf S,\min}(T)=|\bar\Gamma_T|\,\kappa_{\mathsf S}(T)^{1/2}-\bar\Gamma_T\,\gamma_{\mathsf S}(T)\ \ge0 .
\end{equation}
\item If $\bar\Gamma_T=0$, then $V^\star_{\mathsf S}(T)={\varrho}\kappa_{\mathsf S}(T)/2$, which has no minimiser on $(0,\infty)$.
\end{enumerate}
\end{proposition}
\proofref{app:proof_trk_prop_frontier}

The cost of choosing a different effective gain has a simple expression. For $\bar\Gamma_T\ne0$ and $\kappa_{\mathsf S}>0$, let $\varrho^\star=|\bar\Gamma_T|/\sqrt{\kappa_{\mathsf S}}$. Subtraction in the variance formula gives
\[
 V^\star(\varrho)-V^\star(\varrho^\star)
   =\frac{\kappa_{\mathsf S}}{2\varrho}(\varrho-\varrho^\star)^2 .
\]
The additive driver--shock covariance correction cancels. Under zero covariance, taking half or twice the optimal gain multiplies the frozen variance by $1.25$.

\begingroup
Positive scalar rescalings of the driver can be absorbed into the gain, so they produce the same frozen frontier after gain optimisation. The optimal gain also depends on target-motion volatility. Its determination therefore requires the dynamic tracking experiment; a first-order one-step improvement statement alone does not determine it.
\endgroup

Remark~\ref{trk:rem:observer} in \ref{app:moved_section5_more} shows that the frontier is the steady state of a scalar Kalman--Bucy problem in which the rule acts as an observation device with noise $\kappa_{\mathsf S}$; the scaling only fixes units.

The frontier \eqref{trk:eq:vmin} compares rules only up to the units of their targets. The reason is that $\bar\Gamma_T$ contains the target sensitivity $\partial_{\tilde\lambda}\lambda^\star_{\mathsf S,h}$, which depends on the rule (Proposition~\ref{prop::latent_space_recursion}). Measuring errors in units of the latent index removes this dependence.

{Assumption~\ref{trk:ass:dtl} in \ref{app:moved_section5_more} makes the true law differentiable in the latent index, with true score $\tilde s_{h,x}$, Fisher information $\tilde I_h\in(0,\infty)$ and target sensitivity $D_{\mathsf S,h}=\partial_x\lambda^\star_{\mathsf S,h}\neq0$. It also requires these quantities and $\sigma_h$ to converge on the window to $\tilde I_T$, $D_T$ and $\sigma_T$, so that $\bar\Gamma_T=D_T\sigma_T$.}

\begin{proposition}[Target sensitivity and tracking efficiency]\label{trk:prop:eff}
Suppose Assumptions~\ref{trk:ass:frozen} and~\ref{trk:ass:dtl} hold.
\begin{enumerate}[label=(\roman*)]
\item For every small $h$ and $x\in U$,
\begin{equation}\label{trk:eq:sens}
J_{\mathsf S,h}(x)\,D_{\mathsf S,h}(x)=\Cov_x\bigl[\psi_{\mathsf S}\bigl(Y,\lambda^\star_{\mathsf S,h}(x)\bigr),\,\tilde s_{h,x}(Y)\bigr],
\end{equation}
where $J_{\mathsf S,h}(x)=\partial^2_{aa}\widetilde{\mathsf R}^{(h)}_{\mathsf S}(a;x)$ at $a=\lambda^\star_{\mathsf S,h}(x)$. Multiplying by the scaling, the frozen covariance $\Upsilon_T$ between $u_{\mathsf S}(\cdot,m^{(h)}_{kh})$ and the true score exists and equals $b_TD_T\neq0$.
\item Let $\sigma_T>0$. The minimal frozen tracking risk in latent units, $V^{\rm lat}_{\mathsf S,\min}(T)=V^\star_{\mathsf S,\min}(T)/D_T^2$, equals
\begin{equation}\label{trk:eq:lat}
V^{\rm lat}_{\mathsf S,\min}(T)=\sigma_T\,\frac{\Xi_T^{1/2}-\operatorname{sgn}(\Upsilon_T)\,\chi_T}{|\Upsilon_T|}.
\end{equation}
Define $r_{\mathsf S}(T)=\Upsilon_T/(\Xi_T\tilde I_T)^{1/2}$ and $\operatorname{eff}_{\mathsf S}(T)=r_{\mathsf S}(T)^2$. If $\chi_T=0$,
\begin{equation}\label{trk:eq:effmin}
V^{\rm lat}_{\mathsf S,\min}(T)=\sigma_T\,\tilde I_T^{-1/2}\,\operatorname{eff}_{\mathsf S}(T)^{-1/2},
\qquad {\varrho}^{\,\star}_{\mathsf S}(T)=\sigma_T|D_T|\,\kappa_{\mathsf S}(T)^{-1/2}.
\end{equation}
\item $\operatorname{eff}_{\mathsf S}(T)\in(0,1]$. With $\chi_T=0$, $V^{\rm lat}_{\mathsf S,\min}(T)\ge\sigma_T\tilde I_T^{-1/2}$, with equality if and only if $\operatorname{eff}_{\mathsf S}(T)=1$. At each $h$, $u_{\mathsf S}(\cdot,m^{(h)}_{kh})$ and $\tilde s$ have correlation of modulus one if and only if they are proportional almost surely. In the limit, equality means that the conditional mean-square distance between the two centred, unit-variance variables, after multiplying one by $\operatorname{sgn}(\Upsilon_T)$, tends to zero in probability. This follows from $2(1-|r_h|)$, where $r_h$ is their conditional correlation at mesh $h$. Under correct specification at $T$ ($D\equiv1$), $\operatorname{eff}_{\mathsf S}(T)=1/(\tilde I_T\kappa_{\mathsf S}(T))$ and $\operatorname{eff}_{\log}(T)=1$.
\end{enumerate}
\end{proposition}
\proofref{app:proof_trk_prop_eff}
\begingroup
The inequality in part~(iii) is a bound within the class of drivers satisfying $\chi_T=0$. That condition on one driver does not imply that the observation contains no information about the target shock. A sufficient stronger condition for the zero-alignment comparison to cover all observation-based drivers is that the target innovation have zero conditional mean given the contemporaneous observation and the frozen state. Central symmetry alone does not imply this condition; Remark~\ref{rev:rem:leverage} gives a counterexample. Efficiency measures precision about the rule's own target. Accuracy for the untransformed latent state also depends on the target shift, as in Proposition~\ref{trk:prop:biasprec}.
\endgroup

Lemma~\ref{trk:lem:lat} shows that this is the invariant variance of the frozen limit of the latent-state error $h^{-1/4}\{L_h(\lambda^{(h)}_{kh})-\tilde\lambda^{(h)}_{kh}\}$, where $L_h$ is a measurable extension of the inverse target map.

\ref{app:moved_bll} relates Proposition~\ref{trk:prop:eff} to the optimal filter of \citet{beutner2026consistency}.
\endgroup

\subsection{\texorpdfstring{{Scale models and the matrix frontier}}{Scale models and the matrix frontier}}\label{trk:sec:scale}

\begingroup
Let $\Phi$ be a centred symmetric law of unit variance and let $P_\lambda$ be the law of $e^{\lambda/2}Z$, $Z\sim\Phi$, so that the working family at mesh $h$ is $P^{(h)}_\lambda=P_{\lambda+\log h}$. For $a\in\R$ let $\mathsf T_a(y)=e^{a/2}y$; then $\mathsf T_a\#P_\lambda=P_{\lambda+a}$.

Definition~\ref{trk:def:equi} and Assumption~\ref{trk:ass:scale} in \ref{app:moved_section5_more} formalise scale models. The working family is a scale family in the log variance, and the rule is scale-equivariant and reflection-symmetric, as are the logarithmic, Student-$t$ log, CRPS and density-power scores. The innovation law $F$ is fixed.

\begin{corollary}[Scale models]\label{trk:cor:scale}
Suppose Assumption~\ref{trk:ass:scale} holds.
\begin{enumerate}[label=(\roman*)]
\item $\lambda^\star_{\mathsf S,h}(x)=x+\delta_{\mathsf S}(F)$ for all $h$ and $x$. Hence $D_{\mathsf S,h}\equiv1$, and the target loading $\bar\Gamma_{kh}=\sigma_h(\tilde\lambda^{(h)}_{kh})$ is the same for every rule of Definition~\ref{trk:def:equi}. The unit conversion for a displacement is one. Writing $e=\lambda-(x+\delta_{\mathsf S}(F))$, the error relative to the latent log variance is $\lambda-x=e+\delta_{\mathsf S}(F)$.
\item $\kappa_{\mathsf S}$ and $\operatorname{eff}_{\mathsf S}$ depend only on $(S,\Phi,F)$. Explicitly, $\kappa_{\mathsf S}=\kappa_{\mathsf S}(F)=K_F/J_F^2$, where $K_F=\E_F[\psi_{\mathsf S}(\varepsilon,\delta_{\mathsf S}(F))^2]$ and $\psi_{\mathsf S}(\varepsilon,\delta)=-\partial_\delta \mathsf S(P_\delta,\varepsilon)$. Provided identity \eqref{trk:eq:sens} holds with $D=1$ (Lemma~\ref{trk:lem:effid}),
\begin{equation}\label{trk:eq:are}
\operatorname{eff}_{\mathsf S}=\operatorname{eff}_{\mathsf S}(F)=\frac{1}{I_F\,\kappa_{\mathsf S}(F)} .
\end{equation}
\begingroup
This is the asymptotic variance efficiency of the shift-corrected estimator $\widehat\lambda_{\mathsf S}-\delta_{\mathsf S}(F)$ relative to the true-$F$ maximum-likelihood estimator of $x$, for i.i.d.\ observations $e^{x/2}\varepsilon_i$. Equivalently, both estimators can be centred at the common pseudo-target $x+\delta_{\mathsf S}(F)$ (Lemma~\ref{trk:lem:mest}). If $F$ is unknown, this shift correction is generally oracle; estimating it requires accounting for its own uncertainty and covariance.
\endgroup
\item {If the conditional law of $(\varepsilon_{(k+1)h},\eta_{(k+1)h})$ is invariant under $(\varepsilon,\eta)\mapsto(-\varepsilon,-\eta)$, then $\chi_T=0$ for the reflection-symmetric drivers considered here, whatever the correlation between $\varepsilon$ and $\eta$. This is a property of these even drivers, not a statement that leverage has no value for unrestricted filtering.}
\item If $\chi_T=0$ and $\sigma>0$, the constant effective gain ${\varrho}^{\,\star}_{\mathsf S}=\sigma\kappa_{\mathsf S}(F)^{-1/2}$ attains the frontier at every $T$, and
\begin{equation}\label{trk:eq:vscale}
V^\star_{\mathsf S,\min}=\sigma\,\kappa_{\mathsf S}(F)^{1/2}=\sigma\,I_F^{-1/2}\,\operatorname{eff}_{\mathsf S}(F)^{-1/2}.
\end{equation}
For two such rules, $V^\star_{\mathsf S,\min}/V^\star_{S',\min}=\{\operatorname{eff}_{S'}(F)/\operatorname{eff}_{\mathsf S}(F)\}^{1/2}$, free of $\sigma$, $T$ and $h$. Under the scalings of \ref{app::formulas}, $u_{\mathsf S}$ depends on $(y,\lambda)$ only through $z=y/(he^\lambda)^{1/2}$. Then $b_T$ is also free of $x$ and $h$, so ${\varrho}^{\,\star}_{\mathsf S}$ corresponds to a single constant $\rho$.
\end{enumerate}
\end{corollary}
\proofref{app:proof_trk_cor_scale}

For the Gaussian log score, $\kappa_{\log}(F)=\mathrm{kurt}_F-1$ and $\operatorname{eff}_{\log}(F)=1/\{I_F(\mathrm{kurt}_F-1)\}$. For variance-standardised $t_\nu$ innovations, $I_F=\nu/\{2(\nu+3)\}$ and
\[
\operatorname{eff}_{\log}(t_\nu)=\frac{(\nu-4)(\nu+3)}{\nu(\nu-1)},\qquad\nu>4,
\]
which decreases to zero as $\nu\downarrow4$.

\section{Stationary reference-scale risk and gain selection}\label{trk:sec:stationarygain}

Theorem~\ref{pur:thm:persistence}(i) shows why local weak convergence is insufficient for stationary unbounded loss. At a fixed sampling interval, part~(ii) guarantees all polynomial error moments and an exponential moment of some positive order. A particular score risk requires control at its growth exponent, as Remark~\ref{rem:fixed-mesh-exponential-risk} illustrates. Proposition~\ref{trk:prop:selftune} supplies this stronger control for sufficiently small mesh. The latent state is a driftless random walk, the score driver is bounded and globally restoring, and only the gain is optimised. The invariant law is that of the tracking error, not of the latent level.

\begin{remark}[Exponential orders and fixed-mesh score risk]
\label{rem:fixed-mesh-exponential-risk}
At unit persistence, retain the i.i.d. innovation-pair setup and suppose
that $g\geq g_{\min}<0$. Let the target shocks be standard Gaussian,
independent of the observation innovations and the past.
If $M(t):=\mathbb E_\pi e^{tE}$ were finite for some $t>0$,
stationarity would imply
\[
 M(t)
 =e^{t^2s^2/2}\mathbb E_\pi\!\left[
    e^{tE}\mathbb E_\varepsilon e^{ta g(e^{-E/2}\varepsilon)}
   \right]
 \geq e^{ta g_{\min}+t^2s^2/2}M(t).
\]
Consequently, when $s>0$,
\begin{equation}
 M(t)=\infty
 \quad\text{whenever}\quad
 t>\frac{2a|g_{\min}|}{s^2}.
 \label{eq:fixed-mesh-exponential-obstruction}
\end{equation}
This is a necessary restriction on finite exponential orders, without
an assertion about the exact boundary.

For example, the Gaussian CRPS driver is
\[
 g(z)=\frac1{\sqrt\pi}-\sqrt{\frac2\pi}\,e^{-z^2/2},
 \qquad g_{\min}=-c_{\mathrm{CRPS}}.
\]
With standard Gaussian observation innovations, its mean field is
\[
 \bar g(e)=\frac1{\sqrt\pi}
          -\sqrt{\frac2\pi}\,(1+e^{-e})^{-1/2}.
\]
It is strictly decreasing, with unique zero at $e=0$ and nonzero inward
limits. Thus Theorem~\ref{pur:thm:persistence}(ii) gives a unique invariant law
with all polynomial error moments finite. Nevertheless, for $a=0.1$ and
$s=0.5$,
\[
 \frac{2a|g_{\min}|}{s^2}\simeq0.186956<\frac12.
\]
Equations \eqref{eq:fixed-mesh-exponential-obstruction} and
\eqref{eq:gaussian-crps-pointwise-bound} therefore give infinite
stationary reference CRPS risk. Under $a=\rho\sqrt h$ and
$s=\sigma\sqrt h$, the necessary upper bound becomes
$2\rho|g_{\min}|/(\sigma^2\sqrt h)$ and tends to infinity. This removes
the displayed obstruction as $h\to0$; Proposition~\ref{trk:prop:selftune} supplies the sufficient gain-uniform small-mesh exponential integrability for convergence
of stationary score risks.
\end{remark}

\begin{proposition}[Population gain selection for stationary reference-scale risk]\label{trk:prop:selftune}
Let Assumption~\ref{trk:ass:prim}(a) hold with $b_h\equiv0$. Let $(\varepsilon_k,\eta_k)$ be i.i.d.\ pairs independent of the past, with $\varepsilon_k\sim F$ and standard Gaussian $\eta_k$ independent of $\varepsilon_k$. Suppose $g\in\mathfrak B$ has a globally restoring mean field $\bar g(e)=\mathsf E_F[g(we^{-e/2}\varepsilon)]$, where $w=e^{-\delta_{\mathsf S}(F)/2}$: $\bar g$ has its unique zero at $0$ and $\liminf_{|e|\to\infty}|\bar g(e)|>0$. Monotonicity in $|z|$ together with $g(0)<0<g(\pm\infty)$ verifies this condition for the monotone bounded drivers. \ref{pur:sec:dp} verifies the global restoring condition for both density-power drivers and the Gaussian and Student reference laws in Table~\ref{trk:tab:eff}. Write $b_{\mathsf S}=-\bar g'(0)>0$, $\Xi_{\mathsf S}=\mathsf E_F[g(w\varepsilon)^2]$ and $\kappa_{\mathsf S}=\Xi_{\mathsf S}/b_{\mathsf S}^2$. The class $\mathfrak B$ requires $g\in C^2$ and bounded $g$, $zg'(z)$ and $z^2g''(z)$ (Definition~\ref{trk:def:Bc}). Consider
\begin{equation}\label{trk:eq:chain}
 e_{k+1}=e_k+\rho\sqrt h\,g\bigl(we^{-e_k/2}\varepsilon_{k+1}\bigr)-\sqrt h\,\sigma\eta_{k+1}.
\end{equation}
This is the error of \eqref{eq::dynamics} relative to its rule-specific target under a driftless random-walk latent state. Let
\[
 \Delta(e)=R_F(\delta_{\mathsf S}+e)-R_F(\delta_{\mathsf S})
\]
be the excess reference-frame risk, where $R_F(\delta)=\E_F[\mathsf S(P_\delta,\varepsilon)]$, suppose $R_F\in C^3$ near $\delta_{\mathsf S}$ and $\Delta(e)\le C_\Delta e^{a_0|e|}$, and let $P\subset(0,\infty)$ be compact.
\begin{enumerate}[label=(\roman*)]
\item There exist $h_0,a,C>0$ such that, for $h\le h_0$ and $\rho\in P$, \eqref{trk:eq:chain} has a unique invariant law $\pi_{h,\rho}$ with $\E_{\pi_{h,\rho}}\cosh(ah^{-1/4}e)\le C$.
\item Uniformly over $\rho\in P$,
\[
 h^{-1/2}\E_{\pi_{h,\rho}}[e^2]\longrightarrow V^\star(\rho b_{\mathsf S}),\qquad
 h^{-1/2}\E_{\pi_{h,\rho}}[\Delta(e)]\longrightarrow\tfrac12J_F V^\star(\rho b_{\mathsf S}),
\]
where $V^\star(\varrho)=(\varrho^2\kappa_{\mathsf S}+\sigma^2)/(2\varrho)$.
\item If $\rho^\star=\sigma b_{\mathsf S}^{-1}\kappa_{\mathsf S}^{-1/2}$ lies in the interior of $P$, every minimiser over $P$ of $\E_{\pi_{h,\rho}}[\Delta(e)]$, and of $\E_{\pi_{h,\rho}}[e^2]$, converges to $\rho^\star$. Its effective gain is $\rho^\star b_{\mathsf S}=\varrho^\star_{\mathsf S}$.
\end{enumerate}
\end{proposition}
\proofref{app:proof_trk_prop_selftune}

\paragraph{Why the stationary result follows} The proof uses the Lyapunov function $\cosh(a h^{-1/4}e)$ with a drift of order $\sqrt h$ outside a fixed neighbourhood in the rescaled error. This yields a uniform exponential moment bound over the gain set. Starting each chain in its invariant law, the local weak limit must itself be stationary; its unique OU invariant distribution identifies the limiting law. Uniform integrability then permits passage to second moments and reference-scale loss. Compactness of the gain set and uniform risk convergence give convergence of population minimisers. It is this passage to stationary risk, rather than the scalar OU minimisation, that is the substantive conclusion.

For scale-equivariant scores, conditional excess risk at the actual latent level is $A(\log h+\tilde\lambda_{kh}^{(h)})\Delta(e)$, with $A\equiv1$ for logarithmic scores, $A(a)=e^{a/2}$ for CRPS, and $A(a)=e^{-\beta a/2}$ for density power. The joint law of the level weight and the tracking error can depend on the gain even though the weight itself has no explicit gain parameter; it cannot generally be removed from an expectation. Moreover, the random-walk latent level has no invariant probability. The proposition therefore does not establish joint estimation of intercept, persistence and gain, nor minimisation of the unnormalised level-weighted risk. See \ref{app:moved_selftune}.

The following expansion distinguishes the precision of a rule-specific target from accuracy for the latent log variance.

\begingroup
\begin{proposition}[Bias--precision decomposition]\label{trk:prop:biasprec}
Suppose the conditions of Proposition~\ref{trk:prop:selftune} hold, and put $c_{\mathsf S}=\bar g_{\mathsf S}''(0)/b_{\mathsf S}$. Let $\pi_{h,\rho}$ denote the invariant law of the tracking error in its driftless bounded-driver model. Uniformly over gains in compact subsets of $(0,\infty)$,
\[
 \mathsf E_{\pi_{h,\rho}}\bigl[(\lambda_{kh}-\tilde\lambda_{kh})^2\bigr]
 =\delta_{\mathsf S}(F)^2+h^{1/2}\{1+c_{\mathsf S}\delta_{\mathsf S}(F)\}
 V^\star(\rho b_{\mathsf S})+o(h^{1/2}).
\]
When $1+c_{\mathsf S}\delta_{\mathsf S}(F)>0$, the tracking-optimal gain minimises the $h^{1/2}$ coefficient, with $V^\star=\sigma\kappa_{\mathsf S}(F)^{1/2}$. The nonpositive cases and an admissible example are discussed below.
\begin{enumerate}[label=(\roman*)]
\item Leading accuracy for the latent state is governed by $\delta_{\mathsf S}(F)^2$. If two rules have equal leading squared bias, the next-order comparison uses the full coefficient $(1+c_{\mathsf S}\delta_{\mathsf S})\sqrt{\kappa_{\mathsf S}}$. Efficiency alone ranks this coefficient when both shifts are zero; equal nonzero shifts do not remove its dependence on $c_{\mathsf S}$.
\item For two rules with positive coefficients as above, evaluated at their respective tracking-optimal gains, the difference in stationary mean squared errors is
\begin{align*}
 \mathsf{MSE}_{\mathsf S'}-\mathsf{MSE}_{\mathsf S}
 &=\delta_{\mathsf S'}^2-\delta_{\mathsf S}^2\\
 &\quad+h^{1/2}\sigma\bigl\{(1+c_{\mathsf S'}\delta_{\mathsf S'})\kappa_{\mathsf S'}^{1/2}
 -(1+c_{\mathsf S}\delta_{\mathsf S})\kappa_{\mathsf S}^{1/2}\bigr\}+o(h^{1/2}).
\end{align*}
A nonzero leading difference determines the ranking for sufficiently small $h$. The expansion does not by itself determine a finite-step crossover when omitted terms are material.
\end{enumerate}
\end{proposition}
\endgroup
\proofref{app:proof_trk_biasprec}

\paragraph{Nonpositive bias--precision coefficients}
The sign restriction in Proposition~\ref{trk:prop:biasprec} is substantive. Put $A_{\mathsf S}=1+c_{\mathsf S}\delta_{\mathsf S}(F)$ and fix a compact gain interval $P\subset(0,\infty)$. If $A_{\mathsf S}<0$, the uniform expansion implies
\[
 \operatorname{dist}\!\left(\rho_h^{\rm lat},
 \argmax_{\rho\in P}V^\star(\rho b_{\mathsf S})\right)\longrightarrow0
\]
for every sequence $\rho_h^{\rm lat}$ of raw-latent-MSE minimisers over $P$. Strict convexity puts this maximizing set at an endpoint, with a possible tie between endpoints. If $A_{\mathsf S}=0$, this expansion does not select a gain at order $\sqrt h$.

For example, use the Gaussian CRPS driver with the variance-standardised law
\[
 F=0.98N(0,0.01)+0.02N(0,49.51),
\]
where the normal parameters denote variances, and the driftless independent Gaussian target shocks of Proposition~\ref{trk:prop:selftune}. The smooth mixture and the bounded monotone CRPS driver satisfy that proposition's assumptions. For component weights $p_i$ and variances $v_i$, the candidate-scale mean driver is
\[
 M_F(a)=\pi^{-1/2}-\sqrt{2/\pi}\sum_i p_i(1+v_i e^{-a})^{-1/2}.
\]
Its unique zero and derivatives give $\delta_{\mathsf S}\simeq-4.5244$, $c_{\mathsf S}=M_F''(\delta_{\mathsf S})/\{-M_F'(\delta_{\mathsf S})\}\simeq0.27959$, and $A_{\mathsf S}\simeq-0.26498$. This is a population stress-test example; $P$ remains fixed as $h\downarrow0$. Proposition~\ref{trk:prop:selftune} continues to select the centered-tracking and reference-risk optimum.

\begingroup
Proposition~\ref{trk:prop:matrix} gives the vector frozen-covariance frontier under zero driver--target cross-covariance, optimising over all positive-stable gain matrices. The minimum is the geometric mean of the sandwich covariance and target-motion covariance. This attainment statement need not hold under a restricted scalar-gain or preconditioner design, and its interpretation as an unrestricted filtering optimum needs the no-predictable-target-shock condition discussed above.
\endgroup
\endgroup

\section{Numerical illustration}\label{sec::experiments}

The numerical illustrations distinguish precision around a scoring-rule projection from accuracy for the latent state and from stationary accuracy. Table~\ref{trk:tab:eff} reports deterministic population quantities for the six scale-equivariant cells. Two controlled finite-horizon experiments then examine recovery after a large upward error and the gain comparison in the bounded-driver reference model. The score formulas, deterministic driver curves, contamination targets and impulse calculations are collected in \appref{app::formulas}. The stationary conclusions follow from Theorem~\ref{pur:thm:persistence} and Propositions~\ref{trk:prop:selftune} and~\ref{trk:prop:biasprec}, under their respective assumptions.

\subsection{\texorpdfstring{{Tracking efficiency}}{Tracking efficiency}}\label{trk:sec:numerics}
\begingroup
{The zero-covariance frozen frontier of Section~\ref{trk:sec:frozen} turns the curvature--variability constants of Table~\ref{tab::catalogue} into local tracking efficiencies. The information benchmark used in this subsection assumes zero covariance between the update and target shocks; asymmetric-leverage comparisons require the separate covariance correction.} For the six scale-equivariant cells of the factorial, under Gaussian and variance-standardised Student-$t$ innovations, Table~\ref{trk:tab:eff} reports three quantities:
\begin{itemize}
\item the efficiency $\operatorname{eff}_{\mathsf S}(F)$;
\item the minimal frozen tracking risk relative to the efficient filter, $\operatorname{eff}_{\mathsf S}(F)^{-1/2}$;
\item the target shift $\delta_{\mathsf S}(F)$.
\end{itemize}
\begingroup
The MMD cells are not scale-equivariant and are excluded. Finite moments and target roots are computed by adaptive quadrature from the drivers in \appref{app::formulas}. In the finite-variance cells, $1/(I_F\kappa_{\mathsf S}(F))$ agrees with the squared-correlation expression of Proposition~\ref{trk:prop:eff}. The second moment of the Gaussian-log update is infinite at $t_4$ and $t_3$; the corresponding zero-efficiency and infinite-multiplier entries are conventions marking failure of this square-integrable frontier, not regular correlation calculations. The Gaussian-reference values reproduce Table~\ref{tab::catalogue}, and the correctly specified Student-$t_6$ log score has efficiency one. Monotonicity identifies the log and CRPS targets. \ref{pur:sec:dp} proves uniqueness of the density-power targets for the tabulated innovation laws.
\endgroup

\begin{table}[ht]

\centering\small
\caption{Tracking efficiency of scale-equivariant proper-score filters.}\label{trk:tab:eff}
\begin{tabular}{lcccccc}
\toprule
Rule & Gaussian & $t_{10}$ & $t_6$ & $t_5$ & $t_4$ & $t_3$ \\
\midrule
\multicolumn{7}{l}{\emph{Panel A. Efficiency $\operatorname{eff}_{\mathsf S}(F)$}} \\
Gaussian log & 1.000 & 0.867 & 0.600 & 0.400 & 0.000 & 0.000 \\
Student-$t_6$ log & 0.903 & 0.986 & 1.000 & 0.997 & 0.983 & 0.940 \\
Density-power & 0.950 & 0.993 & 0.979 & 0.965 & 0.938 & 0.885 \\
CRPS & 0.808 & 0.912 & 0.954 & 0.969 & 0.984 & 0.993 \\
Student-$t_6$ CRPS & 0.803 & 0.913 & 0.958 & 0.974 & 0.989 & 0.998 \\
Student-$t_6$ density-power & 0.815 & 0.917 & 0.956 & 0.968 & 0.980 & 0.985 \\
\midrule\multicolumn{7}{l}{\emph{Panel B. Minimal tracking risk relative to the efficient filter, $\operatorname{eff}_{\mathsf S}(F)^{-1/2}$}} \\
Gaussian log & 1.000 & 1.074 & 1.291 & 1.581 & $\infty$ & $\infty$ \\
Student-$t_6$ log & 1.052 & 1.007 & 1.000 & 1.001 & 1.009 & 1.031 \\
Density-power & 1.026 & 1.004 & 1.011 & 1.018 & 1.032 & 1.063 \\
CRPS & 1.112 & 1.047 & 1.024 & 1.016 & 1.008 & 1.003 \\
Student-$t_6$ CRPS & 1.116 & 1.047 & 1.022 & 1.013 & 1.005 & 1.001 \\
Student-$t_6$ density-power & 1.107 & 1.044 & 1.023 & 1.016 & 1.010 & 1.008 \\
\midrule\multicolumn{7}{l}{\emph{Panel C. Target shift $\delta_{\mathsf S}(F)$ (rule-specific minus latent log-variance)}} \\
Gaussian log & 0.000 & 0.000 & 0.000 & 0.000 & 0.000 & 0.000 \\
Student-$t_6$ log & +0.139 & +0.071 & 0.000 & -0.047 & -0.140 & -0.390 \\
Density-power & 0.000 & -0.056 & -0.123 & -0.170 & -0.266 & -0.528 \\
CRPS & 0.000 & -0.095 & -0.189 & -0.249 & -0.363 & -0.652 \\
Student-$t_6$ CRPS & +0.186 & +0.092 & 0.000 & -0.059 & -0.171 & -0.457 \\
Student-$t_6$ density-power & +0.187 & +0.094 & 0.000 & -0.061 & -0.176 & -0.469 \\
\midrule
$I_F$ & 0.5000 & 0.3846 & 0.3333 & 0.3125 & 0.2857 & 0.2500 \\
\bottomrule
\end{tabular}
\par\medskip
\parbox{0.95\linewidth}{\footnotesize {Innovations are Gaussian or variance-standardised Student-$t_\nu$; $\beta=0.15$ and Student working shape $\nu=6$. Panel B compares gain-optimised local OU variances in latent units under zero update--target shock covariance. It does not assert finite stationary MSE for the original recursion. Zero efficiency and $\infty$ mark Gaussian-log cells with infinite update variance, outside Theorem~\ref{th::local-ou}. Target shifts in Panel C must be retained when interpreting latent-state accuracy.}}
\end{table}

\begingroup
Table~\ref{trk:tab:eff} illustrates the distinction between precision around a rule-specific target and accuracy for the latent state. Under Gaussian innovations, the own-target local OU-risk multipliers are approximately $1.026$ for density power and $1.112$ for CRPS. At $\nu=4$, the Gaussian-log update variance becomes infinite and the square-integrable frontier ceases to apply. Under $t_3$, Student-$t_6$ CRPS has efficiency $0.998$, but its target shift is $-0.457$, corresponding to about $63\%$ of the latent variance.

Table~\ref{trk:tab:eff} compares local population quantities. Its finite Gaussian-log multiplier under $t_5$ is compatible with the infinite stationary MSE at $\varphi=1$ in Theorem~\ref{pur:thm:persistence}. For $0\le\varphi<1$ the matched model has finite MSE at $t_5$, which diverges as persistence approaches one. The local efficiency of the Gaussian-log driver in this case is $0.4$; the limiting OU risk is finite but it does not describe the stationary second moment at unit persistence.
\endgroup

\endgroup

\paragraph{Recovery after a large upward error}
To isolate the excursion in Theorem~\ref{pur:thm:persistence}, initialise~\eqref{pur:eq:phi} at $E_0=H$ and record $\inf\{k\ge1:E_k\le H/2\}$. Use the Gaussian-log driver $g(z)=z^2-1$, set $a=0.10$ and $s=0$, and take variance-one Student-$t_5$ innovations. For $\varphi=0.90,0.99,0.999,1$, respectively, the median recovery times were $(5,11,13,13)$ at $H=2$, $(6,26,38,41)$ at $H=8$, and $(7,48,130,161)$ at $H=32$. Each cell used 512 independent paths with common innovations across designs. All paths crossed within 4,096 steps. At $\varphi=1$ the pointwise inequality $E_k\ge H-ak$ already forces recovery time to grow at least linearly in $H$. This experiment starts at a specified high state and measures its duration; the invariant frequency of reaching that state enters separately in the proof.

\paragraph{The gain at a finite mesh}
For Proposition~\ref{trk:prop:selftune} use the Gaussian CRPS driver, variance-one Student-$t_5$ innovations, $\sigma=0.5$, independent standard Gaussian target shocks and $e_0=0$. Population integration gives $\delta_{\mathsf S}=-0.249230$, $b_{\mathsf S}=0.129757$, $\Xi_{\mathsf S}=0.0556134$ and $\rho^\star=\sigma/\sqrt{\Xi_{\mathsf S}}=2.120217$. At each mesh we generated 4,096 independent paths and compared six gains on common innovations. Each path contains $\lceil18/(\rho_{\min}b_{\mathsf S}\sqrt h)\rceil$ steps, where $\rho_{\min}=0.5\rho^\star$. At $\rho^\star$, mean $e^2/\sqrt h$ (Monte Carlo SE) was $0.949\ (0.021)$, $0.928\ (0.020)$ and $0.897\ (0.020)$ at $h=2^{-6},2^{-10},2^{-14}$, respectively. The OU value is $0.909$.

\begin{center}
\small\captionsetup{hypcap=false,labelfont=it}
\captionof{table}{Finite-horizon gain comparison at $h=2^{-14}$.}\label{et:tab:gainnew}
\begin{tabular}{rcccc}
\toprule
$\rho/\rho^\star$ & Mean $e^2/\sqrt h$ & OU variance & Paired MSE gap (SE) & Paired risk gap (SE)\\
\midrule
0.50 & 1.115 & 1.136 & $0.2180\ (0.0146)$ & $0.00623\ (0.00042)$\\
0.75 & 0.929 & 0.947 & $0.0329\ (0.0057)$ & $0.00094\ (0.00016)$\\
1.00 & 0.897 & 0.909 & $0$ & $0$\\
1.25 & 0.924 & 0.931 & $0.0277\ (0.0044)$ & $0.00080\ (0.00013)$\\
1.50 & 0.982 & 0.984 & $0.0853\ (0.0082)$ & $0.00246\ (0.00024)$\\
2.00 & 1.141 & 1.136 & $0.2446\ (0.0149)$ & $0.00704\ (0.00043)$\\
\bottomrule
\end{tabular}
\end{center}
Both gaps use the same path at $\rho^\star$ as comparator. The risk column measures $[\Delta(e_\rho)-\Delta(e_{\rho^\star})]/\sqrt h$ in the reference frame. Standard errors use independent path differences. The smallest sample mean on this six-point grid occurs at $\rho^\star$ for both criteria at the finest mesh. The estimates describe a specified finite horizon; Propositions~\ref{trk:prop:selftune} and~\ref{trk:prop:biasprec} give the stationary and latent-state results. Code, seeds, quadrature checks and unrounded results at all meshes accompany the source.

\FloatBarrier

\section{{Conclusion}}\label{sec::conclusion}
\begingroup

The stationary tracking error supplies the clearest distinction among the filters studied here. Under variance-standardised Student-$t_\nu$ innovations and independent Gaussian target shocks, a Gaussian-log filter with matched latent and filter persistence has a finite stationary second moment exactly when $\nu>4$ below unit persistence and when $\nu>6$ at unit persistence. More generally, within the class of Theorem~\ref{pur:thm:persistence}, the boundary is $\nu_\varepsilon/\kappa_g$ below unit persistence and $\nu_\varepsilon/\kappa_g-1$ at unit persistence, with $\nu_\varepsilon/\kappa_g>1$ in the latter case. Bounded restoring updates satisfying part~(ii) have an exponential moment of some positive order. In the intervening range, the stationary laws converge weakly as persistence increases to one, even as their second moments diverge. A Gaussian approximation on a finite local window cannot resolve this long-run question.

For bounded drivers with a globally restoring mean, invariant moment control leads to a positive decision result. In the driftless random-walk model, the gain that minimises stationary excess score at a fixed reference scale converges to the local variance-minimising gain. Scale misspecification introduces a further distinction between precision and accuracy. The Student-$t_6$ CRPS filter under $t_3$ innovations, for example, attains nearly efficient local precision around a target at about $63\%$ of latent variance.

The numerical illustrations separate population target shifts, local precision and finite-horizon tracking behaviour. The controlled recovery experiment measures the duration of an imposed excursion; the invariant frequency of such excursions is a separate part of the stationary proof. The bounded-driver experiment compares reference-scale risks at a specified terminal horizon. Neither calculation verifies the assumptions of the static-parameter estimation theorem. \ref{sec::empirical} reports fixed-fit index forecasts, event diagnostics and a Nile location illustration. Its relative forecast-loss comparisons are descriptive, and the quantile-count tests have the scope stated there. Further theoretical questions include joint estimation of persistence and gain, and stationary risk for level-weighted CRPS and density-power losses under a moving latent scale.

\endgroup

\newpage
\appendix
\section{Proofs of the stationary results}\label{etapp:proofs}
\subsection{A stationary-moment obstruction at unit persistence}\label{etapp:obstruction}
\begin{proposition}[Excursion lower bound]\label{rev:prop:stationary-tail}
Let $\varphi=1$ in \eqref{pur:eq:phi}, and retain the i.i.d. innovation-pair construction, with each pair independent of the past. Let $g$ be bounded below with $g_{\min}=\inf g<0$, let $\eta_n$ have mean zero and finite variance, and let $\pi$ be an invariant probability. In the specialisation to Theorem~\ref{pur:thm:persistence}(i) below, retain the independence of $\varepsilon_n$ and $\eta_n$. Suppose there are a compact set $\mathcal C$ with $\pi(\mathcal C)>0$, a constant $x_0>0$ and a function $J$ with $\inf_{e\in\mathcal C}P(E_1\ge2x\mid E_0=e)\ge J(x)$ for $x\ge x_0$. Then $\pi((x,\infty))\ge c\,x\,J(x)$ for some $c>0$ and all sufficiently large $x$. Under the conditions of Theorem~\ref{pur:thm:persistence}(i), $J(x)\ge c'P(|\varepsilon_1|^{\kappa_g}\ge Cx)$, and consequently $\int(e_+)^r\,\pi(\mathrm de)=\infty$ for every $r\ge p-1$.
\end{proposition}
\begin{proof}
Put $n_x=\lfloor x/(4a|g_{\min}|)\rfloor$. Each step satisfies $E_{n+1}\ge E_n-a|g_{\min}|-s\eta_{n+1}$. Hence, starting from $E_k\ge2x$, the error stays above $x$ for the next $n_x$ dates whenever $\max_{\ell\le n_x}|s\sum_{j=1}^{\ell}\eta_{k+j}|\le x/2$, because $2x-x/4-x/2>x$. By Kolmogorov's maximal inequality, the complementary event has probability at most $4s^2n_x\operatorname{Var}(\eta_1)/x^2=O(x^{-1})$, which is below $1/2$ for large $x$, uniformly over starting values above $2x$.

On a two-sided stationary realisation, for $j=1,\dots,n_x$ let $A_j$ be the event that $E_{-j}\in\mathcal C$, $E_{-j+1}\ge2x$, and $E_{-j+1},\dots,E_0$ all exceed $x$. By the Markov property, $P(A_j)\ge\tfrac12\pi(\mathcal C)J(x)$. For $x>\sup_{e\in\mathcal C}|e|$ the events are disjoint, because an excursion that remains above $x$ excludes a later visit to $\mathcal C$. Therefore
\[
\pi((x,\infty))=P(E_0>x)\ge\sum_{j=1}^{n_x}P(A_j)\ge\tfrac12n_x\pi(\mathcal C)J(x)\ge c\,x\,J(x).
\]
Under part~(i), choose $M$ with $\mathcal C\subset[-M,M]$, and $c_->0$ and $z_0$ with $g(z)\ge c_-|z|^{\kappa_g}$ for $|z|\ge z_0$. For $e\in\mathcal C$ and large $x$, the event $\{ac_-e^{-\kappa_gM/2}|\varepsilon_1|^{\kappa_g}\ge2x+M\}\cap\{\eta_1\le0\}$ implies $|e^{-e/2}\varepsilon_1|\ge z_0$ and $E_1\ge2x$. Independence gives $J(x)\ge P(\eta_1\le0)\,P(|\varepsilon_1|^{\kappa_g}\ge Cx)$, with the first factor replaced by one when $s=0$. Finally,
\[
\int(e_+)^r\,\pi(\mathrm de)=\int_0^\infty rx^{r-1}\pi((x,\infty))\,\mathrm dx\ge c\int_{x_1}^\infty x^{r}P(|\varepsilon_1|^{\kappa_g}\ge Cx)\,\mathrm dx,
\]
which is infinite when $\E|\varepsilon_1|^{\kappa_g(r+1)}=\infty$, that is, when $r\ge p-1$.
\end{proof}

\subsection{Proof of Theorem~\ref{pur:thm:persistence}}\label{pur:sec:proof}
\begin{proof}
Write $P_\varphi$ for the transition kernel of \eqref{pur:eq:phi} and $D_e=a\,g(e^{-e/2}\varepsilon_1)-s\eta_1$, so that $E_1=\varphi e+D_e$ when $E_0=e$. Constants $C,c>0$ do not depend on $\varphi\in[0,1]$ and may change from line to line.

\emph{Step 1: Feller property and irreducibility.} For bounded continuous $f$, the map $(\varphi,e)\mapsto P_\varphi f(e)=\E f(\varphi e+D_e)$ is continuous, by continuity of $g$ and dominated convergence. If $s>0$, $P_\varphi(e,\cdot)$ is the law of $\varphi e+a\,g(e^{-e/2}\varepsilon_1)$ convolved with that of $-s\eta_1$, so it has a density that is positive on $\mathbb R$. If $s=0$, the change of variables $z\mapsto g(z)$ on $(0,\infty)$, valid because $g'>0$ there, gives $P_\varphi(e,\cdot)$ a density that is positive on $(\varphi e+ag(0),\varphi e+a\sup g)$. By induction, the $n$-step kernel has a positive density on an interval whose endpoints converge to $ag(0)/(1-\varphi)$ and $a\sup g/(1-\varphi)$ when $\varphi<1$, and to $\mp\infty$ when $\varphi=1$. In each case the chain is $\psi$-irreducible. A $\psi$-irreducible chain with an invariant probability is recurrent \citep[Proposition~10.1.1]{meyn2009markov}, and a recurrent chain has a unique invariant measure up to scaling \citep[Theorem~10.0.1]{meyn2009markov}, so there is at most one invariant probability.

\emph{Step 2: drift under part~(i).} By continuity, the lower bound and the growth condition, there are $c_-,C,z_0$ with $|g(z)|\le C(1+|z|^{\kappa_g})$ for all $z$ and $g(z)\ge c_-|z|^{\kappa_g}$ for $|z|\ge z_0$. Fix $0<q<p$ and $k$ with $q\kappa_g/2<k\le\theta$ (any $k>q\kappa_g/2$ if $s=0$), and set $V(e)=1+(e_+)^q+e^{-ke}$.

(a) Negative errors. For $e\le0$, $(\varphi e+D_e)_+\le(D_e)_+$, so
\[
\E(\varphi e+D_e)_+^q\le C\bigl(1+e^{-q\kappa_ge/2}\E|\varepsilon_1|^{q\kappa_g}+\E|s\eta_1|^q\bigr)\le C\bigl(1+e^{-q\kappa_ge/2}\bigr),
\]
because $q\kappa_g<\nu_\varepsilon$. Moreover,
\[
\frac{\E e^{-k(\varphi e+D_e)}}{e^{-ke}}=e^{k(1-\varphi)e}\,\E e^{-ka\,g(e^{-e/2}\varepsilon_1)}\,\E e^{ks\eta_1}.
\]
The first factor is at most one for $e\le0$. As $e\to-\infty$, $|e^{-e/2}\varepsilon_1|\to\infty$ almost surely because $P(\varepsilon_1=0)=0$, so $g(e^{-e/2}\varepsilon_1)\to\infty$, and the middle factor tends to zero by dominated convergence with bound $e^{-ka\inf g}$. Since $k>q\kappa_g/2$, $P_\varphi V(e)/V(e)\to0$ as $e\to-\infty$, uniformly in $\varphi$.

(b) Positive errors, $\varphi<1$. As $e\to\infty$, $D_e\to ag(0)-s\eta_1$ almost surely, because $e^{-e/2}\varepsilon_1\to0$ and $g$ is continuous at zero. For $e\ge1$ the envelope $(1+|D_e|)^q\le C(1+|\varepsilon_1|^{q\kappa_g}+|\eta_1|^q)$ is integrable, so $\E(\varphi e+D_e)_+^q/e^q=\E(\varphi+D_e/e)_+^q\to\varphi^q$. On $e\ge0$ the exponential term is bounded by $e^{-ka\inf g}\E e^{ks\eta_1}$. Hence $P_\varphi V\le l_\varphi V+C_\varphi$ for some $l_\varphi<1$.

(c) Positive errors, uniformly in $\varphi$. Now take $1<q<p$, which is possible when $p>1$. For $e\ge1$, $(\varphi e+D_e)_+\le(e+D_e)_+$, and the integral form of the mean-value theorem gives
\[
\frac{\E(e+D_e)_+^q-e^q}{e^{q-1}}=q\,\E\Bigl[D_e\int_0^1(1+tD_e/e)_+^{q-1}\,\mathrm dt\Bigr]\longrightarrow q\,\E[ag(0)-s\eta_1]=q\,a\,g(0)<0,
\]
by dominated convergence with envelope $C(1+|D_e|^q)$. Together with (a) and the boundedness of $P_\varphi V$ on compact sets,
\begin{equation}\label{pur:eq:drift}
P_\varphi V-V\le C-cW,\qquad W(e)=(e_+)^{q-1}+e^{-ke},\qquad 0\le\varphi\le1.
\end{equation}

\emph{Step 3: existence, uniqueness and upper bounds.} By \eqref{pur:eq:drift}, or by (b) when $\varphi<1$, the averaged kernels $n^{-1}\sum_{m<n}P_\varphi^m(e,\cdot)$ are tight: iterating \eqref{pur:eq:drift} gives $n^{-1}\sum_{m<n}P_\varphi^mW(e)\le\{C+V(e)/n\}/c$, iterating (b) gives $\sup_nP_\varphi^nV(e)\le V(e)+C_\varphi/(1-l_\varphi)$, and $W$ and $V$ have compact sublevel sets. The chain is therefore bounded in probability on average, and, being weak Feller, it has an invariant probability \citep[Theorem~12.0.1(i)]{meyn2009markov}; Step~1 gives uniqueness. For the invariant $\pi_\varphi$, integrate the drift of $V\wedge L$: it is nonpositive on $\{V>L\}$ and at most $P_\varphi V-V$ on $\{V\le L\}$. Invariance gives $0\le\int_{\{V\le L\}}(C-cW)\,\mathrm d\pi_\varphi$, and monotone convergence yields $\sup_{0\le\varphi\le1}\pi_\varphi(W)\le C/c$. Taking $q=1+r$ proves finiteness at $\varphi=1$ for $0<r<p-1$. The same truncation applied to (b) gives $\pi_\varphi(V)<\infty$, hence finiteness for $r<p$, when $\varphi<1$. All positive-order moments of the negative part $e^-:=\max\{-e,0\}$ are finite because of the term $e^{-ke}$.

\emph{Step 4: lower bounds.} For $\varphi<1$, choose $M$ with $\pi_\varphi([-M,M])>0$. As in the proof of Proposition~\ref{rev:prop:stationary-tail}, $P(E_1>x\mid E_0=e)\ge c\,P(|\varepsilon_1|^{\kappa_g}\ge Cx)$ for $e\in[-M,M]$ and large $x$, so invariance gives $\pi_\varphi((x,\infty))\ge c\,P(|\varepsilon_1|^{\kappa_g}\ge Cx)$. Integrating against $rx^{r-1}$ shows that $\int(e_+)^r\,\mathrm d\pi_\varphi=\infty$ whenever $\E|\varepsilon_1|^{\kappa_gr}=\infty$, that is, for $r\ge p$. For $\varphi=1$, Proposition~\ref{rev:prop:stationary-tail} gives divergence for $r\ge p-1$.

\emph{Step 5: the limit $\varphi\uparrow1$.} The bound $\sup_\varphi\pi_\varphi(W)<\infty$ gives tightness, because $W$ has compact sublevel sets. For bounded continuous $f$, $P_\varphi f\to P_1f$ uniformly on compact sets by joint continuity. Passing to the limit in $\pi_\varphi(f)=\pi_\varphi(P_\varphi f)$ shows that every weak limit is invariant for $P_1$, hence equals $\pi_1$. If $r<p-1$, choose $1+r<q<p$ in \eqref{pur:eq:drift}; uniform integrability then gives convergence of the $r$th absolute moments. If $p-1\le r<p$, the moments are finite for $\varphi<1$ by Step~3, while $\liminf_{\varphi\uparrow1}\int|e|^r\,\mathrm d\pi_\varphi\ge\int(|e|^r\wedge L)\,\mathrm d\pi_1$ for every $L$, and the right-hand side tends to infinity with $L$ by Step~4.

\emph{Step 6: part~(ii).} Let $V(e)=\cosh(ke)$ with $k>0$ and $ks<\theta$. Since $g$ is bounded and $\E e^{\theta|\eta_1|}<\infty$, the inequality $e^x\le1+x+\tfrac12x^2e^{|x|}$ gives, uniformly in $e$,
\[
\E e^{\pm kD_e}\le1\pm ka\,\E\bigl[g(e^{-e/2}\varepsilon_1)\bigr]+Ck^2 .
\]
By dominated convergence and the restoring-field assumption,
\[
\lim_{e\to\infty}\E[g(e^{-e/2}\varepsilon_1)]=g(0)<0,\qquad \liminf_{e\to-\infty}\E[g(e^{-e/2}\varepsilon_1)]>0.
\] Hence there are $k,c,e_0>0$ with $\E e^{kD_e}\le1-ck$ for $e\ge e_0$ and $\E e^{-kD_e}\le1-ck$ for $e\le-e_0$. Because $0\le\varphi\le1$, $e^{k\varphi e}\le e^{ke}$ for $e\ge0$ and $e^{-k\varphi e}\le e^{-ke}$ for $e\le0$, while the opposite exponentials are bounded by $\E e^{k|D_e|}\le C$ on the respective half-lines. Therefore $P_\varphi V\le(1-ck)V+C$ uniformly in $\varphi$. Steps~1, 3 and~5 then give existence, uniqueness, $\sup_\varphi\pi_\varphi(e^{k|e|})<\infty$, weak convergence as $\varphi\uparrow1$, and convergence of all moments by uniform integrability.
\end{proof}

\begingroup
\subsection{Proof of Proposition~\ref{trk:prop:biasprec}}\label{app:proof_trk_biasprec}
\begin{proof}
By Corollary~\ref{trk:cor:scale}(i), $\lambda-\tilde\lambda=e+\delta_{\mathsf S}(F)$, so
\[
 \mathsf E_\pi[(\lambda-\tilde\lambda)^2]
 =\mathsf E_\pi e^2+2\delta_{\mathsf S}\mathsf E_\pi e+\delta_{\mathsf S}^2.
\]
Proposition~\ref{trk:prop:selftune}(ii) gives $\mathsf E_\pi e^2=h^{1/2}V^\star(\rho b_{\mathsf S})+o(h^{1/2})$, uniformly over $\rho$ in the stated compact set.

Stationarity of \eqref{trk:eq:chain} and centred latent increments imply $\mathsf E_\pi\bar g_{\mathsf S}(e)=0$. The $C^2$ regularity in Lemma~\ref{trk:lem:local} gives the Peano expansion
\[
 \bar g_{\mathsf S}(e)=-b_{\mathsf S}e+\tfrac12\bar g_{\mathsf S}''(0)e^2+r(e),
 \qquad r(e)=o(e^2)\quad(e\to0).
\]
For every $\epsilon>0$, choose $d>0$ such that $|r(e)|\le\epsilon e^2$ on $|e|\le d$. Globally, boundedness of $\bar g_{\mathsf S}$ gives $|r(e)|\le C(1+|e|+e^2)$. The uniform exponential moment bound in Proposition~\ref{trk:prop:selftune}(i) therefore makes
\[
 \sup_{\rho\in P}\mathsf E_{\pi_{h,\rho}}\bigl[|r(e)|\mathbf1_{\{|e|>d\}}\bigr]
 =o(h^{1/2}).
\]
On $|e|\le d$, the expectation is at most $\epsilon\mathsf E_\pi e^2=O(\epsilon h^{1/2})$. Letting $h\downarrow0$ and then $\epsilon\downarrow0$ yields $\mathsf E_\pi r(e)=o(h^{1/2})$, uniformly over the gain set. Thus
\[
 b_{\mathsf S}\mathsf E_\pi e
 =\tfrac12\bar g_{\mathsf S}''(0)\mathsf E_\pi e^2+o(h^{1/2}).
\]
Substitution proves the expansion. When $1+c_{\mathsf S}\delta_{\mathsf S}>0$, minimising its $h^{1/2}$ coefficient is equivalent to minimising $V^\star(\rho b_{\mathsf S})$. The remaining assertions follow by subtracting the expansions and retaining the full coefficients, rather than efficiency alone when nonzero shifts are present.
\end{proof}

\begin{remark}[A target drift changes the bias cross-term]\label{rev:rem:target-drift}
The preceding expansion is specific to the driftless error model. To see the distinction from the leading weak limit, add a constant target drift $\mu h$ to \eqref{trk:eq:chain}, so the error update acquires $-\mu h$. For sufficiently small $h$, the global restoring condition is preserved. The exponential Lyapunov argument in the next proof also applies: the extra increment of $Z=h^{-1/4}e$ is $-\mu h^{3/4}$. Outside a fixed scaled neighbourhood, its $O(h^{3/4})V$ contribution is absorbed by the $O(\sqrt h)V$ negative Lyapunov drift; inside that neighbourhood it is $O(h^{3/4})=O(\sqrt h)$. These bounds are uniform over compact gain sets. The leading invariant second-moment limit is unchanged, but stationarity now gives
\[
 \mathsf E_\pi\bar g_{\mathsf S}(e)=\frac{\mu\sqrt h}{\rho},\qquad
 \mathsf E_\pi e=\frac{c_{\mathsf S}}2\mathsf E_\pi e^2
 -\frac{\mu\sqrt h}{\rho b_{\mathsf S}}+o(\sqrt h).
\]
Consequently the raw latent MSE contains the additional term $-2\delta_{\mathsf S}\mu\sqrt h/(\rho b_{\mathsf S})$. An $O(h)$ drift can therefore leave the leading local OU law unchanged while changing the next-order bias--precision comparison.
\end{remark}
\endgroup

\subsection{\texorpdfstring{{Proof of Proposition~\ref{trk:prop:selftune}}}{Proof of Proposition}}\label{app:proof_trk_prop_selftune}
\begingroup
\begin{proof}
\emph{(i) Drift condition.} Write $Z=h^{-1/4}e$ and $V(e)=\cosh(aZ)$. Given $e$, let
\[\Delta Z=h^{1/4}\{\rho g(we^{-e/2}\varepsilon)-\sigma\eta\}.\]
The identity $\cosh(x+y)=\cosh x\cosh y+\sinh x\sinh y$ gives
\[
PV(e)=\cosh(aZ)\,\E\cosh(a\Delta Z)+\sinh(aZ)\,\E\sinh(a\Delta Z).
\]
Because $|g|\le M_0$ and $\eta$ is Gaussian, for $a\le1$ and small $h$,
\[\E\cosh(a\Delta Z)\le1+C_1a^2\sqrt h,\qquad \E\sinh(a\Delta Z)=a\rho h^{1/4}\bar g(e)+R,\quad|R|\le C_1a^3h^{3/4}.\]
Two regions:
\begin{itemize}[leftmargin=1.5em]
\item On $\{|e|\le r\}$, $\bar g(e)=-\beta(e)e$ with $\beta\ge\mu$, so the second term is at most $-a\rho\mu\sqrt h\,Z\sinh(aZ)+C_1a^3h^{3/4}|\sinh(aZ)|$.
\item On $\{|e|>r\}$, global restoring gives $\mathrm{sgn}(e)\bar g(e)\le-\bar g_r<0$, so the second term is at most $-\tfrac12a\rho\bar g_rh^{1/4}|\sinh(aZ)|$, and $V\ge\cosh(arh^{-1/4})$ is large.
\end{itemize}
In both regions $PV\le(1-c\sqrt h)V$ for $|Z|\ge Z_0$, with $Z_0$ large; on $\{|Z|<Z_0\}$, $PV\le V+C\sqrt h$. Hence
\[PV\le(1-c\sqrt h)V+C\sqrt h\,\mathbf 1\{|Z|<Z_0\}.\]
The kernel has a positive continuous density (Gaussian $\eta$), so compact sets are petite and the chain is Lebesgue-irreducible and aperiodic. The drift inequality yields a unique invariant law with $\pi(V)\le C/c$ \citep[Theorems~14.0.1 and~15.0.1]{meyn2009markov}.

{\emph{(ii) Convergence of stationary moments.} Fix $\rho_h\to\rho\in P$ and start \eqref{trk:eq:chain} from $\pi_{h,\rho_h}$. Every one-time marginal of $Z_h(\tau)=h^{-1/4}e_{\lfloor\tau/\sqrt h\rfloor}$ is the image of $\pi_{h,\rho_h}$ under $e\mapsto h^{-1/4}e$. By (i), its moments of every order are bounded uniformly in $h$.}

\begingroup
The local characteristics in Lemma~\ref{trk:lem:local} are uniform over gains in $P$. The moment bound in part~(i) gives $\E|e_0|^q=O(h^{q/4})$ for every fixed $q$, so the block argument of Lemma~\ref{trk:lem:exit} also gives compact containment on a local window starting at time zero. Along a subsequence for which $Z_h(0)$ converges in law to $\mu$, the stopped decomposition and martingale limit of Theorem~\ref{th::local-ou}, followed by removal of the stop, give
\[
 Z_h\Rightarrow Z,\qquad dZ=-\rho b_{\mathsf S}Z\,d\tau+(\rho^2\Xi_{\mathsf S}+\sigma^2)^{1/2}dW,
 \qquad Z(0)\sim\mu.
\]
For every $\tau\in[0,M]$, continuity of the limiting OU paths gives $Z(\tau)\sim\mu$, since $Z_h(\tau)$ and $Z_h(0)$ have the same law. As the local horizon $M$ is arbitrary, $\mu$ is invariant for the limiting OU semigroup and hence equals its unique invariant law $N(0,V^\star(\rho b_{\mathsf S}))$. The exponential moment bound gives uniform integrability of $Z_h(0)^2$, so $h^{-1/2}\E_\pi[e^2]\to V^\star(\rho b_{\mathsf S})$. The same argument for every sequence $\rho_h\to\rho$ in compact $P$, together with continuity of $V^\star$, gives uniform convergence over $P$.

\emph{Direct second-moment check.} Write $d(e)=-\bar g_{\mathsf S}(e)/e$ for $e\ne0$, $d(0)=b_{\mathsf S}$, and $J_2(e)=\E_F[g(we^{-e/2}\varepsilon)^2]$. Squaring the stationary recursion, subtracting $e^2$ and using independence of the standard Gaussian target shock gives the exact identity
\[
 2\rho\,h^{-1/2}\E_\pi[d(e)e^2]=\rho^2\E_\pi J_2(e)+\sigma^2 .
\]
The function $d$ is bounded and continuous at zero. By part~(i), $e\to0$ in probability uniformly on $P$ and $e^2/\sqrt h$ is uniformly integrable. Thus $d(e)$ may be replaced by $b_{\mathsf S}$ on the left with an $o(1)$ error, uniformly on $P$. Bounded continuity gives $\E_\pi J_2(e)\to\Xi_{\mathsf S}$. This recovers $V^\star(\rho b_{\mathsf S})=(\rho^2\Xi_{\mathsf S}+\sigma^2)/(2\rho b_{\mathsf S})$ directly.

\endgroup

Without loss of generality, increase the growth exponent $a_0$ to a strictly positive value so that $\Delta(e)+e^2\le C\exp(a_0|e|)$. For $\Delta$, Taylor's theorem gives $|\Delta(e)-\tfrac12J_Fe^2|\le C|e|^3$ on $\{|e|\le r\}$, which contributes at most $Ch^{1/4}\E_\pi|Z|^3\to0$ after scaling by $h^{-1/2}$. On $\{|e|>r\}$, by (i) and the growth bound,
\[
h^{-1/2}\E_\pi\bigl[(\Delta(e)+e^2)\mathbf 1\{|e|>r\}\bigr]\le Ch^{-1/2}\bigl(\E_\pi e^{2a_0|e|}\bigr)^{1/2}\cosh(arh^{-1/4})^{-1/2}\to0 .
\]

\begingroup
\emph{(iii) Continuity and minimisers.} Decrease $h_0$, if necessary, so that $a h^{-1/4}>2\max\{a_0,0\}$ for $0<h\le h_0$, where $a$ is the constant in part~(i). Fix such an $h$, and write $\mathcal P_{h,\rho}$ for the transition kernel of \eqref{trk:eq:chain}. For every bounded continuous function $f$, the map $(e,\rho)\mapsto\mathcal P_{h,\rho}f(e)$ is continuous: the update is continuous in $(e,\rho)$ for each innovation pair, and dominated convergence applies. It follows that $\mathcal P_{h,\rho_n}f\to\mathcal P_{h,\rho}f$ uniformly on compact sets of $e$ whenever $\rho_n\to\rho$.

The uniform bound in part~(i) makes $\{\pi_{h,\rho_n}\}$ tight. If a subsequence converges weakly to $\bar\pi$, the invariance identities and the preceding compact convergence give
\[
 \int f\,d\bar\pi
 =\int \mathcal P_{h,\rho}f\,d\bar\pi.
\]
Thus $\bar\pi$ is invariant for $\mathcal P_{h,\rho}$ and equals $\pi_{h,\rho}$ by uniqueness. Consequently $\rho\mapsto\pi_{h,\rho}$ is weakly continuous on $P$. Assumption~\ref{ass:conditional_differentiation}, included in Assumption~\ref{trk:ass:scale}(c), makes $\Delta$ continuous on compact intervals. The same exponential moment bound, together with $0\le\Delta(e)\le C_\Delta e^{a_0|e|}$, gives uniform integrability of both $e^2$ and $\Delta(e)$ over $\rho\in P$. Hence
\[
 \rho\longmapsto\E_{\pi_{h,\rho}}[e^2],\qquad
 \rho\longmapsto\E_{\pi_{h,\rho}}[\Delta(e)]
\]
are continuous and attain their minima on compact $P$.

Under the interior-minimiser premise of part~(iii), both limiting criteria in part~(ii) are positive multiples of $V^\star(\rho b_{\mathsf S})$, which is continuous and strictly convex with unique minimiser $\rho^\star$ (Proposition~\ref{trk:prop:frontier}(iii)). Their uniform convergence on $P$ therefore implies convergence of every sequence of minimisers to $\rho^\star$ \citep[Theorem~5.7]{vaart1998asymptotic}.
\endgroup
\end{proof}
\endgroup
\clearpage
\clearpage
\section{Empirical illustrations}\label{sec::empirical}
\setcounter{table}{0}
\setcounter{figure}{0}
\renewcommand{\thetable}{\Alph{section}.\arabic{table}}
\renewcommand{\thefigure}{\Alph{section}.\arabic{figure}}
\renewcommand{\theHtable}{empirical.\arabic{table}}
\renewcommand{\theHfigure}{empirical.\arabic{figure}}

This appendix presents the historical-data results and associated sensitivity exercises. The forecast comparisons use fixed training fits and are descriptive, apart from the tests explicitly based on the reported quantile counts. The accompanying numerical archive reproduces the population calculations and the two controlled experiments in Section~\ref{sec::experiments}, and permits verification of the reported quantile-count calculations. It does not contain the dated historical inputs, empirical fitting code or fitted forecast paths needed to regenerate the historical-data figures and loss summaries.

The twelve-market application compares two predictive families and four criteria. Each criterion governs both state updating and static fitting. The reported paths show how the joint choices compare in forecast levels, quantile hits and responses to a single shock. Dated loss differences would be needed to test relative forecasting performance.

\paragraph{{Data and fixed-fit design}}
The panel comprises twelve international equity indices observed daily over 2000--2024, with the EURO STOXX~50 beginning in 2007. Table~\ref{tab::panel_data} gives tickers, dates and sample sizes. Daily open--high--low--close quotes from Yahoo Finance are used to construct log returns in per cent, demeaned by the training-sample mean, and Parkinson and Garman--Klass variance proxies \citep{parkinson1980extreme,garman1980estimation}. Each market's first $70\%$ is the training sample; parameters are then fixed over its final $30\%$. Squared-return QLIKE, in the forecast-dependent form $\mathrm{QLIKE}^{\rm fd}$ of~\eqref{eq::qlike}, is the primary close-to-close loss \citep{patton2011volatility}; intraday range proxies provide comparisons for different variance targets. Descriptive event paths additionally use the Thai baht/US dollar rate from FRED \citep{fred_dexthus}, with pre-event preprocessing and fitting.

\paragraph{{Fitting and scope}}
The working density is Gaussian or variance-standardised Student-$t_6$, and the criterion is log score, density power, CRPS or MMD. The reported coefficients $\widehat\vartheta_{\mathsf S}=(\widehat\omega,\widehat\phi,\widehat\alpha)'$ result from numerical minimisation of \eqref{eq:main-est-sample} on the training sample, starting the recursion at $\widehat\lambda_1(\vartheta)=\omega/(1-\phi)$. The fitting restrictions impose $|\phi|<1$, $\alpha>0$ and a bounded stationary level. The shape $\nu=6$ and density-power exponent $\beta=0.15$ are fixed. The coefficients and reported numerical diagnostics are listed in \ref{app::panel_fit_diagnostics}.

The results concern a single training fit, a fixed Student shape $\nu=6$ and the stated common filter family. Estimating $\nu$, refitting monthly and fitting a quantile-specific CAViaR model \citep{engle2004caviar} would change the forecasting procedures being compared.

For Gaussian CRPS, Student-$t_6$ log and Student-$t_6$ CRPS, the reported fitted vectors satisfy the monotone-driver contraction inequality. That is a fitted-point statement. Theorem~\ref{thm:static-estimation} additionally requires a fixed compact parameter domain satisfying the uniform inequality and its population, identification, moment, derivative and score-CLT conditions. Those requirements have not been established for the actual estimation procedure. Density-power applicability remains conditional on Proposition~\ref{inv:prop:rstep}; the unrestricted Gaussian-log and empirically bandwidth-selected MMD cells have no verified route here. Estimation of the training mean and bandwidth is not covered by simply conditioning on their realised values. We therefore use the fitted paths descriptively and do not report asymptotic parameter inference for these empirical fits.

Supplementary event-path and controlled-shock comparisons, together with the relevant moment qualifications, are given in \ref{sec::emp_outlier} and~\ref{sec::emp_heavytail}. They are illustrative and do not identify the latent persistence needed to apply Theorem~\ref{pur:thm:persistence} to individual markets.

\subsection{The density \texorpdfstring{$\times$}{x} criterion factorial}\label{sec::emp_factorial}
Equal-market squared-return QLIKE ranges from $1.007$ to $1.088$ across the eight density--criterion combinations (Figure~\ref{fig::panel_factorial} and Table~\ref{tab::emp_qlike}). The Gaussian GARCH comparator records $0.9998$ in the fixed-fit sensitivity table (Table~\ref{tab::emp_robustness_support}). For $q_t=y_t^2$, QLIKE differs from twice the Gaussian log loss by a forecast-independent constant, so it gives the same ranking as Gaussian log evaluation on the same variance paths. The four Gaussian-family log losses in Table~\ref{tab::emp_factorial_support} reproduce their squared-return QLIKE values to the displayed precision via $2\bar S_{\log}-\log(2\pi)$. The Parkinson and Garman--Klass columns use intraday ranges and give different rankings. These are descriptive point comparisons: the Gaussian GARCH comparator has a smaller squared-return QLIKE mean than all eight cells. The Student-log procedure also has the smallest displayed mean within the Student family under each of the four proper-score evaluators. The design changes the updating rule and estimation criterion jointly, so these differences cannot identify the isolated effect of either choice. A crossed update-by-training-criterion design would be needed for that purpose.

At the $1\%$ quantile (Table~\ref{tab::emp_var}), changing from the Gaussian to the Student-$t_6$ family moves the equal-market exceedance frequency closer to $1\%$ for each criterion: the Gaussian values are $1.91\%$--$2.63\%$, and the Student values $1.46\%$--$1.56\%$. Every displayed mean still exceeds the nominal level. At $5\%$, Gaussian log is closest in this descriptive comparison, with $5.04\%$, whereas the Student values are $5.39\%$--$5.44\%$. Closeness of a sample exceedance frequency to its target is not evidence of conditional calibration or comparative quantile-forecast superiority. Marginal PIT histograms provide a further descriptive view \citep{gneiting2007probabilistic}; no calibrated marginal-uniformity or dynamic-PIT test is claimed.

Criterion choice also changes forecast levels. In the latent scale experiment of Corollary~\ref{trk:cor:scale}, the $t_5$ target shifts are $0$, $-0.170$ and $-0.249$ for Gaussian log, density power and CRPS, corresponding to approximately $100\%$, $84\%$ and $78\%$ of latent variance. This ordering is compatible with the reported Gaussian-family $1\%$ exceedance frequencies, but it does not identify their empirical cause. Observation-conditional laws under latent stochastic volatility are generally mixtures, the autoregressive pull changes the composite centre, and the innovation law and fitted parameters also matter. The oracle calculation illustrates a mechanism; it is not an exact characterisation of the daily forecasts.

\begin{figure}[t]
\centering
\includegraphics[width=\linewidth]{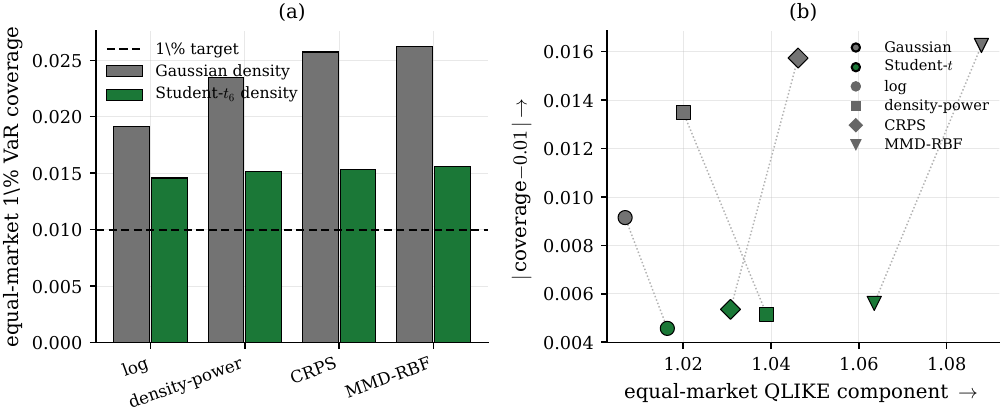}
\caption{{Descriptive density~$\times$~criterion comparison across twelve equity indices. Panel (a) reports equal-market $1\%$ exceedance frequencies; the dashed line is the nominal target. Panel (b) compares equal-market squared-return QLIKE with the absolute sample coverage gap. Neither panel reports an inferential ranking.}}
\label{fig::panel_factorial}
\end{figure}

\begin{table}[H]
\begingroup
\centering\small
\caption{{Equal-market mean realised exceedance frequencies on each market's final $30\%$, using fixed training fits. Entries are sample percentages, not calibration-test decisions. Market-level sample sizes and hit counts are in Table~\ref{tab::emp_market_calibration}.}}\label{tab::emp_var}
\begin{tabular}{llrr}
\toprule
Density & Criterion & $1\%$ exceedances (\%) & $5\%$ exceedances (\%) \\
\midrule
Gaussian & log & 1.91 & 5.04 \\
Gaussian & density-power & 2.35 & 5.71 \\
Gaussian & CRPS & 2.57 & 6.22 \\
Gaussian & MMD & 2.63 & 6.22 \\
\midrule
Student-$t_6$ & log & 1.46 & 5.39 \\
Student-$t_6$ & density-power & 1.52 & 5.41 \\
Student-$t_6$ & CRPS & 1.54 & 5.44 \\
Student-$t_6$ & MMD & 1.56 & 5.41 \\
\bottomrule
\end{tabular}

\endgroup
\end{table}

Table~\ref{tab::emp_factorial_support} evaluates the reported paths under all four proper scores. Comparisons are made within evaluator columns because the scoring rules have different scales. The training criterion does not generally give the smallest reported out-of-sample mean under its own score. Together with the proxy-dependent QLIKE ordering, this illustrates why a forecasting comparison must specify its evaluation loss. The limited fixed-$\nu$ and bandwidth sensitivities in \appref{sec::emp_robustness} do not establish robustness to arbitrary tuning, forecast origins or estimation schemes.

\paragraph{What the reported quantile counts establish}
Table~\ref{tab::emp_market_calibration} gives the observed number $H$ of hits and the number $n$ of forecasts for each market, model and nominal probability $p$. For predictable one-step forecasts on predetermined evaluation dates, exact conditional coverage means $\mathsf P(H_t=1\mid\mathcal F_{t-1})=p$ at every date. Repeated conditioning gives independent Bernoulli$(p)$ hits within the tested series and hence $H\sim{\rm Binomial}(n,p)$. We calculate the two-sided exact $p$-value by summing the probabilities of counts whose binomial mass is at most the observed mass. A single Holm correction \citep{holm1979} at $5\%$ over $12\times8\times2=192$ hypotheses gives 49 rejections, all with more hits than expected. Of these, 45 concern the $1\%$ quantile and four the $5\%$ quantile. The $1\%$ results comprise 40 Gaussian-family and five Student-family rejections. The calculation permits dependence between tests. A rejection identifies a failure of this conditional-coverage null detectable from the total. Nonrejection is not evidence of full dynamic calibration: total-count tests cannot detect all forms of dependence or hit clustering. Tests based on the ordered hit sequence and paired quantile-loss comparisons require dated forecasts and observations.

The binomial calculation uses an evaluation calendar and sample length fixed independently of the realised hits, and a continuous conditional return law at the issued quantile (or a pre-specified randomised convention for ties). The reported Holm-adjusted $p$-values on either side of the rejection boundary are $0.03844$ and $0.06354$. As a sensitivity to the unresolved Nasdaq identifier, we exclude its eight rows and repeat the entire adjustment across $176$ hypotheses. There are $45$ rejections: $41$ at $1\%$ (36 Gaussian and five Student) and four at $5\%$ (all Gaussian). The boundary values become $0.03524$ and $0.05821$. This calculation uses the reported counts and does not identify the Nasdaq return series.

Further permanent-break and observation-jump experiments illustrating the robustness--adaptation trade-off are reported in \ref{sec::emp_misspec}.

\FloatBarrier

\begingroup
\subsection{Empirical protocols and panel diagnostics}\label{app::panel_fit_diagnostics}
The tables report coefficients and sample summaries for 96 market--criterion--density fits. The count tests in \ref{sec::emp_factorial} use their published denominators and numerators. The remaining forecast loss and path summaries are descriptive.

\paragraph{{Replication scope}}
The tables provide sufficient inputs to recalculate hit frequencies and the count tests. Independent refitting, checking forecast timing and testing paired losses require the dated input data, fitting code and forecast paths. The parameter tables use the label NDX for a market with 1,887 test observations, while the metadata list the NASDAQ Composite (\texttt{\string^IXIC}) with this test length. The source series used in fitting requires verification against the dated inputs.

The reported labels and numerical values are retained. The count-based sensitivity excluding this market is reported in \ref{sec::emp_factorial}, so its rejection totals can be assessed without resolving that identifier. Forecast-loss summaries involving this market retain the identification limitation.

\paragraph{{Reported fitting procedure}}
Each market's MMD bandwidth is the median pairwise absolute distance among training returns, computed on a deterministic evenly spaced subsample of at most $1{,}024$ observations and then fixed. Log, density-power, CRPS and Gaussian MMD losses are evaluated in closed form; Student CRPS uses \eqref{eq::crps_student}, and Student MMD uses a $256$-node Gauss--Legendre rule after a probability-integral-transform change of variables. Numerical fitting uses L-BFGS-B with automatic-differentiation gradients and eight fixed starts. Tables~\ref{tab::emp_optimiser} and~\ref{tab::emp_stability} summarise the reported endpoints and computed path derivatives. They do not establish global minimisation, a population near-minimiser condition or uniform invertibility over the estimation domain.

The bandwidth statistic uses at most $1{,}024$ training \emph{observations}. Its pairwise distances share observations. Even for independent returns, a cap fixed as the training length grows can leave a nondegenerate random bandwidth. The fixed-tuning asymptotics apply directly to a bandwidth specified in advance. For joint inference with an estimated bandwidth, a growing training subsample requires its own rate and a joint limit for the distance quantile, training centre and estimating score, including their cross-covariances. The fixed-fit summaries below do not require that parameter CLT.

The Gaussian GARCH comparator uses two starts and the reported restrictions $0\le a_{\rm G}<0.3$, $0\le b_{\rm G}<0.98$ and $a_{\rm G}+b_{\rm G}<0.999$. EWMA uses decay $0.94$ and the variance of the first $50$ training-mean-demeaned observations for initialization.

\paragraph{{Descriptive aggregation}}
Forecast losses are averaged within each market and then equally across the twelve markets. Exceedance frequencies use the corresponding market-specific evaluation sample. Market-level sample sizes and hit counts are retained in Table~\ref{tab::emp_market_calibration}. PIT histograms likewise give each market weight $1/12$. These summaries hold fitted parameters, preprocessing and tuning values fixed. They do not account for repeated training or parameter uncertainty, and differences between sample means are not interpreted as statistically resolved.
\endgroup

\begin{table}[H]
\centering\small
\resizebox{\textwidth}{!}{%
\begin{tabular}{@{}lllccrrr@{}}
\toprule
Index & Ticker & Country & Start & End & Obs. & Train & Test \\
\midrule
S\&P 500       & \texttt{\textasciicircum GSPC}     & United States  & 2000-01-04 & 2024-12-31 & 6288 & 4401 & 1887 \\
NASDAQ Composite & \texttt{\textasciicircum IXIC}   & United States  & 2000-01-04 & 2024-12-31 & 6288 & 4401 & 1887 \\
Dow Jones      & \texttt{\textasciicircum DJI}      & United States  & 2000-01-04 & 2024-12-31 & 6288 & 4401 & 1887 \\
FTSE 100       & \texttt{\textasciicircum FTSE}     & United Kingdom & 2000-01-05 & 2024-12-31 & 6313 & 4419 & 1894 \\
DAX            & \texttt{\textasciicircum GDAXI}    & Germany        & 2000-01-04 & 2024-12-30 & 6347 & 4442 & 1905 \\
CAC 40         & \texttt{\textasciicircum FCHI}     & France         & 2000-01-04 & 2024-12-31 & 6389 & 4472 & 1917 \\
EURO STOXX 50  & \texttt{\textasciicircum STOXX50E} & Euro area      & 2007-04-02 & 2024-12-30 & 4451 & 3115 & 1336 \\
Nikkei 225     & \texttt{\textasciicircum N225}     & Japan          & 2000-01-05 & 2024-12-30 & 6125 & 4287 & 1838 \\
Hang Seng      & \texttt{\textasciicircum HSI}      & Hong Kong      & 2000-01-04 & 2024-12-31 & 6159 & 4311 & 1848 \\
S\&P/TSX       & \texttt{\textasciicircum GSPTSE}   & Canada         & 2000-01-05 & 2024-12-31 & 6278 & 4394 & 1884 \\
ASX 200        & \texttt{\textasciicircum AXJO}     & Australia      & 2000-01-05 & 2024-12-31 & 6315 & 4420 & 1895 \\
Bovespa        & \texttt{\textasciicircum BVSP}     & Brazil         & 2000-01-04 & 2024-12-30 & 6190 & 4333 & 1857 \\
\bottomrule
\end{tabular}
}
\caption{The twelve-index equity panel. Daily open--high--low--close quotes from Yahoo Finance
(download tickers as shown). ``Obs.'' is the number of daily log-returns after listwise deletion of
days with a missing quote; ``Train'' is the first $70\%$ of each series, on which the static
parameters are estimated, and ``Test'' the final $30\%$, on which all diagnostics are computed out of
sample. Returns are close-to-close log-returns in per cent, demeaned by the training-sample mean; the
Parkinson and Garman--Klass range proxies are formed from the same OHLC quotes. The out-of-sample
QLIKE target is the squared-return proxy. The Nasdaq metadata and the NDX labels in the fitted-result tables are retained as reported; the identifier limitation is discussed in \ref{app::panel_fit_diagnostics}.}
\label{tab::panel_data}
\end{table}

\begin{table}[H]
\begingroup
\centering\scriptsize
\begin{tabular}{llrrrr}
\toprule
Density & Training criterion & predictive log & CRPS & density-power & MMD \\
\midrule
Gaussian & log & 1.4223 & 0.5978 & -5.4447 & 0.3790 \\
Gaussian & density-power & 1.4290 & 0.5969 & -5.4476 & 0.3773 \\
Gaussian & CRPS & 1.4420 & 0.5976 & -5.4444 & 0.3776 \\
Gaussian & MMD & 1.4629 & 0.5998 & -5.4385 & 0.3785 \\
Student-$t_6$ & log & 1.3940 & 0.5958 & -5.4622 & 0.3762 \\
Student-$t_6$ & density-power & 1.3994 & 0.5974 & -5.4588 & 0.3771 \\
Student-$t_6$ & CRPS & 1.3978 & 0.5968 & -5.4596 & 0.3769 \\
Student-$t_6$ & MMD & 1.4049 & 0.5991 & -5.4554 & 0.3780 \\
\bottomrule
\end{tabular}
\caption{{Reported out-of-sample mean scores for each fitted path, averaged within market and then equally across markets. Lower is better within an evaluator column; the columns have different scales. Density power uses beta $0.15$, and MMD uses each market's fixed training bandwidth. No standard errors or comparative tests are reported.}}\label{tab::emp_factorial_support}
\endgroup
\end{table}

\begin{table}[H]
\begingroup
\centering\scriptsize
\textit{Panel A: GARCH family comparison}\par\smallskip
\begin{tabular}{lrr}
\toprule
Comparator & QLIKE & $1\%$ exceedance frequency \\
\midrule
Gaussian GARCH & 0.9998 & 0.0200 \\
Student-$t_6$ GARCH & 1.0057 & 0.0137 \\
\bottomrule
\end{tabular}
\par\medskip
\textit{Panel B: Fixed Student shape sensitivity}\par\smallskip
\begin{tabular}{lrrrrr}
\toprule
$\nu$ & 4 & 5 & 6 & 8 & 10 \\
\midrule
Student-log exceedance frequency & 0.0103 & 0.0127 & 0.0146 & 0.0168 & 0.0182 \\
Student-CRPS exceedance frequency & 0.0102 & 0.0132 & 0.0154 & 0.0179 & 0.0198 \\
\bottomrule
\end{tabular}
\par\medskip
\textit{Panel C: Fixed MMD bandwidth sensitivity}\par\smallskip
\begin{tabular}{llrrrr}
\toprule
Density & bandwidth & predictive log & CRPS & QLIKE & $\Delta$ QLIKE \\
\midrule
Gaussian & $0.5\times$ & 1.5031 & 0.6015 & 1.1684 & +0.0805 \\
Gaussian & $1\times$ & 1.4629 & 0.5998 & 1.0879 & +0.0000 \\
Gaussian & $2\times$ & 1.4492 & 0.5997 & 1.0604 & -0.0275 \\
Student-$t_6$ & $0.5\times$ & 1.4132 & 0.6009 & 1.1138 & +0.0503 \\
Student-$t_6$ & $1\times$ & 1.4049 & 0.5991 & 1.0635 & +0.0000 \\
Student-$t_6$ & $2\times$ & 1.4048 & 0.5982 & 1.0547 & -0.0088 \\
\bottomrule
\end{tabular}
\caption{{Descriptive sensitivity of reported fixed-fit sample summaries. Panel A compares Gaussian and fixed-shape Student GARCH specifications. Panel B reports separate Student-log and Student-CRPS fits at each fixed shape; it does not estimate the degrees of freedom. Panel C reports MMD fits at one-half, one and twice the training-only median-distance bandwidth. Scores are equal-market out-of-sample means, and $\Delta$ QLIKE is relative to the corresponding $1\times$ row. Exceedance frequencies are proportions. No calibration decisions or formal robustness test are reported.}}
\label{tab::emp_robustness_support}
\endgroup
\end{table}

\begingroup
\begingroup\small\setlength{\tabcolsep}{4pt}
\begin{longtable}{llrrrrr}
\caption{{Descriptive market-level exceedances for all 96 fixed-fit market--cell combinations. $N$ is the evaluation sample size, and hits count returns below the forecast quantile. Rates are $100$ times hits divided by $N$. Abbreviations combine density (G, Gaussian; T, Student-$t_6$) and criterion (L, log; DP, density power; C, CRPS; M, MMD). \ref{sec::emp_factorial} uses these counts for exact binomial tests with a joint Holm correction.}}\label{tab::emp_market_calibration}\\
\toprule
Cell & Market & $N$ & $1\%$ hits & rate (\%) & $5\%$ hits & rate (\%) \\
\midrule\endfirsthead
\toprule
Cell & Market & $N$ & $1\%$ hits & rate (\%) & $5\%$ hits & rate (\%) \\
\midrule\endhead
\bottomrule\endfoot
\textsc{G-L} & SPX & 1887 & 41 & 2.17 & 100 & 5.30 \\
\textsc{G-L} & NDX & 1887 & 39 & 2.07 & 105 & 5.56 \\
\textsc{G-L} & DJI & 1887 & 37 & 1.96 & 100 & 5.30 \\
\textsc{G-L} & FTSE & 1894 & 39 & 2.06 & 95 & 5.02 \\
\textsc{G-L} & DAX & 1905 & 33 & 1.73 & 101 & 5.30 \\
\textsc{G-L} & CAC & 1917 & 34 & 1.77 & 92 & 4.80 \\
\textsc{G-L} & STOXX50 & 1336 & 29 & 2.17 & 68 & 5.09 \\
\textsc{G-L} & NIKKEI & 1838 & 29 & 1.58 & 84 & 4.57 \\
\textsc{G-L} & HSI & 1848 & 30 & 1.62 & 89 & 4.82 \\
\textsc{G-L} & TSX & 1884 & 39 & 2.07 & 97 & 5.15 \\
\textsc{G-L} & ASX & 1895 & 50 & 2.64 & 109 & 5.75 \\
\textsc{G-L} & BVSP & 1857 & 21 & 1.13 & 70 & 3.77 \\
\textsc{G-DP} & SPX & 1887 & 50 & 2.65 & 108 & 5.72 \\
\textsc{G-DP} & NDX & 1887 & 41 & 2.17 & 108 & 5.72 \\
\textsc{G-DP} & DJI & 1887 & 49 & 2.60 & 109 & 5.78 \\
\textsc{G-DP} & FTSE & 1894 & 45 & 2.38 & 104 & 5.49 \\
\textsc{G-DP} & DAX & 1905 & 45 & 2.36 & 123 & 6.46 \\
\textsc{G-DP} & CAC & 1917 & 48 & 2.50 & 111 & 5.79 \\
\textsc{G-DP} & STOXX50 & 1336 & 36 & 2.69 & 80 & 5.99 \\
\textsc{G-DP} & NIKKEI & 1838 & 39 & 2.12 & 103 & 5.60 \\
\textsc{G-DP} & HSI & 1848 & 36 & 1.95 & 97 & 5.25 \\
\textsc{G-DP} & TSX & 1884 & 45 & 2.39 & 109 & 5.79 \\
\textsc{G-DP} & ASX & 1895 & 53 & 2.80 & 121 & 6.39 \\
\textsc{G-DP} & BVSP & 1857 & 29 & 1.56 & 84 & 4.52 \\
\textsc{G-C} & SPX & 1887 & 51 & 2.70 & 119 & 6.31 \\
\textsc{G-C} & NDX & 1887 & 46 & 2.44 & 117 & 6.20 \\
\textsc{G-C} & DJI & 1887 & 50 & 2.65 & 120 & 6.36 \\
\textsc{G-C} & FTSE & 1894 & 52 & 2.75 & 118 & 6.23 \\
\textsc{G-C} & DAX & 1905 & 52 & 2.73 & 133 & 6.98 \\
\textsc{G-C} & CAC & 1917 & 52 & 2.71 & 119 & 6.21 \\
\textsc{G-C} & STOXX50 & 1336 & 41 & 3.07 & 89 & 6.66 \\
\textsc{G-C} & NIKKEI & 1838 & 44 & 2.39 & 110 & 5.98 \\
\textsc{G-C} & HSI & 1848 & 40 & 2.16 & 113 & 6.11 \\
\textsc{G-C} & TSX & 1884 & 51 & 2.71 & 111 & 5.89 \\
\textsc{G-C} & ASX & 1895 & 57 & 3.01 & 132 & 6.97 \\
\textsc{G-C} & BVSP & 1857 & 29 & 1.56 & 89 & 4.79 \\
\textsc{G-M} & SPX & 1887 & 55 & 2.91 & 120 & 6.36 \\
\textsc{G-M} & NDX & 1887 & 51 & 2.70 & 122 & 6.47 \\
\textsc{G-M} & DJI & 1887 & 56 & 2.97 & 122 & 6.47 \\
\textsc{G-M} & FTSE & 1894 & 48 & 2.53 & 118 & 6.23 \\
\textsc{G-M} & DAX & 1905 & 48 & 2.52 & 134 & 7.03 \\
\textsc{G-M} & CAC & 1917 & 47 & 2.45 & 115 & 6.00 \\
\textsc{G-M} & STOXX50 & 1336 & 40 & 2.99 & 84 & 6.29 \\
\textsc{G-M} & NIKKEI & 1838 & 40 & 2.18 & 104 & 5.66 \\
\textsc{G-M} & HSI & 1848 & 48 & 2.60 & 117 & 6.33 \\
\textsc{G-M} & TSX & 1884 & 53 & 2.81 & 114 & 6.05 \\
\textsc{G-M} & ASX & 1895 & 62 & 3.27 & 131 & 6.91 \\
\textsc{G-M} & BVSP & 1857 & 29 & 1.56 & 89 & 4.79 \\
\textsc{T-L} & SPX & 1887 & 31 & 1.64 & 102 & 5.41 \\
\textsc{T-L} & NDX & 1887 & 30 & 1.59 & 100 & 5.30 \\
\textsc{T-L} & DJI & 1887 & 34 & 1.80 & 106 & 5.62 \\
\textsc{T-L} & FTSE & 1894 & 28 & 1.48 & 98 & 5.17 \\
\textsc{T-L} & DAX & 1905 & 30 & 1.57 & 116 & 6.09 \\
\textsc{T-L} & CAC & 1917 & 26 & 1.36 & 106 & 5.53 \\
\textsc{T-L} & STOXX50 & 1336 & 26 & 1.95 & 75 & 5.61 \\
\textsc{T-L} & NIKKEI & 1838 & 19 & 1.03 & 96 & 5.22 \\
\textsc{T-L} & HSI & 1848 & 19 & 1.03 & 94 & 5.09 \\
\textsc{T-L} & TSX & 1884 & 29 & 1.54 & 104 & 5.52 \\
\textsc{T-L} & ASX & 1895 & 33 & 1.74 & 111 & 5.86 \\
\textsc{T-L} & BVSP & 1857 & 14 & 0.75 & 79 & 4.25 \\
\textsc{T-DP} & SPX & 1887 & 32 & 1.70 & 106 & 5.62 \\
\textsc{T-DP} & NDX & 1887 & 30 & 1.59 & 102 & 5.41 \\
\textsc{T-DP} & DJI & 1887 & 37 & 1.96 & 108 & 5.72 \\
\textsc{T-DP} & FTSE & 1894 & 28 & 1.48 & 100 & 5.28 \\
\textsc{T-DP} & DAX & 1905 & 34 & 1.78 & 119 & 6.25 \\
\textsc{T-DP} & CAC & 1917 & 29 & 1.51 & 106 & 5.53 \\
\textsc{T-DP} & STOXX50 & 1336 & 27 & 2.02 & 75 & 5.61 \\
\textsc{T-DP} & NIKKEI & 1838 & 19 & 1.03 & 95 & 5.17 \\
\textsc{T-DP} & HSI & 1848 & 19 & 1.03 & 94 & 5.09 \\
\textsc{T-DP} & TSX & 1884 & 30 & 1.59 & 103 & 5.47 \\
\textsc{T-DP} & ASX & 1895 & 33 & 1.74 & 108 & 5.70 \\
\textsc{T-DP} & BVSP & 1857 & 14 & 0.75 & 75 & 4.04 \\
\textsc{T-C} & SPX & 1887 & 33 & 1.75 & 104 & 5.51 \\
\textsc{T-C} & NDX & 1887 & 30 & 1.59 & 99 & 5.25 \\
\textsc{T-C} & DJI & 1887 & 39 & 2.07 & 109 & 5.78 \\
\textsc{T-C} & FTSE & 1894 & 28 & 1.48 & 100 & 5.28 \\
\textsc{T-C} & DAX & 1905 & 33 & 1.73 & 119 & 6.25 \\
\textsc{T-C} & CAC & 1917 & 28 & 1.46 & 105 & 5.48 \\
\textsc{T-C} & STOXX50 & 1336 & 28 & 2.10 & 77 & 5.76 \\
\textsc{T-C} & NIKKEI & 1838 & 20 & 1.09 & 97 & 5.28 \\
\textsc{T-C} & HSI & 1848 & 19 & 1.03 & 94 & 5.09 \\
\textsc{T-C} & TSX & 1884 & 31 & 1.65 & 107 & 5.68 \\
\textsc{T-C} & ASX & 1895 & 34 & 1.79 & 112 & 5.91 \\
\textsc{T-C} & BVSP & 1857 & 13 & 0.70 & 75 & 4.04 \\
\textsc{T-M} & SPX & 1887 & 37 & 1.96 & 107 & 5.67 \\
\textsc{T-M} & NDX & 1887 & 29 & 1.54 & 105 & 5.56 \\
\textsc{T-M} & DJI & 1887 & 37 & 1.96 & 109 & 5.78 \\
\textsc{T-M} & FTSE & 1894 & 29 & 1.53 & 104 & 5.49 \\
\textsc{T-M} & DAX & 1905 & 30 & 1.57 & 110 & 5.77 \\
\textsc{T-M} & CAC & 1917 & 30 & 1.56 & 106 & 5.53 \\
\textsc{T-M} & STOXX50 & 1336 & 27 & 2.02 & 73 & 5.46 \\
\textsc{T-M} & NIKKEI & 1838 & 19 & 1.03 & 95 & 5.17 \\
\textsc{T-M} & HSI & 1848 & 20 & 1.08 & 93 & 5.03 \\
\textsc{T-M} & TSX & 1884 & 31 & 1.65 & 105 & 5.57 \\
\textsc{T-M} & ASX & 1895 & 38 & 2.01 & 112 & 5.91 \\
\textsc{T-M} & BVSP & 1857 & 15 & 0.81 & 74 & 3.98 \\
\end{longtable}
\endgroup
\endgroup

\begin{table}[ht]
\centering\small
\begin{tabular}{@{}llrrrrr@{}}
\toprule
Density & Criterion & min $\bar\ell$ & max $\bar\ell$ & max $|A_t|$ & train hits & test hits \\
\midrule
Gaussian & log & -0.1317 & -0.0625 & 1.9689 & 0 & 0 \\
Gaussian & density-power & -0.1134 & -0.0586 & 1.4501 & 0 & 0 \\
Gaussian & CRPS & -0.1152 & -0.0631 & 0.9952 & 0 & 0 \\
Gaussian & MMD & -0.0944 & -0.0474 & 1.1100 & 0 & 0 \\
\midrule
Student-$t_6$ & log & -0.1113 & -0.0573 & 0.9942 & 0 & 0 \\
Student-$t_6$ & density-power & -0.1028 & -0.0510 & 1.0725 & 0 & 0 \\
Student-$t_6$ & CRPS & -0.1153 & -0.0633 & 0.9950 & 0 & 0 \\
Student-$t_6$ & MMD & -0.0937 & -0.0468 & 1.1041 & 0 & 0 \\
\bottomrule
\end{tabular}
 \caption{{Per-cell path-derivative diagnostics across twelve markets. Here $A_t$ is the derivative of the recursion evaluated along its computed training path, and $\bar\ell$ is the sample mean of $\log|A_t|$; its minimum and maximum are over markets. These are local realised-path quantities, not random Lipschitz suprema over all states or parameter vectors. Negative sample means are fit diagnostics and do not verify the population uniform-invertibility condition. ``max $|A_t|$'' is the maximum over reported dates and markets. ``Hits'' counts proposed-to-actual projection changes; no state projection is used, so all counts are zero.}}
\label{tab::emp_stability}
\end{table}

\begin{table}[ht]
\begingroup
\centering\scriptsize
\begin{tabular}{llrrrrr}
\toprule
Density & Criterion & gap min & gap max & max $\|\nabla\|$ & min eig. & max cond. \\
\midrule
Gaussian & log & 2.98e-13 & 1.55e-10 & 4.29e-06 & 5.74e-03 & 9.15 \\
Gaussian & density-power & 1.63e-12 & 3.31e-10 & 1.66e-06 & 1.28e-03 & 21.47 \\
Gaussian & CRPS & 2.15e-11 & 3.24e-09 & 4.08e-06 & 2.11e-04 & 25.04 \\
Gaussian & MMD & 9.49e-11 & 3.69e-09 & 4.79e-06 & 1.33e-04 & 28.84 \\
Student-$t_6$ & log & 1.41e-13 & 2.71e-10 & 1.99e-06 & 1.08e-03 & 24.19 \\
Student-$t_6$ & density-power & 1.03e-12 & 3.64e-10 & 1.90e-06 & 3.46e-04 & 54.47 \\
Student-$t_6$ & CRPS & 5.32e-12 & 3.69e-09 & 3.07e-06 & 2.43e-04 & 18.95 \\
Student-$t_6$ & MMD & 1.63e-11 & 3.91e-09 & 1.68e-06 & 1.14e-04 & 30.47 \\
\bottomrule
\end{tabular}
\caption{{Reported numerical endpoint summaries across twelve markets per cell. Gap denotes the best-to-second-best objective separation; the remaining columns are the worst final gradient norm, smallest Hessian eigenvalue and largest condition number in the reported computations. These finite diagnostics do not certify global minimisation, the asymptotic near-minimiser premise or nonsingular population sensitivity.}}\label{tab::emp_optimiser}
\endgroup
\end{table}

\begin{table}[p]
\centering\footnotesize
\resizebox{\textwidth}{!}{\begin{tabular}{@{}lrrrr@{}}
\toprule
Market & log & density-power & CRPS & MMD \\
\midrule
SPX & $(7.798e-05,0.98388,0.0620)$ & $(0.0001482,0.98792,0.1137)$ & $(-0.0003954,0.99080,0.5824)$ & $(-0.0001821,0.99161,0.3931)$ \\
NDX & $(0.005519,0.99134,0.0584)$ & $(0.005192,0.99383,0.0995)$ & $(0.005261,0.99524,0.4900)$ & $(0.002682,0.99875,0.3157)$ \\
DJI & $(-0.002278,0.98120,0.0658)$ & $(-0.001875,0.98657,0.1195)$ & $(-0.001644,0.99016,0.5908)$ & $(-0.001008,0.99106,0.3920)$ \\
FTSE & $(-0.0003076,0.98317,0.0785)$ & $(-0.0004186,0.98484,0.1278)$ & $(-0.0001894,0.98445,0.6811)$ & $(-0.0001834,0.98683,0.4685)$ \\
DAX & $(0.007825,0.98498,0.0646)$ & $(0.005315,0.98969,0.1027)$ & $(0.004885,0.99038,0.5179)$ & $(0.004436,0.99186,0.3601)$ \\
CAC & $(0.007888,0.98414,0.0665)$ & $(0.005154,0.98916,0.1006)$ & $(0.004834,0.98889,0.5200)$ & $(0.004337,0.98995,0.3604)$ \\
STOXX50 & $(0.007505,0.97954,0.0627)$ & $(0.003157,0.98655,0.1092)$ & $(0.002512,0.98665,0.5837)$ & $(0.002903,0.98657,0.3901)$ \\
NIKKEI & $(0.0202,0.96747,0.0830)$ & $(0.009393,0.98393,0.1024)$ & $(0.00885,0.98440,0.5005)$ & $(0.006634,0.98770,0.2850)$ \\
HSI & $(0.005366,0.99237,0.0502)$ & $(0.003846,0.99334,0.0734)$ & $(0.003481,0.99357,0.4139)$ & $(0.002425,0.99606,0.2473)$ \\
TSX & $(-0.0008664,0.99114,0.0585)$ & $(-0.001635,0.98953,0.1003)$ & $(-0.0004857,0.99276,0.5158)$ & $(-4.324e-05,0.99419,0.3378)$ \\
ASX & $(-0.004189,0.98520,0.0642)$ & $(-0.002757,0.98934,0.0935)$ & $(-0.002444,0.99048,0.5047)$ & $(-0.002254,0.99150,0.3304)$ \\
BVSP & $(0.02158,0.97938,0.0477)$ & $(0.01528,0.98503,0.0780)$ & $(0.01323,0.98694,0.3960)$ & $(0.0125,0.98731,0.2454)$ \\
\bottomrule
\end{tabular}
 }
\caption{{Reported Gaussian-family coefficients $(\omega,\varphi,\alpha)$ used in the fixed-fit comparisons. These coefficient values are descriptive outputs of the fitting procedure; they do not identify or verify a compact uniform-contraction optimisation domain.}}
\label{tab::emp_parameters_gauss}
\end{table}

\begin{table}[p]
\centering\footnotesize
\resizebox{\textwidth}{!}{\begin{tabular}{@{}lrrrr@{}}
\toprule
Market & log & density-power & CRPS & MMD \\
\midrule
SPX & $(0.001117,0.98834,0.1139)$ & $(0.001388,0.98931,0.1381)$ & $(0.001154,0.99076,0.6831)$ & $(0.001475,0.99157,0.4485)$ \\
NDX & $(0.005596,0.99424,0.0993)$ & $(0.005394,0.99545,0.1194)$ & $(0.006119,0.99503,0.5770)$ & $(0.002907,0.99881,0.3588)$ \\
DJI & $(-0.0005241,0.98709,0.1192)$ & $(-6.159e-05,0.98766,0.1480)$ & $(6.035e-05,0.99012,0.6913)$ & $(0.0007717,0.99098,0.4488)$ \\
FTSE & $(0.001504,0.98446,0.1281)$ & $(0.002359,0.98496,0.1568)$ & $(0.002438,0.98431,0.7961)$ & $(0.002388,0.98681,0.5351)$ \\
DAX & $(0.006437,0.98979,0.1030)$ & $(0.006548,0.99034,0.1268)$ & $(0.006542,0.99040,0.6079)$ & $(0.00609,0.99184,0.4149)$ \\
CAC & $(0.006586,0.98867,0.1025)$ & $(0.006813,0.98895,0.1278)$ & $(0.006908,0.98874,0.6112)$ & $(0.006382,0.99004,0.4139)$ \\
STOXX50 & $(0.005085,0.98562,0.1130)$ & $(0.005808,0.98526,0.1446)$ & $(0.004978,0.98645,0.6827)$ & $(0.005527,0.98659,0.4455)$ \\
NIKKEI & $(0.012,0.98291,0.1012)$ & $(0.01004,0.98653,0.1121)$ & $(0.01183,0.98396,0.5955)$ & $(0.009352,0.98735,0.3331)$ \\
HSI & $(0.004272,0.99380,0.0733)$ & $(0.003847,0.99473,0.0874)$ & $(0.004634,0.99355,0.4849)$ & $(0.003243,0.99611,0.2812)$ \\
TSX & $(-0.0002089,0.99066,0.0987)$ & $(0.0003842,0.99157,0.1167)$ & $(0.0007502,0.99284,0.5979)$ & $(0.001143,0.99416,0.3830)$ \\
ASX & $(-0.001552,0.98943,0.0957)$ & $(-0.0008771,0.99049,0.1143)$ & $(-0.0008149,0.99036,0.5906)$ & $(-0.0005385,0.99145,0.3768)$ \\
BVSP & $(0.01676,0.98534,0.0764)$ & $(0.01638,0.98617,0.0900)$ & $(0.01555,0.98693,0.4618)$ & $(0.0153,0.98715,0.2828)$ \\
\bottomrule
\end{tabular}
 }
\caption{{Reported Student-$t_6$ coefficients $(\omega,\varphi,\alpha)$ used in the fixed-fit comparisons. The shape $\nu=6$ is fixed. As for Table~\ref{tab::emp_parameters_gauss}, fitted-point properties do not verify the population estimation assumptions.}}
\label{tab::emp_parameters_student}
\end{table}

\FloatBarrier
\paragraph{\textbf{Event-path design}} For each descriptive path, the demeaning mean and recursion
parameters are estimated over a declared five-year window ending strictly before the event and are
then frozen. Annualised volatility is evaluated on a $\log_{10}$ scale directly from log variance,
without exponentiation or numerical clipping. For the controlled SPX counterfactual, parameters are
estimated on 2012--2016 and frozen over 2017. The 3 July 2017 return alone receives the additive
perturbation $B\widehat\sigma_{\rm ref}$, where $\widehat\sigma_{\rm ref}$ is the root mean squared
demeaned return over the preceding 60 sessions and is common to all methods. No state projection is
applied to the clean or perturbed paths. The reported growth and plateau comparisons are computed
directly from those paths for every displayed $B$, without imposing the asymptotic shape in
Proposition~\ref{prop::outlier}.

\paragraph{\textbf{Location-path design}} The Nile series contains 100 annual observations over
1871--1970, measured in $10^8\,\mathrm{m}^3$. The Gaussian working density is fixed.
{All displayed filters use parameters estimated on 1871--1944 and frozen thereafter. The displayed cells are Gaussian log, CRPS and density power; fixed-bandwidth MMD is omitted because a near-singular pre-1945 Hessian was reported. These paths are descriptive, and no acceptance-gate or population-inference claim is made. The 1899 shift is inside the estimation period, so its response is retrospective. The 1945 injection is post-training.} The controlled arm adds a common
$6\widehat\sigma_{\rm log}$ outlier at 1945 and measures displacement relative to each method's
clean path; the first half-decay years are 1949, 1948, and 1949 for Gaussian log, CRPS, and
density-power, respectively.

\paragraph{\textbf{Break and jump protocols}} The reported permanent-break design uses $20{,}000$ paired replications, $100$ pre-break and $300$ post-break dates, and gain $0.05$. The reported jump design has length $12{,}000$, a $60\%$ training fraction and jump probability $0.02$. The reported simulation seed is $20260430$. Both designs use no state projection. The
permanent-break comparison holds the Gaussian density and innovations fixed, starts log-variance at
zero, and applies the gain to the raw score divided by current working-model curvature. Adaptation
delay is the first post-break date on which the ten-date trailing mean reaches $90\%$ of the new
log-variance; no path is right-censored. Loss is twice cumulative population Gaussian log-risk
excess relative to the oracle, $v^\star/v_t+\log v_t-1-\log v^\star$.

In the jump design, $-5\exp(x_t/2)|Z_t|$ contamination, with $Z_t$ standard Gaussian, leaves the
latent component $x_t$ unchanged; its half-normal scale is five and its conditional mean absolute
size is $5\sqrt{2/\pi}\simeq3.99$ standard deviations. The frozen-fit arm applies clean-training
{parameters to both paths; the refitted arm estimates under jumps. The retained contrasts are reported point summaries, without a general numerical-acceptance claim. Gaussian-log completed paths cross the stated $\pm12$ diagnostic bound six times in the frozen-fit arm and once after refitting; these failures are reported without clipping.} {The comparison uses the reported point contrasts only, without confidence statements or comparative tests.}

\subsection{Event paths and controlled outlier propagation}\label{sec::emp_outlier}
A common injected return isolates the one-observation mechanism in Proposition~\ref{prop::outlier}. In the SPX counterfactual, only the return of 3 July 2017 is changed, by $B$ times the root mean squared return over the previous 60 sessions; parameters are fixed from the pre-event training sample. At $B=100$, the reported peak changes are $541.64$ log-variance units for Gaussian log, $1.29$ for Student log, $0.25$ for Student density power and $0.42$ for Student MMD. The reported Gaussian-log log--log slope over $B\in\{35,60,100\}$ is $1.99$. These large shocks are algebraic stress tests. They illustrate quadratic growth versus bounded response, not plausible magnitudes for an ordinary daily equity return. \appref{app:moved_empirics} distinguishes this controlled experiment from historical event paths.

\subsection{Moment requirements under heavy-tailed innovations}\label{sec::emp_heavytail}

With Gaussian log scale and $G_{\log}=2$, the update is $z^2-1$. Its variance is finite exactly when the standardised observation has a finite fourth moment. Theorem~\ref{pur:thm:persistence} gives a separate, fixed-mesh statement about the accumulated tracking error. In its matched-persistence experiment, $t_5$ observations produce finite stationary MSE at every $\varphi<1$ and infinite MSE at $\varphi=1$. The fitted indices supply no observations of the latent persistence needed to apply this classification market by market.

A finite computed sample correlation does not establish that the corresponding population correlation exists. We therefore make no numerical population-correlation comparison for Gaussian-log cells with $\nu\le4$. Bounded implemented drivers remove the immediate infinite-update-variance obstruction, while the remaining target, moment, stability and containment conditions still require separate verification. A Cauchy innovation law also gives infinite Gaussian-log and CRPS risks and lies outside the finite-moment MSE comparison.

\subsection{Misspecification and robustness--adaptation trade-offs}\label{sec::emp_misspec}
The trade-off between resistance to isolated outliers and adaptation to persistent changes is familiar in robust score-driven filtering \citep{harvey2013dynamic}. A controlled permanent-break experiment holds the Gaussian family, innovations and curvature-normalised gain fixed while variance changes from $1$ to $9$. The reported mean adaptation delays are $29.425$ dates for Gaussian log and $42.817$ for density power, {with cumulative oracle excess QLIKE values $26.234$ and $41.600$ (twice the Gaussian log-risk excess)}. These are finite-simulation summaries under the reported protocol in \ref{app::panel_fit_diagnostics}.

A general displacement bound must include autoregressive drift:
\[
\lambda_{t+1}-\lambda_t=(1-\phi)(\bar\lambda-\lambda_t)+\alpha g_{\mathsf S}(z_t),
\qquad \bar\lambda=\omega/(1-\phi).
\]
Thus $|\lambda_{t+1}-\lambda_t|\le |1-\phi|\,|\lambda_t-\bar\lambda|+\alpha\sup|g_{\mathsf S}|$. The lower delay bound $\Delta/(\alpha\sup|g_{\mathsf S}|)$ applies to a pure integrator, or a specified direction in which the autoregressive term opposes the displacement; it is not a universal bound for arbitrary level changes.

The paired jump experiment instead contaminates observations while leaving the latent diffusive scale unchanged. Gaussian-log completed paths cross the declared $\pm12$ diagnostic bound six times with clean fits held fixed and once after refitting. These failures are reported rather than clipped. {Among the reported cells}, jump-minus-clean RMSE changes are $0.016$ in both arms for Gaussian density power, $0.047$ and $0.049$ for Student log, and $0.023$ and $0.025$ for Student density power. They describe this realised design, without a formal significance comparison. \appref{sec::emp_paths} gives a separate descriptive location-filter illustration; latent-location consistency is not claimed.

\subsection{\texorpdfstring{{Supplementary empirical results}}{Supplementary empirical results}}\label{app:moved_empirics}
\paragraph{\textbf{Diagnostics}} We evaluate each cell out of sample through three lenses. Write
$v_t=e^{\lambda_t}$ for the one-step predictive variance. Both working densities are
variance-standardised, so no additional factor $\nu/(\nu-2)$ enters the Student-$t_\nu$ prediction or
evaluation. Point-variance tracking uses the zero-safe, forecast-dependent QLIKE component against a
nonnegative proxy $q_t$,
\begin{equation}\label{eq::qlike}
\mathrm{QLIKE}^{\rm fd}(q,v)=\frac{1}{T_{\rm te}}\sum_t
\left(\log v_t+\frac{q_t}{v_t}\right).
\end{equation}
For $q_t>0$, each summand differs from conventional QLIKE,
$q_t/v_t-\log(q_t/v_t)-1$, by the target-only term $\log q_t+1$; model rankings
and inference based on loss differences therefore agree. The displayed version
also remains defined at $q_t=0$. With $q_t=y_t^2$, the summand equals twice the
Gaussian logarithmic score, up to a $v_t$-free constant. Squared-return QLIKE is
thus a \emph{home} rule for the Gaussian log-score filter. Rankings of filters
targeting different pseudo-true paths require this qualification. The Parkinson
and Garman--Klass intraday-range proxies provide target-changing sensitivity
checks; because they omit overnight return variation, they do not measure the
same close-to-close target.

Let $r_t$ be the uncentred return, $\widehat\mu$ the training mean known at the forecast origin, $\mu_t=\E[r_t\mid\mathcal F_{t-1}]$ and $\sigma_t^2=\Var(r_t\mid\mathcal F_{t-1})$. With $q_t=(r_t-\widehat\mu)^2$ one has $\E_{t-1}q_t=\sigma_t^2+(\mu_t-\widehat\mu)^2$. For two positive predictable variance forecasts the conditional expected loss difference is
\[
 \log(v_{A,t}/v_{B,t})+\E_{t-1}q_t\,(v_{A,t}^{-1}-v_{B,t}^{-1}).
\]
Thus squared-return QLIKE estimates risk for the conditional second moment about the training mean. It ranks forecasts for conditional variance directly if the centring error vanishes, or under a separately justified proxy condition. Equal-market means determine descriptive differences. Confidence intervals for these contrasts require common dated losses.

If the conditional mean is correctly specified and volatility is latent, the observational target is $\E[e^{x_t}\mid\mathcal F_{t-1}]$ in the centred scale model, provided the innovation has conditional unit variance given the full latent and observed past. A paired loss interval also needs sampling assumptions. For $q_t$ from a Student return with $2<\nu\le4$, a predictable nonzero difference of inverse forecast variances can give a QLIKE differential with infinite second moment. A short-memory, root-$T$ approximation or a conventional block bootstrap therefore requires suitable loss-difference moments and time-dependence conditions.

\begingroup
The fitted one-step quantile is $\mathrm{VaR}_t(p)=q_\Phi(p)\sqrt{e^{\lambda_t}}$, where $q_\Phi(p)$ is the variance-standardised lower-tail quantile of the working density. Its sample exceedance frequency is $\widehat p=T_{\rm te}^{-1}\sum_t\mathbf1\{y_t<\mathrm{VaR}_t(p)\}$. \ref{sec::emp_factorial} also evaluates its count under the exact conditional-coverage null. The PIT is $u_t=F_\Phi(y_t/\sqrt{e^{\lambda_t}})$. Exact predictive specification implies uniform independent PITs, but the reported histograms are descriptive displays rather than tests of either property.
\endgroup

Descriptive event paths around the 1997 baht float and an S\&P~500 episode appear in Figure~\ref{fig::emp_impulse}; they are event-specific and do not identify the effect of an isolated observation.

Figure~\ref{fig::emp_injection} isolates the one-observation mechanism in a fixed-parameter SPX
counterfactual with a common additive shock on one 2017 date. {The perturbed return is that of 3 July 2017, in a year of unusually low volatility, so the downstream response is not confounded with other large returns. The shock $B$ is measured in units of the root mean squared return over the preceding 60 sessions and is common to all methods. Injecting a shock, rather than using a historical extreme, isolates the effect of one observation; historical extremes arrive with clustered neighbours, as the event paths above show. The content of the exercise is how the response scales with $B$, which Proposition~\ref{prop::outlier} predicts: quadratically for the Gaussian log driver, boundedly otherwise.}

At $B=100$, the Gaussian-log peak is $541.64$ log-variance units, versus $1.29$ (Student log),
$0.25$ (Student density-power), and $0.42$ (Student MMD); its log--log slope over
$B\in\{35,60,100\}$ is $1.99$. This separation accords with Proposition~\ref{prop::outlier}.
The density-power and MMD curves decline after interior maxima on this grid, which does not
establish that the tails redescend to zero. \appref{app::panel_fit_diagnostics} records the reported
event-window, shock-scaling, and path-computation protocol.

\begingroup
\paragraph{{Scope of estimation theory}}
The fixed-fit comparisons do not establish the conditions of Theorem~\ref{thm:static-estimation}. For the three bounded monotone cells, Table~\ref{inv:tab:fit} records the analytic contraction inequality at the reported fitted coefficients. A fixed parameter domain satisfying the uniform restriction, population identification and the score CLT remain separate requirements. The density-power block conditions are unverified for the observed-data law. The Gaussian-log cell has no global contraction bound on the unrestricted state space, and the capped-subsample MMD bandwidth has no supplied deterministic-limit or joint influence argument. Training-sample demeaning also requires the preliminary-estimation treatment of Corollary~\ref{cor:est-nuisance} for inferential use. None of the sample path diagnostics is treated as a population proof.
\endgroup

The two exercises serve different purposes. Figure~\ref{fig::emp_impulse} is descriptive: it
follows the Thai-baht float of 2 July 1997 and the S\&P~500 episode of 29 September 2008 using
estimates based only on pre-event data and then held fixed. The former is a regime change and the
latter a multi-day crisis, so neither path identifies the effect of an isolated observation. The
event returns are $20.77\%$ and $-9.22\%$, respectively. Around the baht float, Gaussian log reaches
$10^{7.17}$ annualised volatility, whereas the other three cells peak between $10^{1.63}$ and
$10^{1.70}$; around the S\&P episode, all four peaks lie between $10^{1.81}$ and $10^{1.88}$. The
contrast is therefore event-specific, not a universal ranking of the four filters.

\begin{table}[H]
\begingroup
\centering\small
\begin{tabular}{llrrr}
\toprule
Density & Criterion & $r^2$ QLIKE & Parkinson & Garman--Klass \\
\midrule
Gaussian & log & 1.007 & 0.637 & 0.597 \\
Gaussian & density-power & 1.020 & 0.596 & 0.551 \\
Gaussian & CRPS & 1.046 & 0.586 & 0.537 \\
Gaussian & MMD & 1.088 & 0.605 & 0.556 \\
Student-$t_6$ & log & 1.016 & 0.641 & 0.601 \\
Student-$t_6$ & density-power & 1.039 & 0.656 & 0.616 \\
Student-$t_6$ & CRPS & 1.031 & 0.650 & 0.610 \\
Student-$t_6$ & MMD & 1.063 & 0.670 & 0.629 \\
GARCH(1,1)-N & reference & 1.000 & 0.619 & 0.578 \\
EWMA/RiskMetrics & reference & 1.031 & 0.592 & 0.545 \\
\bottomrule
\end{tabular}
\caption{{Reported equal-market mean out-of-sample $\mathrm{QLIKE}^{\rm fd}$ on the final $30\%$ of each market. The columns use squared returns, Parkinson and Garman--Klass proxies, respectively. The latter two use intraday prices and concern a different target from close-to-close squared returns. Lower sample means indicate lower reported loss under that column's evaluator; no significance test or Model Confidence Set is reported.}}\label{tab::emp_qlike}
\endgroup
\end{table}
\FloatBarrier

\begingroup
\paragraph{{Interpretation of the reported fits}}
The optimiser and realised-path summaries in \appref{app::panel_fit_diagnostics} are numerical diagnostics. All forecast summaries hold the fitted coefficients and tuning values fixed; no uncertainty for a refitted forecasting procedure is inferred from them. \ref{sec::emp_factorial} assesses the published hit counts. Relative loss comparisons remain descriptive.
\endgroup

\begin{figure}[H]
\centering
\includegraphics[width=\linewidth]{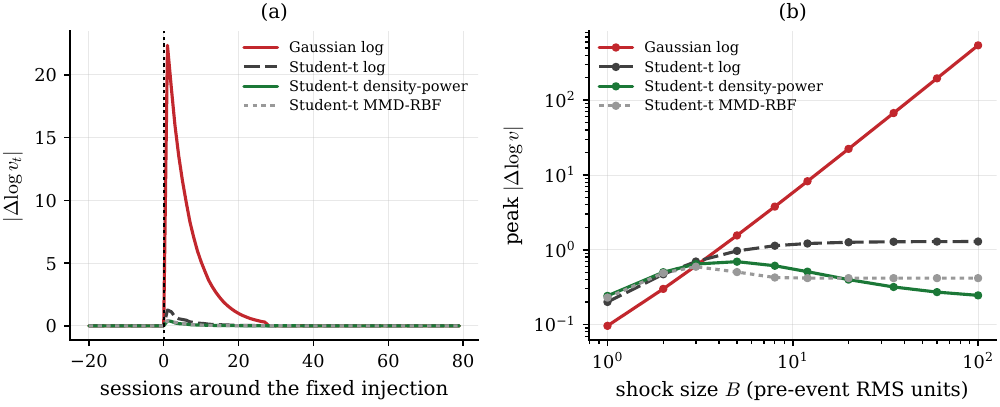}
\caption{{Descriptive fixed-parameter SPX counterfactual: downstream log-variance perturbation for $B=20$ in panel (a), and peak perturbation over the reported $B$ grid in panel (b). The common return shock is applied with fixed coefficients and without state projection. Large values of $B$ are algebraic stress tests.}}
\label{fig::emp_injection}
\end{figure}

\begin{figure}[!htbp]
\centering
\begin{tikzpicture}
\node[anchor=south west,inner sep=0] (etart) at (0,0)
  {\includegraphics[width=\linewidth]{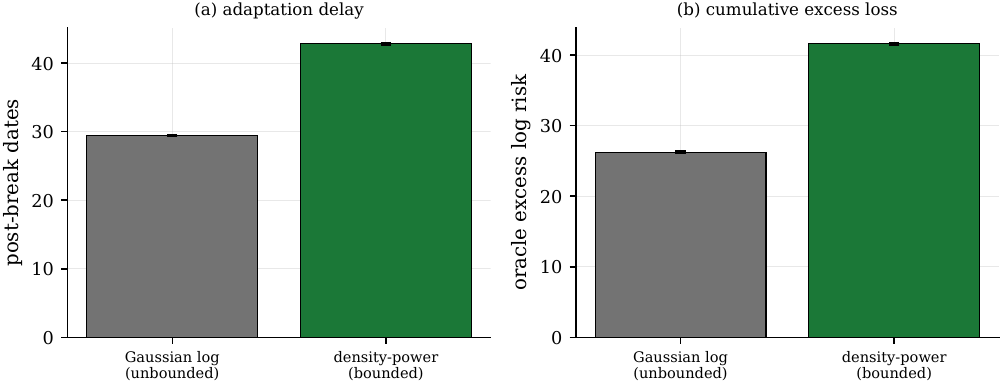}};
\begin{scope}[x={(etart.south east)},y={(etart.north west)}]
\fill[white] (.506,.225) rectangle (.533,.816);
\node[font=\small,inner sep=0pt,rotate=90,text=black] at (.519,.525) {oracle excess QLIKE};
\end{scope}
\end{tikzpicture}
\caption{{Controlled permanent variance break, $1\to9$, across $R=20000$ paired simulated paths. Panels report (a) the ten-date $90\%$ adaptation delay and (b) cumulative oracle excess QLIKE. The plotted means $26.234$ and $41.600$ are in the stated QLIKE units. The corresponding Gaussian log-risk excesses are $13.117$ and $20.800$. Bars show $\pm1$ Monte Carlo standard error. The reported density-power-minus-log differences are $13.392$ dates (SE $0.117$) and $15.366$ loss units (SE $0.100$). The error bars concern numerical precision for this finite-horizon simulation, not sampling uncertainty for the financial-data fits. The protocol is in \ref{app::panel_fit_diagnostics}.}}
\label{fig::permanent_break}
\end{figure}
\FloatBarrier

\begingroup
\subsection{Descriptive sensitivity of the fixed-fit comparison}\label{sec::emp_robustness}
Table~\ref{tab::emp_robustness_support} reports a limited set of point-summary sensitivities. The Gaussian and Student-$t_6$ GARCH fits have squared-return QLIKE means $0.9998$ and $1.0057$, and mean $1\%$ exceedance frequencies $2.00\%$ and $1.37\%$. For separately fitted Student-log and Student-CRPS cells, the reported $1\%$ exceedance frequencies decrease toward the target as the fixed shape moves from $\nu=10$ to $\nu=4$ over this grid. This does not select or estimate $\nu$ on unseen data. At one-half, one and twice the fixed training MMD bandwidth, the Student-MMD sample mean is lower than the Gaussian-MMD mean under each displayed evaluator. These finite comparisons do not establish statistical superiority, general tuning robustness or performance under monthly refitting.
\endgroup

\FloatBarrier

\subsection{Beyond scale: location filtering}\label{sec::emp_paths}
To show that criterion-dependent propagation is not specific to volatility, Figure~\ref{fig::location}
applies Gaussian location filters, $y_t=\mu_t+\sigma\varepsilon_t$, to the Nile annual-flow series
(100 observations, 1871--1970, in $10^8\,\mathrm{m}^3$). Panel~(a) gives a descriptive comparison
within the training period across the 1899 downward level shift \citep{cobb1978nile}; panel~(b) isolates propagation
from a common artificial outlier. Holding the Gaussian density fixed, all parameters are estimated
on 1871--1944 and frozen before the 1945 injection. The resulting 1946 displacement is
$1.615\widehat\sigma_{\rm log}$ for Gaussian log, $0.465\widehat\sigma_{\rm log}$ for CRPS, and
$0.066\widehat\sigma_{\rm log}$ for density-power. Thus the bounded CRPS and redescending
density-power location updates attenuate isolated contamination while retaining adaptation to the
sustained historical shift. {These location results are descriptive. The Nile observations are annual. For a different observation scheme, $dY_t=\mu_t\,dt+\sigma\,dW_t$, one increment of length $h$ gives only $h/\sigma^2$ Fisher information for a locally constant $\mu$. On a fixed time interval this scaling can preclude consistent pointwise tracking, as discussed by \citet{beutner2026consistency}. The Nile comparison gives descriptive paths. Consistency of these filters for the latent location is therefore not claimed.}

\FloatBarrier

\subsection{Fitted-point contraction diagnostics}\label{app:empirical_contraction}
Table~\ref{inv:tab:fit} applies the analytic derivative bounds in Table~\ref{inv:tab:const} to the reported fitted coefficients. These calculations concern the fitted points; parameter-uniform population invertibility and the other assumptions of the estimation theorem remain separate requirements.
\begin{table}[ht]
\begingroup
\centering\scriptsize
\begin{tabular}{lccccc}
\toprule
Cell & $\widehat\phi$ & $\widehat\alpha$ & max $\widehat\alpha M_+$ & $L(\widehat\vartheta)$ & path max \\
\midrule
Gaussian CRPS & [0.9844,0.9952] & [0.396,0.681] & 0.200 & [0.9844,0.9952] & 0.9952 \\
Student-$t_6$ log & [0.9829,0.9942] & [0.073,0.128] & 0.224 & [0.9829,0.9942] & 0.9942 \\
Student-$t_6$ CRPS & [0.9840,0.9950] & [0.462,0.796] & 0.184 & [0.9840,0.9950] & 0.9950 \\
\bottomrule
\end{tabular}
\caption{{Analytic contraction calculations at the reported fitted coefficients. For these monotone cells, $\widehat\alpha M_+\le2\widehat\phi$ and the modulus equals $\widehat\phi$. This does not verify the uniform inequality over the actual estimation domain or the remaining population assumptions. The final column is the reported maximum derivative along the fitted paths from Table~\ref{tab::emp_stability}.}}\label{inv:tab:fit}
\endgroup
\end{table}
\FloatBarrier

\subsection{Historical-event, PIT and location figures}\label{app:empirical_figures}
\begin{figure}[t]
\centering
\includegraphics[width=\linewidth]{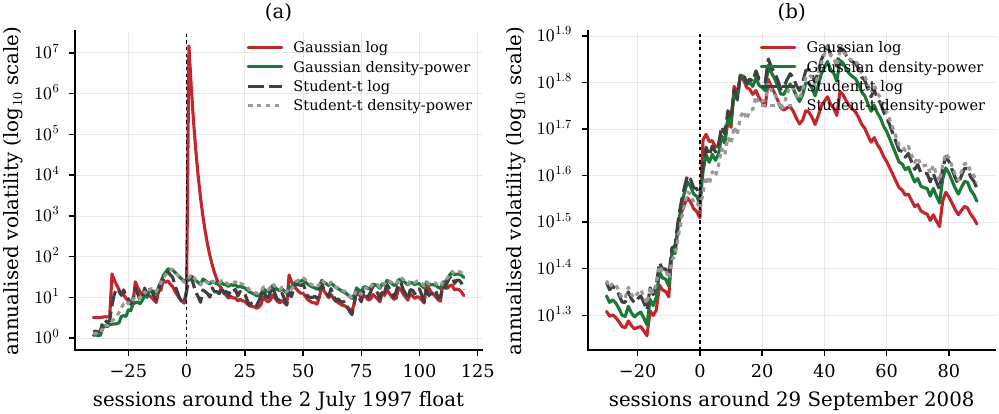}
\caption{Pre-event fixed-fit paths around the 2 July 1997 Thai-baht float (left) and
29 September 2008 S\&P~500 episode (right), for four density--criterion cells. Each fit is frozen
before the event, and annualised volatility is shown on a $\log_{10}$ scale. These regime/crisis
paths are descriptive and do not causally identify an isolated-observation response.}
\label{fig::emp_impulse}
\end{figure}

\begin{figure}[t]
\centering
\includegraphics[width=\linewidth]{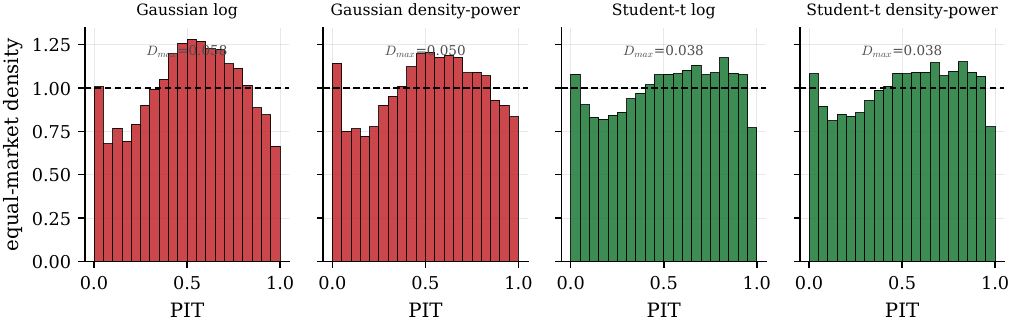}
\caption{{Descriptive equal-market PIT histograms, with normalised market histograms weighted by $1/12$. Red denotes Gaussian-family cells and green Student-$t_6$ cells. The dashed line marks uniform density. The displayed $D_{\max}$ is the largest computed discrepancy between the equal-market PIT CDF and the uniform CDF on the evaluation grid. It is a descriptive distance, not a p-value or a calibrated uniformity or serial-independence test.}}
\label{fig::pit}
\end{figure}

\begin{figure}[t]
\centering
\includegraphics[width=\linewidth]{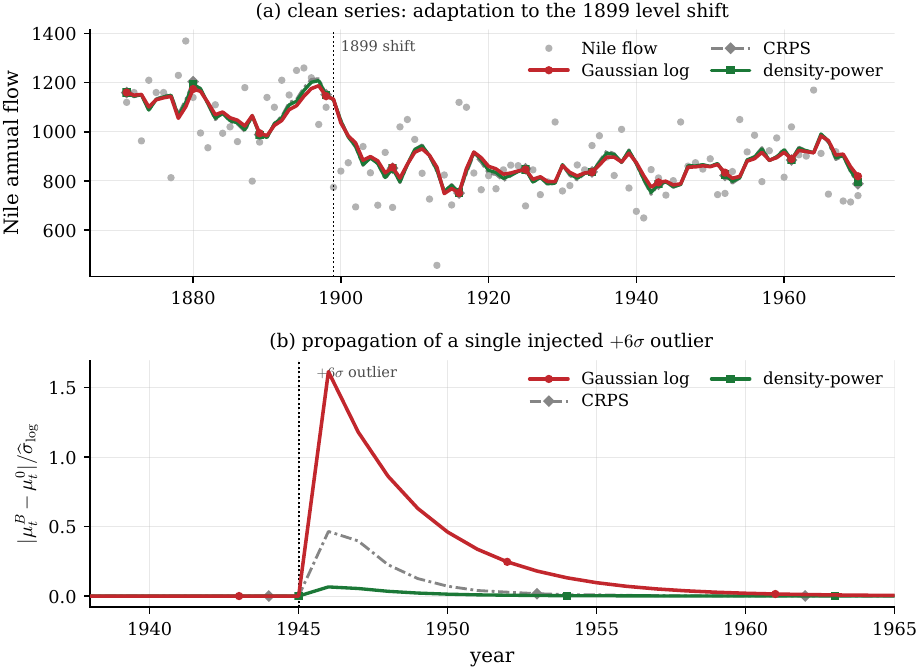}
\caption{Criterion-only Gaussian location exercise, Nile annual flow. (a)~Descriptive clean paths
across the in-sample 1899 level shift; (b)~the difference between a common
$+6\widehat\sigma_{\rm log}$ injection at 1945 and each method's clean path. Peak displacements in
1946 are $1.615\widehat\sigma_{\rm log}$ for Gaussian log, $0.465\widehat\sigma_{\rm log}$ for CRPS,
and $0.066\widehat\sigma_{\rm log}$ for density-power. The protocol is given
above.}
\label{fig::location}
\end{figure}

\FloatBarrier

\clearpage
\setcounter{section}{0}
\setcounter{equation}{0}
\setcounter{theorem}{0}
\setcounter{table}{0}
\setcounter{figure}{0}
\renewcommand{\thesection}{S\arabic{section}}
\renewcommand{\thetheorem}{\thesection.\arabic{theorem}}
\renewcommand{\theequation}{\thesection.\arabic{equation}}
\renewcommand{\thetable}{S\arabic{table}}
\renewcommand{\thefigure}{S\arabic{figure}}
\renewcommand{\theHsection}{S\arabic{section}}
\renewcommand{\theHtheorem}{\thesection.\arabic{theorem}}
\renewcommand{\theHequation}{\thesection.\arabic{equation}}
\renewcommand{\theHtable}{S\arabic{table}}
\renewcommand{\theHfigure}{S\arabic{figure}}
\phantomsection
\addcontentsline{toc}{section}{Supplementary Material}
\begin{center}\Large\bfseries Supplementary Material\end{center}
\noindent This supplement contains the conditional batch-estimation theory, continuous-time limits, supporting proofs, score formulas, deterministic illustrations and vector extension. Numbers prefixed by S refer to this supplement; the stationary proofs are in Appendix A of the main paper.
\medskip
\begingroup
\section{{Conditional static-parameter estimation}}\label{sec:static-estimation}

For completeness, we give the fixed-frequency batch estimation result for a specified density--criterion--scaling cell. The density, scoring rule, scaling and tuning constants are held fixed. This is a $T\to\infty$ problem, distinct from population gain selection as $h\downarrow0$. \ref{app:static-estimation} states the conditions and proof.

The comparison with PRADA is material here. Theorems~2--4 in the 20 November 2026 version of \citet{depunder2026prada} cover consistency, functional correct specification and misspecification. Our fixed-frequency result records the assumptions and normalisations for the criteria used in this paper. The persistence-dependent stationary moment result in the main text concerns a different object.

Let $\Theta\subset\mathbb R\times(-1,1)\times(0,\infty)$ be the admissible
parameter set, and write $\vartheta=(\omega,\phi,\alpha)'$ and
$u_{\mathsf S}(y,\lambda)
=G_{\mathsf S}(\lambda)\psi_{\mathsf S}(y,\lambda)$. For each
$\vartheta\in\Theta$, let
$f_\vartheta(\lambda,y)
=\omega+\phi\lambda+\alpha u_{\mathsf S}(y,\lambda)$.
We consider parameters for which the recursions and scoring losses below are
well defined on their state domain.
The computed filter starts at
$\widehat\lambda_1(\vartheta)=\omega/(1-\phi)$ and follows
$\widehat\lambda_{t+1}(\vartheta)
=f_\vartheta(\widehat\lambda_t(\vartheta),y_t)$. Its sample criterion is
\begin{equation}\label{eq:main-est-sample}
  \widehat Q_T(\vartheta)
  =\frac1T\sum_{t=1}^T
    \mathsf S\bigl(P_{\widehat\lambda_t(\vartheta)},y_t\bigr).
\end{equation}
The population criterion instead uses the stationary, infinite-past causal
solution $\lambda_t(\vartheta)$ of the same recursion, with
$\lambda_t(\vartheta)$ measurable with respect to $\mathcal F_{t-1}$. Define
\begin{equation}\label{eq:main-est-population}
  q_t(\vartheta)
  =\mathsf S\bigl(P_{\lambda_t(\vartheta)},y_t\bigr),
  \qquad
  Q(\vartheta)=\mathsf E[q_0(\vartheta)],
  \qquad
  \vartheta^\star=\underset{\vartheta\in\Theta}{\arg\min}\,Q(\vartheta).
\end{equation}
The assumptions, including stationary-solution existence, finite-start
forgetting, and identification of the minimiser, are stated in
\appref{app:static-estimation}.

\begin{theorem}[Batch minimum-scoring-risk estimation]
\label{thm:static-estimation}
Suppose Assumptions~\ref{ass:est-data}--\ref{ass:est-criterion} hold.  Let
$\widehat\vartheta_T$ be a measurable numerical near-minimiser satisfying
\[
  \widehat Q_T(\widehat\vartheta_T)
  \le\inf_{\vartheta\in\Theta}\widehat Q_T(\vartheta)+{o_p(1)}.
\]
Then
\[
  \sup_{\vartheta\in\Theta}
  |\widehat Q_T(\vartheta)-Q(\vartheta)|\longrightarrow0
  \quad\text{almost surely},
  \qquad
  \widehat\vartheta_T\xrightarrow{p}\vartheta^\star.
\]
If, in addition, Assumptions~\ref{ass:est-differentiability} and
\ref{ass:est-clt} hold, $\vartheta^\star$ is interior, and
$\|\nabla_\vartheta\widehat Q_T(\widehat\vartheta_T)\|
=o_p(T^{-1/2})$, then
\[
  \sqrt T(\widehat\vartheta_T-\vartheta^\star)
  \Longrightarrow
  \mathcal{N}\!\left(0,\mathcal A^{-1}\Omega(\mathcal A^{-1})^{\top}\right),
\]
where
\(\mathcal A=\mathsf E[\nabla_\vartheta^2q_0(\vartheta^\star)]\)
and $\Omega$ is the limiting covariance in the score CLT for
$\mathfrak s_t(\vartheta^\star)=\nabla_\vartheta q_t(\vartheta^\star)$.
\end{theorem}
\proofref{app:est-proof}

The static pseudo-true parameter $\vartheta^\star$ minimises the stationary risk of an entire recursion, not the datewise scoring-risk projection or its autoregressive composite centre. The dependent-data sandwich accounts for the long-run covariance of the score under misspecification. Corollary~\ref{cor:est-correct-spec} gives the correctly specified martingale case, and Corollary~\ref{cor:est-nuisance} gives the joint influence expansion when preliminary quantities are estimated. Propositions~\ref{inv:prop:mono} and~\ref{inv:prop:rstep} provide sufficient uniform state and derivative-forgetting conditions for selected bounded drivers. For example, for a monotone driver the deterministic modulus is $L(\vartheta)=\max\{|\varphi|,|\varphi-\alpha M_+|\}$, and the admissible parameter domain must satisfy $\sup_{\vartheta\in\Theta}L(\vartheta)<1$. These are conditional applications of the established minimum-score and stochastic-recurrence literature. The numerical illustrations do not verify their population, identification or parameter-uniform assumptions; these are stated in \ref{app:static-estimation}.

\section{{Continuous-time specialisations}}\label{section::continuous_time}

For completeness, we record the two diffusion-approximation regimes. For the score-driven scale model with predictable variance $c_t=\Gamma(\lambda_t)$, the recursions and conditional characteristics are specified in \ref{app:ct_formal}. There are two distinct regimes. In the centred regime $\alpha_h/\sqrt h\to\alpha$, $\omega_h/h\to\omega$, $(1-\varphi_h)/h\to\kappa$ and conditional means $o(\sqrt h)$, Theorem~\ref{th::diffusion_limit} gives a joint price--scale diffusion under the stated moment, containment and well-posedness conditions. The filtered log-variance drift is $\omega-\kappa\lambda_t$, its instantaneous variance is $\alpha^2\Xi_{\mathsf S}(\lambda_t)$, and its covariance with the return is determined by $C_{\mathsf S}(\lambda_t)$ (zero for the symmetric examples in Proposition~\ref{prop:verified_score_arrays}). In the non-centred regime with gain $a_h/h\to\bar\alpha$, Theorem~\ref{th::mean_flow} gives an ordinary differential equation for the scale state from a fixed initial scale, while the return remains stochastic. Both results specialise standard diffusion approximations \citep{nelson1990arch,corradi2000reconsidering,buccheri2021continuous,wuhe2026continuous}; their full statements, primitive checks and proofs are in \ref{app::appendix_section_6}. Neither by itself establishes convergence of invariant tracking loss.

\endgroup

\FloatBarrier

\section{\texorpdfstring{{A Gaussian scale example for Section~\ref{section::introduction}}}{A Gaussian scale example}}
\label{app::appendix_section_2}
{The supplementary results are grouped by topic: filter construction and one-step identities (\ref{app::appendix_section_2}--\ref{app::appendix_section_5}), estimation (\ref{app:static-estimation}), diffusion limits (\ref{app::appendix_section_6}), local tracking (\ref{app::appendix_section_7}), score formulas and numerical illustrations (\ref{app::formulas}), and the vector extension (\ref{app::multivariate}). Proof subsections have labels for direct cross-reference.}

\subsection{Gaussian scale example}
\begin{example}\label{ex::gaussian-scale}
Consider the following Gaussian scale family model: $\mathfrak{p}(y \mid \lambda)=e^{-\frac{\lambda}{2}}\phi_{\mathcal N}(e^{-\frac{\lambda}{2}}y)$, where $z=e^{-\frac{\lambda}{2}}y$ and $\phi_{\mathcal N}$ is the standard normal density. In this case $\psi_{\mathsf{S}_{\rm log}}(y,\lambda)=\frac{1}{2}(z^2-1)$, which is unbounded and has finite variance only if the fourth moment of the standardised true innovation is finite. For $\beta>0$, use the paper's fixed density-power normalisation
\begin{equation*}
    \mathsf{S}_{\beta}(P,y)=\int_{-\infty}^{+\infty}\mathfrak{p}(x \mid \lambda)^{1+\beta}\,\ud x -\left(1+\frac1\beta\right)\mathfrak{p}(y \mid \lambda)^{\beta},
\end{equation*}
where $I_{\beta}(\lambda)=\int_{-\infty}^{+\infty}\mathfrak{p}(x \mid \lambda)^{1+\beta}\,\ud x=e^{-\frac{\beta\lambda}{2}}(2 \pi)^{-\beta/2}(1+\beta)^{-1/2}$. Therefore,
\begin{equation*}
    \psi_{\mathsf{S}_{\beta}}(y,\lambda)=\frac{\beta}{2}I_{\beta}(\lambda)+(1+\beta)\mathfrak{p}(y \mid \lambda)^{\beta}\frac{z^2-1}{2},
\end{equation*}
which is bounded in $z$ at each fixed state. The implemented scaling is $G_\beta(\lambda)=2e^{\beta\lambda/2}$, so the fully scaled update is
\[
 u_\beta(z)=\beta(2\pi)^{-\beta/2}(1+\beta)^{-1/2}+(1+\beta)\phi_{\mathcal N}(z)^\beta(z^2-1),
\]
which is state free and bounded. It is this fully scaled object, not the raw derivative alone, that determines path robustness and update moments.
\end{example}

\section{\texorpdfstring{{One-step conditional criteria}}{One-step conditional criteria}}\label{app:onestep}
\begingroup
\begin{assumption}\label{ass:scalar_conditional_geometry}
Fix $t$. Let $I_t\subset\Lambda$ be an open interval containing $\lambda_t$ and $\lambda^\star_{\mathsf S,t}$, and suppose the following hold almost surely:
\begin{enumerate}[label=(\roman*)]
\item $\mathsf R_{\mathsf S,t}$ is twice continuously differentiable on $I_t$;
\item $\lambda^\star_{\mathsf S,t}$ is its unique minimiser on $I_t$;
\item $H_{\mathsf S,t}(\lambda)=G_{\mathsf S,t}(\lambda)\mathsf R'_{\mathsf S,t}(\lambda)$ is continuously differentiable and strictly increasing on $I_t$, with $0<H'_{\mathsf S,t}\le\bar c_t$ for an $\mathcal F_{t-1}$-measurable $\bar c_t<\infty$.
\end{enumerate}
\end{assumption}

\begin{lemma}[One-step conditional criteria]\label{lem:onestep}
Suppose Assumptions~\ref{ass:measurable_target}, \ref{ass:conditional_differentiation} and~\ref{ass:scalar_conditional_geometry} hold.
\begin{enumerate}[label=(\roman*)]
\item If $0<\alpha<2/\bar c_t$, then $|\mathsf E_{t-1}[\lambda_{t+1}]-\lambda^\star_{\mathsf S,t}|<|\lambda_t-\lambda^\star_{\mathsf S,t}|$ whenever $\lambda_t\neq\lambda^\star_{\mathsf S,t}$.
\item If $\mathsf E_{t-1}[\psi_{\mathsf S}^2(y_t,\lambda_t)]<\infty$, then
\begin{equation}\label{eq::mse}
\mathsf E_{t-1}[(\lambda_{t+1}-\lambda^\star_{\mathsf S,t})^2]=(\lambda_t-\lambda^\star_{\mathsf S,t})^2-2\alpha H_{\mathsf S,t}(\lambda_t)(\lambda_t-\lambda^\star_{\mathsf S,t})+\alpha^2G_{\mathsf S,t}(\lambda_t)^2\mathsf E_{t-1}[\psi_{\mathsf S}^2(y_t,\lambda_t)].
\end{equation}
Off target, the conditional mean-squared distance falls if and only if $0<\alpha<2\alpha^{\rm or}_t$, where
\[
 \alpha^{\rm or}_t=\frac{H_{\mathsf S,t}(\lambda_t)(\lambda_t-\lambda^\star_{\mathsf S,t})}{G_{\mathsf S,t}(\lambda_t)^2\,\mathsf E_{t-1}[\psi_{\mathsf S}^2(y_t,\lambda_t)]},
\]
and the fall is largest at $\alpha^{\rm or}_t$. On target, a non-degenerate innovation raises it for every $\alpha>0$.
\item If $H'_{\mathsf S,t}\le\bar c$ for a deterministic $\bar c$ at every date and the recursion stays in the intervals $I_t$, then every fixed $\alpha\in(0,2/\bar c)$ gives~(i) at all dates.
\end{enumerate}
\end{lemma}

\begin{proof}
\emph{(i)} By~\eqref{eq::gradient_update} and $H_{\mathsf S,t}(\lambda^\star_{\mathsf S,t})=0$, the mean value theorem gives $\mathsf E_{t-1}[\lambda_{t+1}]-\lambda^\star_{\mathsf S,t}=\{1-\alpha H'_{\mathsf S,t}(\zeta)\}(\lambda_t-\lambda^\star_{\mathsf S,t})$ for some $\zeta\in I_t$. Since $0<\alpha H'_{\mathsf S,t}(\zeta)<2$, the factor lies in $(-1,1)$.

\emph{(ii)} Expanding $(\lambda_t-\lambda^\star_{\mathsf S,t}+\alpha G_{\mathsf S,t}\psi_{\mathsf S})^2$ and using $\mathsf E_{t-1}[\psi_{\mathsf S}(y_t,\lambda_t)]=-\mathsf R'_{\mathsf S,t}(\lambda_t)$ gives~\eqref{eq::mse}. The change in mean-squared distance is the quadratic $\alpha\{\alpha G^2\mathsf E\psi_{\mathsf S}^2-2H(\lambda_t-\lambda^\star)\}$. Off target, $H(\lambda_t)(\lambda_t-\lambda^\star)>0$ by strict monotonicity, and $\mathsf E\psi_{\mathsf S}^2\ge(\mathsf R')^2>0$. The quadratic is therefore negative exactly on $(0,2\alpha^{\rm or}_t)$, with its minimum at $\alpha^{\rm or}_t$.

\emph{(iii)} Immediate from~(i).
\end{proof}

Parts~(i) and~(iii) adapt the contraction step of \citet[Theorem~1]{gorgi2024optimality}, and part~(ii) makes their Corollary~1 exact. \citet{de2024kullback} show that these criteria share their first-order condition with the expected Kullback--Leibler criterion and require a monotone mean field, here Assumption~\ref{ass:scalar_conditional_geometry}(iii). PRADA \citep{depunder2026prada} establishes the expected-divergence counterpart for proper scoring rules. The moving-target analysis of Section~\ref{sec::local_consistency} replaces these one-date statements.
\endgroup

\section{\texorpdfstring{Proofs for Section~\ref{section::introduction}: {local structure and propagation}}{Proofs for Section 2: local structure and propagation}}\label{app::appendix_section_5}
\subsection{\texorpdfstring{{The autoregressive recursion}}{The autoregressive recursion}}\label{subsec:applied_ar_recursion}
\begingroup
Definition~\ref{def::scoring-rule-filter} concerns the pure-integrator update $\lambda_{t+1}=\lambda_t+\alpha u_{\mathsf S,t}$. The continuous-time and static-estimation sections also consider $\lambda_{t+1}=\omega+\varphi\lambda_t+\alpha u_{\mathsf S,t}$ with $|\varphi|<1$. With $\bar\lambda=\omega/(1-\varphi)$ and $H_{\mathsf S,t}(\lambda)=G_{\mathsf S,t}(\lambda)\mathsf R'_{\mathsf S,t}(\lambda)$, its conditional mean field is
\[
 {H^{\mathrm{AR}}_{\mathsf S,t}}(\lambda)=\alpha^{-1}(1-\varphi)(\lambda-\bar\lambda)+H_{\mathsf S,t}(\lambda),
 \qquad
 \mathsf E_{t-1}[\lambda_{t+1}]=\lambda_t-\alpha{H^{\mathrm{AR}}_{\mathsf S,t}}(\lambda_t).
\]
Its zero ${\lambda^{\mathrm{AR}}_{\mathsf S,t}}$ is the composite centre. It coincides with $\lambda^\star_{\mathsf S,t}$ when $\omega=(1-\varphi)\lambda^\star_{\mathsf S,t}$, and is close to it when the autoregressive pull is negligible at the relevant scale. In the high-frequency analysis of Section~\ref{sec::local_consistency} the pull is $O(h)$, and the target is $\lambda^\star_{\mathsf S,t}$.
\endgroup

\subsection{\texorpdfstring{{Conditional differentiation and dynamic equivalence}}{Conditional differentiation and dynamic equivalence}}\label{app:moved_section2}
\begin{proposition}[Local mean--noise decomposition]\label{prop::local_mean}
Let $\lambda^\dagger$ be an interior centred point, so that $\mathsf R'_{\mathsf S,t}(\lambda^\dagger)=0$. Suppose Assumption~\ref{ass:conditional_differentiation} holds on a neighbourhood of $\lambda^\dagger$, the conditional risk is twice continuously differentiable there, $J_{\mathsf S,t}(\lambda^\dagger)>0$, $K_{\mathsf S,t}(\lambda^\dagger)<\infty$, $G_{\mathsf S,t}$ is continuous at $\lambda^\dagger$, and $\lambda\mapsto\mathsf E_{t-1}[\psi_{\mathsf S}^2(y_t,\lambda)]$ is continuous at $\lambda^\dagger$. Then, almost surely as $\lambda_t\to\lambda^\dagger$,
\begin{equation}\label{eq::mean_local_mean}
    \mathsf{E}_{t-1}[\lambda_{t+1}-\lambda^\dagger]
    =\{1-\alpha G_{\mathsf S,t}(\lambda^\dagger)J_{\mathsf S,t}(\lambda^\dagger)\}(\lambda_t-\lambda^\dagger)+o(|\lambda_t-\lambda^\dagger|),
\end{equation}
and
\begin{equation}\label{eq::variance_local_mean}
    \mathsf{Var}_{t-1}(\lambda_{t+1})=\alpha^2G_{\mathsf S,t}(\lambda^\dagger)^2K_{\mathsf S,t}(\lambda^\dagger)+o(1).
\end{equation}
\end{proposition}
\proofref{app:proof_local_mean}

\begin{assumption}[Conditional differentiation]
\label{ass:conditional_differentiation}
Fix $t$ and $\theta$. For every compact $K\subset\operatorname{int}(\Lambda)$, choose an anchor $\lambda_K\in K$. The loss $(y,\lambda)\mapsto\ell_{\mathsf S}(y,\lambda;\theta)$ is jointly measurable; for $\widetilde P_t$-almost every $y$, it is continuously differentiable in $\lambda$ on $K$; and
\[
 \mathsf E_{t-1}[|\ell_{\mathsf S}(y_t,\lambda_K;\theta)|]<\infty,
 \qquad
 \sup_{\lambda\in K}|\partial_\lambda\ell_{\mathsf S}(y,\lambda;\theta)|\le M_{t,K}(y),
 \qquad
 \mathsf E_{t-1}[M_{t,K}(y_t)]<\infty
\]
almost surely. The conditional risk is understood as the jointly measurable version obtained by integrating against the conditional probability kernel $\widetilde P_t(\omega,\mathrm dy)$. Then, almost surely, $\mathsf R_{\mathsf S,t}$ is continuously differentiable on $K$ and, simultaneously for every $\lambda\in K$,
\(
 \mathsf R'_{\mathsf S,t}(\lambda)
 =\mathsf E_{t-1}[\partial_\lambda\mathsf S(P_{\lambda,\theta},y_t)]
\).
\end{assumption}

\begingroup
\begin{remark}[Scaling and dynamic equivalence]\label{rem:dynamic_equivalence}
The filtered path depends on the rule only through the scaled derivative $\alpha G_{\mathsf S,t}(\lambda)\psi_{\mathsf S}(y,\lambda)$. Two rule--scaling pairs with common intercept, autoregressive coefficient and initial state therefore generate the same path when $\alpha_1G_{1,t}\psi_1=\alpha_2G_{2,t}\psi_2$. In particular, if $\psi_2=a_t\psi_1$ with $a_t>0$, path equality holds when $\alpha_2G_{2,t}a_t=\alpha_1G_{1,t}$.

If $\psi_2(y,\lambda)=a(\lambda)\psi_1(y,\lambda)$ in a neighbourhood of a common centred target, with positive differentiable $a$ independent of $y$, then $J_2=aJ_1$ at the target because the term containing $a'$ has zero expectation. Also $K_2=a^2K_1$, so the sandwich ratio is unchanged. Cross-rule comparisons therefore use $GJ$ and $G^2K$ under generic scaling, and $1$ and $K/J^2$ under curvature scaling. The tail class below likewise belongs to the fully scaled update.
\end{remark}
{The remark is elementary. It is kept because the invariance of $K/J^2$ justifies comparing rules through $\kappa_{\mathsf S}$ in Section~\ref{sec::local_consistency}.}
\endgroup

\subsection{Proof of Proposition~\ref{prop::local_mean}}\label{app:proof_local_mean}
By twice continuous differentiability and centring,
\begin{equation}
    \mathsf{R}^{'}_{\mathsf{S},t}(\lambda_t)=J_{\mathsf{S},t}(\lambda^{\dagger})(\lambda_t-\lambda^{\dagger}) + o(\mid \lambda_t-\lambda^{\dagger} \mid).
\end{equation}
By continuity of the predictable scaling,
\begin{equation*}
    G_{\mathsf{S},t}(\lambda_t)\mathsf{R}^{'}_{\mathsf{S},t}(\lambda_t)=G_{\mathsf{S},t}(\lambda^\dagger)J_{\mathsf{S},t}(\lambda^\dagger)(\lambda_t-\lambda^{\dagger}) + o(|\lambda_t-\lambda^{\dagger}|).
\end{equation*}
Substituting this expansion into the conditional-mean {identity~\eqref{eq::gradient_update}} gives \eqref{eq::mean_local_mean}. We turn now to the conditional variance,
\begin{equation*}
    \mathsf{Var}_{t-1}[\lambda_{t+1}]=\alpha^2 G^2_{\mathsf{S},t}(\lambda_t)\mathsf{Var}_{t-1}[\psi_{\mathsf{S}}(y_t,\lambda_t)].
\end{equation*}
Moreover, $\mathsf{E}_{t-1}[\psi_{\mathsf{S}}(y_t,\lambda_t)]=-\mathsf{R}^{'}_{\mathsf{S},t}(\lambda_t)=O(|\lambda_t-\lambda^{\dagger}|)$, so its squared conditional mean is $o(1)$. By continuity of the second moment, $\mathsf{Var}_{t-1}[\psi_{\mathsf{S}}(y_t,\lambda_t)]=K_{\mathsf{S},t}(\lambda^\dagger)+o(1)$. Substituting this and $G^2_{\mathsf S,t}(\lambda_t)=G^2_{\mathsf S,t}(\lambda^\dagger)+o(1)$ gives~\eqref{eq::variance_local_mean}.

\subsection{\texorpdfstring{{Proof of Proposition~\ref{prop::outlier}}}{Proof of the outlier-propagation proposition}}\label{app:proof_outlier}
\begingroup
\emph{(i)} The two paths satisfy the same linear recursion after $t$, so their difference evolves as $d_{t+1+k}=\varphi^kd_{t+1}$, with $d_{t+1}=\alpha\{g_{\mathsf S}(B)-g_{\mathsf S}(z_0)\}$.

\emph{(ii)} At date $t$ the common map gives $|d_{t+1}|=\alpha|u_{\mathsf S,t}(B,\lambda_t)-u_{\mathsf S,t}(y_0,\lambda_t)|$. Afterwards the common $q$-Lipschitz maps give $|d_{t+1+k}|\le q^k|d_{t+1}|$. The difference of the two values is bounded by $\operatorname{osc}_y u_{\mathsf S,t}(\cdot,\lambda_t)$, with the environment and common state held fixed. This quantity is the same for all subsequent indices $k$, so summing the geometric series gives the stated cumulative bound. Taking a deterministic uniform upper bound $\Omega_{\mathsf S}$ yields the global version.

\emph{(iii)} By~\eqref{eq::scale_derivative} and $0\le m_{\mathsf S}\le M_+$, $\partial_\lambda f\in[\varphi-\alpha M_+,\varphi]$, so $|\partial_\lambda f|\le\max\{|\varphi|,|\varphi-\alpha M_+|\}$, and the mean value theorem gives the Lipschitz bound. If $0\le\alpha M_+\le2\varphi$, then $|\varphi-\alpha M_+|\le\varphi$. For the Gaussian log score, $\sup_\lambda|\varphi-\alpha y^2e^{-\lambda}|=\infty$ for every $y\neq0$. \qed
\endgroup

\section{Static-parameter estimation}\label{app:static-estimation}
This appendix establishes the large-sample properties of the batch minimum-scoring-risk estimator introduced in Section~\ref{sec:static-estimation}. Throughout the appendix, we fix one density-criterion-scaling cell. Thus, the working family, the negatively oriented differentiable scoring rule, the time-homogeneous state-based scaling, and all density and criterion tuning constants are held fixed. We suppress the cell index and write
\[
  \ell(y,\lambda)=\mathsf{S}(P_\lambda,y),
  \qquad
  \psi_{\mathsf{S}}(y,\lambda)=-\partial_\lambda \ell(y,\lambda),
  \qquad
  u_{\mathsf{S}}(y,\lambda)=G_{\mathsf{S}}(\lambda)\psi_{\mathsf{S}}(y,\lambda).
\]
For \(\vartheta=(\omega,\phi,\alpha)'\), write
\(f_\vartheta(\lambda,y)=\omega+\phi\lambda+\alpha u_{\mathsf{S}}(y,\lambda).\)
The finite-start filter used in computation is
\begin{equation}\label{eq:est-finite-filter}
  \widehat\lambda_{t+1}(\vartheta)
  =f_\vartheta(\widehat\lambda_t(\vartheta),y_t),
  \qquad
  \widehat\lambda_1(\vartheta)=a(\vartheta)=\frac{\omega}{1-\phi}.
\end{equation}
Its sample criterion is $\widehat Q_T$ in~\eqref{eq:main-est-sample},
with $\ell$ as shorthand for the same scoring loss.
We allow for a numerical near-minimiser satisfying
\begin{equation}\label{eq:est-near-minimiser}
  \widehat Q_T(\widehat\vartheta_T)
  \leq \inf_{\vartheta\in\Theta}\widehat Q_T(\vartheta)+r_T,
  \qquad r_T\geq0.
\end{equation}

The population objective must instead be defined from the stationary,
infinite-past filter.  Therefore, let \(\lambda_t(\vartheta)\) be a stationary causal
solution of
\begin{equation}\label{eq:est-stationary-filter}
  \lambda_{t+1}(\vartheta)
  =f_\vartheta(\lambda_t(\vartheta),y_t),
  \qquad
  \lambda_t(\vartheta)\ \text{is }\mathcal F_{t-1}\text{-measurable},
\end{equation}
where \(\mathcal F_t=\sigma(y_s:s\leq t)\). The stationary criterion
$Q(\vartheta)$ and its minimiser $\vartheta^\star$ are defined in
\eqref{eq:main-est-population}; here
$q_t(\vartheta)=\ell(y_t,\lambda_t(\vartheta))$. The distinction between
this static recursion parameter and the date-specific targets is discussed
after Theorem~\ref{thm:static-estimation}.

\subsection{Assumptions}\label{app:est-assumptions}

For a continuous function \(x:\Theta\to\mathbb R\), write
\(\|x\|_\Theta=\sup_{\vartheta\in\Theta}|x(\vartheta)|\); we use the same notation with the operator norm for vector- and matrix-valued functions.

\begin{assumption}[Data and parameter set]\label{ass:est-data}
The process \(\{y_t\}_{t\in\mathbb Z}\) is strictly stationary and ergodic.
The parameter set \(\Theta\) is a compact subset of
\(\mathbb R\times(-1,1)\times(0,\infty)\), and there is
\(\varepsilon_\phi>0\) such that
\(|\phi|\leq1-\varepsilon_\phi\) on \(\Theta\).  The initialisation map
\(a(\vartheta)=\omega/(1-\phi)\) is continuous and bounded on \(\Theta\).
\end{assumption}

\begin{assumption}[Uniform invertibility and finite-start forgetting]
\label{ass:est-invertibility}
There is a stationary and ergodic \(C(\Theta)\)-valued process
\(\lambda_t(\cdot)\) satisfying~\eqref{eq:est-stationary-filter}.  It is
unique among \(C(\Theta)\)-valued stationary causal solutions.  Moreover, there are
a deterministic \(\rho_\lambda\in(0,1)\) and an almost surely finite random
variable \(C_\lambda\) such that
\begin{equation}\label{eq:est-state-forgetting}
  \|\widehat\lambda_t-\lambda_t\|_\Theta
  \leq C_\lambda\rho_\lambda^t,
  \qquad t\geq1,
  \quad\text{almost surely}.
\end{equation}
\end{assumption}

\begin{assumption}[Criterion regularity and identification]
\label{ass:est-criterion}
For almost every sample path, \(\vartheta\mapsto q_t(\vartheta)\) is
continuous on \(\Theta\).  There are stationary nonnegative random
variables \(B_t,M_t,D_t\), with
\(\mathsf E[B_0+M_0+D_0]<\infty\), and one fixed
\(\bar\vartheta\in\Theta\), such that, almost surely,
\begin{align}
  |q_t(\bar\vartheta)|&\leq B_t,\label{eq:est-anchor}\\
  |q_t(\vartheta)-q_t(\vartheta')|
  &\leq M_t\|\vartheta-\vartheta'\|,
  \qquad \vartheta,\vartheta'\in\Theta,\label{eq:est-param-lipschitz}\\
  \sup_{\vartheta\in\Theta}
  \left|
  \ell(y_t,\widehat\lambda_t(\vartheta))
  -\ell(y_t,\lambda_t(\vartheta))
  \right|
  &\leq D_t\|\widehat\lambda_t-\lambda_t\|_\Theta.
  \label{eq:est-state-lipschitz}
\end{align}
The population criterion \(Q\) has the unique minimiser
\(\vartheta^\star\) in~\eqref{eq:main-est-population}.
\end{assumption}

For asymptotic normality, let \(N\) be an open neighbourhood of
\(\vartheta^\star\) whose closure is contained in
\(\operatorname{int}(\Theta)\).  Define
\[
  \dot\lambda_t(\vartheta)
  =\nabla_\vartheta\lambda_t(\vartheta),
  \qquad
  \ddot\lambda_t(\vartheta)
  =\nabla_\vartheta^2\lambda_t(\vartheta),
\]
and let the hatted versions be obtained by differentiating~\eqref{eq:est-finite-filter}.  Also put
\[
  \mathfrak s_t(\vartheta)
  =\nabla_\vartheta q_t(\vartheta),
  \qquad
  \mathfrak h_t(\vartheta)
  =\nabla_\vartheta^2 q_t(\vartheta),
\]
with finite-start analogues
\(\widehat{\mathfrak s}_t\) and \(\widehat{\mathfrak h}_t\).

\begin{assumption}[Twice differentiable invertibility]
\label{ass:est-differentiability}
On \(N\), the loss and update are twice continuously differentiable in
\(\lambda\) for almost every outcome.  The stationary filter is twice
continuously differentiable in \(\vartheta\), and
\(\dot\lambda_t(\cdot)\) and \(\ddot\lambda_t(\cdot)\) are stationary,
ergodic, and \(\mathcal F_{t-1}\)-measurable.  There are stationary
nonnegative random variables \(R_t,L_t,R_t^D\), with
\(\mathsf E[R_0+L_0+R_0^D]<\infty\), a deterministic
\(\rho_D\in(0,1)\), and an almost surely finite \(C_D\), such that
\begin{align}
  \sup_{\vartheta\in N}
  \bigl\{\|\mathfrak s_t(\vartheta)\|
       +\|\mathfrak h_t(\vartheta)\|\bigr\}
  &\leq R_t,\label{eq:est-derivative-envelope}\\
  \|\mathfrak h_t(\vartheta)-\mathfrak h_t(\vartheta')\|
  &\leq L_t\|\vartheta-\vartheta'\|,
  \qquad \vartheta,\vartheta'\in N,\label{eq:est-hessian-lipschitz}\\
  \sup_{\vartheta\in N}
  \bigl\{
  \|\widehat{\mathfrak s}_t(\vartheta)-\mathfrak s_t(\vartheta)\|
  +\|\widehat{\mathfrak h}_t(\vartheta)-\mathfrak h_t(\vartheta)\|
  \bigr\}
  &\leq C_D R_t^D\rho_D^t.
  \label{eq:est-derivative-forgetting}
\end{align}
\end{assumption}

\begin{assumption}[Score central limit theorem and sensitivity]
\label{ass:est-clt}
At the pseudo-true parameter,
\begin{equation}\label{eq:est-score-clt}
  \frac1{\sqrt T}\sum_{t=1}^T
  \mathfrak s_t(\vartheta^\star)
  \ \Longrightarrow\ \mathcal{N}(0,\Omega),
\end{equation}
where $\Omega$ is finite. When the covariance-series representation is used, assume absolute summability and the identification
\begin{equation}\label{eq:est-lrv}
  \Omega
  =\sum_{k=-\infty}^{\infty}
  \operatorname{Cov}\bigl(
  \mathfrak s_0(\vartheta^\star),
  \mathfrak s_k(\vartheta^\star)
  \bigr).
\end{equation}
For example, absolute summability gives convergence of the finite-sample score-sum covariance by the triangular-weight identity. Uniform integrability of the squared norm of the normalised score sum then identifies that limit with the covariance in \eqref{eq:est-score-clt}.
The matrix
\begin{equation}\label{eq:est-A}
  \mathcal A
  =\mathsf E[\mathfrak h_0(\vartheta^\star)]
\end{equation}
exists and is positive definite.
\end{assumption}

Assumption~\ref{ass:est-clt} is stated directly because, under
misspecification, the stationary parameter score need not be a martingale
difference.  A mixing or near-epoch-dependent central limit theorem is one
route to~\eqref{eq:est-score-clt}.  Under correct conditional
specification, Corollary~\ref{cor:est-correct-spec} below supplies the more
primitive martingale route.

\subsection{A primitive sufficient condition for uniform invertibility}
\label{app:est-sre}

Assumption~\ref{ass:est-invertibility} is stated in terms of the property used
by the estimator.  The following proposition gives a convenient global
sufficient condition.  It is not necessary, and the Gaussian-log cell requires
the non-global qualification stated at the end of this appendix.

\begin{proposition}[Uniform stochastic-recurrence contraction]
\label{prop:est-sre}
Suppose Assumption~\ref{ass:est-data} holds.
For $x\in C(\Theta)$ define the random map
$[\Phi_t(x)](\vartheta)=f_\vartheta(x(\vartheta),y_t)$.
Suppose that $\Phi_t$ maps $C(\Theta)$ into itself and has the finite random
Lipschitz coefficient
\[
  \mathcal L_t
  =\sup_{\vartheta\in\Theta}\sup_{z\ne z'}
  \frac{|f_\vartheta(z,y_t)-f_\vartheta(z',y_t)|}{|z-z'|}
  \quad\text{almost surely}.
\]
Assume $0<\mathcal L_0<\infty$ almost surely and, for one constant function
$x_0\in C(\Theta)$,
\begin{equation}\label{eq:est-sre-conditions}
  \mathsf E[|\log\mathcal L_0|]<\infty,
  \qquad
  \mathsf E[\log\mathcal L_0]<0,
  \qquad
  \mathsf E[\log^+\|\Phi_0(x_0)-x_0\|_\Theta]<\infty.
\end{equation}
If $z\mapsto f_\vartheta(z,y_t)$ is differentiable, the inner supremum may be
bounded by $\sup_z|\partial_z f_\vartheta(z,y_t)|$.  Then a $C(\Theta)$-valued
stationary causal solution of~\eqref{eq:est-stationary-filter} exists, is
unique among $C(\Theta)$-valued stationary causal solutions, and the exponential
finite-start forgetting in~\eqref{eq:est-state-forgetting} holds for every
continuous bounded initialisation map.
\end{proposition}

The proposition uses a globally defined action on real-valued $C(\Theta)$.
If the scoring-rule state domain is a proper subset $\Lambda\subsetneq\mathbb R$,
applying it to that filter additionally requires its constructed stationary
and finite-start paths to remain in $\Lambda$.

\begin{proof}
Work in the complete metric space $C(\Theta)$ with the supremum norm.  For a
fixed $t$, define the backward iterates
\(X_{t,n}=\Phi_{t-1}\circ\Phi_{t-2}\circ\cdots\circ\Phi_{t-n}(x_0)\), $n\ge1$.
Repeated use of the Lipschitz bounds gives
\[
  \|X_{t,n+1}-X_{t,n}\|_\Theta
  \le
  \left(\prod_{j=1}^{n}\mathcal L_{t-j}\right)
  \|\Phi_{t-n-1}(x_0)-x_0\|_\Theta.
\]
Put $\gamma=-\mathsf E[\log\mathcal L_0]>0$. By the ergodic theorem,
$\prod_{j=1}^{n}\mathcal L_{t-j}\le e^{-3\gamma n/4}$ eventually almost
surely. The logarithmic anchor moment in~\eqref{eq:est-sre-conditions} gives,
for every $\eta>0$,
$\|\Phi_{t-n-1}(x_0)-x_0\|_\Theta\le e^{\eta n}$ eventually almost surely.
This follows from the tail-sum formula and the
Borel--Cantelli lemma applied to
$\log^+\|\Phi_{t-n-1}(x_0)-x_0\|_\Theta$.  Choosing $\eta<\gamma/4$, the
successive differences are eventually bounded by $e^{-\gamma n/2}$.
Consequently $(X_{t,n})_n$ is almost surely Cauchy in $C(\Theta)$ and converges
uniformly to a finite continuous random function $\lambda_t(\cdot)$.

Passing to the limit in
$X_{t+1,n}=\Phi_t(X_{t,n-1})$ gives
$\lambda_{t+1}=\Phi_t(\lambda_t)$.  The limit is measurable with respect to
the infinite past, stationary, and ergodic because it is a measurable factor
of the stationary ergodic data sequence.

For uniqueness, let $\lambda_t^{(1)}$ and $\lambda_t^{(2)}$ be two
$C(\Theta)$-valued stationary causal solutions and put
$D_t=\|\lambda_t^{(1)}-\lambda_t^{(2)}\|_\Theta$.  Iteration gives
\[
  D_t\le P_{t,n}D_{t-n},
  \qquad
  P_{t,n}=\prod_{j=1}^{n}\mathcal L_{t-j}\longrightarrow0
  \quad\text{almost surely}.
\]
For $a,M>0$, marginal stationarity of the two solutions implies
\[
  \mathbb P(D_t>a)
  \le \mathbb P(P_{t,n}>a/M)
     +\mathbb P(\|\lambda_t^{(1)}\|_\Theta>M/2)
     +\mathbb P(\|\lambda_t^{(2)}\|_\Theta>M/2).
\]
Letting first $n\to\infty$ and then $M\to\infty$ proves $D_t=0$ almost surely.

Finally, a forward recursion started from any continuous bounded map satisfies
\[
  \|\widehat\lambda_t-\lambda_t\|_\Theta
  \le
  \left(\prod_{j=1}^{t-1}\mathcal L_j\right)
  \|\widehat\lambda_1-\lambda_1\|_\Theta.
\]
For every deterministic
$\rho_\lambda\in(\exp\{\mathsf E[\log\mathcal L_0]\},1)$, the right-hand side
is bounded by $C_\lambda\rho_\lambda^t$ for an almost surely finite
$C_\lambda$.  This proves~\eqref{eq:est-state-forgetting}.
\end{proof}

\subsection{Proof of Theorem~\ref{thm:static-estimation}}
\label{app:est-proof}
The proof is divided into several steps.

\paragraph{\emph{Step 1: continuity and finiteness of the population criterion}}
Equations~\eqref{eq:est-anchor}--\eqref{eq:est-param-lipschitz} give
$|q_t(\vartheta)|\le B_t+M_t\operatorname{diam}(\Theta)$, so the stationary
criterion is integrable uniformly over $\Theta$. Also,
$|Q(\vartheta)-Q(\vartheta')|\le\mathsf E[M_0]\|\vartheta-\vartheta'\|$,
and hence $Q$ is continuous on $\Theta$.

\paragraph{\emph{Step 2: uniform law of large numbers for the stationary
criterion}}
Let $Q_T^\circ(\vartheta)=T^{-1}\sum_{t=1}^Tq_t(\vartheta)$.
Fix \(\delta>0\), and choose a finite \(\delta\)-net
\(\{\vartheta_1,\ldots,\vartheta_{N_\delta}\}\) of the compact set
\(\Theta\).  For each \(\vartheta\), choose \(j(\vartheta)\) such that
\(\|\vartheta-\vartheta_{j(\vartheta)}\|\leq\delta\).  Then
\begin{equation*}
  |Q_T^\circ(\vartheta)-Q(\vartheta)|
  \leq
  |Q_T^\circ(\vartheta_{j(\vartheta)})
    -Q(\vartheta_{j(\vartheta)})|+
  \delta\left\{\frac1T\sum_{t=1}^TM_t+\mathsf E[M_0]\right\}.
\end{equation*}
The stationary criterion at each net point is an ergodic function of the data history. For the stationary integrable envelope $M_t$, Birkhoff gives an almost surely finite limit for its time average, possibly conditional on its invariant field. Taking a countable sequence of nets with $\delta\downarrow0$ gives
\begin{equation}\label{eq:est-stationary-ulln}
  \sup_{\vartheta\in\Theta}
  |Q_T^\circ(\vartheta)-Q(\vartheta)|
  \longrightarrow0
  \qquad\text{almost surely}.
\end{equation}

\paragraph{\emph{Step 3: the finite initial value is asymptotically negligible}}
By~\eqref{eq:est-state-lipschitz} and
\eqref{eq:est-state-forgetting},
\begin{align*}
  \sup_{\vartheta\in\Theta}
  |\widehat Q_T(\vartheta)-Q_T^\circ(\vartheta)|
  &\leq
  \frac{C_\lambda}{T}\sum_{t=1}^TD_t\rho_\lambda^t.
\end{align*}
Since \(D_t\) is stationary and integrable,
\[
  \mathsf E\left[\sum_{t=1}^{\infty}D_t\rho_\lambda^t\right]
  =\frac{\rho_\lambda}{1-\rho_\lambda}\mathsf E[D_0]<\infty.
\]
The infinite series is therefore finite almost surely, and
\begin{equation}\label{eq:est-init-objective}
  \sup_{\vartheta\in\Theta}
  |\widehat Q_T(\vartheta)-Q_T^\circ(\vartheta)|
  \longrightarrow0
  \qquad\text{almost surely}.
\end{equation}
Combining~\eqref{eq:est-stationary-ulln} and
\eqref{eq:est-init-objective} proves the uniform convergence asserted in
Theorem~\ref{thm:static-estimation}.

\paragraph{\emph{Step 4: consistency of a numerical near-minimiser}}
Fix \(\varepsilon>0\) and set
\(C_\varepsilon=\{\vartheta\in\Theta:\|\vartheta-\vartheta^\star\|\geq\varepsilon\}.\)
Compactness, continuity of \(Q\), and uniqueness of the minimiser imply the
strict identification gap
\begin{equation}\label{eq:est-id-gap}
  \eta_\varepsilon
  =\inf_{\vartheta\in C_\varepsilon}
  \{Q(\vartheta)-Q(\vartheta^\star)\}>0.
\end{equation}
On the event
\[
  \sup_{\vartheta\in\Theta}
  |\widehat Q_T(\vartheta)-Q(\vartheta)|
  <\eta_\varepsilon/4,
  \qquad
  r_T<\eta_\varepsilon/4,
\]
every \(\vartheta\in C_\varepsilon\) satisfies
\[
  \widehat Q_T(\vartheta)
  >\widehat Q_T(\vartheta^\star)+r_T
  \geq \inf_{\vartheta'\in\Theta}
  \widehat Q_T(\vartheta')+r_T.
\]
Thus a near-minimiser satisfying~\eqref{eq:est-near-minimiser} cannot
belong to \(C_\varepsilon\).  The uniform law of large numbers and
\(r_T=o_p(1)\) prove
\(\widehat\vartheta_T\xrightarrow{p}\vartheta^\star\).  If \(r_T\to0\) almost surely,
the same argument gives almost-sure consistency.

\paragraph{\emph{Step 5: population first-order condition}}
The integrable score envelope in
Assumption~\ref{ass:est-differentiability} permits differentiation under the
expectation on \(N\), so
\(\nabla_\vartheta Q(\vartheta)=\mathsf E[\mathfrak s_0(\vartheta)].\)
Since \(\vartheta^\star\) is an interior minimiser,
\begin{equation}\label{eq:est-pop-foc}
  \mathsf E[\mathfrak s_0(\vartheta^\star)]=0.
\end{equation}

\paragraph{\emph{Step 6: finite-start score replacement}}
By~\eqref{eq:est-derivative-forgetting},
\begin{align*}
  \sqrt T\left\|
  \nabla_\vartheta\widehat Q_T(\vartheta^\star)
  -\frac1T\sum_{t=1}^T\mathfrak s_t(\vartheta^\star)
  \right\|
  &\leq
  \frac{C_D}{\sqrt T}\sum_{t=1}^TR_t^D\rho_D^t.
\end{align*}
The infinite weighted sum is finite almost surely because \(R_t^D\) is
stationary and integrable.  Consequently,
\begin{equation}\label{eq:est-gradient-replacement}
  \sqrt T\nabla_\vartheta\widehat Q_T(\vartheta^\star)
  =\frac1{\sqrt T}\sum_{t=1}^T
  \mathfrak s_t(\vartheta^\star)+o_p(1),
\end{equation}
and Assumption~\ref{ass:est-clt} gives
\begin{equation}\label{eq:est-feasible-score-clt}
  \sqrt T\nabla_\vartheta\widehat Q_T(\vartheta^\star)
  \Longrightarrow \mathcal{N}(0,\Omega).
\end{equation}

\paragraph{\emph{Step 7: uniform convergence of the feasible Hessian}}
Let \(N_0\) be a closed ball neighbourhood of \(\vartheta^\star\) contained
in \(N\).  The finite-net argument from \emph{Step~2}, now applied to the
stationary matrix-valued function \(\mathfrak h_t(\vartheta)\), together
with~\eqref{eq:est-derivative-envelope} and
\eqref{eq:est-hessian-lipschitz}, yields
\begin{equation}\label{eq:est-stationary-hessian-ulln}
  \sup_{\vartheta\in N_0}
  \left\|
  \frac1T\sum_{t=1}^T\mathfrak h_t(\vartheta)
  -\mathsf E[\mathfrak h_0(\vartheta)]
  \right\|
  \longrightarrow0
  \qquad\text{almost surely}.
\end{equation}
The finite-start replacement in~\eqref{eq:est-derivative-forgetting} implies
\begin{equation}\label{eq:est-feasible-hessian-ulln}
  \sup_{\vartheta\in N_0}
  \left\|
  \nabla_\vartheta^2\widehat Q_T(\vartheta)
  -\mathsf E[\mathfrak h_0(\vartheta)]
  \right\|
  \longrightarrow0
  \qquad\text{in probability}.
\end{equation}
The expected Hessian is continuous on \(N_0\) by the integrable Lipschitz
envelope.

\paragraph{\emph{Step 8: Taylor expansion of the empirical first-order condition}}
Consistency and interiority imply that, with probability tending to one,
the line segment joining \(\vartheta^\star\) and
\(\widehat\vartheta_T\) lies in \(N_0\).  The integral Taylor expansion is
\begin{equation}
  \label{eq:est-taylor}
  \nabla_\vartheta\widehat Q_T(\widehat\vartheta_T)
  =\nabla_\vartheta\widehat Q_T(\vartheta^\star)
  +\overline{\mathcal A}_T
  (\widehat\vartheta_T-\vartheta^\star), \qquad
  \overline{\mathcal A}_T
  =\int_0^1
  \nabla_\vartheta^2\widehat Q_T
  \bigl(\vartheta^\star+s(\widehat\vartheta_T-\vartheta^\star)\bigr)
  \,ds.
\end{equation}
Equation~\eqref{eq:est-feasible-hessian-ulln}, consistency, and continuity
give
\(\overline{\mathcal A}_T\xrightarrow{p}\mathcal A\).  Positive definiteness of
\(\mathcal A\) makes \(\overline{\mathcal A}_T\) nonsingular with
probability tending to one.  If
\(\|\nabla_\vartheta\widehat Q_T(\widehat\vartheta_T)\|
=o_p(T^{-1/2})\), rearranging~\eqref{eq:est-taylor} gives
\[
  \sqrt T(\widehat\vartheta_T-\vartheta^\star)
  =-\overline{\mathcal A}_T^{-1}
  \sqrt T\nabla_\vartheta\widehat Q_T(\vartheta^\star)+o_p(1).
\]
Equation~\eqref{eq:est-feasible-score-clt} and Slutsky's theorem prove the
asymptotic normality in Theorem~\ref{thm:static-estimation}.
\hfill\(\square\)

\subsection{Derivative recursions}\label{app:est-derivatives}

The score and Hessian in the theorem can be computed recursively.  Suppress
the parameter argument and put
\[
  u_t=u(y_t,\lambda_t),
  \qquad
  u_{\lambda,t}=\partial_\lambda u(y_t,\lambda_t),
  \qquad
  u_{\lambda\lambda,t}=\partial_\lambda^2u(y_t,\lambda_t).
\]
Define
\begin{equation}\label{eq:est-ABC}
  A_t=\phi+\alpha u_{\lambda,t},
  \qquad
  b_t=\begin{pmatrix}1\\ \lambda_t\\ u_t\end{pmatrix},
  \qquad
  c_t=\begin{pmatrix}0\\ 1\\ u_{\lambda,t}\end{pmatrix}.
\end{equation}
Then
\begin{equation*}
  \dot\lambda_{t+1}
  =A_t\dot\lambda_t+b_t, \qquad
  \ddot\lambda_{t+1}
  =A_t\ddot\lambda_t
  +\alpha u_{\lambda\lambda,t}
   \dot\lambda_t\dot\lambda_t^{\top}
  +c_t\dot\lambda_t^{\top}
  +\dot\lambda_t c_t^{\top}.
\end{equation*}
For the finite-start recursion,
\begin{equation}\label{eq:est-initial-derivatives}
  \widehat{\dot\lambda}_1
  =\begin{pmatrix}
  (1-\phi)^{-1}\\[1mm]
  \omega(1-\phi)^{-2}\\[1mm]
  0
  \end{pmatrix},
  \qquad
  \widehat{\ddot\lambda}_1
  =\begin{pmatrix}
  0&(1-\phi)^{-2}&0\\
  (1-\phi)^{-2}&2\omega(1-\phi)^{-3}&0\\
  0&0&0
  \end{pmatrix}.
\end{equation}
Within a fixed cell, the loss has no direct dependence on
\((\omega,\phi,\alpha)\), and therefore
\begin{align}
  \mathfrak s_t(\vartheta)
  &=\ell_\lambda(y_t,\lambda_t)\dot\lambda_t
  =-\psi(y_t,\lambda_t)\dot\lambda_t,
  \label{eq:est-score-explicit}\\
  \mathfrak h_t(\vartheta)
  &=\ell_{\lambda\lambda}(y_t,\lambda_t)
  \dot\lambda_t\dot\lambda_t^{\top}
  +\ell_\lambda(y_t,\lambda_t)\ddot\lambda_t.
  \label{eq:est-hessian-explicit}
\end{align}
The feasible versions use the hatted state and derivative filters.  These
are the analytic quantities computed by the automatic-differentiation
implementation.

\subsection{Correct specification and the dynamic form}
\label{app:est-correct-spec}

\begin{corollary}[Correct conditional specification]\label{cor:est-correct-spec}
Suppose Assumptions~\ref{ass:est-data} and~\ref{ass:est-invertibility} hold, as do the continuity and integrable-envelope conditions of Assumption~\ref{ass:est-criterion} (without presupposing its unique-minimiser clause). Suppose the domain and simultaneous finiteness conditions of~{\eqref{eq::criteria}} hold at every date on a common probability-one event, and suppose there is \(\vartheta_0\in\Theta\) such that
\[
  \widetilde P_t=P_{\lambda_t(\vartheta_0)}
  \qquad\text{almost surely for every }t.
\]
Assume that \(\mathsf S\) is strictly proper and that the following dynamic
identification condition holds:
\[
  P_{\lambda_t(\vartheta)}
  =P_{\lambda_t(\vartheta_0)}
  \quad\text{almost surely for every }t
  \quad\Longrightarrow\quad
  \vartheta=\vartheta_0.
\]
Then \(\vartheta^\star=\vartheta_0\).

For the remaining conclusions, suppose that \(\vartheta_0\in\operatorname{int}(\Theta)\), that \(\lambda_t(\vartheta_0)\in\operatorname{int}(\Lambda)\) almost surely for every $t$, and that Assumptions~\ref{ass:conditional_differentiation} and~\ref{ass:est-differentiability} hold along the correctly specified path. Assume that differentiation may be passed twice through the true conditional expectation, that both terms in~\eqref{eq:est-hessian-explicit} are integrable, and that the conditional moments below and their displayed unconditional expectations are finite. Then
\[
  \mathsf E[
  \mathfrak s_t(\vartheta_0)\mid\mathcal F_{t-1}]=0.
\]
Let
\begin{equation*}
  J_{\mathsf{S},t}
  =\mathsf E\bigl[
  \ell_{\lambda\lambda}(y_t,\lambda_t(\vartheta_0))
  \mid\mathcal F_{t-1}\bigr], \qquad
  K_{\mathsf{S},t}
  =\mathsf E\bigl[
  \psi(y_t,\lambda_t(\vartheta_0))^2
  \mid\mathcal F_{t-1}\bigr].
\end{equation*}
Then
\begin{equation}\label{eq:est-godambe}
  \mathcal A
  =\mathsf E[J_{\mathsf{S},t}\dot\lambda_t\dot\lambda_t^{\top}],
  \qquad
  \Omega
  =\mathsf E[K_{\mathsf{S},t}\dot\lambda_t\dot\lambda_t^{\top}].
\end{equation}
If
\(\mathsf E[\|\mathfrak s_0(\vartheta_0)\|^{2+\delta}]<\infty\) for some
\(\delta>0\), the martingale central limit theorem yields~\eqref{eq:est-score-clt};
positive definiteness of $\mathcal A$ remains a separate condition. For the
logarithmic score under the conditional Bartlett identity,
\(J_{\mathsf{S},t}=K_{\mathsf{S},t}\) almost surely and the sandwich covariance
in Theorem~\ref{thm:static-estimation} reduces to \(\mathcal A^{-1}\) when
its remaining conditions hold. For a general proper score, the inverse
Hessian alone need not be the asymptotic covariance, even under correct
specification.
\end{corollary}

\begin{proof}
Conditional on \(\mathcal F_{t-1}\), strict propriety and the divergence
representation {\eqref{eq::criteria}} give
\begin{align*}
 \mathsf E\bigl[
  \ell(y_t,\lambda_t(\vartheta))
  -\ell(y_t,\lambda_t(\vartheta_0))
  \mid\mathcal F_{t-1}\bigr]=
  \mathsf D_{\mathsf S}\bigl(\widetilde P_t,P_{\lambda_t(\vartheta)}\bigr)\geq0.
\end{align*}
Taking expectations shows that \(Q(\vartheta)\geq Q(\vartheta_0)\).
Equality forces the nonnegative divergence to vanish almost surely at date zero. Strict stationarity extends equality of the predictive laws to every fixed date; a countable intersection supplies one common probability-one event. Strict propriety and dynamic identification then yield \(\vartheta=\vartheta_0\).

At the correctly specified state, differentiable propriety gives
\(\mathsf E[\psi(y_t,\lambda_t(\vartheta_0))\mid\mathcal F_{t-1}]=0\).
Since \(\dot\lambda_t\) is predictable,~\eqref{eq:est-score-explicit} shows that the parameter score is a
martingale difference.  Hence all nonzero-lag score autocovariances vanish,
and
\[
  \mathsf E[
  \mathfrak s_t\mathfrak s_t^{\top}\mid\mathcal F_{t-1}]
  =K_{\mathsf S,t}\dot\lambda_t\dot\lambda_t^{\top}.
\]
Taking expectations gives the formula for \(\Omega\).  Conditioning~\eqref{eq:est-hessian-explicit} on \(\mathcal F_{t-1}\), the term
containing \(\ell_\lambda\ddot\lambda_t\) vanishes by conditional
centring, leaving
\[
  \mathsf E[\mathfrak h_t\mid\mathcal F_{t-1}]
  =J_{\mathsf S,t}\dot\lambda_t\dot\lambda_t^{\top}.
\]
This gives the formula for \(\mathcal A\). The stated moment condition verifies conditional Lindeberg, and stationarity and ergodicity give convergence of the predictable quadratic variation. The martingale central limit theorem therefore gives~\eqref{eq:est-score-clt}. Positive definiteness of $\mathcal A$ must still be checked separately. The final assertion follows from the conditional Bartlett identity in the logarithmic case.
\end{proof}

\paragraph{A sufficient condition for dynamic identification}
For the fixed-driver recursion, suppose $\lambda\mapsto P_\lambda$ is injective, $\lambda_t(\vartheta_0)$ is nonconstant, and $u_{\mathsf S}(Y_t,\lambda_t(\vartheta_0))$ has a finite, strictly positive conditional variance on an event of positive probability. Equality of all predictive laws gives equality of the predictable states. Subtracting their recursions yields $\Delta\omega+\Delta\varphi\,\lambda_t+\Delta\alpha\,u_{\mathsf S}(Y_t,\lambda_t)=0$ almost surely. Conditional variance then forces $\Delta\alpha=0$. Nonconstancy of $\lambda_t$ gives $\Delta\varphi=\Delta\omega=0$. This verifies the identification clause of Corollary~\ref{cor:est-correct-spec} under correct specification. A constant state identifies only the combination $\omega+\varphi\lambda_t$ from the predictable part of this equation.

\subsection{Feasible covariance and preliminary estimated quantities}
\label{app:est-feasible}

At \(\widehat\vartheta_T\), compute
\[
  \widehat{\mathcal A}_T
  =\frac1T\sum_{t=1}^T
  \widehat{\mathfrak h}_t(\widehat\vartheta_T).
\]
Under misspecification, \(\Omega\) is a long-run covariance.  A feasible
estimator is therefore a HAC estimator based on the fitted parameter-score
sequence
\(\widehat{\mathfrak s}_t(\widehat\vartheta_T)\).  If
\(\widehat\Omega_T\) is any consistent HAC estimator, then
\(
  \widehat V_T
  =\widehat{\mathcal A}_T^{-1}
  \widehat\Omega_T
  (\widehat{\mathcal A}_T^{-1})^{\top}
\)
{consistently estimates the covariance matrix of the limiting normal law in Theorem~\ref{thm:static-estimation}. Accordingly, $\widehat V_T/T$ is the usual first-order asymptotic covariance approximation for $\widehat\vartheta_T$. Convergence of the actual finite-sample covariance additionally requires uniform integrability of the squared normalised estimation error.}  An unlagged OPG/BHHH matrix is generally
insufficient under misspecification because it estimates only the lag-zero
score variability.  Under correct specification the parameter score is a
martingale difference. The inverse Hessian is justified when sensitivity and
variability coincide, in particular for the logarithmic score under the
conditional Bartlett identity.

\begingroup
Theorem~\ref{thm:static-estimation} treats the quantities defining a cell as fixed nonrandom inputs, as with the stated values of $\nu$ and $\beta$. Treating a same-sample estimate as fixed does not establish a conditional version of the theorem. When a preliminary quantity is estimated from the same training sample, the following corollary supplies the required additional conditions.
\endgroup

\begin{corollary}[Estimated preliminary quantities]\label{cor:est-nuisance}
Suppose Assumption~\ref{ass:est-data} holds. Let $\eta\in\mathcal H\subseteq\mathbb R^q$ collect finite-dimensional preliminary quantities, such as a demeaning mean. Make their dependence explicit in the loss $\ell(y,\lambda;\eta)$ and, where applicable, the scaling $G_{\mathsf S}(\lambda;\eta)$. For the demeaning example, $y_t$ here denotes the raw return, and the loss and update use $y_t-\eta$ internally. Set
\[
 u_{\mathsf S}(y,\lambda;\eta)
 =-G_{\mathsf S}(\lambda;\eta)\partial_\lambda\ell(y,\lambda;\eta),
 \qquad
 f_{\vartheta,\eta}(\lambda,y)
 =\omega+\phi\lambda+\alpha u_{\mathsf S}(y,\lambda;\eta).
\]
For each $(\vartheta,\eta)$ under consideration, let the finite-start and stationary causal filters satisfy
\[
 \widehat\lambda_{t+1}(\vartheta,\eta)
 =f_{\vartheta,\eta}(\widehat\lambda_t(\vartheta,\eta),y_t),
 \quad \widehat\lambda_1(\vartheta,\eta)=a(\vartheta),
 \qquad
 \lambda_{t+1}(\vartheta,\eta)
 =f_{\vartheta,\eta}(\lambda_t(\vartheta,\eta),y_t).
\]
Define
\[
 \widehat Q_T(\vartheta,\eta)
 =\frac1T\sum_{t=1}^T
   \ell(y_t,\widehat\lambda_t(\vartheta,\eta);\eta),
 \quad
 q_t(\vartheta,\eta)
 =\ell(y_t,\lambda_t(\vartheta,\eta);\eta),
 \quad Q(\vartheta,\eta)=\mathsf E[q_0(\vartheta,\eta)].
\]
Suppose $\widehat\eta_T\xrightarrow{p}\eta_0$ and
\begin{equation}\label{eq:est-plugin-ulln}
  \sup_{\vartheta\in\Theta}
  |\widehat Q_T(\vartheta,\widehat\eta_T)-Q(\vartheta,\eta_0)|
  \xrightarrow{p}0,
\end{equation}
where $Q(\cdot,\eta_0)$ is finite, continuous, and has the unique minimiser $\vartheta^\star$. Then every $o_p(1)$ numerical near-minimiser of $\widehat Q_T(\cdot,\widehat\eta_T)$ is consistent for $\vartheta^\star$.

For asymptotic normality, suppose $\vartheta^\star\in\operatorname{int}(\Theta)$ and the empirical objective is twice continuously differentiable jointly in $(\vartheta,\eta)$ near $(\vartheta^\star,\eta_0)$. Put $\Delta_{\vartheta,T}=\widehat\vartheta_T-\vartheta^\star$ and $\Delta_{\eta,T}=\widehat\eta_T-\eta_0$. Suppose the segment joining these two points lies in that neighbourhood with probability tending to one. Along the segment define the path-averaged matrices
\begin{align}
 \overline{\mathcal A}_T
 &=\int_0^1\nabla_\vartheta^2\widehat Q_T
   (\vartheta^\star+u\Delta_{\vartheta,T},
    \eta_0+u\Delta_{\eta,T})\,\mathrm du, \notag\\
 \overline{\mathcal B}_T
 &=\int_0^1\partial_\eta\nabla_\vartheta\widehat Q_T
   (\vartheta^\star+u\Delta_{\vartheta,T},
    \eta_0+u\Delta_{\eta,T})\,\mathrm du.
 \label{eq:est-plugin-path-matrices}
\end{align}
Suppose the stationary criterion contribution is twice differentiable at $(\vartheta^\star,\eta_0)$, with finite limits
\[
 \mathcal A=\mathsf E[\nabla_\vartheta^2q_0(\vartheta^\star,\eta_0)],
 \qquad
 \mathcal B=\mathsf E[\partial_\eta\nabla_\vartheta
                         q_0(\vartheta^\star,\eta_0)],
\]
where $\mathcal A$ is nonsingular and
\begin{equation}\label{eq:est-plugin-derivative-limits}
 \overline{\mathcal A}_T\xrightarrow{p}\mathcal A,
 \qquad
 \overline{\mathcal B}_T\xrightarrow{p}\mathcal B.
\end{equation}
Here $\partial_\eta\nabla_\vartheta q_t$ includes both direct dependence of the loss on $\eta$ and dependence through the stationary filter. With $\mathfrak s_t(\vartheta,\eta)=\nabla_\vartheta q_t(\vartheta,\eta)$, assume the finite-start gradient replacement
\begin{equation}\label{eq:est-plugin-score-replacement}
 \sqrt T\left\|
 \nabla_\vartheta\widehat Q_T(\vartheta^\star,\eta_0)
 -\frac1T\sum_{t=1}^T\mathfrak s_t(\vartheta^\star,\eta_0)
 \right\|\xrightarrow{p}0.
\end{equation}
Finally, assume separate root-$T$ tightness of the preliminary estimator and stationary score sum:
\[
 \sqrt T(\widehat\eta_T-\eta_0)
 =\frac1{\sqrt T}\sum_{t=1}^Th_t+o_p(1)=O_p(1),
 \qquad
 \frac1{\sqrt T}\sum_{t=1}^T
   \mathfrak s_t(\vartheta^\star,\eta_0)=O_p(1).
\]
If the retained solution satisfies
$\|\nabla_\vartheta\widehat Q_T(\widehat\vartheta_T,\widehat\eta_T)\|
=o_p(T^{-1/2})$, then
\begin{equation}\label{eq:est-plugin}
  \sqrt T(\widehat\vartheta_T-\vartheta^\star)
  =-\mathcal A^{-1}
  \left\{
  \frac1{\sqrt T}\sum_{t=1}^T
  \mathfrak s_t(\vartheta^\star,\eta_0)
  +\mathcal B\frac1{\sqrt T}\sum_{t=1}^Th_t
  \right\}+o_p(1).
\end{equation}
Consequently, if the combined influence process
$\mathfrak s_t+\mathcal B \operatorname{IF}_{\eta,t}$ satisfies a central limit theorem with
long-run covariance $\Omega_\eta$, then
\[
  \sqrt T(\widehat\vartheta_T-\vartheta^\star)
  \Longrightarrow
   \mathcal{N}\!\left(0,
  \mathcal A^{-1}\Omega_\eta(\mathcal A^{-1})^{\top}\right).
\]
\end{corollary}

\begin{proof}
Equation~\eqref{eq:est-plugin-ulln}, compactness of $\Theta$, continuity of $Q(\cdot,\eta_0)$, and its unique minimiser give the strict identification gap used in Step~4 of the proof of Theorem~\ref{thm:static-estimation}. This proves consistency for an $o_p(1)$ near-minimiser.

For the distributional result, put
$e_T=\nabla_\vartheta\widehat Q_T(\widehat\vartheta_T,\widehat\eta_T)$.
On the high-probability event where the joint segment lies in the differentiability neighbourhood, the fundamental theorem of calculus gives the exact identity
\[
 e_T
 =\nabla_\vartheta\widehat Q_T(\vartheta^\star,\eta_0)
  +\overline{\mathcal A}_T\Delta_{\vartheta,T}
  +\overline{\mathcal B}_T\Delta_{\eta,T}.
\]
The limit of $\overline{\mathcal A}_T$ makes it invertible with probability tending to one. Multiply this identity by $\sqrt T$, solve for $\sqrt T\Delta_{\vartheta,T}$, and use~\eqref{eq:est-plugin-score-replacement}:
\[
 \sqrt T\Delta_{\vartheta,T}
 =-\overline{\mathcal A}_T^{-1}
 \left\{\frac1{\sqrt T}\sum_{t=1}^T
   \mathfrak s_t(\vartheta^\star,\eta_0)
   +\overline{\mathcal B}_T\sqrt T\Delta_{\eta,T}
   -\sqrt T e_T\right\}+o_p(1).
\]
The two separate $O_p(1)$ inputs and $\sqrt T e_T=o_p(1)$ permit replacement of both path-averaged matrices by their limits. Substituting the preliminary estimator's influence representation gives~\eqref{eq:est-plugin}. The combined-process central limit theorem and Slutsky's theorem give the stated normal limit.
\end{proof}

Thus unconditional inference must use the long-run covariance of the combined
influence process, including all cross-lag covariances between the criterion
score and the preliminary estimator.  If optimisation is performed in
unconstrained coordinates and then mapped smoothly to
$(\omega,\phi,\alpha)$, the covariance is transformed by the Jacobian of that
map.

\paragraph{\textbf{Population invertibility versus sample diagnostics}}
Assumption~\ref{ass:est-invertibility} is a population condition uniform over the admissible parameter set. Proposition~\ref{prop:est-sre} gives a standard global route through a negative top Lyapunov exponent for a uniform random Lipschitz coefficient, together with logarithmic moment conditions. A derivative calculation along one fitted path does not establish this condition. For Gaussian log scale, $\partial_\lambda f_\vartheta(\lambda,y)=\phi-\alpha y^2e^{-\lambda}$ has infinite absolute supremum over $\lambda\in\mathbb R$ whenever $y\ne0$. The following subsection treats the available alternatives and the sufficient bounds for selected bounded drivers. If tuning quantities are estimated, Corollary~\ref{cor:est-nuisance} requires a joint influence expansion; fixing their realised values does not remove that requirement.

\begingroup
\subsection{Primitive conditions for bounded drivers}\label{subsec:primitive_invertibility}
These are applications of the stochastic-recurrence route of \citet{bougerol1993kalman} and \citet{straumann2006quasi}, as used by \citet{blasques2018feasible,blasques2022maximum}. The Student-$t$ log cell is the Beta-$t$-EGARCH filter \citep{harvey2013dynamic}, and the monotone case is a scale-model instance of the Lipschitz sufficient condition of \citet{donker2025stability}. The results below concern a fixed compact admissible parameter set. They separate an analytic contraction inequality from numerical diagnostics at fitted points.

For the scale-invariant bounded cells, $u_{\mathsf S}(y,\lambda)=g_{\mathsf S}(z)$ with $z=ye^{-\lambda/2}$, and
\[
 A(\lambda,y)=\partial_\lambda f_\vartheta(\lambda,y)
 =\phi-\alpha m_{\mathsf S}(z),\qquad m_{\mathsf S}(z)=\tfrac12zg'_{\mathsf S}(z).
\]
Fixed-data-bandwidth MMD is not scale-invariant and is excluded from this representation.

\begin{definition}[Bounded scale-invariant drivers of order three]\label{inv:def:B3}
A driver $g\in C^3(\mathbb R)$ belongs to $\mathfrak B_3$ if $z^kg^{(k)}(z)$ is bounded for $k=0,1,2,3$. Write $M_0=\sup|g|$, $\operatorname{osc}(g)=\sup g-\inf g$, $[m_-,m_+]$ for the closure of the range of $m(z)=zg'(z)/2$, and $M_1=\max\{|m_-|,|m_+|\}$. We call the driver monotone in scale if $m\ge0$ and then put $M_+=m_+$. For an even driver, as in the examples below, this is equivalent to $g$ being non-decreasing in $|z|$.
\end{definition}

\begin{proposition}[Uniform contraction and primitive conditions]\label{inv:prop:mono}
Let $g\in\mathfrak B_3$ be monotone in scale and set
\[
 L(\vartheta)=\max\{|\phi|,|\phi-\alpha M_+|\}.
\]
Suppose Assumption~\ref{ass:est-data} holds and choose a deterministic $\bar L\in(0,1)$ with $\sup_{\vartheta\in\Theta}L(\vartheta)\le\bar L$. Such a bound exists when $\Theta$ is compact and $|\phi|<1$, $\alpha M_+<1+\phi$ throughout $\Theta$. Then:
\begin{enumerate}[label=(\roman*)]
\item For every $y$ and $\vartheta\in\Theta$, $\lambda\mapsto f_\vartheta(\lambda,y)$ is a contraction on $\mathbb R$ with modulus at most $L(\vartheta)$. If $\phi\ge0$ and $\alpha M_+\le2\phi$, then $L(\vartheta)=\phi$.
\item Assumption~\ref{ass:est-invertibility} holds with $\rho_\lambda=\bar L$ and bounded $C_\lambda$. The stationary filter satisfies
\[
 \left|\lambda_t(\vartheta)-\frac{\omega}{1-\phi}\right|
 \le\frac{\alpha M_0}{1-|\phi|}.
\]
\item The stationary filter has three parameter derivatives, bounded uniformly in $t$ and $\vartheta$ by deterministic constants. The finite-start state and its first two parameter derivatives differ from their stationary counterparts by at most $C(1+t)^3\bar L^t$. No rate for the third derivative difference is asserted under $\mathfrak B_3$ alone.
\item For Gaussian and Student CRPS, Assumption~\ref{ass:est-criterion}, apart from identification, holds if $\mathsf E|y_0|<\infty$; for Student log score it holds if $\mathsf E\log(1+y_0^2)<\infty$. For these losses, Assumption~\ref{ass:est-differentiability} holds on every neighborhood $N$ with $\overline N\subset\operatorname{int}\Theta$, with deterministic envelopes and any $\rho_D\in(\bar L,1)$.
\item {For these losses, $\mathfrak s_t(\vartheta^\star)$ is bounded and $L^2$ near-epoch dependent on $(y_t)$ with coefficients at most $C\rho_0^k$ for any fixed $\rho_0\in(\bar L,1)$. Consequently the centred score satisfies the mixing CLT of \citet[Ch.~24]{davidson1994stochastic} when $(y_t)$ is strong mixing of size $-r/(r-2)$ for some $r>2$. If $\vartheta^\star\in\operatorname{int}\Theta$, differentiation of the population criterion under the bounded score envelope gives $\mathsf E[\mathfrak s_t(\vartheta^\star)]=0$, so this verifies~\eqref{eq:est-score-clt}. At a boundary minimiser the score need not have zero mean. Under correct conditional specification, the martingale route of Corollary~\ref{cor:est-correct-spec} is available instead.}
\end{enumerate}
The Gaussian CRPS, Student-$t_6$ log and Student-$t_6$ CRPS satisfy the stated driver conditions. Application of Theorem~\ref{thm:static-estimation} additionally requires the declared data conditions, identification, nonsingular sensitivity, the numerical near-minimisation conditions, and a fixed admissible parameter set obeying the uniform restriction above.
\end{proposition}
\proofref{app:proof_mono_invertibility}

For density-power drivers $m$ takes both signs. The one-step Lipschitz coefficient is bounded by
$\max\{|\phi-\alpha m_-|,|\phi-\alpha m_+|\}$ and may exceed one. Contraction over several steps is then a separate property of the observation-driven composition; boundedness of the driver alone does not establish it.

\begin{proposition}[$r$-step contraction]\label{inv:prop:rstep}
Suppose Assumption~\ref{ass:est-data} holds and $g\in\mathfrak B_3$. For an observation block $(y_t,\ldots,y_{t+r-1})$, define
\[
 \Phi^{(r)}_{t,\vartheta}=f_\vartheta(\cdot,y_{t+r-1})\circ\cdots\circ f_\vartheta(\cdot,y_t),
 \qquad
 \Lambda_{t,r}(\vartheta)=\sup_{\lambda\in\mathbb R}
       |\partial_\lambda\Phi^{(r)}_{t,\vartheta}(\lambda)|.
\]
For $t=1$ write $\Lambda_r=\Lambda_{1,r}$. By the chain rule,
\[
 \Lambda_r(\vartheta)=\sup_{\lambda\in\mathbb R}
       \prod_{j=1}^r|\phi-\alpha m(y_je^{-\lambda_j(\lambda)/2})|,
 \quad \lambda_1=\lambda,\quad
 \lambda_{j+1}=f_\vartheta(\lambda_j,y_j).
\]
\begin{enumerate}[label=(\roman*)]
\item If $\mathsf E\log\sup_{\vartheta\in\Theta}\Lambda_r(\vartheta)<0$ for some $r\ge1$, then Assumption~\ref{ass:est-invertibility} holds with a deterministic $\rho_\lambda<1$ and almost surely finite $C_\lambda$. The stationary filter is bounded as in Proposition~\ref{inv:prop:mono}(ii). A value $-\infty$ of the logarithmic expectation is allowed.
\item For consistency, the parameter-Lipschitz condition~\eqref{eq:est-param-lipschitz} may be replaced by continuity of $q_t$ on $\Theta$ and $\mathsf E\sup_\Theta|q_0(\vartheta)|<\infty$, retaining identification and the finite-start loss bound~\eqref{eq:est-state-lipschitz}. For both density-power scores used here, the loss and its first state derivative are bounded on the deterministic state range, so the envelope and finite-start loss bound hold. The consistency proof of Theorem~\ref{thm:static-estimation} then applies in this envelope form.
\item For the density-power losses, suppose in addition that $(\mathcal G_t)$ is a stationary filtration containing the observation history and the relevant driving information, with future observation blocks adapted to it, and that
\begin{equation}\label{rev:eq:conditional-block}
 \mathsf E\!\left[\left(\sup_{\vartheta\in\Theta}
          \Lambda_{t,r}(\vartheta)\right)^3\middle|\mathcal G_{t-1}\right]
 \le c<1\quad\text{almost surely for every }t.
\end{equation}
\begingroup
Then, writing $\lambda_t^{(3)}=\nabla_\vartheta^3\lambda_t$ and using the parameter supremum norm defined above, the stationary derivative processes satisfy
\[
 \mathsf E\|\dot\lambda_0\|_\Theta^3
 +\mathsf E\|\ddot\lambda_0\|_\Theta^{3/2}
 +\mathsf E\|\lambda_0^{(3)}\|_\Theta<\infty.
\]
The first two initialized derivative processes forget their initial values exponentially almost surely, uniformly in $\vartheta$. Assumption~\ref{ass:est-differentiability} therefore holds on neighborhoods compactly contained in $\operatorname{int}\Theta$. The score CLT and nonsingular sensitivity in Assumption~\ref{ass:est-clt} remain separate conditions.
\endgroup
\end{enumerate}
\end{proposition}
\proofref{app:proof_rstep}

\paragraph{{Gaussian log score}}
Here $m(z)=z^2$, so the one-step global coefficient is infinite for every nonzero observation. For the positive-persistence case $0<\phi<1$ and $y_1\ne0$, $\Lambda_r=\infty$ for every fixed $r$: as the initial state tends to $-\infty$, the first derivative diverges, while all subsequent states tend to $+\infty$ and their state derivatives tend to $\phi>0$. This multi-step argument is not asserted at $\phi=0$. Possible alternatives are non-global continuous invertibility \citep{wintenberger2013continuous}, a feasible admissible-region argument \citep{blasques2018feasible}, or an explicitly restricted state recursion. These require assumptions beyond a bounded-driver argument.

\paragraph{Gaussian log on its attainable states}
When $\alpha>0$ and $0\le\varphi<1$ throughout a compact parameter set, put $\ell=\inf_{\Theta}(\omega-\alpha)/(1-\varphi)$. Initialisation at $\omega/(1-\varphi)$ lies in $[\ell,\infty)$, and $f_\vartheta(x,y)\ge\omega-\alpha+\varphi x\ge\ell$ there. On this invariant half-line a uniform Lipschitz bound is
\[
 L_\ell(y)=\sup_{\vartheta\in\Theta}
       \max\{|\varphi|,|\varphi-\alpha e^{-\ell}y^2|\}.
\]
A strictly negative $\mathsf E\log L_\ell(Y)$, together with the log integrability and anchor conditions of Proposition~\ref{prop:est-sre}, suffices for uniform state forgetting on this range, by applying its proof on the corresponding closed function space. Derivative forgetting and identification still require separate checks.

\paragraph{{MMD cells}}
A fixed data-scale bandwidth leaves a state-dependent driver, so the residual-only derivative formula used in this subsection does not apply. An estimated bandwidth additionally requires the preliminary-estimation conditions of Corollary~\ref{cor:est-nuisance} or a separate argument.

\paragraph{{Scope of the sufficient conditions}}
{The exact monotone-driver inequality can be checked at a fitted vector, but the estimation theorem requires it throughout the fixed compact parameter set. For the density-power route, a finite-grid maximum is only a lower bound on the state supremum, and an unconditional second-moment diagnostic does not establish the conditional, parameter-uniform third-moment premise~\eqref{rev:eq:conditional-block}. The numerical illustrations do not verify these population conditions.}

\subsection{Proof of Proposition~\ref{inv:prop:mono}}\label{app:proof_mono_invertibility}
\begin{proof}
\emph{(i)--(ii).} Since $0\le m\le M_+$, the state derivative lies in $[\phi-\alpha M_+,\phi]$ and is bounded in absolute value by $L(\vartheta)$. The mean value theorem gives contraction. Apply Proposition~\ref{prop:est-sre} using the strictly positive deterministic bound $\bar L<1$; the anchor is bounded because $g$ and the parameter set are bounded. The forward difference is at most $\bar L^{t-1}\|\widehat\lambda_1-\lambda_1\|_\Theta$. Iterating
$\lambda_{t+1}-\bar\lambda=\phi(\lambda_t-\bar\lambda)+\alpha g(z_t)$,
where $\bar\lambda=\omega/(1-\phi)$, gives the state bound.

\emph{(iii).} The first derivative solves
$\dot\lambda_{t+1}=A_t\dot\lambda_t+b_t$ with
$b_t=(1,\lambda_t,g(z_t))^\top$. Both $b_t$ and the state are bounded. The second and third derivative equations have the same contracting coefficient $A_t$ and polynomial forcing in the lower derivatives, with uniformly bounded coefficients. The bounds on $zg'$, $z^2g''$ and $z^3g^{(3)}$ control the first three state derivatives of $u$. Solving these affine derivative recursions successively gives bounded stationary derivatives through order three. For the first two initialized derivative differences, the required coefficient differences are Lipschitz in the state; iteration gives polynomial factors multiplying $\bar L^t$, bounded by $C(1+t)^3\bar L^t$. Comparing third derivative differences would additionally require a Lipschitz condition on the third state derivative, which is not assumed or used here.

\emph{(iv).} On the deterministic bounded state range, the first three state derivatives of each specified loss are bounded. For CRPS,
$\partial_\lambda S=-\tfrac12e^{\lambda/2}g(z)$; for Student log,
$\partial_\lambda S=-\tfrac12g(z)$.
Together with part~(iii), the chain rule gives bounded score, Hessian and parameter-Lipschitz Hessian envelopes. The first two derivative-forgetting bounds imply~\eqref{eq:est-derivative-forgetting} after increasing the geometric rate to any $\rho_D\in(\bar L,1)$. Finally, $|S|\le|y|+Ce^{\lambda/2}$ for CRPS and $|S|\le C+C\log(1+y^2)$ for Student log supply the loss-level envelope.

{\emph{(v).} Replace the state and derivative histories before $t-k$ by fixed bounded values. The score changes by at most $C(1+k)^3\bar L^k\le C_{\rho_0}\rho_0^k$ for any $\rho_0\in(\bar L,1)$. Conditional expectation gives the same order of $L^2$ near-epoch approximation. The cited mixing CLT applies to $\mathfrak s_t(\vartheta^\star)-\mathsf E[\mathfrak s_t(\vartheta^\star)]$ under its stated mixing-rate premise. At an interior population minimiser, differentiation under the bounded score envelope gives the zero-mean condition and hence~\eqref{eq:est-score-clt}. Alternatively, Corollary~\ref{cor:est-correct-spec} supplies the martingale route under correct conditional specification.}
\end{proof}

\subsection{Proof of Proposition~\ref{inv:prop:rstep}}\label{app:proof_rstep}
\begin{proof}
\emph{(i).} The Lipschitz constants of compositions are submultiplicative. The one-step coefficient is bounded above by a deterministic constant because $\Theta$ is compact and $m$ is bounded. Thus its positive logarithmic part is integrable; no lower bound on $|\phi|$ is required. The stationary-ergodic subadditive contraction argument for stochastic recurrences gives a negative top Lipschitz exponent whenever one block satisfies the displayed negative expected logarithm \citep{bougerol1993kalman,straumann2006quasi}. Zero coefficients can be handled by replacing the block coefficient by its maximum with a sufficiently small positive constant. Bounded anchors and exponential decay of composition coefficients yield a unique stationary causal solution and exponential finite-start forgetting. The deterministic state bound follows from $|\phi|<1$ and $|g|\le M_0$.

\emph{(ii).} Continuity and the integrable supremum envelope give a uniform ergodic law of large numbers on the compact parameter set. This replaces the finite-net Lipschitz argument in the consistency proof; initialization is still removed by~\eqref{eq:est-state-lipschitz}. For the two density-power losses, boundedness of the loss and its state derivative on the bounded state range supplies both requirements. Identification remains an assumption.

\emph{(iii).} Work at block boundaries and put
$B_{t,r}=\sup_\Theta\Lambda_{t,r}$.
For $1\le p\le3$, conditional Jensen gives
$\mathsf E[B_{t,r}^p\mid\mathcal G_{t-1}]\le c^{p/3}$.
{All partial derivatives through order three of a fixed-length block map are uniformly bounded on the deterministic state range, apart from the incoming parameter derivatives themselves. Let $U=\|\dot\lambda\|_\Theta$, $V=\|\ddot\lambda\|_\Theta$ and $W=\|\lambda^{(3)}\|_\Theta$ be the parameter supremum norms of the incoming derivatives. Their outgoing values satisfy}
\[
 U^+\le B_{t,r}U+C,\qquad
 V^+\le B_{t,r}V+C(1+U^2),\qquad
 W^+\le B_{t,r}W+C(1+U^3+UV+V).
\]
\begingroup
Here $U,V,W$ are $\mathcal G_{t-1}$-measurable. Put $q=c^{1/3}<1$ and write $\|X\|_p=(\mathsf E|X|^p)^{1/p}$. Conditional expectation, Minkowski's inequality and H\"older's inequality give
\begin{align*}
 \|U^+\|_3&\le q\|U\|_3+C,\\
 \|V^+\|_{3/2}&\le q\|V\|_{3/2}+C(1+\|U\|_3^2),\\
 \|W^+\|_1&\le q\|W\|_1
   +C(1+\|U\|_3^3+\|U\|_3\|V\|_{3/2}+\|V\|_{3/2}).
\end{align*}
Constructing the stationary derivative recursions from the infinite past gives the asserted moments of their parameter supremum norms. The conditioning in~\eqref{rev:eq:conditional-block} is what permits control of the products of the future block coefficient and the incoming derivatives.
\endgroup

\begingroup
For the finite-start comparison, the blockwise state difference has $L^3$ norm $O(q^n)$ after $n$ blocks. The first derivative difference has $L^{3/2}$ norm $O((1+n)q^n)$: its forcing is bounded by a state difference times a constant or a first derivative. H\"older's inequality and the moment hierarchy above then give $L^1$ norm $O((1+n)^2q^n)$ for the second derivative difference. These bounds use parameter supremum norms throughout. Markov's inequality and Borel--Cantelli yield almost sure geometric forgetting at every sufficiently slower deterministic rate; the finitely many dates within each block do not change this conclusion. The bounded density-power loss derivatives yield the integrable score/Hessian envelopes. The third parameter derivative contributes only $U^3$, $UV$ and $W$ to the parameter-Lipschitz Hessian envelope. Stationary factors with finite moments grow subexponentially almost surely, so another arbitrarily small slowing of the rate gives~\eqref{eq:est-derivative-forgetting} for the feasible score and Hessian. This establishes~\eqref{eq:est-derivative-envelope}--\eqref{eq:est-derivative-forgetting}. No score CLT follows solely from the contraction argument.
\endgroup
\end{proof}

\begin{remark}[Envelope form of the Lipschitz clause]\label{rem:envelope}
For consistency, the parameter-Lipschitz clause~\eqref{eq:est-param-lipschitz} can be replaced by almost-sure continuity on $\Theta$ and $\mathsf E\sup_\Theta|q_0(\vartheta)|<\infty$, retaining the finite-start loss bound and identification. Dominated convergence makes $Q$ continuous, and the uniform ergodic theorem supplies the uniform LLN \citep{straumann2006quasi}. This is the envelope version used in Proposition~\ref{inv:prop:rstep}(ii).
\end{remark}
\endgroup

\section{Proofs for Section~\ref{section::continuous_time}: continuous-time limits}\label{app::appendix_section_6}
\subsection{\texorpdfstring{{Formal statements for Section~\ref{section::continuous_time}}}{Formal continuous-time statements}}\label{app:ct_formal}
At mesh $h$, let $Q_{\lambda_0}^{(h)}$ be the true local law at state $\lambda_0$ and $P_a^{(h)}$ the working law at candidate $a$.
We define the canonical two-argument risk
\(
\mathsf R_{\mathsf S,h}(a;\lambda_0)
 =\mathsf E_{Q_{\lambda_0}^{(h)}}[\mathsf S(P_a^{(h)},Y)]
\).
Only the candidate $a$ is differentiated; the true law is held fixed. Write $a^\star=a^\star_{\mathsf S,h}(\lambda_0)$ for a unique interior minimiser and define
\[
 \psi_{\mathsf S,h}(a,y)=-\partial_a\mathsf S(P_a^{(h)},y), \quad
 J_{\mathsf S,h}(\lambda_0)=\left.\partial_a^2\mathsf R_{\mathsf S,h}(a;\lambda_0)\right|_{a=a^\star},\enspace
 K_{\mathsf S,h}(\lambda_0)=\mathsf{Var}_{Q_{\lambda_0}^{(h)}}[\psi_{\mathsf S,h}(a^\star,Y)]
.\]
Under correct specification, identification yields $a^\star=\lambda_0$, and differentiable propriety gives $\mathsf E_{Q_{\lambda_0}^{(h)}}[\psi_{\mathsf S,h}(\lambda_0,Y)]=0$. For Gaussian log score, $J_{\log,h}=1/2$; varying candidate and truth together would instead differentiate entropy along the diagonal.
The limit theorems below use the scaled update
\begin{equation}\label{eq::array_driver}
 u_{\mathsf S,h,k}(\lambda,y)
 =G_{\mathsf S,h,k}(\lambda)\psi_{\mathsf S,h}(\lambda,y),
 \qquad G_{\mathsf S,h,k}(\lambda)>0,
\end{equation}
where the scaling is positive and predictable. The two-state limit theorems
require every predictable input to the transition law to be current-state
measurable. A random plug-in scaling is covered if it is asymptotically
replaceable by a deterministic state-based design, with accumulated stopped
drift and martingale replacement errors $o_p(1)$ on every finite horizon;
otherwise a separate augmented-state theorem is required.

Under misspecification, inverse true-risk curvature $J_{\mathsf S,h}^{-1}$ provides an oracle scaling. A feasible alternative uses working-model sensitivity:
\[
 J_{\mathsf S,h}^{\mathrm{work}}(a)
 =
 \mathsf E_{P_a^{(h)}}\!\left[\partial_a^2\mathsf S(P_a^{(h)},Y)\right],
 \qquad
 G_{\mathsf S,h}^{\mathrm{work}}(a)
 =
 \{J_{\mathsf S,h}^{\mathrm{work}}(a)\}^{-1},
\]
assuming $0<J_{\mathsf S,h}^{\mathrm{work}}(a)<\infty$. At a correctly specified, identified target
$a^\star=\lambda_0$, the working and true sensitivities coincide:
$J_{\mathsf S,h}^{\mathrm{work}}(a^\star)=J_{\mathsf S,h}(\lambda_0)$.
\appref{app::formulas} gives the other feasible designs.

At a centred target, let $Z$ denote the standardised innovation and set $C=\mathsf E[Z\psi]$. Under generic scaling $G$, the local mean-reversion loading, update variance, and update--observation covariance are $GJ$, $G^2K$, and $GC$; curvature scaling $G=J^{-1}$ yields $1$, $K/J^2$, and $C/J$. The convergence arguments use predictable characteristics and conditional Lindeberg conditions \citep{stroock2007multidimensional,ethier1986markov,ethier1994convergence,kushner1984approximation,nelson1990arch}.

At mesh $h$, allow static parameters to depend on $h$ and write
\begin{align}\label{eq::discretised}
    \begin{split}
        x_{(k+1)h}^{(h)}-x_{kh}^{(h)}
        &=
        \sqrt{h\,\Gamma(\lambda_{kh}^{(h)})}\,
        z_{(k+1)h},
        \\
        \lambda_{(k+1)h}^{(h)}-\lambda_{kh}^{(h)}
        &=
        \omega_h-(1-\varphi_h)\lambda_{kh}^{(h)}
        +
        \alpha_h u_{\mathsf S,h,k}(\lambda_{kh}^{(h)},Y_{(k+1)h}^{(h)}),
        \qquad
        Y_{(k+1)h}^{(h)}=\sqrt{h\Gamma(\lambda_{kh}^{(h)})}\,z_{(k+1)h}.
    \end{split}
\end{align}

The initial state has law $\pi^{(h)}_0$. In the frozen-family verification, $z_{(k+1)h}$ is drawn conditionally independently from the specified standardised Gaussian or Student law; the limit theorems also permit conditional triangular arrays. For $kh\le t<(k+1)h$, use the piecewise-constant interpolation of~\eqref{eq::discretised}:
\(x_t^{(h)}=x_{kh}^{(h)},\quad\lambda_t^{(h)}=\lambda_{kh}^{(h)}.\)
Write $\mathsf P_0$ for the limiting initial law of $(x_0,\lambda_0)$. Driver centring distinguishes the two limits below.

The limiting transition characteristics are built from the conditional update mean, second moment, and update--innovation cross-moment:
\begingroup
\begin{equation}\label{eq::localvarandcov}
\begin{aligned}
 \mu_{\mathsf S,h,k}(\lambda)
   &=\mathsf E_{kh}[u_{\mathsf S,h,k}(\lambda,Y)],\\
 \Xi_{\mathsf S,h,k}(\lambda)
   &=\mathsf E_{kh}[u_{\mathsf S,h,k}(\lambda,Y)^2],\\
 C_{\mathsf S,h,k}(\lambda)
   &=\mathsf E_{kh}[Z\,u_{\mathsf S,h,k}(\lambda,Y)].
\end{aligned}
\end{equation}
\endgroup
Here $Y=\sqrt{h\Gamma(\lambda)}Z$. When the update is centred, $\Xi$ is its conditional variance and $C$ its covariance with $Z$; otherwise they are, respectively, a second moment and a cross-moment.

Write $m_h(\lambda)=\omega_h-(1-\varphi_h)\lambda$ for the
autoregressive increment. The finite-mesh conditional-moment
expansions, including the mixed characteristic and a fourth-moment
sufficient bound for Lindeberg, are in
{\ref{app:conditional_characteristics}}.

Both regimes use the common drift scaling
\begin{equation}\label{eq::common_drift_scaling}
    \lim_{h \rightarrow 0} h^{-1}\omega_h = \omega,
    \qquad
    \lim_{h \rightarrow 0} h^{-1}(1-\varphi_h) = \kappa.
\end{equation}
Hence, locally uniformly for $\lambda$ in compact sets, $m_h(\lambda)=\omega_h-(1-\varphi_h)\lambda=h(\omega-\kappa\lambda)+o(h)$. The gain order then separates the limits: a centred $O(\sqrt h)$ update retains a quadratic characteristic for $\lambda$, whereas an order-one update mean requires an $O(h)$ gain, which removes the $\lambda$-update noise and leaves the Brownian driver only in $x$.\footnote{The regimes have different option-pricing implications when $\lambda$ is non-tradeable volatility.}
\subsection{Centred diffusion limit}
{Assumption~\ref{ass::assumption_61} in \ref{app:moved_section4} collects the conditions on the scaled update for this regime: conditional centring, convergence of the conditional second moments, and a Lindeberg-type moment bound.}
Nonzero first-order limits $h^{-1/2}\mathsf E_{kh}[Z]\to r_x(\lambda)$ or
$h^{-1/2}\mathsf E_{kh}[u]\to r_\lambda(\lambda)$ would add
$\sqrt{\Gamma(\lambda)}r_x(\lambda)$ or $\alpha r_\lambda(\lambda)$ to the
limiting drift and are excluded.
{These requirements are imposed on the scaled update. Proposition~\ref{prop:verified_score_arrays} in \ref{app:moved_section4} verifies the centring, characteristic and Lindeberg conditions for the score-specific arrays used in the paper, except fixed-data-bandwidth MMD.}

The centred limiting second moments, calculated in {\ref{app:conditional_characteristics}}, determine the covariance matrix
\begin{equation}\label{eq::covariance}
    \mathcal{A}_{\mathsf{S}}(\lambda)=
    \begin{bmatrix}
    \Gamma(\lambda) \zeta_2(\lambda)
    &
    \alpha\sqrt{\Gamma(\lambda)}\,C_{\mathsf{S}}(\lambda)
    \\
    \alpha\sqrt{\Gamma(\lambda)}\,C_{\mathsf{S}}(\lambda)
    &
    \alpha^2\Xi_{\mathsf{S}}(\lambda)
    \end{bmatrix}.
\end{equation}

Since $C_{\mathsf S}^2(\lambda)\le\zeta_2(\lambda)\Xi_{\mathsf S}(\lambda)$ by Cauchy--Schwarz and $\zeta_2(\lambda)>0$, define
$\Delta_{\mathsf S}(\lambda)=\Xi_{\mathsf S}(\lambda)-\zeta_2(\lambda)^{-1}C_{\mathsf S}^2(\lambda)\ge0$,
the component of the scaled-update variance orthogonal to $Z$. An independent second Brownian factor is needed only where $\Delta_{\mathsf S}>0$. A convenient lower-triangular square root of $\mathcal A_{\mathsf S}$ is
\begin{equation}\label{eq::triangular}
   \Sigma^{\rm tr}_{\mathsf{S}}(\lambda)=
    \begin{bmatrix}
    \sqrt{\zeta_2(\lambda)\Gamma(\lambda)}
    &
    0
    \\
    \alpha \zeta_2(\lambda)^{-1/2} C_{\mathsf{S}}(\lambda)
    &
    \alpha \sqrt{\Delta_{\mathsf{S}}(\lambda)}
    \end{bmatrix}\,.
\end{equation}

\begin{samepage}
The covariance matrix identifies the only possible stochastic component of the limit, and
\eqref{eq::triangular} gives a convenient factorisation. To turn these local characteristics into a
well-defined diffusion, we now impose regularity and well-posedness conditions.

{The additional regularity conditions for this regime are collected in Assumption~\ref{ass::assumption_62} in \ref{app:moved_section4}.}
\end{samepage}
\begin{theorem}\label{th::diffusion_limit}
Suppose that Assumptions \ref{ass::assumption_61} and \ref{ass::assumption_62} hold. Then, for every finite $T$,
\[
 (x^{(h)},\lambda^{(h)})\Rightarrow(x,\lambda)
 \quad\text{in }D([0,T],\mathbb R\times\Lambda),
\]
where $(x,\lambda)$ is the unique weak solution of
\begin{equation}
    \ud x_t = \bigl[\Sigma^{\rm tr}_{\mathsf{S}}(\lambda_t)\,\ud W_t\bigr]_1,
    \qquad
    \ud \lambda_t = (\omega-\kappa\lambda_t)\,\ud t
    + \bigl[\Sigma^{\rm tr}_{\mathsf{S}}(\lambda_t)\,\ud W_t\bigr]_2.
\end{equation}
where $W$ is a two-dimensional Brownian motion with independent components and $\Sigma^{\rm tr}_{\mathsf{S}}$ is given by \eqref{eq::triangular}.
\end{theorem}
\proofref{app:proof_diffusion}
Under Assumption~\ref{ass::assumption_61}, the score update contributes no limiting drift, so the $\lambda$-drift is $\omega-\kappa\lambda_t$ for every covered rule. Correct specification, {propriety and differentiability of the rule}, interiority, and differentiation interchange ensure this centring. The logarithmic case recovers the corresponding result of \citet{buccheri2021continuous}, while the other rules modify the diffusion factor through $\Xi_{\mathsf S}$ and $C_{\mathsf S}$.

{Proposition~\ref{prop:compact_containment} in \ref{app:moved_section4} reduces the compact-containment condition to a one-step Lyapunov drift bound.}

\subsection{\texorpdfstring{{Supplementary results for Section~\ref{section::continuous_time}}}{Supplementary results for Section~}}\label{app:moved_section4}
\begin{assumption}\label{ass::assumption_62}
The following additional conditions hold.
\begin{enumerate}
    \item[(i)] The functions $\zeta_2(\cdot)$, $C_{\mathsf{S}}(\cdot)$ and $\sqrt{\Delta_{\mathsf{S}}}(\cdot)$ are locally Lipschitz on $\Lambda$, while $\sqrt{\Gamma(\cdot)}$ is continuous and locally bounded.
    \item[(ii)] The scalar $\lambda$-equation with drift $\omega-\kappa\lambda$ and diffusion coefficients given by the second row of \eqref{eq::triangular} admits a pathwise unique, non-explosive strong solution in $\Lambda$ for every initial law under consideration. Given this solution, the first row of \eqref{eq::triangular}, driven by the same $W^1$, defines $x$ uniquely on every finite horizon.
\end{enumerate}
\end{assumption}

\begin{assumption}\label{ass::assumption_61}
For each $h$, $(x^{(h)},\lambda^{(h)})$ is a time-homogeneous Markov chain
on $\mathbb R\times\Lambda$, and every predictable transition input is
current-state measurable. For every compact $K\subset\Lambda$ and finite
$T$, the following hold locally uniformly in the state on $\mathbb R\times K$;
every displayed transition-characteristic limit is deterministic. In (ii)--(iv),
write $u=u_{\mathsf S,h,k}(\lambda,Y)$.
\begin{enumerate}[leftmargin=2em,label=(\roman*)]
    \item In addition to~\eqref{eq::common_drift_scaling},
    \begin{equation}\label{eq::condition_3}
        h^{-1/2}\alpha_h\to\alpha.
    \end{equation}
    \item In the baseline, the true conditional law at state $\lambda$
    is $P_\lambda^{(h)}$; $Z$ and $u$ are centred state by state. More
    generally, the theorem permits
    \[
      \sup_{k\le T/h,\lambda\in K}
      \max\{|\mathsf E_{kh}[Z]|,|\mathsf E_{kh}[u]|\}=o(\sqrt h).
    \]
    At each fixed mesh, conditional differentiation in the candidate
    parameter is valid up to second order in a neighbourhood of the target:
    \[
      \partial_a^j\mathsf R_{\mathsf S,h}(a;\lambda)
      =\mathsf E_{kh}\!\left[\partial_a^j\mathsf S(P_a^{(h)},Y)\right],
      \qquad j=1,2.
    \]
    With interiority, this gives exact baseline update centring and its
    mean-field derivative.
    \item The conditional characteristics converge:
    \[
      \mathsf E_{kh}[Z^2]\to\zeta_2(\lambda),\qquad
      \mathsf E_{kh}[u^2]\to\Xi_{\mathsf S}(\lambda),\qquad
      \mathsf E_{kh}[Zu]\to C_{\mathsf S}(\lambda),
    \]
    with continuous limits and $\zeta_2(\lambda)>0$.
    \item For every $\varepsilon>0$, the residual and update Lindeberg conditions hold:
    \[
      \sup_{k,\lambda\in K}\mathsf E_{kh}\!\left[Z^2\mathbf 1\{\sqrt{h\Gamma(\lambda)}|Z|>\varepsilon\}\right]\to0,
      \qquad
      \sup_{k,\lambda\in K}\mathsf E_{kh}\!\left[u^2\mathbf 1\{|\alpha_hu|>\varepsilon\}\right]\to0.
    \]
    A locally uniform $(2+\delta)$ moment bound on both $Z$ and $u$, for some $\delta>0$, suffices.
    \item The initial laws converge and the joint array satisfies compact containment.
    \item The limiting martingale problem with the drift and covariance below is well posed.
\end{enumerate}
\end{assumption}

These requirements are imposed on the scaled update and supplement {propriety}. Proposition~\ref{prop:compact_containment} gives a criterion for containment when its drift bound holds; on $\Lambda=\mathbb R$, one may use $V(\lambda)=1+\lambda^2$ after checking that bound.
The next result verifies the score-specific centring, characteristic, and Lindeberg conditions.

\begin{proposition}[Score-specific verification]\label{prop:verified_score_arrays}
Let $q_h(\lambda)=he^\lambda$, $Y=\sqrt{q_h(\lambda)}Z$, and let $Z$ have either the standard Gaussian law or the variance-standardised Student law with $\nu=6$. For the logarithmic, density-power with fixed $\beta>0$, and CRPS rules, choose
\(
 G_{\log,h}=2
\),
\(
 G_{\beta,h}=2q_h^{\beta/2}
\),
\(
 G_{{\rm CRPS},h}=2q_h^{-1/2}
\).
The resulting scaled updates depend only on $Z$ and satisfy
\[
 \mathsf E[u_{\mathsf S}]=0,
 \qquad
 \mathsf E[Zu_{\mathsf S}]=0,
 \qquad
 0<\mathsf E[u_{\mathsf S}^2]<\infty.
\]
Consequently, the centring, characteristic, and Lindeberg requirements in Assumption~\ref{ass::assumption_61}-(ii)--(iv) hold for all six score--density pairs. Initial convergence, compact containment, and well posedness remain separate conditions.
\end{proposition}
\proofref{app:proof_verified_arrays}
The only unbounded update among these cases is Gaussian log: a $(2+\delta)$ update moment requires a residual moment of order $4+2\delta$.
Fixed-data-bandwidth MMD remains part of the finite-step catalogue but is not covered by Proposition~\ref{prop:verified_score_arrays}, whose verification exploits exact scale homogeneity; its pointwise curvature-normalised limit and the distinct mesh-dependent-bandwidth design are discussed in \appref{app::formulas}.

Theorem~\ref{th::diffusion_limit} leaves compact containment as an abstract condition. The following
Lyapunov criterion makes that requirement operational by reducing it to a one-step drift bound.

\begin{proposition}[Lyapunov criterion for compact containment]\label{prop:compact_containment}
Let $V:\Lambda\to[1,\infty)$ be $C^2$ with compact sublevel sets in $\Lambda$. Suppose $\sup_h\mathsf E[V(\lambda_0^{(h)})]<\infty$ and, for all sufficiently small $h$ under the actual recursion,
\[
 \mathsf E_{kh}[V(\lambda_{(k+1)h}^{(h)})-V(\lambda_{kh}^{(h)})]
 \le ch\{1+V(\lambda_{kh}^{(h)})\}.
\]
Then $\lambda^{(h)}$ is compactly contained on every finite horizon. Suppose additionally that $x_0^{(h)}$ is tight and that $\Gamma$ and the conditional residual second moments are locally bounded. If, after stopping on each compact subset of $\Lambda$ and for every finite $T$, the predictable partial-sum maxima
\(
 \sup_{t\le T}\left|
 \sum_{k<t/h}\sqrt{h\Gamma(\lambda_{kh}^{(h)})}\,
 \mathsf E_{kh}[Z]
 \right|
\)
form a tight family, then $(x^{(h)},\lambda^{(h)})$ is jointly compactly contained. Exact conditional centring, $\mathsf E_{kh}[Z]=0$, makes the last condition automatic.
\end{proposition}
\proofref{app:proof_compact_containment}
Thus the abstract containment requirement reduces to a one-step Lyapunov drift bound together with standard control of the observation coordinate.

\begin{assumption}\label{ass::assumption_63}
The variance map $\Gamma:\Lambda\to[0,\infty)$ is continuous and locally bounded.
For each $h$, $(x^{(h)},\lambda^{(h)})$ is a time-homogeneous Markov chain on $\mathbb R\times\Lambda$, and every predictable transition input is current-state measurable. For every compact $K\subset\Lambda$ and finite horizon $T$, the conditions below hold locally uniformly over the joint state in $\mathbb R\times K$.
\begin{enumerate}[leftmargin=2em,label=(\roman*)]
\item The actual recursion has an $O(h)$ gain. With
\[
{\Delta\lambda_{(k+1)h}^{(h)}=\lambda_{(k+1)h}^{(h)}-\lambda_{kh}^{(h)}}
\]
and~\eqref{eq::common_drift_scaling},
\begin{equation}\label{eq::condition_3_deg}
\begin{gathered}
\Delta\lambda_{(k+1)h}^{(h)}
=\omega_h-(1-\varphi_h)\lambda_{kh}^{(h)}
+a_hu_{\mathsf S,h,k}(\lambda_{kh}^{(h)},Y_{(k+1)h}^{(h)}),\\
h^{-1}a_h\to\bar\alpha.
\end{gathered}
\end{equation}
\item For (ii)--(iv), put $u=u_{\mathsf S,h,k}(\lambda,Y)$. Locally uniformly,
$\tilde\mu_{\mathsf S,h}(\lambda)\to\tilde\mu_{\mathsf S}(\lambda)$, where
$\tilde\mu_{\mathsf S}$ is continuous and not identically zero on $\Lambda$.
The family $u^2$ is locally uniformly integrable.
\item $\mathsf E_{kh}[Z]=0$ and
$\mathsf E_{kh}[Z^2]\to\tilde\zeta_2(\lambda)$ locally uniformly,
with continuous $\tilde\zeta_2>0$.
\item For every $\varepsilon>0$, the residual and update conditional Lindeberg conditions hold locally uniformly:
\[
  \sup_{k,\lambda\in K}\mathsf E_{kh}\!\left[Z^2\mathbf 1\{\sqrt{h\Gamma(\lambda)}|Z|>\varepsilon\}\right]\to0,
  \qquad
  \sup_{k,\lambda\in K}h^{-1}a_h^2\mathsf E_{kh}\!\left[u^2\mathbf 1\{|a_hu|>\varepsilon\}\right]\to0.
\]
\item $\pi^{(h)}_0\Rightarrow\mathsf P_0$, a specified law on $\mathbb R\times\Lambda$ with
$\lambda$-marginal concentrated at $\lambda_{\rm init}$; the joint array
satisfies compact containment.
\item The ordinary differential equation
\[
\ud\lambda_t=(\omega-\kappa\lambda_t+\bar\alpha\tilde\mu_{\mathsf S}(\lambda_t))\ud t,
\qquad \lambda_0=\lambda_{\rm init}\in\Lambda,
\]
has a unique non-explosive solution in $\Lambda$. For every finite $T$, its range $\{\lambda_t:0\le t\le T\}$ is contained in a compact subset $K_T\subset\Lambda$.
\end{enumerate}
\end{assumption}

\subsection{Non-centred mean-flow limit}\label{subsec:mean_flow_moved}
Let $\tilde\pi_{z,\lambda}^{(h)}$ denote the true conditional residual law at working state $\lambda$, and define the mean field of the actual scaled update by $\tilde\mu_{\mathsf S,h}(\lambda)=\mathsf E^{\tilde\pi_{z,\lambda}^{(h)}}[u_{\mathsf S,h}(\lambda,Y)]$.
The non-centred regime is characterised by $\tilde\mu_{\mathsf S,h}\to\tilde\mu_{\mathsf S}$, with $\tilde\mu_{\mathsf S}\ne0$ on part of the state space. This does not itself imply misspecification: a misspecified update can be centred at its pseudo-true value. The decomposition $u_{\mathsf S,h}=\tilde\mu_{\mathsf S,h}+(u_{\mathsf S,h}-\tilde\mu_{\mathsf S,h})$ shows why the gain changes. With $a_h=O(h)$, the mean term contributes to the drift, whereas the centred fluctuation vanishes.

{The conditions for this regime parallel Assumption~\ref{ass::assumption_61} and are collected in Assumption~\ref{ass::assumption_63} in \ref{app:moved_section4}.}
Proposition~\ref{prop:compact_containment} again gives a criterion for containment, with $V(\lambda)=1+\lambda^2$ available on $\mathbb R$ after checking its drift bound.

\begin{theorem}\label{th::mean_flow}
Suppose that Assumption \ref{ass::assumption_63} holds. Then, for every finite $T$,
\[
 (x^{(h)},\lambda^{(h)})\Rightarrow(x,\lambda)
 \quad\text{in }D([0,T],\mathbb R\times\Lambda),
\]
where $(x,\lambda)$ is the unique weak solution of
\begin{equation}
        \ud x_t = \sqrt{\tilde\zeta_2(\lambda_t)\Gamma(\lambda_t)}\,\ud W_t, \qquad
        \ud \lambda_t = (\omega-\kappa\lambda_t+\bar\alpha\tilde\mu_{\mathsf S}(\lambda_t))\,\ud t,
\end{equation}
with initial law $\mathsf P_0$, where $W$ is a one-dimensional Brownian motion.
\end{theorem}
\proofref{app:proof_mean_flow}
Theorem~\ref{th::mean_flow} complements Theorem~\ref{th::diffusion_limit}. With a centred update and an $O(\sqrt h)$ gain, score noise survives in the $\lambda$-diffusion; an order-one non-centred mean instead requires an $O(h)$ gain and contributes only to deterministic drift. Thus centring determines the qualitative form of the continuous-time limit, not merely its coefficients. Section~\ref{sec::local_consistency} exploits this distinction by re-centring the update around a moving pseudo-true path and recovering an $O(\sqrt h)$ stochastic tracking regime.

The matched-scale restriction $\tilde\zeta_2(\lambda)=\zeta_2(\lambda)$ holds residual variance fixed so that differences across criteria isolate shape misspecification. It does not define non-centring, which is governed by the mean field.

\subsection{Conditional-characteristic calculations}\label{app:conditional_characteristics}
For the calculations below, suppress $(h,k)$ and set $g_{\mathsf S}(\lambda,z)=u_{\mathsf S,h,k}(\lambda,\sqrt{h\Gamma(\lambda)}z)$. Let $\mathsf E_{kh}^{\pi_z}$ denote expectation under the current conditional law of the standardised innovation $Z$, and write $\zeta^{(j)}=\mathsf E_{kh}^{\pi_z}[Z^j]$ for $j=2,4$. The abstract limit assumptions require local uniform convergence of the residual second moment where stated; a vanishing bound based on the fourth-moment identity below is one sufficient route to Lindeberg, not an abstract convergence premise.

Writing $m_h(\lambda)=\omega_h-(1-\varphi_h)\lambda$, direct conditioning gives
\begin{equation}\label{eq::moments}
\begin{split}
    \mathsf E_{kh}^{\pi_z}
    [x_{(k+1)h}^{(h)}-x_{kh}^{(h)}]
    &=
    h^{1/2}\sqrt{\Gamma(\lambda_{kh}^{(h)})}\,
        \mathsf E_{kh}^{\pi_z}[Z],
    \\
    \mathsf E_{kh}^{\pi_z}
    [(x_{(k+1)h}^{(h)}-x_{kh}^{(h)})^2]
    &=
    h \Gamma(\lambda_{kh}^{(h)})\zeta^{(2)},
    \\
    \mathsf E_{kh}^{\pi_z}
    [\lambda_{(k+1)h}^{(h)}-\lambda_{kh}^{(h)}]
    &=
    m_h(\lambda_{kh}^{(h)})
    +
    \alpha_h
    \mathsf E_{kh}^{\pi_z}
    [g_{\mathsf S}(\lambda_{kh}^{(h)},Z)],
    \\
    \mathsf E_{kh}^{\pi_z}
    [(\lambda_{(k+1)h}^{(h)}-\lambda_{kh}^{(h)})^2]
    &=
    m_h^2(\lambda_{kh}^{(h)})
    +
    2m_h(\lambda_{kh}^{(h)})\alpha_h
    \mathsf E_{kh}^{\pi_z}
    [g_{\mathsf S}(\lambda_{kh}^{(h)},Z)]
    +
    \alpha_h^2
    \mathsf E_{kh}^{\pi_z}
    [g_{\mathsf S}^2(\lambda_{kh}^{(h)},Z)].
\end{split}
\end{equation}

The mixed second moment supplies the limiting off-diagonal covariance:
\begin{equation}\label{eq::mix_term}
\begin{aligned}
\mathsf E_{kh}^{\pi_z}\!\left[
(x_{(k+1)h}^{(h)}-x_{kh}^{(h)})
(\lambda_{(k+1)h}^{(h)}-\lambda_{kh}^{(h)})\right]
&=h^{1/2}\alpha_h\sqrt{\Gamma(\lambda_{kh}^{(h)})}
C_{\mathsf S,h,k}(\lambda_{kh}^{(h)})\\
&\quad+h^{1/2}\sqrt{\Gamma(\lambda_{kh}^{(h)})}
m_h(\lambda_{kh}^{(h)})\mathsf E_{kh}^{\pi_z}[Z].
\end{aligned}
\end{equation}

Fourth moments provide a sufficient large-jump negligibility condition,\footnote{{It suffices that $h^{-1}\mathsf E[\|\Delta X^{(h)}\|^{2+\delta}\mid X^{(h)}=x]\to0$ locally uniformly on compact state sets for some $\delta>0$ \citep{nelson1990arch}; the required normalisation is by $h^{-1}$. Here we use $\delta=2$ for the joint state increment.}}
\begin{equation}\label{eq::fourth_moments}
\begin{split}
    h^{-1}\mathsf E_{kh}^{\pi_z}
    [(x_{(k+1)h}^{(h)}-x_{kh}^{(h)})^4]
    &=
    h\,\Gamma^2(\lambda_{kh}^{(h)})\zeta^{(4)},
    \\
    h^{-1}\mathsf E_{kh}^{\pi_z}
    [(\lambda_{(k+1)h}^{(h)}-\lambda_{kh}^{(h)})^4]
    &\le
    8h^{-1}m_h^4(\lambda_{kh}^{(h)}) +
    8h^{-1}\alpha_h^4
    \mathsf E_{kh}^{\pi_z}
    [g_{\mathsf S}^4(\lambda_{kh}^{(h)},Z)].
\end{split}
\end{equation}

\subsection{Proof of Proposition~\ref{prop:verified_score_arrays}}\label{app:proof_verified_arrays}
Write $q=q_h(\lambda)$ and change variables from $y$ to
$z=y/\sqrt q$. Let $f$ and $F_f$ denote the density and distribution function
of $Z$, put $s_f(z)=\partial_z\log f(z)$, and, for independent
$Z',Z''\sim f$, define
\[
 A_f(z)=\mathsf E[|Z'-z|]-\tfrac12\mathsf E[|Z'-Z''|],
 \qquad
 c_f(\beta)=\int f^{1+\beta}.
\]
Since $\partial_\lambda q=q$ and $\partial_\lambda z=-z/2$, direct
differentiation gives
\[
 U_{\log}(z)=-1-zs_f(z),
\]
\[
 U_\beta(z)=\beta c_f(\beta)+(1+\beta)f(z)^\beta[-1-zs_f(z)],
 \qquad
 U_{\rm CRPS}(z)=z\{2F_f(z)-1\}-A_f(z),
\]
and
\[
\begin{array}{c|ccc}
\mathsf S & \log & \text{density-power} & \mathrm{CRPS} \\
\hline
\psi_{\mathsf S,h} & U_{\log}/2 & q^{-\beta/2}U_\beta/2 & q^{1/2}U_{\rm CRPS}/2 \\
G_{\mathsf S,h} & 2 & 2q^{\beta/2} & 2q^{-1/2}
\end{array}
\]
The powers of $q$ follow from $p_q(y)=q^{-1/2}f(z)$ and
$\operatorname{CRPS}(P_q,y)=q^{1/2}\operatorname{CRPS}(P_1,z)$; the displayed
scalings therefore leave the residual-only updates $U_{\log}$, $U_\beta$, and
$U_{\rm CRPS}$.

At the frozen true law, differentiable propriety and differentiation under the
integral give $\mathsf E[U_{\mathsf S}(Z)]=0$. For the two symmetric densities,
each $U_{\mathsf S}$ is even, so $ZU_{\mathsf S}(Z)$ is odd and has expectation
zero. Squaring gives $\Xi_{\mathsf S}=\mathsf E[U_{\mathsf S}(Z)^2]$.
Differentiating the mean driver once more at the target gives
\[
 J_{\log,h}=j_{\log},
 \qquad J_{\beta,h}=q^{-\beta/2}j_\beta,
 \qquad J_{{\rm CRPS},h}=q^{1/2}j_{\rm CRPS},
\]
where
\[
 j_{\log}=\frac14\int f(1+zs_f)^2,
 \quad j_\beta=\frac{1+\beta}{4}\int f^{1+\beta}(1+zs_f)^2,
 \quad j_{\rm CRPS}=\frac12\int z^2f(z)^2\,\mathrm dz.
\]
These quantities are positive and finite for the two stated densities. The
Gaussian logarithmic update is $Z^2-1$; its $(2+\delta)$ moment requires a
$Z$ moment of order $4+2\delta$. The Student logarithmic, density-power, and
CRPS updates are bounded. Gaussian innovations have moments of every order,
while Student--6 has the fourth moment, so a common $(2+\delta)$ sufficient
bound is available in every cell. Because the frozen laws and residual-only
updates do not vary with $h$, the required local uniformity follows. This
verifies Assumption~\ref{ass::assumption_61}-(ii)--(iv); conditions (v)--(vi)
remain separate.

\subsection{Proof of Theorem~\ref{th::diffusion_limit}}\label{app:proof_diffusion}
Write $\Delta_k^{(h)}=\big(x_{(k+1)h}^{(h)}-x_{kh}^{(h)},\,\lambda_{(k+1)h}^{(h)}-\lambda_{kh}^{(h)}\big)^\top$, and let $b(x,\lambda)=(0,\,\omega-\kappa\lambda)^\top$ and $\mathcal A_{\mathsf S}(\lambda)$ be the candidate covariance in~\eqref{eq::covariance}. We verify conditions (2.4)--(2.6) in Chapter~11 of \citet{stroock2007multidimensional} (see also \citealp[p.~40]{arnold1974stochastic}), locally uniformly on every compact $K\subset\Lambda$.

\emph{(i) Drift.} Equations~\eqref{eq::common_drift_scaling} and~\eqref{eq::condition_3}, with Assumption~\ref{ass::assumption_61}-(ii), give
\[
 h^{-1}\mathsf E_{kh}[\Delta x_k^{(h)}]
 =\frac{\sqrt{\Gamma(\lambda)}}{\sqrt h}\mathsf E_{kh}[Z]=o(1),
 \qquad
 h^{-1}\alpha_h\mathsf E_{kh}[u]=o(1),
\]
and hence $h^{-1}\mathsf E_{kh}[\Delta_k^{(h)}]\to b(\cdot,\lambda)$. Thus the $o(\sqrt h)$ approximate-centring rate imposed in Assumption~\ref{ass::assumption_61}-(ii) controls both coordinates.

\emph{(ii) Covariance.} The three characteristic limits in Assumption~\ref{ass::assumption_61}-(iii), together with $\alpha_h/\sqrt h\to\alpha$ and $m_h=O(h)$, give $h^{-1}\mathsf E_{kh}[\Delta_k^{(h)}(\Delta_k^{(h)})^\top]\to\mathcal A_{\mathsf S}(\lambda)$. The vanishing mean corrections do not affect the quadratic characteristic. Cauchy--Schwarz makes $\mathcal A_{\mathsf S}$ nonnegative definite, including the degenerate case.

\emph{(iii) Negligible jumps.} The two conditional Lindeberg conditions in Assumption~\ref{ass::assumption_61}-(iv) imply that the predictable compensator of jumps larger than any fixed $\varepsilon$ vanishes. The fourth-moment display~\eqref{eq::fourth_moments} remains one sufficient score-specific verification, but is not an abstract premise.

Initial convergence, compact containment, and well posedness are imposed in
Assumption~\ref{ass::assumption_61}-(v)--(vi). The locally uniform transition-characteristic
limits and negligible-jump condition therefore verify the Markov-chain martingale-problem
criterion of \citet{stroock2007multidimensional} and \citet{ethier1986markov}, which yields the asserted weak
convergence. The Markov restriction in Assumption~\ref{ass::assumption_61} is essential to this
proof.

\subsection{Proof of Proposition~\ref{prop:compact_containment}}\label{app:proof_compact_containment}
\begin{proof}
Put $U=1+V$ and, for $R>0$, define
$\tau_R^{(h)}=\inf\{k:V(\lambda_{kh}^{(h)})\ge R\}$.  The one-step drift condition gives, for $k<\tau_R^{(h)}$,
\[
  \mathsf E_k[U(\lambda_{(k+1)h}^{(h)})]
  \le(1+ch)U(\lambda_{kh}^{(h)}).
\]
In particular, after stopping and iterating,
\[
  \mathsf E[U(\lambda_{(n\wedge\tau_R^{(h)})h}^{(h)})]
  \le(1+ch)^n\mathsf E[U(\lambda_0^{(h)})]
  \le e^{cT}\sup_h\mathsf E[U(\lambda_0^{(h)})]
\]
for $n\le T/h$.  On $\{\tau_R^{(h)}\le T/h\}$ the stopped value is at least
$1+R$; hence,
\[
  \sup_h\mathbb P(\tau_R^{(h)}\le T/h)
  \le\frac{e^{cT}\sup_h\mathsf E[U(\lambda_0^{(h)})]}{1+R}
  \longrightarrow0
  \qquad(R\to\infty).
\]
Compact sublevel sets of $V$ prove compact containment of the $\lambda$ array; we notice that no update Lindeberg condition is needed for this Lyapunov step.

For the $x$ coordinate, stop when $\lambda^{(h)}$ leaves one of the compact
sublevel sets just obtained and decompose $x_t^{(h)}-x_0^{(h)}=A_t^{(h)}+M_t^{(h)}$,
where $A_t^{(h)}=\sum_{k<t/h}\sqrt{h\Gamma(\lambda_{kh}^{(h)})}\,\mathsf E_k[Z_{k+1}]$ and $M_t^{(h)}$ is the corresponding martingale.  The stated
predictable-part condition makes the process maximum $\sup_{t\le T}|A_t^{(h)}|$ tight; under exact centring it is identically zero.  On the stopped compact, local boundedness of $\Gamma$ and of the residual second moments gives $\mathsf E\!\left[\langle M^{(h)}\rangle_T\right]\le C_T$. Therefore, Doob's inequality makes the martingale maximum $\sup_{t\le T}|M_t^{(h)}|$ tight.  Together with tightness of $x_0^{(h)}$ and the preceding $\lambda$ containment, this proves joint compact containment.
\end{proof}

\subsection{Proof of Theorem~\ref{th::mean_flow}}\label{app:proof_mean_flow}
The $O(h)$ gain and local-uniform mean convergence give
\[
h^{-1}\mathsf E_{kh}[\Delta\lambda]
=h^{-1}m_h(\lambda)+h^{-1}a_h\tilde\mu_{\mathsf S,h}(\lambda)
\longrightarrow\omega-\kappa\lambda+\bar\alpha\tilde\mu_{\mathsf S}(\lambda).
\]
Local uniform integrability of $u^2$ and $a_h=O(h)$ imply $h^{-1}\mathsf{Var}_{kh}(a_hu)=O(h)\to0$; Cauchy--Schwarz gives order $O(\sqrt h)$ for the mixed characteristic, which also vanishes. Continuity and local boundedness of $\Gamma$ in Assumption~\ref{ass::assumption_63}, together with locally uniform convergence of the residual second moment, give locally uniform convergence of the observation quadratic characteristic to $\Gamma(\lambda)\tilde\zeta_2(\lambda)$. The two conditional Lindeberg conditions remove large jumps, so the limiting generator is
\[
(\mathcal A f)(x,\lambda)=\{\omega-\kappa\lambda+\bar\alpha\tilde\mu_{\mathsf S}(\lambda)\}\partial_\lambda f+\tfrac12\Gamma(\lambda)\tilde\zeta_2(\lambda)\partial^2_{xx}f.
\]
The unique ODE solution in Assumption~\ref{ass::assumption_63} remains in a compact subset of $\Lambda$ on each finite horizon. Thus $\Gamma(\lambda_t)\tilde\zeta_2(\lambda_t)$ is continuous and bounded there, and
\[
 x_t=x_0+\int_0^t\sqrt{\Gamma(\lambda_s)\tilde\zeta_2(\lambda_s)}\,\ud W_s
\]
defines the unique weak observation limit for the specified initial law. Initial convergence, compact containment, and this uniqueness then allow the same Markov-chain martingale-problem criterion as in Theorem~\ref{th::diffusion_limit}.

\section{Proofs for Section~\ref{sec::local_consistency}: local tracking and distributional approximation}\label{app::appendix_section_7}
\subsection{\texorpdfstring{{Supplementary results for Section~\ref{sec::local_consistency}}}{Supplementary results for Section}}\label{app:moved_section5_more}
\begingroup
\begin{definition}[Scale-equivariant, reflection-symmetric rules]\label{trk:def:equi}
A scoring rule $\mathsf S$ is \emph{scale-equivariant} on $\{P_\lambda\}$ if there are functions $A:\R\to(0,\infty)$ and $B:\R\to\R$ such that
\begin{equation}\label{trk:eq:equi}
\mathsf S\bigl(P_{\lambda+a},\mathsf T_a(y)\bigr)=A(a)\,\mathsf S(P_\lambda,y)+B(a)\qquad\text{for all }a,\lambda,y .
\end{equation}
It is \emph{reflection-symmetric} if $\mathsf S(P_\lambda,-y)=\mathsf S(P_\lambda,y)$.
\end{definition}

The following are scale-equivariant and reflection-symmetric: the log score ($A\equiv1$, $B(a)=a/2$), the Student-$t$ log score, the CRPS ($A(a)=e^{a/2}$), and the density-power score in the normalisation of \ref{app::formulas} ($A(a)=e^{-\beta a/2}$). The fixed-data-bandwidth MMD score is reflection-symmetric but not scale-equivariant.

\begin{assumption}[Scale model]\label{trk:ass:scale}
\leavevmode
\begin{enumerate}[label=(\alph*)]
\item The working family is $\{P^{(h)}_\lambda\}$ above, and $\mathsf S$ is scale-equivariant and reflection-symmetric.
\item Conditionally on $\mathcal G^{(h)}_{kh}$, $y^{(h)}_{(k+1)h}=\sqrt h\,e^{\tilde\lambda^{(h)}_{kh}/2}\varepsilon_{(k+1)h}$, where $\varepsilon_{(k+1)h}$ has a fixed law $F$ with a Lebesgue density $f$ that is locally absolutely continuous as a function on $\mathbb R$. For the log-variance interpretation impose $\E_F\varepsilon=0$ and $\E_F\varepsilon^2=1$. Then $\Var_{kh}(y^{(h)}_{(k+1)h})=h e^{\tilde\lambda^{(h)}_{kh}}$. Without fixing the innovation scale, replacing $x$ by $x+c$ and $\varepsilon$ by $e^{-c/2}\varepsilon$ leaves the observations unchanged. For an arbitrary fixed $F$ of positive finite variance the log conditional variance per unit time is $x+\log\Var_F(\varepsilon)$. Derivatives are understood almost everywhere. Define the log-variance score by $s_F(y)=-\tfrac12-yf'(y)/\{2f(y)\}$ on $\{f>0\}$ and set it to zero elsewhere; assume $I_F=\E_F[s_F^2]\in(0,\infty)$.
\item $R_F(\delta)=\E_F[\mathsf S(P_\delta,\varepsilon)]$ has a unique interior minimiser $\delta_{\mathsf S}(F)$ with $J_F=R_F''(\delta_{\mathsf S}(F))>0$, and Assumption~\ref{ass:conditional_differentiation} holds for the scale family generated by $F$.
\item In \eqref{eq::latent_space_recursion}, $\Var_{kh}(\eta)=1$ and $\sigma_h\to\sigma\ge0$ uniformly on compacts.
\end{enumerate}
\end{assumption}
\endgroup

\begin{assumption}[Local Ornstein--Uhlenbeck approximation]
\label{ass::local-ou}
{Fix $T\in(0,T_0)$ and $M<\infty$. Throughout this assumption, $b_{\mathsf S}(T)>0$ and $q_{\mathsf S}(T)\ge0$ are finite deterministic constants. Put $k_T=\lfloor T/h\rfloor$ and $k_j=k_T+j$; work on $k_T\le k\le k_T+\lfloor M/\sqrt h\rfloor$. For each $K<\infty$, define the common local exit time}
\(
 \sigma_{h,K}=\inf\{\tau\in[0,M]:|Z_{\mathsf S,h}(\tau)|\ge K\}\wedge M
\).
\begin{enumerate}[leftmargin=2em,label=(\roman*)]
\item \emph{Initial condition and local tube.} $Z_{\mathsf S,h}(0)$ is $\mathcal G_{k_Th}^{(h)}$-measurable and $Z_{\mathsf S,h}(0)\Rightarrow Z_0$. For every fixed $K<\infty$, the local tube $m_{kh}^{(h)}+h^{1/4}[-K,K]$ is contained in $\Lambda$, uniformly over the stated window, almost surely for all sufficiently small $h$; this is automatic when $\Lambda=\mathbb R$.
\item \emph{Frozen local drift.} At the target, $\bar u_{kh}(m_{kh}^{(h)})=0$ and $b_{\mathsf S,h,kh}=-\partial_\lambda\bar u_{kh}(m_{kh}^{(h)})\to_p b_{\mathsf S}(T)>0$ uniformly on the window.
\item \emph{Mean-field linearisation.} For each $K<\infty$, the stopped accumulated Taylor remainder satisfies
\[
 \sup_{0\le\tau\le M}\left|
 \sum_{j<n_h(\tau\wedge\sigma_{h,K})}\rho h^{1/4}
 \left\{\bar u_{k_jh}\!\left(m_{k_jh}^{(h)}+e_{k_jh}^{(h)}\right)
 +b_{\mathsf S,h,k_jh}e_{k_jh}^{(h)}\right\}
 \right|\to_p0.
\]
\item \emph{Evaluation noise.} There is a finite deterministic $L$, independent of $h$, $k$, and $K$, such that, for each fixed $K>0$ and all sufficiently small $h$, the following conditional bound holds almost surely for every $k$ in the stated window and every $\mathcal G_{kh}^{(h)}$-measurable offset $x$ with $|x|\le Kh^{1/4}$ almost surely:
\begin{equation}
\mathsf E_{kh}\!\left[\left|
\begin{aligned}
&u_{\mathsf S}^{(h)}(y_{(k+1)h},m_{kh}^{(h)}+x)
-u_{\mathsf S}^{(h)}(y_{(k+1)h},m_{kh}^{(h)})\\[-0.2em]
&\quad-\{\bar u_{kh}(m_{kh}^{(h)}+x)-\bar u_{kh}(m_{kh}^{(h)})\}
\end{aligned}
\right|^2\right]\le L^2x^2.
\label{eq::local_driver_lipschitz}
\end{equation}
\item \emph{Martingale limit.} Set
$\mathcal G_{h,j}=\mathcal G_{k_jh}^{(h)}$. For $0\le j<n_h(M)$, suppose $\tilde\xi_{(k_j+1)h}^{(h)}$ is
$\mathcal G_{h,j+1}$-measurable, square-integrable, and
$\mathsf E[\tilde\xi_{(k_j+1)h}^{(h)}\mid\mathcal G_{h,j}]=0$.
Uniformly in $0\le\tau\le M$,
\[
 \sum_{j<n_h(\tau)}h^{1/2}q^2_{\mathsf S,h,k_jh}\to_p q_{\mathsf S}(T)^2\tau,
\]
and the conditional Lindeberg condition
\begin{equation}
\sum_{j<n_h(M)}h^{1/2}\mathsf E_{k_jh}\!\left[(\tilde\xi_{(k_j+1)h}^{(h)})^2\mathsf1_{\{|h^{1/4}\tilde\xi_{(k_j+1)h}^{(h)}|>\varepsilon\}}\right]\to_p0
\label{eq::local_lindeberg_composite}
\end{equation}
holds for every $\varepsilon>0$.
\end{enumerate}
\end{assumption}

\begingroup
\begin{remark}[Observer interpretation and the role of the scaling]\label{trk:rem:observer}
With $\gamma_{\mathsf S}(T)=0$, the frozen limit of Theorem~\ref{th::local-ou} has the same law as the solution of $dZ=-{\varrho} Z\,d\tau+{\varrho}\kappa_{\mathsf S}^{1/2}dW^{(1)}-\bar\Gamma_T\,dW^{(2)}$, with independent Brownian motions. This is the error equation of the constant-gain observer $d\hat X={\varrho}(dY-\hat X\,d\tau)$ for a latent Brownian motion $dX=\bar\Gamma_T\,dW^{(2)}$ observed through $dY=X\,d\tau+\kappa_{\mathsf S}^{1/2}dW^{(1)}$.

Proposition~\ref{trk:prop:frontier}(iii) is therefore the scalar steady-state solution of \citet{kalman1961new}. The optimal gain is $|\bar\Gamma_T|/\kappa_{\mathsf S}^{1/2}$, and the minimal error variance is the geometric mean $(\kappa_{\mathsf S}\bar\Gamma_T^2)^{1/2}$ of the two noise intensities. This is why efficiency enters the frontier through a square root. Locally, a proper-score filter is an observer whose effective observation noise is $\kappa_{\mathsf S}$.

The invariance to $G_{\mathsf S}$ is the local counterpart of Remark~\ref{rem:dynamic_equivalence}: any positive rescaling of the update can be undone by the gain. Curvature scaling only fixes units ($b_T=1$, ${\varrho}=\rho$).
\end{remark}
\endgroup

\begin{proposition}[Recentring at the pseudo-true target]\label{prop::proposition_recentering}
Fix \(h>0\) and a latent state \(\tilde\lambda\). Suppose that the minimiser
\(\lambda^{\star}_{\mathsf{S},h}(\tilde\lambda)\) in
\eqref{eq::minimiser_pseudo_true} exists and lies in the interior of
\(\Lambda\). Assume that differentiation under the expectation is valid
in a neighbourhood of this minimiser, so that
\(\partial_\lambda \widetilde{\mathsf R}_{\mathsf S}^{(h)}(\lambda;\tilde\lambda)
= \mathsf E^{\widetilde P_{\tilde\lambda}^{(h)}}
[\partial_\lambda \mathsf S(P_\lambda^{(h)},y)]\).
Then
\begin{equation}\label{eq::recentering}
    \mathsf{E}^{\widetilde P_{\tilde\lambda}^{(h)}}[\psi_{\mathsf{S}}^{(h)}(y,\lambda^{\star}_{\mathsf{S},h}(\tilde\lambda))]=0.
\end{equation}
If, in addition, $G_{\mathsf{S}}^{(h)}(\lambda^{\star}_{\mathsf{S},h}(\tilde\lambda))$ is finite and does not depend on the observation $y$, then
\begin{equation}\label{eq::recentering2}
    \mathsf E^{\widetilde P_{\tilde\lambda}^{(h)}}
    \left[
        u_{\mathsf S}^{(h)}
        \bigl(y,\lambda^{\star}_{\mathsf S,h}(\tilde\lambda)\bigr)
    \right]
    =
    G_{\mathsf S}^{(h)}
    \bigl(\lambda^{\star}_{\mathsf S,h}(\tilde\lambda)\bigr)
    \mathsf E^{\widetilde P_{\tilde\lambda}^{(h)}}
    \left[
        \psi_{\mathsf S}^{(h)}
        \bigl(y,\lambda^{\star}_{\mathsf S,h}(\tilde\lambda)\bigr)
    \right]
    =
    0.
\end{equation}

\end{proposition}
\proofref{app:proof_recentering}

Within the setting of Theorem~\ref{th::local-ou}, consider the deterministic-target special case.
Over the local window, write
\[
 J_{\mathsf S,h,kh}
 =\partial_{\lambda\lambda}^2
 \widetilde{\mathsf R}_{\mathsf S,kh}^{(h)}(m_{kh}^{(h)}),
 \qquad
 K_{\mathsf S,h,kh}
 =\mathsf{Var}_{kh}\!\left[
 \psi_{\mathsf S}^{(h)}(y_{(k+1)h},m_{kh}^{(h)})
 \right].
\]
Suppose that, eventually and throughout this window, $\bar\Gamma_{kh}=0$,
$0<J_{\mathsf S,h,kh}<\infty$, and
$G_{\mathsf S}^{(h)}(m_{kh}^{(h)})=J_{\mathsf S,h,kh}^{-1}$, and that
\(
 \sup_{0\le j<n_h(M)}
 \left|
 \frac{K_{\mathsf S,h,(k_T+j)h}}
 {J_{\mathsf S,h,(k_T+j)h}^{2}}
 -\kappa_{\mathsf S}
 \right|\to_p0
\).
Then~\eqref{eq::local_curvature_interpretation} gives
$b_{\mathsf S,h,kh}=1$, while~\eqref{eq::q_variance_decomposition} and
Assumption~\ref{ass::local-ou}-(v) give
$q_{\mathsf S}(T)^2=\rho^2\kappa_{\mathsf S}$. Therefore the frozen OU invariant variance reduces to
\(
 V_{\mathsf S}^{\star}(T)=\frac{\rho\kappa_{\mathsf S}}{2}
\).
If target-motion variance or update--target covariance is non-zero, this reduction to $K/J^2$ is false and the full composite variance in~\eqref{eq::q_variance_decomposition} must be used.

\begingroup
\begin{lemma}[Identification of the frozen coefficients]\label{trk:lem:ident}
Assume the target first-order condition $\bar u_{kh}(m_{kh}^{(h)})=0$ at each mesh, as obtained at an interior projection by differentiating the conditional risk. Under Assumption~\ref{trk:ass:frozen}, the slope clause of Assumption~\ref{ass::local-ou}-(ii) then holds with $b_{\mathsf S}(T)=b_T$, and the bracket condition in Assumption~\ref{ass::local-ou}-(v) holds with
\begin{equation}\label{trk:eq:q}
q_{\mathsf S}(T)^2=\rho^2\Xi_T+\bar\Gamma_T^2-2\rho\,\bar\Gamma_T\chi_T .
\end{equation}
Items (i), (iii) and (iv) of Assumption~\ref{ass::local-ou}, and the conditional Lindeberg condition in~(v), are not implied. \ref{trk:sec:primitive} verifies them for bounded scale-invariant drivers.
\end{lemma}
\proofref{app:proof_trk_lem_ident}
\endgroup

\begingroup
\begin{assumption}[Differentiable true law]\label{trk:ass:dtl}
Let $U\subset\R$ be an open interval such that $\tilde\lambda^{(h)}_{kh}\in U$ for all $k\in\mathcal W_h$ with probability tending to one, and let $V$ be an open set containing $\lambda^\star_{\mathsf S,h}(U)$ for all small $h$. For each small $h$ and $x\in U$, the true local law $\widetilde P^{(h)}_x$ has a $\nu$-density $\tilde p^{(h)}_x$, and:
\begin{enumerate}[label=(\alph*)]
\item for $\nu$-a.e.\ $y$, the map $x\mapsto\tilde p^{(h)}_x(y)$ is continuously differentiable. Moreover,
\[\max\{1,|\psi_{\mathsf S}(y,a)|\}\,|\partial_x\tilde p^{(h)}_x(y)|\le H_h(y)\quad\text{for all }(a,x)\in V\times U,\qquad \int H_h\,d\nu<\infty,\]
and the conditions of Lemma~\ref{lem:pseudotrue_implicit} hold on $V\times U$;
\item the true score $\tilde s_{h,x}(y)=\partial_x\log\tilde p^{(h)}_x(y)$ has $\tilde I_h(x)=\E_x[\tilde s_{h,x}(Y)^2]\in(0,\infty)$, and $\psi_{\mathsf S}(\cdot,\lambda^\star_{\mathsf S,h}(x))\in L^2(\widetilde P^{(h)}_x)$;
\item $D_{\mathsf S,h}(x)=\partial_x\lambda^\star_{\mathsf S,h}(x)\neq0$ on $U$;
\item uniformly over $k\in\mathcal W_h$ and in probability, $\tilde I_h(\tilde\lambda^{(h)}_{kh})\to\tilde I_T\in(0,\infty)$, $D_{\mathsf S,h}(\tilde\lambda^{(h)}_{kh})\to D_T\neq0$ and $\sigma_h(\tilde\lambda^{(h)}_{kh})\to\sigma_T\ge0$. Hence $\bar\Gamma_T=D_T\sigma_T$ by Proposition~\ref{prop::latent_space_recursion}.
\end{enumerate}
\end{assumption}
\endgroup

\subsection{\texorpdfstring{{Mean-square tracking on the survival event}}{Mean-square tracking on the survival event}}\label{subsec:local_tracking_bound}
Proposition~\ref{prop::latent_space_recursion} is a primitive verification device. The tracking theory below requires only the target expansion~\eqref{eq::increments_m}, local mean reversion, and a second-moment bound for the composite innovation; it does not otherwise depend on the latent-state recursion~\eqref{eq::latent_space_recursion}. Write $\bar u_{kh}(\lambda)=\mathsf E_{kh}[u_{\mathsf S}^{(h)}(y_{(k+1)h},\lambda)]$ for the conditional mean update direction. We use a jointly measurable conditional-mean version of $\bar u_{kh}$ whose defining identity also holds at every $\mathcal G_{kh}^{(h)}$-measurable local state.

\begin{assumption}[Local tracking]\label{ass:local-tracking}
    Fix $T_0<\infty$ and set $N_h=\lfloor T_0/h\rfloor$. There exist constants $r,\mu,L>0$ and $C<\infty$ such that the following conditions hold uniformly for $0\le k\le N_h$:
    \begin{enumerate}
        \item[(i)] {\emph{Target motion.} The pseudo-true path satisfies \eqref{eq::increments_m}, where $\bar\Gamma_{kh}$ is $\mathcal G_{kh}^{(h)}$-measurable, $\mathsf E_{kh}[\eta_{(k+1)h}]=0$, $\mathsf E_{kh}[\eta_{(k+1)h}^2]\le C_\eta$, and $\mathsf E_{kh}[(r_{(k+1)h}^{(h)})^2]\le Ch^2$. The second moment of the composite innovation is imposed separately in item~(iii). The fourth moment in Proposition~\ref{prop::latent_space_recursion} is a sufficient condition for its nonlinear target expansion, not a requirement of the present tracking argument.}
        \item[(ii)] \emph{Local mean reversion.} For every $\lambda$ such that $|\lambda-m_{k h}^{(h)}|\leq r$,
        \[
        \bar u_{kh}(m_{kh}^{(h)})=0, \quad
        (\lambda-m_{kh}^{(h)})\bar u_{kh}(\lambda)\leq-\mu(\lambda-m_{kh}^{(h)})^2, \quad
         |\bar u_{kh}(\lambda)|\leq L|\lambda-m_{kh}^{(h)}|\,.
        \]
        \item[(iii)] \emph{Composite innovation.} For the rescaling rate $\rho>0$ and every $\lambda$ such that $|\lambda-m_{k h}^{(h)}|\leq r$, the difference between the filter and target innovations has a uniformly bounded conditional second moment:
        \[
            \mathsf{E}_{k h}\left[\left(\rho (u_{\mathsf{S}}^{(h)}(y_{(k+1)h},\lambda)-\bar{u}_{k h}(\lambda)) - \bar{\Gamma}_{k h} \eta_{(k+1)h}\right)^2\right] \leq C.
        \]

    \end{enumerate}
\end{assumption}

Items~(i)--(iii) control, respectively, target motion, deterministic contraction towards the current pseudo-true value, and filter noise net of target innovation. These are precisely the ingredients used in Theorem~\ref{th::local-tracking}.
Proposition~\ref{prop::condition-2} gives a score-specific sufficient condition
for Assumption~\ref{ass:local-tracking}-(ii). Under its conditions on an interior target, differentiation under conditional expectation,
finite predictable scaling, and smoothness, define
\[
\mathcal H_{h,k}(\lambda)
=G_{\mathsf S}^{(h)}(\lambda)
\partial_\lambda\widetilde{\mathsf R}_{\mathsf S,kh}^{(h)}(\lambda)
=-\bar u_{kh}(\lambda).
\]
If \(0<\underline b\le\mathcal H'_{h,k}(\lambda)\le\overline b<\infty\) uniformly for
\(0\le k\le N_h\) and \(|\lambda-m_{kh}^{(h)}|\le r\), then item~(ii) holds with
\(\mu=\underline b\) and \(L=\overline b\). Thus local stability depends on the
derivative of the scaled risk gradient, rather than raw risk curvature alone. For density-power
and CRPS drivers, feasible scaling may vary with the mesh, so this uniform bound must be
verified for the chosen scaling.

\begingroup
\begin{theorem}[Local tracking bound on the survival event]\label{th::local-tracking}
Under Assumption~\ref{ass:local-tracking}, consider the recursion in \eqref{eq::dynamics} and the tracking error $e_{kh}^{(h)}=\lambda_{k h}^{(h)}-m_{k h}^{(h)}$. Define
\(
    \tau_{r,h}^{(h)}=\inf\{0 \leq k \leq N_h\,:\,|e_{k h}^{(h)}|>r\}
\)
,
with the convention $\inf\emptyset=N_h+1$. If $\sup_{h}\mathsf{E}[(e_{0}^{(h)})^2]<\infty$, then there exist constants $c>0$ and $C_r<\infty$, independent of $h$ and $k$, such that, for all sufficiently small $h$ and every $0\le k\le N_h$,
\begin{equation}
    \mathsf{E}[(e_{k h}^{(h)})^2 \mathsf{1}_{\{k < \tau_{r,h}^{(h)}\}}] \leq (1-c \sqrt{h})^k \mathsf{E}[(e_{0}^{(h)})^2] + C_r \sqrt{h}.
\end{equation}
Moreover, for each fixed $t\in(0,T_0]$, let $k_h(t)=\lfloor t/h\rfloor$. If
\(
\mathsf{E}[(e_{k_h(t) h}^{(h)})^2
\mathsf{1}_{\{\tau_{r,h}^{(h)}\leq k_h(t)\}}]=O(h^{1/2})
\)
,
then $\lambda_{k_h(t) h}^{(h)}-m_{k_h(t) h}^{(h)}=O_{L^2}(h^{1/4})$, where $O_{L^2}(\cdot)$ denotes big-$O$ in mean square.
\end{theorem}
\endgroup
\proofref{app:proof_local_tracking}
\begingroup
The displayed bound concerns $\mathsf E[(e_{kh}^{(h)})^2\mathsf1_{\{k<\tau_{r,h}^{(h)}\}}]$, not the second moment of the stopped value $e_{(k\wedge\tau_{r,h}^{(h)})h}^{(h)}$. The gain $\rho\sqrt h$ produces contraction of order $1-c\sqrt h$ and per-step noise variance $O(h)$, giving an $O(h^{1/2})$ survival contribution after the initial transient. The exit contribution must be $O(\sqrt h)$ for the stated unstopped $O_{L^2}(h^{1/4})$ conclusion. Weak compact containment does not supply that moment condition.
\endgroup

Proposition~\ref{prop::latent_space_recursion} and Proposition~\ref{prop::condition-2}
provide routes to the target expansion and mean reversion in
Theorem~\ref{th::local-tracking}; its innovation moment, initial-error, and
exit conditions must also hold for the stated rate.

\subsection{A fourth-moment boundary}\label{subsec:local_moment_boundary}
Finally, specialise Section~\ref{section::continuous_time} to a Gaussian working scale with $\Gamma(\lambda)=e^\lambda$. At date $T$, let the true law be
\(
    Y = \sqrt{h e^{\chi_T}}\varepsilon = \sqrt h\,e^{\chi_T/2}\varepsilon
\),
where $\chi_T$ is true log variance and $\varepsilon$ need not be Gaussian. Logarithmic and bounded drivers then require different moments for local tracking variance.

\begin{proposition}[Fourth-moment boundary for local tracking variance]
\label{prop::heavy-tail-diagnostic}
In the Gaussian scale specialisation above, suppose that $0<\mathsf E[\varepsilon^2]<\infty$ and $\mathsf E[\varepsilon^4]=\infty$. Then:

\begin{enumerate}
\item[(i)] For the logarithmic score, the pseudo-true log variance and the update
second moment satisfy
\[
m_{\log}(T)=\chi_T+\log \mathsf E[\varepsilon^2],
\qquad
\mathsf E\!\left[u_{\log}^{(h)}(y,m_{\log}(T))^2\right]=\infty.
\]
The infinite second moment persists under any predictable scaling of the
log-score derivative that is bounded away from zero. Consequently, if
$|\bar \Gamma_T|<\infty$ and $\mathsf E[\eta^2]<\infty$, the composite
innovation $\rho u_{\log}^{(h)}(y,m_{\log}(T))-\bar\Gamma_T\eta$ is not
square-integrable. Assumption~\ref{ass::local-ou}-(v) therefore fails, and
Theorem~\ref{th::local-ou} does not apply.
\item[(ii)] Let $\mathsf S$ have a locally existing pseudo-true value
$m_{\mathsf S}(T)$, and let $U$ be a neighbourhood of that value. Suppose
that its scaled driver has the bounded standardised form
\begin{equation}
    u_{\mathsf S}^{(h)}(y,\lambda)
    =
    g_{\mathsf S}
    \left(
        \lambda,
        \frac{y}{\sqrt{h e^{\lambda}}}
    \right),
    \qquad
    \sup_{\lambda\in U}\sup_{z\in\mathbb R}
    |g_{\mathsf S}(\lambda,z)|<\infty
    \label{eq::bounded_standardised_driver}
\end{equation}
for $\lambda\in U$. Then the update is locally uniformly square-integrable.
If $|\bar \Gamma_T|<\infty$ and $\mathsf E[\eta^2]<\infty$, the composite
innovation is also locally square-integrable. If, in addition, the remaining
martingale-limit conditions of Assumption~\ref{ass::local-ou}-(v) hold with
$q_{\mathsf S}(T)<\infty$ and $0<b_{\mathsf S}(T)<\infty$, then
$V_{\mathsf S}^{\star}(T)<\infty$.
\end{enumerate}
\end{proposition}
\proofref{app:proof_heavy_tail}
For the logarithmic score, the scaled driver is
\(
    u_{\log}^{(h)}(y,\lambda)
    =
    \left(
        \frac{y}{\sqrt{h e^{\lambda}}}
    \right)^2
    -
    1
\).
Thus the local Gaussian OU approximation for the log update uses a finite fourth innovation moment. Bounded drivers remove this moment obstruction but do not by themselves establish the OU limit: the martingale-limit, initial-law, linearisation, local-tube and identification conditions still apply. Under these conditions, the rule shapes both the selected path and tracking variance.

\begingroup
\subsection{Relation to the optimal filter of Beutner, Lin and Lucas}\label{app:moved_bll}
\begin{remark}[Information bound and predictable target shocks]\label{rev:rem:leverage}
For a given rule with zero driver--target covariance, \eqref{trk:eq:effmin} compares its frozen risk with $\sigma_T\tilde I_T^{-1/2}$. The gap is $\operatorname{eff}_{\mathsf S}^{-1/2}$ after conversion to latent units, and efficiency is the squared correlation with the true score at the rule's target. This is a within-class information inequality. To identify this benchmark with an unrestricted observation-based filtering optimum, one must additionally rule out information about the target innovation in the contemporaneous observation. A sufficient condition in the scalar scale experiment is $\mathsf E[\eta\mid\varepsilon]=0$, conditional on the frozen state. This is stronger than zero covariance for a particular driver.

For example, let $(\varepsilon,\eta)$ be jointly standard Gaussian with correlation $r\ne0$, and let $\sigma>0$ be the latent loading. Central symmetry gives $\Cov(g(\varepsilon),\eta)=0$ for every even driver $g$. Nevertheless, the generic centred update
\[
 v_a(z)=a(z^2-1)+\sigma r z,\qquad a>0,
\]
has mean $a(e^{-e}-1)$ when $z=e^{-e/2}\varepsilon$, and hence restoring coefficient $a$. Write $\eta=r\varepsilon+\sqrt{1-r^2}\zeta$, with independent standard normal $\zeta$. At the target the composite innovation is
\[
 v_a(\varepsilon)-\sigma\eta
 =a(\varepsilon^2-1)-\sigma\sqrt{1-r^2}\zeta.
\]
Its frozen variance is therefore
\[
 V_a=a+\frac{\sigma^2(1-r^2)}{2a},\qquad
 \inf_{a>0}V_a=\sigma\sqrt{2(1-r^2)}<\sigma\sqrt2.
\]
Since $I_F=1/2$, the final quantity $\sigma\sqrt2$ is the zero-alignment information benchmark. The odd update term uses information about the target shock which reflection-symmetric scale drivers discard. The example belongs to a larger class than the even proper-score drivers studied here and does not contradict Proposition~\ref{trk:prop:eff}. It shows why central symmetry removes leverage from their frontier without removing the value of leverage for unrestricted filtering.
\end{remark}
\endgroup

\begingroup
\subsection{\texorpdfstring{Scope of population gain selection}{Scope of population gain selection}}\label{app:moved_selftune}
\begin{remark}[Open extension: estimated autoregressive filters]\label{trk:conj:self}
Consider the scale model with a stationary latent log variance and a bounded driver. The unresolved question is whether, for a suitably identified in-fill autoregressive family with $\varphi_h=1-O(h)$, the gain component of its population minimum-score parameter converges to the frozen tracking-optimal gain. Proposition~\ref{trk:prop:selftune} does not establish this extension, the joint estimation of $(\omega,\varphi)$, or the level weighting of nonlogarithmic scores.
\end{remark}

\paragraph{{What is proved}} Proposition~\ref{trk:prop:selftune} concerns gain-only optimisation for a driftless random-walk latent state and the excess risk in the reference frame $x=0$, $h=1$. It establishes convergence of invariant moments and of minimisers of that population criterion. For logarithmic rules, scale equivariance has $A\equiv1$, so the excess reference-frame loss also equals the actual excess loss up to gain-free terms; the bounded Student-$t_6$ log driver is covered under the proposition's other assumptions. For CRPS and density-power scores, the actual excess risk is multiplied by $A(\log h+\tilde\lambda)$, which can be correlated with the tracking error. A stationary latent-state and two-time-scale argument would be needed to cover that criterion. Even though $A$ itself does not involve $\rho$, the joint law of the latent level and error can depend on $\rho$. It is therefore not legitimate to pull $A$ out of the stationary risk. The fixed-$h$ estimation theorem by itself does not bridge this additional in-fill problem.

\paragraph{{Evidence and scope}}
The objective in Proposition~\ref{trk:prop:selftune} is the stationary excess risk $\mathsf E_{\pi_{h,\rho}}\Delta(e)$. In a numerical implementation, average this quantity over the tracking error at a fixed latent reference scale. Averaging realised CRPS or density-power loss along the unnormalised latent path would answer a different question.

The formal statement is given in Section~\ref{trk:sec:stationarygain}, with proof in \ref{app:proof_trk_prop_selftune}.

For other innovation laws, check the global sign of the density-power mean field. \ref{pur:sec:dp} supplies that check for the reference laws used here.

\begin{remark}[Scope]
\begin{itemize}[leftmargin=1.5em]
\item Predictable $O(h)$ drifts leave the leading local weak limit unchanged, but do not automatically preserve an invariant-risk theorem or the next-order bias expansion; see Remark~\ref{rev:rem:target-drift}.
\item For fixed-data-bandwidth MMD, scale equivariance fails. The general sensitivity comparison can be used when its assumptions hold, but $D_T$, $\kappa_{\mathsf S}$ and efficiency can depend on the state; a random frozen-coefficient extension is not proved here.
\item With nonzero driver--target covariance, \eqref{trk:eq:lat} retains the appropriate covariance correction. Zero covariance for a chosen even driver is weaker than the no-predictable-target-shock condition required for an unrestricted zero-alignment optimality interpretation.
\end{itemize}
\end{remark}
\endgroup

\subsection{Smoothness of the pseudo-true map}\label{app:proof_pseudotrue_implicit}
\begin{lemma}[Smoothness of the pseudo-true map]\label{lem:pseudotrue_implicit}
Fix $h>0$ and a true-law index $\tilde\lambda_0$, and write
$F_h(a,\tilde\lambda)=\partial_a\widetilde{\mathsf R}_{\mathsf S}^{(h)}
(a;\tilde\lambda)$. Let
$a_0=\lambda^\star_{\mathsf S,h}(\tilde\lambda_0)$ be an interior global
minimiser. Suppose there are open intervals $U\ni\tilde\lambda_0$ and
$V\ni a_0$, with $V\subset\operatorname{int}(\Lambda)$, such that $F_h$ is
twice continuously differentiable on $V\times U$ and
$\partial_aF_h=\partial_{aa}^2\widetilde{\mathsf R}_{\mathsf S}^{(h)}$
is bounded below by a positive constant there. Assume also that for every
$x\in U$ the global argmin of
$a\mapsto\widetilde{\mathsf R}_{\mathsf S}^{(h)}(a;x)$ is nonempty and
contained in $V$. Then this global pseudo-true map is $C^2$ on $U$ and
\(
 \partial_{\tilde\lambda}\lambda^\star_{\mathsf S,h}
 =
 -\frac{\partial_{a\tilde\lambda}^2
 \widetilde{\mathsf R}_{\mathsf S}^{(h)}}
 {\partial_{aa}^2\widetilde{\mathsf R}_{\mathsf S}^{(h)}}
\).
All risk derivatives are evaluated at
$a=\lambda^\star_{\mathsf S,h}(\tilde\lambda)$. If the same intervals $U,V$
and argmin-containment condition work for $0<h\le h_0$, the curvature has a
positive lower bound uniform in $h$ there, and the
mixed second derivative
$\partial_{a\tilde\lambda}^2\widetilde{\mathsf R}_{\mathsf S}^{(h)}$ and third
derivatives $\partial_{aaa}^3\widetilde{\mathsf R}_{\mathsf S}^{(h)}$,
$\partial_{aa\tilde\lambda}^3\widetilde{\mathsf R}_{\mathsf S}^{(h)}$, and
$\partial_{a\tilde\lambda\tilde\lambda}^3\widetilde{\mathsf R}_{\mathsf S}^{(h)}$
are uniformly bounded on $V\times U$, then the first two derivatives of the
pseudo-true map are uniformly bounded in $h$ on $U$.
\end{lemma}
\begin{proof}
At each $x\in U$, a global minimiser exists in the open interval $V$ by
assumption and therefore satisfies $F_h(a,x)=0$. Since $\partial_aF_h$ is
strictly positive throughout $V\times U$, this zero is unique in $V$.
The $C^2$ implicit-function theorem gives a local $C^2$ stationary branch
through that zero. The argmin-containment premise and uniqueness in $V$
identify this branch with the global pseudo-true map throughout $U$.
Differentiating the identity
\(F_h\bigl(\lambda^\star_{\mathsf S,h}(\tilde\lambda),\tilde\lambda\bigr)=0\)
once with respect to \(\tilde\lambda\) gives the first-derivative formula in the
lemma. Differentiating once more gives
\[
 \partial_{\tilde\lambda\tilde\lambda}^2\lambda^\star_{\mathsf S,h}
 =
 -\frac{
 \partial_{aaa}^3\widetilde{\mathsf R}_{\mathsf S}^{(h)}
 (\partial_{\tilde\lambda}\lambda^\star_{\mathsf S,h})^2
 +2\partial_{aa\tilde\lambda}^3\widetilde{\mathsf R}_{\mathsf S}^{(h)}
 \partial_{\tilde\lambda}\lambda^\star_{\mathsf S,h}
 +\partial_{a\tilde\lambda\tilde\lambda}^3
 \widetilde{\mathsf R}_{\mathsf S}^{(h)}
 }{\partial_{aa}^2\widetilde{\mathsf R}_{\mathsf S}^{(h)}},
\]
with every risk derivative evaluated at the pseudo-true point. The stated
mixed second-derivative bound and curvature lower bound make the first-derivative
formula uniformly bounded; the third-derivative bounds then do the same for the
displayed second derivative.

\end{proof}

The argmin-containment premise holds, for example, locally when $\Lambda$ is
compact, the risk is jointly continuous on a neighbourhood of
$\Lambda\times\{\tilde\lambda_0\}$, and its global minimiser at
$\tilde\lambda_0$ is unique. For noncompact $\Lambda$, the same conclusion
follows if, for some $c>\widetilde{\mathsf R}_{\mathsf S}^{(h)}(a_0;\tilde\lambda_0)$
and index neighbourhood $U_0$, all sublevel sets below $c$ for indices in
$U_0$ lie in one compact subset of $\Lambda$, and the risk is jointly
continuous there, provided the global minimiser at $\tilde\lambda_0$ is
unique. When the lemma's uniform conditions hold with each deterministic
region $D_h$ of Proposition~\ref{prop::latent_space_recursion} contained in
the index interval $U$, it provides a
risk-based route to that proposition's smooth-target conditions. The
proposition also permits direct verification.

\subsection{Proof of Proposition~\ref{prop::proposition_recentering}}\label{app:proof_recentering}
By assumption, the minimiser $\lambda^{\star}_{\mathsf{S},h}(\tilde\lambda)$ is
an interior point. Hence its first-order condition is
\(
\left.\partial_{\lambda}
\widetilde{\mathsf{R}}_{\mathsf{S}}^{(h)}(\lambda;\tilde\lambda)
\right|_{\lambda=\lambda^{\star}_{\mathsf{S},h}(\tilde\lambda)}=0
\).
Exchanging differentiation and integration and using the definition of
$\psi_{\mathsf{S}}^{(h)}$ gives~\eqref{eq::recentering}. The scaled driver
satisfies
\(
u_{\mathsf{S}}^{(h)}
  (y,\lambda^{\star}_{\mathsf{S},h}(\tilde\lambda))
=G_{\mathsf{S}}^{(h)}
  (\lambda^{\star}_{\mathsf{S},h}(\tilde\lambda))
 \psi_{\mathsf{S}}^{(h)}
  (y,\lambda^{\star}_{\mathsf{S},h}(\tilde\lambda))
\)
. The finite scaling is observation-independent, so it factors out of the
conditional expectation. This gives~\eqref{eq::recentering2} and completes
the proof.

\subsection{Proof of Proposition~\ref{prop::latent_space_recursion}}\label{app:proof_latent_recursion}
We set $\Delta \tilde\lambda^{(h)}_{(k+1)h}=\tilde\lambda^{(h)}_{(k+1)h}-\tilde\lambda^{(h)}_{k h}$ (see \eqref{eq::latent_space_recursion}). Taylor's formula applied to $\tilde\lambda \mapsto \lambda_{\mathsf{S},h}^{\star}(\tilde\lambda)$ implies that, for some $\bar{\tilde\lambda}^{(h)} \in [\tilde\lambda^{(h)}_{k h},\tilde\lambda^{(h)}_{(k+1)h}]$, we can write:

\begin{equation*}
    m_{(k+1) h}^{(h)}-m_{k h}^{(h)} = \partial_{\tilde\lambda}\lambda_{\mathsf{S},h}^{\star}(\tilde\lambda_{k h}^{(h)})\Delta \tilde\lambda^{(h)}_{(k+1)h}+\frac{1}{2}\partial^2_{\tilde\lambda \tilde\lambda}\lambda_{\mathsf{S},h}^{\star}(\bar{\tilde\lambda}^{(h)})(\Delta \tilde\lambda^{(h)}_{(k+1)h})^2
\end{equation*}
Substituting the definition of $\Delta \tilde\lambda^{(h)}_{(k+1)h}$ in the previous equation and setting
\begin{equation*}
r_{(k+1)h}^{(h)}=h\,\partial_{\tilde\lambda}\lambda_{\mathsf{S},h}^{\star}(\tilde\lambda_{k h}^{(h)})b_h(\tilde\lambda_{k h}^{(h)})+\frac{1}{2}\partial^2_{\tilde\lambda \tilde\lambda}\lambda_{\mathsf{S},h}^{\star}(\bar{\tilde\lambda}^{(h)})(\Delta \tilde\lambda^{(h)}_{(k+1)h})^2
\end{equation*}
gives
\begin{equation}
    m_{(k+1) h}^{(h)}-m_{k h}^{(h)} = \sqrt{h} \,\partial_{\tilde\lambda}\lambda^{\star}_{\mathsf{S},h}(\tilde\lambda_{k h}^{(h)}) \sigma_h(\tilde\lambda_{k h}^{(h)})\eta_{(k+1)h} + r_{(k+1)h}^{(h)},
\end{equation}
which gives \eqref{eq::increments_m}, where $\bar{\Gamma}_{k h}$ is $\mathcal G_{kh}^{(h)}$-measurable. It remains to bound the remainder $r_{(k+1)h}^{(h)}$. Let $\bar b,\bar\sigma,L_1,L_2$ denote uniform bounds for $b_h(\cdot)$, $\sigma_h(\cdot)$, $\partial_{\tilde\lambda}\lambda^{\star}_{\mathsf{S},h}(\cdot)$ and $\partial^2_{\tilde\lambda\tilde\lambda}\lambda^{\star}_{\mathsf{S},h}(\cdot)$. From $\Delta\tilde\lambda^{(h)}_{(k+1)h}=h\,b_h(\tilde\lambda^{(h)}_{kh})+\sqrt h\,\sigma_h(\tilde\lambda^{(h)}_{kh})\,\eta_{(k+1)h}$ and the conditional moments $\mathsf E_{kh}[\eta_{(k+1)h}]=0$, $\mathsf E_{kh}[\eta_{(k+1)h}^{4}]\le C_\eta$, collecting powers of $h$ (the leading contribution being $h^{2}\bar\sigma^{4}C_\eta$) gives
\begin{equation*}
    \mathsf E_{kh}\big[(\Delta\tilde\lambda^{(h)}_{(k+1)h})^{4}\big]\le C' h^{2}\qquad (h\le1),
\end{equation*}
for a finite $C'$ depending only on $\bar b,\bar\sigma,C_\eta$. Since $|r_{(k+1)h}^{(h)}|\le h\,L_1\bar b+\tfrac12 L_2(\Delta\tilde\lambda^{(h)}_{(k+1)h})^{2}$, the inequality $(x+y)^2\le2x^2+2y^2$ and the display above yield
\begin{equation*}
    \mathsf E_{kh}[(r_{(k+1)h}^{(h)})^{2}]\le 2h^{2}L_1^{2}\bar b^{2}+\tfrac12 L_2^{2}\,\mathsf E_{kh}\big[(\Delta\tilde\lambda^{(h)}_{(k+1)h})^{4}\big]\le C\,h^{2},
\end{equation*}
with $C=2L_1^{2}\bar b^{2}+\tfrac12 L_2^{2}C'$ finite and independent of $h$ and $k$. This is the stated $O(h^{2})$ remainder bound and concludes the proof.

\subsection{Criterion for local mean reversion}\label{app:proof_condition_two}
\begin{proposition}[Criterion for local mean reversion]\label{prop::condition-2}
Suppose, uniformly over $0\le k\le N_h$, that $m_{kh}^{(h)}$ belongs to the interior of $\Lambda$ and
\(\partial_\lambda\widetilde{\mathsf R}_{\mathsf S,kh}^{(h)}(m_{kh}^{(h)})=0.\)
On the local region $|\lambda-m_{kh}^{(h)}|\le r$, suppose that differentiation can be passed under the conditional expectation and
\[
u_{\mathsf S}^{(h)}(y,\lambda)
=
G_{\mathsf S}^{(h)}(\lambda)
\psi_{\mathsf S}^{(h)}(y,\lambda),
\qquad
\psi_{\mathsf S}^{(h)}(y,\lambda)
=
-\partial_\lambda \mathsf S(P_\lambda^{(h)},y).
\]
Assume also that $G_{\mathsf S}^{(h)}(\lambda)$ is finite and jointly measurable with respect to $\mathcal G_{kh}^{(h)}\otimes\mathcal B(\Lambda)$ on this region. In the feasible case it is jointly measurable with respect to $\mathcal F_{kh}^{\mathrm{obs},(h)}\otimes\mathcal B(\Lambda)$. Define the scaled risk derivative
\[
\mathcal H_{h,k}(\lambda)
=
G_{\mathsf S}^{(h)}(\lambda)
\partial_\lambda\widetilde{\mathsf R}_{\mathsf S,kh}^{(h)}(\lambda)
=-\bar u_{kh}(\lambda).
\]
If $\mathcal H_{h,k}$ is continuously differentiable on the local region and, uniformly over $0\le k\le N_h$ and $|\lambda-m_{kh}^{(h)}|\le r$,
\(0<\underline b\le\mathcal H_{h,k}'(\lambda)\le\overline b<\infty,\)
then Assumption~\ref{ass:local-tracking}-(ii) holds with
$\mu=\underline b$ and $L=\overline b$.
\end{proposition}
\begin{proof}
By differentiation under conditional expectation, the predictability of $G_{\mathsf S}^{(h)}(\lambda)$, and the definition of the scoring-rule innovation, we obtain
\[
\bar u_{kh}(\lambda)
=
\mathsf E_{kh}
\left[
u_{\mathsf S}^{(h)}(y_{(k+1)h},\lambda)
\right]
=
-G_{\mathsf S}^{(h)}(\lambda)
\partial_\lambda
\widetilde{\mathsf R}_{\mathsf S,kh}^{(h)}(\lambda)
=-\mathcal H_{h,k}(\lambda)\,.
\]
Since
$
\partial_\lambda
\widetilde{\mathsf R}_{\mathsf S,kh}^{(h)}(m_{kh}^{(h)})
=
0,
$
it follows that $\mathcal H_{h,k}(m_{kh}^{(h)})=\bar u_{kh}(m_{kh}^{(h)})=0$. Now, let $\lambda$ satisfy $|\lambda-m_{kh}^{(h)}|\le r$. By the mean-value theorem, there exists a point $\xi$ between $\lambda$ and $m_{kh}^{(h)}$ such that
\[
\mathcal H_{h,k}(\lambda)
=\mathcal H_{h,k}'(\xi)(\lambda-m_{kh}^{(h)}).
\]
Using the lower derivative bound and $\bar u_{kh}=-\mathcal H_{h,k}$ gives
\[
(\lambda-m_{kh}^{(h)})\bar u_{kh}(\lambda)
=-\mathcal H_{h,k}'(\xi)(\lambda-m_{kh}^{(h)})^2
\le-\underline b(\lambda-m_{kh}^{(h)})^2.
\]
This proves the local mean-reversion inequality in Assumption~\ref{ass:local-tracking}-(ii) with
$\mu=\underline b$. The same mean-value identity and the upper derivative bound give
\(|\bar u_{kh}(\lambda)|\le\overline b|\lambda-m_{kh}^{(h)}|.\)
Thus the Lipschitz bound in Assumption~\ref{ass:local-tracking}-(ii) holds with
$L=\overline b$.
This completes the proof.

\end{proof}

\subsection{Proof of Theorem~\ref{th::local-tracking}}\label{app:proof_local_tracking}
First, notice that by the definition of $\bar u_{k h}$, the conditional mean-zero assumption on $\eta_{(k+1)h}$, and the predictability of $\bar{\Gamma}_{k h}$, the following martingale difference
\begin{equation}
\xi_{(k+1)h}^{(h)}=\rho \bigl(u_{\mathsf{S}}^{(h)}(y_{(k+1)h},\lambda_{k h}^{(h)})-\bar{u}_{k h}(\lambda_{k h}^{(h)})\bigr) - \bar{\Gamma}_{k h} \eta_{(k+1)h}
\end{equation}
has zero conditional expectation. Using the pseudo-true path expansion in \eqref{eq::increments_m}, the tracking error satisfies
\begin{equation}
    e_{(k+1)h}^{(h)} = e_{k h}^{(h)} + \rho \sqrt{h} \bar{u}_{k h}(\lambda_{k h}^{(h)}) + \sqrt{h} \xi_{(k+1)h}^{(h)} - r_{(k+1)h}^{(h)}.
\end{equation}
Set $A_{k h}=e_{k h}^{(h)} + \rho \sqrt{h} \bar{u}_{k h}(\lambda_{k h}^{(h)})$ and $B_{k h}=\{k < \tau_{r,h}^{(h)}\}$. On $B_{k h}$, $|e_{k h}^{(h)}|\leq r$, so Assumption~\ref{ass:local-tracking}-(ii) gives $A_{k h}^2 \leq (1-c_0\sqrt{h})(e_{k h}^{(h)})^2$ for some $c_0>0$ and all sufficiently small $h$. Since $\mathsf{E}_{k h}[\xi_{(k+1)h}^{(h)}]=0$, we obtain
\begin{equation*}
\begin{aligned}
\mathsf{E}_{k h}[(e_{(k+1)h}^{(h)})^2]
&= A_{k h}^2+h\mathsf{E}_{k h}[(\xi_{(k+1)h}^{(h)})^2]
   +\mathsf{E}_{k h}[(r_{(k+1)h}^{(h)})^2]\\
&\quad-2\left(A_{k h}\mathsf{E}_{k h}[r_{(k+1)h}^{(h)}]
   +\sqrt h\,\mathsf{E}_{k h}[\xi_{(k+1)h}^{(h)}r_{(k+1)h}^{(h)}]\right).
\end{aligned}
\end{equation*}
Assumption~\ref{ass:local-tracking}-(iii) and -(i) give $h\mathsf{E}_{k h}[(\xi_{(k+1)h}^{(h)})^2] \leq C h$ and $\mathsf{E}_{k h}[(r_{(k+1)h}^{(h)})^2] \leq C h^2$, respectively. On $B_{kh}$, Jensen's and Cauchy--Schwarz inequalities also give
$|A_{k h}\mathsf{E}_{k h}[r_{(k+1) h}^{(h)}]| \leq C_r h$ and
$\sqrt h\,|\mathsf{E}_{k h}[\xi_{(k+1)h}^{(h)}r_{(k+1) h}^{(h)}]| \leq C h^{3/2}$.
Because $B_{k h}\in\mathcal G_{kh}^{(h)}$ and $B_{(k+1) h}\subset B_{k h}$, these bounds imply, for constants $c>0$ and $C_r<\infty$,
\begin{equation*}
\mathsf E[(e_{(k+1)h}^{(h)})^2\mathsf 1_{B_{(k+1)h}}]
\le \mathsf E\!\left[\mathsf 1_{B_{kh}}\mathsf E_{kh}[(e_{(k+1)h}^{(h)})^2]\right]
\le (1-c\sqrt h)\mathsf E[(e_{kh}^{(h)})^2\mathsf 1_{B_{kh}}]+C_rh.
\end{equation*}
Iterating gives
\begin{equation}
    \mathsf{E}[(e_{k h}^{(h)})^2\mathsf{1}_{B_{k h}}] \leq (1-c \sqrt{h})^k \mathsf{E}[(e_0^{(h)})^2] + C_r h \sum_{\ell=0}^{k-1} (1-c \sqrt{h})^{\ell} \leq  (1-c \sqrt{h})^k \mathsf{E}[(e_0^{(h)})^2]  + C_r \sqrt{h}.
\end{equation}
{This proves the $L^2$ bound on the survival event.} Finally, fix $t\in(0,T_0]$ and let $k_h(t)=\lfloor t/h \rfloor$. For all sufficiently small $h$,
$(1-c \sqrt{h})^{k_h(t)} \leq \exp(-c t / (2 \sqrt{h}))$.
{The survival contribution is therefore $O(h^{1/2})$.} If the exit contribution is $O(\sqrt h)$, then $\mathsf{E}[(e_{k_h(t) h}^{(h)})^2]=O(h^{1/2})$, which is the asserted $O_{L^2}(h^{1/4})$ rate.

\subsection{Joint martingale FCLT}\label{app:proof_primitive_joint_fclt}
Define $M_h(\tau)=\sum_{j<n_h(\tau)}h^{1/4}\tilde\xi_{(k_T+j+1)h}^{(h)}$.
\begin{corollary}[Joint martingale FCLT]\label{cor:primitive_joint_fclt}
Under Assumption~\ref{ass::local-ou}-(i) and (v),
\[
 \left(Z_{\mathsf S,h}(0),M_h\right)
 \Rightarrow
 \left(Z_0,q_{\mathsf S}(T)W\right)
 \quad\text{in }\mathbb R\times D([0,M]),
\]
where $W$ is a standard Brownian motion independent of $Z_0$.
\end{corollary}
\begin{proof}
With $\mathcal H_{h,\tau}=\mathcal G_{h,n_h(\tau)}$, the process $M_h$ is a
square-integrable martingale with bracket $\langle M_h\rangle(\tau)=
\sum_{j<n_h(\tau)}h^{1/2}q_{\mathsf S,h,(k_T+j)h}^{2}$. Assumption~\ref{ass::local-ou}-(v)
gives uniform convergence of this bracket to $q_{\mathsf S}^{2}(T)\tau$ in probability and the
conditional Lindeberg condition, so $M_h$ is $C$-tight by the martingale functional central limit
theorem \cite[Theorem~7.1.4]{ethier1986markov}. Assumption~\ref{ass::local-ou}-(i) makes
$Z_{\mathsf S,h}(0)$ tight; hence the pair is jointly tight.

Take any joint subsequential limit $(Z_0,M)$ and let $\mathcal H_\tau$ be the completed,
right-continuous filtration generated by $Z_0$ and $\{M(s):s\le\tau\}$. For $s<t$ and every
bounded continuous cylinder functional of $(Z_{\mathsf S,h}(0),M_h|_{[0,s]})$, the martingale
identities for $M_h$ and $M_h^2-\langle M_h\rangle$ hold after multiplication by that functional.
Localisation at a bounded predictable bracket and truncation of the jumps make the increments
uniformly integrable; conditional Lindeberg removes the truncation, and uniform bracket convergence
removes the localisation. Passing to the limit shows that $M$ and
$M(\tau)^2-q_{\mathsf S}^{2}(T)\tau$ are $\mathcal H_\tau$-martingales. Thus $M$ is continuous with
bracket $q_{\mathsf S}^{2}(T)\tau$. If $q_{\mathsf S}(T)>0$, L\'{e}vy's characterisation gives
$M=q_{\mathsf S}(T)W$, where $W$ is Brownian in $(\mathcal H_\tau)$ and therefore independent of
the $\mathcal H_0$-measurable variable $Z_0$. If $q_{\mathsf S}(T)=0$, then $M\equiv0$ and an
independent Brownian motion may be added on an extension. Every joint subsequential limit therefore
has the asserted law, proving convergence of the full sequence.

\end{proof}

\subsection{Proof of Theorem~\ref{th::local-ou}}\label{app:proof_local_ou}
We first isolate the two ingredients used in the stopped approximation.
\begin{lemma}[Fast-scale decomposition]
\label{lem:fast-scale-decomposition}
Let $k_T=\lfloor T/h\rfloor$, $k_j=k_T+j$, and define $Z_{\mathsf S,h,j}=h^{-1/4} e_{k_jh}^{(h)}$. Suppose Assumption~\ref{ass:local-tracking}(i) holds and
$\bar u_{k_jh}(m_{k_jh}^{(h)})=0$.
Use the composite innovation $\tilde\xi_{(k+1)h}^{(h)}$ of
\eqref{eq:composite_innovation}, and define
\begin{equation*}
\begin{aligned}
    \zeta_{(k+1)h}^{(h)}(e)
    &=
    \rho
    \left(
    u_{\mathsf S}^{(h)}(y_{(k+1)h},m_{kh}^{(h)}+e)
    -
    u_{\mathsf S}^{(h)}(y_{(k+1)h},m_{kh}^{(h)})
    \right)
    \\
    &\quad
    -
    \rho
    \left(
    \bar u_{kh}(m_{kh}^{(h)}+e)
    -
    \bar u_{kh}(m_{kh}^{(h)})
    \right)
\end{aligned}
\end{equation*}
Then
\begin{equation*}
    \begin{aligned}
    Z_{\mathsf S,h,j+1}
    &=
    Z_{\mathsf S,h,j}
    +
    \rho h^{1/4}
    \bar u_{k_jh}
    (m_{k_jh}^{(h)}+h^{1/4}Z_{\mathsf S,h,j})
    \\
    &\quad
    +
    h^{1/4}\tilde\xi_{(k_j+1)h}^{(h)}
    +
    h^{1/4}
    \zeta_{(k_j+1)h}^{(h)}
    (h^{1/4}Z_{\mathsf S,h,j})
    -
    h^{-1/4}r^{(h)}_{(k_j+1)h}.
\end{aligned}
\end{equation*}
Moreover, $\mathsf E_{kh}[\tilde\xi_{(k+1)h}^{(h)}]=0$ and $\mathsf E_{kh}[\zeta_{(k+1)h}^{(h)}(e)]=0$ for every $\mathcal G_{kh}^{(h)}$-measurable local $e$.
\end{lemma}
To obtain the recursion, substitute $e_{k_jh}^{(h)}=h^{1/4}Z_{\mathsf S,h,j}$ into the tracking-error recursion~\eqref{eq::recursion_tracking_error} at the index $k=k_j=k_T+j$ and divide by $h^{1/4}$. Splitting the scaled innovation about the pseudo-true point $m_{k_jh}^{(h)}$ separates the centred term $\tilde\xi_{(k_j+1)h}^{(h)}$ from the evaluation-noise term $\zeta_{(k_j+1)h}^{(h)}$ defined above, and the remainder contributes $-h^{-1/4}r^{(h)}_{(k_j+1)h}$; collecting these terms gives the stated recursion. The two zero-mean claims follow from the conditional centring of the update at the pseudo-true point (Proposition~\ref{prop::proposition_recentering}) and from $\mathsf E_{kh}[\eta_{(k+1)h}]=0$ (Assumption~\ref{ass:local-tracking}).
\begin{lemma}[Negligible fast-scale remainders]
\label{lem:fast-scale-remainders}
Suppose Assumption~\ref{ass:local-tracking}(i) and
Assumption~\ref{ass::local-ou}-(ii)--(iv) hold. Define
$a_{kh}(e)=\bar u_{kh}(m_{kh}^{(h)}+e)+b_{\mathsf S,h,kh}e$ and use
the exit time $\sigma_{h,K}$ in Assumption~\ref{ass::local-ou}.
For every fixed $K,M<\infty$, the following stopped sums vanish in probability:
\begin{equation*}
\begin{gathered}
    \sup_{0\le \tau\le M}
    \left|
    \sum_{j<n_h(\tau\wedge\sigma_{h,K})}
    h^{1/4}
    \zeta_{(k_j+1)h}^{(h)}
    (h^{1/4}Z_{\mathsf S,h,j})
    \right|
    \overset{\rm p}{\rightarrow} 0,\\
    \sup_{0\le \tau\le M}
    \left|
    \sum_{j<n_h(\tau\wedge\sigma_{h,K})}
    \rho h^{1/4}
    a_{k_jh}(h^{1/4}Z_{\mathsf S,h,j})
    \right|
      \overset{\rm p}{\rightarrow} 0.
\end{gathered}
\end{equation*}
The stopped target-motion remainder also vanishes in $L^1$:
\begin{equation*}
    \sup_{0\le \tau\le M}
    \left|
    \sum_{j<n_h(\tau\wedge\sigma_{h,K})}
    h^{-1/4}r^{(h)}_{(k_j+1)h}
    \right|
     \overset{\rm L^{1}}{\rightarrow} 0.
\end{equation*}
\end{lemma}
\begin{proof}
For the evaluation noise, Assumption~\ref{ass::local-ou}-(iv) and Jensen's
inequality give, for local $e$, $\mathsf E_{kh}[|\zeta_{(k+1)h}^{(h)}(e)|^2] \le C e^2$. On the stopped event, $e=h^{1/4}Z_{\mathsf S,h,j}$ satisfies
$|e|\le h^{1/4}K$. Therefore the predictable quadratic variation of the stopped
martingale in the first stopped sum of the lemma is bounded by
$C\sum_{j\le M/\sqrt h}h^{1/2}h^{1/2}K^2\le CMK^2\sqrt h$.
Since this bound vanishes, Doob's inequality gives the first claim. The stopped accumulated mean field Taylor remainder is $o_p(1)$ directly by Assumption~\ref{ass::local-ou}-(iii). Finally, Assumption~\ref{ass:local-tracking}(i) and conditional Cauchy--Schwarz give $\mathsf E_{kh}[|r_{(k+1)h}^{(h)}|]\le C h$, so
\begin{equation*}
    \mathsf E
    \left[
    \sup_{\tau\le M}
    \left|
    \sum_{j<n_h(\tau\wedge\sigma_{h,K})}
    h^{-1/4}r_{(k_j+1)h}
    \right|
    \right]
    \le
    C\frac{M}{\sqrt h}h^{-1/4}h
    =
    CMh^{1/4}
    \to0.
\end{equation*}
This concludes the proof.
\end{proof}

We are now ready to prove Theorem \ref{th::local-ou}.

By Lemmas~\ref{lem:fast-scale-decomposition} and
\ref{lem:fast-scale-remainders},
\begin{equation*}
    Z_{\mathsf S,h}^K(\tau)
    =
    Z_{\mathsf S,h}(0)
    -
    \rho
    \sum_{j<n_h(\tau\wedge\sigma_{h,K})}
    \sqrt h\,
    b_{\mathsf S,h,k_jh}Z_{\mathsf S,h,j}
    +
    M_h(\tau\wedge\sigma_{h,K})
    +
    o_p(1),
\end{equation*}
uniformly on $[0,M]$. Assumption~\ref{ass::local-ou}-(v) gives uniform
convergence of the predictable bracket of $M_h$ to
$q_{\mathsf S}^{2}(T)\tau$. Corollary~\ref{cor:primitive_joint_fclt}
then supplies the joint limit of $(Z_{\mathsf S,h}(0),M_h)$, with Brownian
motion independent of $Z_0$.

It remains to identify the stopped drift. By boundedness before $\sigma_{h,K}$ and Assumption~\ref{ass::local-ou}-(ii),
\begin{equation*}
    \sup_{\tau\le M}
    \left|
    \sum_{j<n_h(\tau\wedge\sigma_{h,K})}
    \sqrt h\,
    b_{\mathsf S,h,k_jh}Z_{\mathsf S,h,j}
    -
    b_{\mathsf S}(T)
        \int_0^{\tau\wedge\sigma_{h,K}} Z_{\mathsf S,h}(s)\,ds
    \right|
 \overset{\rm p}{\rightarrow} 0.
\end{equation*}
Let $\widehat Z_h$ be the unique c\`adl\`ag solution of the linear integral equation
\begin{equation*}
 \widehat Z_h(\tau)
 =Z_{\mathsf S,h}(0)
 -\rho b_{\mathsf S}(T)\int_0^\tau\widehat Z_h(s)\,ds
 +M_h(\tau).
\end{equation*}
The solution map for this linear equation is continuous at continuous driving paths. Hence the
joint convergence above and the continuous-mapping theorem give
$\widehat Z_h\Rightarrow Z_{\mathsf S}$; in particular, $\widehat Z_h$ is $C$-tight.

Write $\sigma=\sigma_{h,K}$ for the exit time of the actual filter, capped at $M$. Subtract the global linear equation up to $\sigma$. The remainder bounds use the filter state at the \emph{start} of each step, so they cover the increment that reaches the boundary. They impose no restriction on $\widehat Z_h$. Gronwall's inequality gives
\[
 \sup_{\tau\le\sigma}|Z_{\mathsf S,h}(\tau)-\widehat Z_h(\tau)|=o_p(1).
\]
If the filter reaches $K$ by $M$, the comparison process reaches at least $K-\epsilon$ except on an event with probability tending to zero. By weak convergence of the comparison process to a continuous limit and the closed-set bound,
\[
 \limsup_{h\downarrow0}\mathsf P\!\left(\sup_{\tau\le M}|Z_{\mathsf S,h}(\tau)|\ge K\right)
 \le \mathsf P\!\left(\sup_{\tau\le M}|Z_{\mathsf S}(\tau)|\ge K-\epsilon\right)
 \quad(\epsilon>0).
\]
Let $K$ grow. The right side vanishes because continuous OU paths are bounded on $[0,M]$. First using the comparison at fixed $K$ and then this exit bound gives $\sup_{\tau\le M}|Z_{\mathsf S,h}(\tau)-\widehat Z_h(\tau)|\to_p0$. Thus the full processes converge weakly. Applying the stopping map at a continuity radius gives the stated stopped limit.

The explicit solution is
\begin{equation*}
    Z_{\mathsf S}(\tau)
    =
    e^{-\rho b_{\mathsf S}(T)\tau}Z_0
    +
    q_{\mathsf S}(T)
    \int_0^\tau
    e^{-\rho b_{\mathsf S}(T)(\tau-s)}
    dW_s.
\end{equation*}
If $Z_0=z$ is deterministic, the transition law is Gaussian with mean $e^{-\rho b_{\mathsf S}(T)\tau}z$ and variance $\frac{q_{\mathsf S}^{2}(T)}
    {2\rho b_{\mathsf S}(T)}
    (1-e^{-2\rho b_{\mathsf S}(T)\tau})$.
Because $b_{\mathsf S}(T)>0$, the transient factor decays as $\tau\to\infty$, leaving the invariant variance of the frozen limiting OU, $V_{\mathsf S}^{\star}(T)
    =
    \frac{q_{\mathsf S}^{2}(T)}
    {2\rho b_{\mathsf S}(T)}$.
\qed

\subsection{Proof of Corollary~\ref{cor:unstopped_local_ou}}\label{app:proof_unstopped_local_ou}

The comparison up to the filter's exit in the proof of Theorem~\ref{th::local-ou} gives both the compact-containment bound and convergence on the full finite window.

\subsection{Proof of Proposition~\ref{prop::heavy-tail-diagnostic}}\label{app:proof_heavy_tail}
We write $z_{\lambda}=\frac{y}{\sqrt{h e^\lambda}}=\frac{y}{\sqrt h\,e^{\lambda/2}}$. Hence, the logarithmic scoring loss is, up to the constant $\frac12\log(2\pi)$, given by $\mathsf S_{\log}(P_\lambda^{(h)},Y)
=\frac12\log(he^\lambda)+\frac12 z_\lambda^2$. Its raw innovation is $\psi_{\log}^{(h)}=\frac12(z_\lambda^2-1)$; the finite-step scaling $G_{\log}=2$ gives the actual update $u_{\log}^{(h)}=z_\lambda^2-1$. The pseudo-true first-order condition is $\mathsf E[u_{\log}^{(h)}(y,\lambda)]=0\Longleftrightarrow e^{\chi_T-\lambda}\mathsf E[\varepsilon^2]=1$. Hence, $m_{\log}(T)=\chi_T+\log \mathsf E[\varepsilon^2]$. At this pseudo-true value,
\begin{equation*}
     u_{\log}^{(h)}(y,m_{\log}(T))
    =
    \left(
    \frac{\varepsilon^2}{\mathsf E[\varepsilon^2]}-1
    \right).
\end{equation*}
In particular, if $\mathsf E[\varepsilon^4]=\infty$, then $\mathsf E[u_{\log}^{(h)}(Y,m_{\log}(T))^2]=\infty$, because the squared driver is proportional to $\varepsilon^4$. If a predictable scaling $G_T$ satisfies $|G_T|\ge c>0$ almost surely, then $\mathsf E[(G_Tu_{\log}^{(h)})^2]\ge c^2\mathsf E[(u_{\log}^{(h)})^2]=\infty$, so the conclusion persists. Now suppose $|\bar\Gamma_T|<\infty$ and $\mathsf E[\eta^2]<\infty$. If the composite innovation $\rho u_{\log}^{(h)}(Y,m_{\log}(T))-\bar\Gamma_T\eta$ were square integrable, then adding back the square-integrable term
$\bar\Gamma_T\eta$ would imply that $u_{\log}^{(h)}(Y,m_{\log}(T))$ is square integrable, a contradiction. Thus the square-integrability requirement in Assumption~\ref{ass::local-ou}-(v) fails, so Theorem~\ref{th::local-ou} cannot be invoked.

For the second claim, let $U$ be the stated neighbourhood. Then
\begin{align*}
C_U&=\sup_{\lambda\in U}\sup_{z\in\mathbb R}
|g_{\mathsf S}(\lambda,z)|<\infty,\\
\sup_{\lambda\in U}\mathsf E
\left[u_{\mathsf S}^{(h)}(y,\lambda)^2\right]
&\le C_U^2<\infty.
\end{align*}
If also $|\bar\Gamma_T|<\infty$ and $\mathsf E[\eta^2]<\infty$, the composite
innovation $\rho u_{\mathsf S}^{(h)}(y,m_{\mathsf S}(T))-\bar\Gamma_T\eta$
is square integrable. This verifies the moment requirement but does not imply
the predictable variance convergence or conditional Lindeberg condition in
Assumption~\ref{ass::local-ou}-(v). If those remaining conditions hold with
$q_{\mathsf S}(T)<\infty$ and the pseudo-true target is locally identified with
$0<b_{\mathsf S}(T)<\infty$, then
$V_{\mathsf S}^{\star}(T)=
    \frac{q_{\mathsf S}^{2}(T)}
    {2\rho b_{\mathsf S}(T)}
    <
    \infty.
$
This concludes the proof.

\subsection{\texorpdfstring{{Tracking error in latent units}}{Tracking error in latent units}}

\begingroup
\begin{lemma}[Tracking error in latent units]\label{trk:lem:lat}
Suppose the conditions of Theorem~\ref{th::local-ou} and Assumption~\ref{trk:ass:dtl} hold. Suppose also that $\lambda^\star_{\mathsf S,h}$ is a $C^2$ diffeomorphism of $U$ onto an open set containing the local tube of Assumption~\ref{ass::local-ou}-(i), with $|D_{\mathsf S,h}|\ge d>0$ and $|\partial^2_{xx}\lambda^\star_{\mathsf S,h}|\le C$ on $U$, uniformly in small $h$.

Let $I_h=\lambda^\star_{\mathsf S,h}(U)$ and define the measurable extension
\[
 L_h(a)=\begin{cases}
 (\lambda^\star_{\mathsf S,h})^{-1}(a),&a\in I_h,\\
 0,&a\notin I_h.
 \end{cases}
\]
Put $\hat x^{(h)}_{kh}=L_h(\lambda^{(h)}_{kh})$ and let $Z^{\rm lat}_{\mathsf S,h}(\tau)=h^{-1/4}(\hat x^{(h)}-\tilde\lambda^{(h)})$ at index $k_T+n_h(\tau)$. For a well-defined quotient on the whole probability space, extend $D_{\mathsf S,h}$ outside $U$ by the nonzero constant $D_T$. Then, uniformly on the finite window $[0,M]$,
\[
 Z^{\rm lat}_{\mathsf S,h}(\tau)=Z_{\mathsf S,h}(\tau)\big/D_{\mathsf S,h}\bigl(\tilde\lambda^{(h)}_{(k_T+n_h(\tau))h}\bigr)+O_p(h^{1/4}).
\]
In particular, $Z^{\rm lat}_{\mathsf S,h}\Rightarrow Z_{\mathsf S}/D_T$. If $K$ is a continuity radius for the original error, explicitly set $Z^{\rm lat,K}_{\mathsf S,h}(\tau)=Z^{\rm lat}_{\mathsf S,h}(\tau\wedge\sigma_{h,K})$. Then
\[
 Z^{\rm lat,K}_{\mathsf S,h}\Rightarrow \frac{\Phi_K(Z_{\mathsf S})}{D_T}
 =\Phi_{K/|D_T|}(Z_{\mathsf S}/D_T).
\]
The stopping radius is measured in the original error units. The invariant variance of the frozen latent-unit limit is $V^\star_{\mathsf S}(T)/D_T^2$.

\end{lemma}

\begin{proof}

By the inverse function theorem, the inverse on $I_h$ is $C^2$, with first derivative $1/D_{\mathsf S,h}$ at the corresponding latent index and second derivative bounded in absolute value by $C/d^3$. For any fixed $L>0$, on the event that the latent indices lie in $U$ throughout the window and $\sup_{\tau\le M}|Z_{\mathsf S,h}(\tau)|\le L$, the local-tube condition ensures that the inverse is evaluated in $I_h$ for all sufficiently small $h$. As $I_h$ is an interval, Taylor's theorem between $m_{kh}^{(h)}$ and $\lambda_{kh}^{(h)}$ gives
\[
 \hat x_{kh}^{(h)}-\tilde\lambda_{kh}^{(h)}
 =\frac{e_{kh}^{(h)}}{D_{\mathsf S,h}(\tilde\lambda_{kh}^{(h)})}+R_{kh}^{(h)},
 \qquad |R_{kh}^{(h)}|\le \frac{C}{2d^3}|e_{kh}^{(h)}|^2.
\]
On this event the supremum of the scaled remainder is at most $CL^2 h^{1/4}/(2d^3)$. Assumption~\ref{trk:ass:dtl} and Corollary~\ref{cor:unstopped_local_ou} make the complement arbitrarily small by first taking $h\downarrow0$ and then $L\uparrow\infty$. Thus the probability of using the inverse extension vanishes, and the asserted uniform $O_p(h^{1/4})$ expansion holds, including any finite-mesh exit overshoot. The uniform convergence of the derivative in Assumption~\ref{trk:ass:dtl}(d), boundedness away from zero and Corollary~\ref{cor:unstopped_local_ou} give full-window weak convergence. For the stopped version, the expansion at $\tau\wedge\sigma_{h,K}$ differs uniformly by $o_p(1)$ from $\Phi_K(Z_{\mathsf S,h})/D_T$; Theorem~\ref{th::local-ou} yields the stated limit. Multiplication of a path by $1/D_T$ changes the absolute exit radius by $1/|D_T|$, proving the identity of the two stopped limits.

\end{proof}
\endgroup

\subsection{\texorpdfstring{{Proof of Lemma~\ref{trk:lem:ident}}}{Proof of Lemma}}\label{app:proof_trk_lem_ident}
\begingroup
\begin{proof}
By \eqref{eq::local_curvature_interpretation}, $b_{\mathsf S,h,kh}$ is the frozen drift, which converges uniformly on $\mathcal W_h$. Because $G_{\mathsf S}(m^{(h)}_{kh})$ is predictable, \eqref{eq::q_variance_decomposition} reads $q^2_{\mathsf S,h,kh}=\rho^2\Xi_{\mathsf S,h,kh}+\bar\Gamma_{kh}^2s^2_{\eta,h,kh}-2\rho\bar\Gamma_{kh}\chi_{\mathsf S,h,kh}$, which converges to \eqref{trk:eq:q} uniformly on $\mathcal W_h$. Since $h^{1/2}n_h(\tau)\in(\tau-\sqrt h,\tau]$,
\[
\sup_{0\le\tau\le M}\Bigl|\sum_{j<n_h(\tau)}h^{1/2}q^2_{\mathsf S,h,k_jh}-\tau q_{\mathsf S}(T)^2\Bigr|
\le M\max_{j<n_h(M)}\bigl|q^2_{\mathsf S,h,k_jh}-q_{\mathsf S}(T)^2\bigr|+\sqrt h\,q_{\mathsf S}(T)^2\to_p0 .\qedhere
\]
\end{proof}
\endgroup

\subsection{\texorpdfstring{{Proof of Proposition~\ref{trk:prop:frontier}}}{Proof of Proposition}}\label{app:proof_trk_prop_frontier}
\begingroup
\begin{proof}
(i) By Lemma~\ref{trk:lem:ident} and \eqref{eq::stationary_local_variance}, $V^\star_{\mathsf S}(T)=(\rho^2\Xi_T+\bar\Gamma_T^2-2\rho\bar\Gamma_T\chi_T)/(2\rho b_T)$. Substituting $\rho={\varrho}/b_T$ gives \eqref{trk:eq:frontier}.

(ii) By conditional Cauchy--Schwarz, $\chi_{\mathsf S,h,kh}^2\le\Xi_{\mathsf S,h,kh}s^2_{\eta,h,kh}$, so $\chi_T^2\le\Xi_T$ and $\gamma_{\mathsf S}^2\le\kappa_{\mathsf S}$. Then
${\varrho}^{\,2}\kappa_{\mathsf S}-2{\varrho}\bar\Gamma_T\gamma_{\mathsf S}+\bar\Gamma_T^2=\kappa_{\mathsf S}({\varrho}-\bar\Gamma_T\gamma_{\mathsf S}/\kappa_{\mathsf S})^2+\bar\Gamma_T^2(1-\gamma_{\mathsf S}^2/\kappa_{\mathsf S})\ge0$.

(iii) Write $V^\star_{\mathsf S}={\varrho}\kappa_{\mathsf S}/2-\bar\Gamma_T\gamma_{\mathsf S}+\bar\Gamma_T^2/(2{\varrho})$. Its second derivative is $\bar\Gamma_T^2/{\varrho}^{\,3}>0$, and its first derivative vanishes only at ${\varrho}=|\bar\Gamma_T|\kappa_{\mathsf S}^{-1/2}$, where the value is \eqref{trk:eq:vmin}. Non-negativity follows from (ii).

(iv) Immediate.
\end{proof}
\endgroup

\subsection{\texorpdfstring{{Proof of Proposition~\ref{trk:prop:eff}}}{Proof of Proposition}}\label{app:proof_trk_prop_eff}
\begingroup
\begin{proof}
(i) Let $\Psi_h(a,x)=\E_x[\psi_{\mathsf S}(Y,a)]$. By Assumption~\ref{ass:conditional_differentiation}, $\Psi_h(a,x)=-\partial_a\widetilde{\mathsf R}^{(h)}_{\mathsf S}(a;x)$, so $\partial_a\Psi_h=-\partial^2_{aa}\widetilde{\mathsf R}^{(h)}_{\mathsf S}$.

By Assumption~\ref{trk:ass:dtl}(a) and dominated convergence, $\partial_x\Psi_h(a,x)=\E_x[\psi_{\mathsf S}(Y,a)\tilde s_{h,x}(Y)]$, continuous on $V\times U$, and $\E_x[\tilde s_{h,x}]=0$.

The first-order condition $\Psi_h(\lambda^\star_{\mathsf S,h}(x),x)=0$ holds on $U$, and Lemma~\ref{lem:pseudotrue_implicit} makes $\lambda^\star_{\mathsf S,h}$ continuously differentiable. Differentiating gives $-J_{\mathsf S,h}(x)D_{\mathsf S,h}(x)+\E_x[\psi_{\mathsf S}(Y,\lambda^\star_{\mathsf S,h}(x))\tilde s_{h,x}(Y)]=0$. Both factors in the expectation are centred, which proves \eqref{trk:eq:sens}.

Multiplying by the predictable $G_{\mathsf S}(\lambda^\star_{\mathsf S,h}(x))$ gives $b_{\mathsf S,h}D_{\mathsf S,h}=\Cov_x[u_{\mathsf S},\tilde s_{h,x}]$. Evaluating at $x=\tilde\lambda^{(h)}_{kh}$ and using Assumptions~\ref{trk:ass:frozen} and~\ref{trk:ass:dtl}(d) gives $\Upsilon_T=b_TD_T\ne0$.

(ii) By \eqref{trk:eq:vmin} and $\bar\Gamma_T=D_T\sigma_T$,
\[
V^\star_{\mathsf S,\min}/D_T^2=\sigma_T\kappa_{\mathsf S}^{1/2}/|D_T|-\sigma_T\gamma_{\mathsf S}/D_T=\sigma_T\Xi_T^{1/2}/|b_TD_T|-\sigma_T\chi_T/(b_TD_T).
\]
Using $b_TD_T=\Upsilon_T$ gives \eqref{trk:eq:lat}. With $\chi_T=0$, $\sigma_T\Xi_T^{1/2}/|\Upsilon_T|=\sigma_T\tilde I_T^{-1/2}|r_{\mathsf S}(T)|^{-1}$. The minimiser follows from Proposition~\ref{trk:prop:frontier}(iii).

(iii) At each $h$, $\Cov_{kh}[u_{\mathsf S},\tilde s]^2\le\Xi_{\mathsf S,h,kh}\tilde I_h(\tilde\lambda^{(h)}_{kh})$ by Cauchy--Schwarz, with equality if and only if the centred variables are proportional. Passing to the limit gives $\Upsilon_T^2\le\Xi_T\tilde I_T$; positivity follows from $\Upsilon_T\ne0$. Under correct specification, $D\equiv1$, so $\operatorname{eff}_{\mathsf S}=b_T^2/(\Xi_T\tilde I_T)=1/(\tilde I_T\kappa_{\mathsf S})$. The log-score driver is proportional to $\tilde s$, so $r=1$.
\end{proof}
\endgroup

\subsection{\texorpdfstring{{Proof of Corollary~\ref{trk:cor:scale}}}{Proof of Corollary}}\label{app:proof_trk_cor_scale}
\begingroup
\begin{proof}
(i) Write $P^{(h)}_\lambda=\mathsf T_{\log h+x}\#P_{\lambda-x}$ and $y=\mathsf T_{\log h+x}(\varepsilon)$. By \eqref{trk:eq:equi},
\[
\mathsf S\bigl(P^{(h)}_\lambda,y\bigr)=A(\log h+x)\,\mathsf S(P_{\lambda-x},\varepsilon)+B(\log h+x).
\]
Hence $\widetilde{\mathsf R}^{(h)}_{\mathsf S}(\lambda;x)=A(\log h+x)R_F(\lambda-x)+B(\log h+x)$. Since $A>0$, the minimiser is $x+\delta_{\mathsf S}(F)$.

(ii) Differentiating in $\lambda$ gives $\psi_{\mathsf S}(y,\lambda)=A(\log h+x)\,\psi_{\mathsf S}(\varepsilon,\lambda-x)$. At the target, $K_{\mathsf S,h}=A^2K_F$ and $J_{\mathsf S,h}=AJ_F$, so the factor $A$ cancels from $\kappa_{\mathsf S}$. The true density is $(he^x)^{-1/2}f(y(he^x)^{-1/2})$, so $\tilde s_{h,x}(y)=s_F(\varepsilon)$ and $\tilde I_h=I_F$. With $D=1$ and \eqref{trk:eq:sens}, $\operatorname{eff}_{\mathsf S}=J_{\mathsf S,h}^2/(K_{\mathsf S,h}I_F)=1/(I_F\kappa_{\mathsf S}(F))$.

(iii) By reflection symmetry and the representation in (ii), $u_{\mathsf S}(y,m^{(h)}_{kh})=q(|\varepsilon|)$ for a measurable $q$. Hence
\[\chi=\E[q(|\varepsilon|)\,\eta]=\E[q(|{-\varepsilon}|)\,(-\eta)]=-\chi,\]
so $\chi=0$.

(iv) Combine (i)--(iii) with Propositions~\ref{trk:prop:frontier}(iii) and~\ref{trk:prop:eff}(ii). With $u_{\mathsf S}=g(z)$,
\[b_T=-\partial_\lambda\E\bigl[g\bigl(y/(he^\lambda)^{1/2}\bigr)\bigr]\Big|_{\lambda=m}=\E_F\bigl[\tfrac z2\,g'(z)\bigr],\qquad z=e^{-\delta_{\mathsf S}/2}\varepsilon,\]
which is free of $x$ and $h$.
\end{proof}
\endgroup

\subsection{\texorpdfstring{{Primitive conditions in the scale model}}{Primitive conditions in the scale model}}\label{trk:sec:primitive}

\begingroup
Throughout this section the working family, rule and scaling are those of Corollary~\ref{trk:cor:scale}. The scaled update is a function of the standardised residual, $u_{\mathsf S}(y,\lambda)=g\bigl(y/(he^\lambda)^{1/2}\bigr)$, as under the scalings of \ref{app::formulas}. Write $m(z)=zg'(z)/2$.

\begin{definition}[Bounded scale-invariant drivers]\label{trk:def:Bc}
A driver $g\in C^2(\R)$ belongs to the class $\mathfrak{B}$ if $g$, $m$ and $z\mapsto zm'(z)$ are bounded. Write $M_0=\sup|g|$, $M_1=\sup|m|$ and $M_2=\sup|zm'(z)|$.
\end{definition}

{Since $zm'(z)=\{zg'(z)+z^2g''(z)\}/2$, membership requires bounded $g$, $zg'$ and $z^2g''$. The Student-$t_6$ log, Gaussian density-power ($\beta=0.15$) and Gaussian CRPS drivers in \appref{app::formulas} belong to $\mathfrak B$, so their constants $M_0,M_1,M_2$ are finite. Their explicit formulas establish these boundedness properties; no rounded numerical value is used as an exact global supremum.}
The function $m$ is nonnegative for the Student-$t_6$ log and CRPS drivers and changes sign for the density-power driver. The Gaussian log driver $z^2-1$ is not in $\mathfrak{B}$.

\begin{assumption}[Primitive scale model]\label{trk:ass:prim}
\leavevmode
\begin{enumerate}[label=(\alph*)]
\item Assumption~\ref{trk:ass:scale} holds with $\sigma_h\equiv\sigma>0$ and $\sup_{h,x}|b_h(x)|\le\bar b<\infty$, and $u_{\mathsf S}(y,\lambda)=g(y/(he^\lambda)^{1/2})$ with $g\in\mathfrak{B}$.
\item The pairs $(\varepsilon_{(k+1)h},\eta_{(k+1)h})$ are independent of $\mathcal G^{(h)}_{kh}$ and identically distributed, with $\varepsilon\sim F$, $\E\eta=0$, $\Var\eta=1$ and $\E|\eta|^q=C_\eta<\infty$ for some $q>2$.
\item $\sup_h\E|e^{(h)}_0|^q<\infty$ and $\E|e^{(h)}_0|^q\to0$, with $q$ as in (b).
\end{enumerate}
\end{assumption}

Let $w=e^{-\delta_{\mathsf S}(F)/2}$ and define the mean field $\bar g(e)=\E_F[g(we^{-e/2}\varepsilon)]$. The first-order condition of the target gives $\bar g(0)=0$. Moreover $b_{\mathsf S}:=\E_F[m(w\varepsilon)]$ is the frozen drift \eqref{eq::local_curvature_interpretation}, so $b_{\mathsf S}>0$ by Assumption~\ref{trk:ass:scale}(c). Put $\Xi_{\mathsf S}=\E_F[g(w\varepsilon)^2]$ and $\chi_{\mathsf S}=\E[g(w\varepsilon)\eta]$.

\begin{lemma}[Local conditions]\label{trk:lem:local}
Under Assumption~\ref{trk:ass:prim}(a)--(b):
\begin{enumerate}[label=(\roman*)]
\item $\bar u^{(h)}_{kh}(m^{(h)}_{kh}+e)=\bar g(e)$ for all $h$, $k$ and $e$. Moreover $\bar g\in C^2$, with $\bar g'(e)=-\E_F[m(we^{-e/2}\varepsilon)]$ and $\bar g''(e)=\tfrac12\E_F[(zm'(z))|_{z=we^{-e/2}\varepsilon}]$. Hence $|\bar g'|\le M_1$, $|\bar g''|\le M_2/2$ and $\bar g'(0)=-b_{\mathsf S}$.
\item Assumption~\ref{ass:local-tracking} holds with $\bar\Gamma_{kh}=\sigma$ and $r_{(k+1)h}=hb_h(\tilde\lambda^{(h)}_{kh})$. Take any $r>0$ with $\bar g'\le-b_{\mathsf S}/2$ on $[-r,r]$, together with $\mu=b_{\mathsf S}/2$, $L=M_1$ and $C=\max\{\bar b^2,\,8\rho^2M_0^2+2\sigma^2\}$.
\item For every $T\in(0,T_0)$ and $M<\infty$, items (ii)--(v) of Assumption~\ref{ass::local-ou} hold, with $b_{\mathsf S}(T)=b_{\mathsf S}$ and $q_{\mathsf S}(T)^2=\rho^2\Xi_{\mathsf S}+\sigma^2-2\rho\sigma\chi_{\mathsf S}$. In item~(iv) one may take $L=M_1$. In item~(iii) the remainder is bounded deterministically by $\rho M_2(M+1)K^2h^{1/4}/4$.
\item Assumption~\ref{trk:ass:frozen} holds with $b_T=b_{\mathsf S}$, $\Xi_T=\Xi_{\mathsf S}$, $\bar\Gamma_T=\sigma$ and $\chi_T=\chi_{\mathsf S}$, all free of $T$ and $h$.
\end{enumerate}
\end{lemma}

\begin{proof}
(i) By Corollary~\ref{trk:cor:scale}(i), $y^{(h)}_{(k+1)h}/(he^{m+e})^{1/2}=we^{-e/2}\varepsilon_{(k+1)h}$, so (b) gives $\bar u(m+e)=\bar g(e)$. The map $e\mapsto g(we^{-e/2}\varepsilon)$ has derivatives $-m(we^{-e/2}\varepsilon)$ and $\tfrac12(zm'(z))|_{z=we^{-e/2}\varepsilon}$, bounded by $M_1$ and $M_2/2$. Dominated convergence gives the derivatives of $\bar g$.

(ii) By Corollary~\ref{trk:cor:scale}(i) and \eqref{eq::latent_space_recursion}, $m_{(k+1)h}-m_{kh}=\sqrt h\sigma\eta_{(k+1)h}+hb_h(\tilde\lambda_{kh})$, so $\E_{kh}[r^2]\le\bar b^2h^2$.

Local mean reversion follows from the mean value theorem: $(\lambda-m)\bar u(\lambda)=\bar g'(\zeta)e^2\le-\tfrac{b_{\mathsf S}}2e^2$ and $|\bar u(\lambda)|\le M_1|e|$.

Since $|u_{\mathsf S}|\le M_0$, the composite innovation satisfies $\E_{kh}[(\rho(u_{\mathsf S}-\bar u)-\sigma\eta)^2]\le8\rho^2M_0^2+2\sigma^2$.

(iii) \emph{Item (ii).} $b_{\mathsf S,h,kh}=-\bar g'(0)=b_{\mathsf S}$ exactly.

\emph{Item (iii).} $|\bar u(m+e)+b_Se|=|\bar g(e)-\bar g'(0)e|\le\tfrac{M_2}4e^2$. Before $\sigma_{h,K}$, $|e|\le Kh^{1/4}$, and there are at most $M/\sqrt h+1$ terms, giving the stated bound for $h\le1$.

\emph{Item (iv).} $u_{\mathsf S}(y,m+x)-u_{\mathsf S}(y,m)=-\int_0^xm(we^{-s/2}\varepsilon)\,ds$, so the centred difference has conditional second moment at most $M_1^2x^2$.

\emph{Item (v).} By (b), $\tilde\xi_{(k+1)h}=\rho g(w\varepsilon_{(k+1)h})-\sigma\eta_{(k+1)h}$ is a martingale difference with constant conditional variance $\rho^2\Xi_{\mathsf S}+\sigma^2-2\rho\sigma\chi_{\mathsf S}$, so the bracket condition holds as in Lemma~\ref{trk:lem:ident}. With $C_\xi=2^{q-1}(\rho^qM_0^q+\sigma^qC_\eta)\ge\E|\tilde\xi|^q$, for every $c>0$,
\[
\sum_{j<n_h(M)}h^{1/2}\E\bigl[\tilde\xi^2\mathbf 1\{|h^{1/4}\tilde\xi|>c\}\bigr]\le(M/\sqrt h+1)\,h^{1/2}\,C_\xi\,(h^{1/4}/c)^{q-2}\to0 .
\]

(iv) Immediate from (iii) and the definitions.
\end{proof}

\begin{lemma}[Exit, tightness and compact containment]\label{trk:lem:exit}
Under Assumption~\ref{trk:ass:prim}, let $r$ and $\mu$ be as in Lemma~\ref{trk:lem:local}, let $\tau_{r,h}$ be the exit time of Theorem~\ref{th::local-tracking}, let $q>2$ be as in Assumption~\ref{trk:ass:prim}(b), and put $c=\rho\mu/2$. There are $h_0>0$ and $C<\infty$, independent of $h$ and $k$, such that for $h\le h_0$:
\begin{enumerate}[label=(\roman*)]
\item $\E[|e^{(h)}_{kh}|^q\mathbf 1\{k<\tau_{r,h}\}]\le(1-c\sqrt h)^k\,\E|e^{(h)}_0|^q+Ch^{q/4}$ for $0\le k\le N_h$;
\item $P(\tau_{r,h}\le N_h)\le Cr^{-q}\bigl(\E|e^{(h)}_0|^q+h^{(q-2)/4}\bigr)\to0$;
\item for every $T\in(0,T_0)$ and $M<\infty$, $h^{-1/4}e^{(h)}_{k_Th}$ is tight, and
\[\lim_{K\to\infty}\limsup_{h\downarrow0}P\Bigl(\sup_{0\le\tau\le M}|Z_{\mathsf S,h}(\tau)|>K\Bigr)=0.\]
\end{enumerate}
Consequently, Assumption~\ref{ass::local-ou}-(i) holds along every subsequence on which $Z_{\mathsf S,h}(0)$ converges in distribution, and every subsequence has such a further subsequence. Along it, Theorem~\ref{th::local-ou} and Corollary~\ref{cor:unstopped_local_ou} give $Z_{\mathsf S,h}\Rightarrow Z_{\mathsf S}$ in $D([0,M])$, with the coefficients of Lemma~\ref{trk:lem:local}(iii). \begingroup
The frozen invariant variance is independent of $Z_0$, so Propositions~\ref{trk:prop:frontier} and~\ref{trk:prop:eff} define the same local benchmark along these subsequences. This does not identify the initial limiting law as invariant or imply convergence of actual squared-error expectations. Such moment conclusions require additional conditions, as supplied by Proposition~\ref{trk:prop:selftune} for its particular stationary-error model.
\endgroup
\end{lemma}

\begin{proof}
Write $e_k=e^{(h)}_{kh}$. By Lemma~\ref{trk:lem:local}(i) and Corollary~\ref{trk:cor:scale}(i),
\begin{equation}\label{trk:eq:rec}
e_{k+1}=e_k+\rho\sqrt h\,\bar g(e_k)+\sqrt h\,\xi_{k+1}-hb_h(\tilde\lambda_{kh}),\qquad
\xi_{k+1}=\rho\{g(we^{-e_k/2}\varepsilon_{(k+1)h})-\bar g(e_k)\}-\sigma\eta_{(k+1)h}.
\end{equation}
Here $\E_{kh}[\xi_{k+1}]=0$. Since $|g-\bar g|\le2M_0$, $\E_{kh}[\xi_{k+1}^2]\le C_2=8\rho^2M_0^2+2\sigma^2$ and $\E_{kh}|\xi_{k+1}|^q\le C_q=2^{q-1}\{(2\rho M_0)^q+\sigma^qC_\eta\}$. On $\{|e_k|\le r\}$, $\bar g(e_k)=-\beta_ke_k$ with $\beta_k\in[\mu,M_1]$, and $|\bar g(e)|\le M_1|e|$ for all $e$.

\emph{(i) Moment drift.} For $q\ge2$ and real $x,y$, Taylor's theorem gives $|x+y|^q\le|x|^q+q|x|^{q-2}xy+c_q(|x|^{q-2}y^2+|y|^q)$, with $c_q=\tfrac12q(q-1)\max\{1,2^{q-3}\}$. Let $|e_k|\le r$, take $h_0$ with $\rho M_1\sqrt{h_0}\le1$, and put $a_k=(1-\rho\sqrt h\beta_k)e_k-hb_h$. Then $e_{k+1}=a_k+\sqrt h\xi_{k+1}$ and $|a_k|\le(1-\rho\mu\sqrt h)|e_k|+h\bar b$. The linear term vanishes in conditional expectation, so
\[
\E_{kh}|e_{k+1}|^q\le|a_k|^q+c_q\bigl(C_2h|a_k|^{q-2}+C_qh^{q/2}\bigr).
\]
Three bounds complete the step:
\begin{itemize}[leftmargin=1.5em]
\item Young's inequality with exponents $q/(q-2)$ and $q/2$ gives $c_qC_2h|a|^{q-2}\le\tfrac{\rho\mu\sqrt h}4|a|^q+Ch^{(q+2)/4}$.
\item $h^{q/2}\le h^{(q+2)/4}$ for $h\le1$ and $q\ge2$.
\item From $(x+y)^q\le(1+\theta)^{q-1}x^q+(1+\theta^{-1})^{q-1}y^q$ with $\theta=\rho\mu\sqrt h/4$, and $(1+\theta)^{q-1}(1-4\theta)^q\le1-4\theta$, we get $|a_k|^q\le(1-\rho\mu\sqrt h)|e_k|^q+Ch^{(q+1)/2}$.
\end{itemize}
Collecting terms and using $(1+x/4)(1-x)\le1-x/2$,
\begin{equation}\label{trk:eq:drift4}
\E_{kh}|e_{k+1}|^q\le(1-c\sqrt h)\,|e_k|^q+Ch^{(q+2)/4}\qquad\text{on }\{|e_k|\le r\}.
\end{equation}
Since $\{k<\tau\}\in\mathcal G_{kh}$ and $\{k+1<\tau\}\subset\{k<\tau\}$, iterating \eqref{trk:eq:drift4} gives (i), because $Ch^{(q+2)/4}/(c\sqrt h)=Ch^{q/4}$. The same bound holds for $\E[|e_{k+1}|^q\mathbf 1\{k<\tau\}]$, which includes the exit step.

\begingroup
\emph{(ii) Exit probability.} Let $n=\lceil h^{-1/2}\rceil$. For each $j\ge0$ with $jn<N_h$, put $b_j=\min\{(j+1)n,N_h\}$ and use the block $\{jn,\dots,b_j\}$. On $\{jn<\tau\}$, for $jn\le k\le b_j\wedge\tau$, \eqref{trk:eq:rec} gives
\[
|e_k|\le|e_{jn}|+\rho M_1\sqrt h\sum_{i=jn}^{k-1}|e_i|+nh\bar b+\max_{jn\le i\le b_j}|S_i|,
\qquad S_i=\sqrt h\sum_{l=jn}^{i-1}\xi_{l+1}\mathbf 1\{l<\tau\}.
\]
By the discrete Gronwall inequality, and since $n\sqrt h\le2$ and $nh\le2\sqrt h$,
\[
\max_{jn\le k\le b_j\wedge\tau}|e_k|\le e^{2\rho M_1}\bigl(|e_{jn}|+2\sqrt h\bar b+\max_{jn\le i\le b_j}|S_i|\bigr).
\]
For $h$ so small that $2e^{2\rho M_1}\sqrt h\bar b\le r/3$, Markov's inequality gives
\[
P(jn<\tau\le b_j)\le(3e^{2\rho M_1}/r)^q\bigl\{\E[|e_{jn}|^q\mathbf 1\{jn<\tau\}]+\E\max_{jn\le i\le b_j}|S_i|^q\bigr\}.
\]
The event on the left contains a genuine exit; the no-exit value $\tau=N_h+1$ is excluded by $b_j\le N_h$. Doob's $L^q$ inequality and the Burkholder--Rosenthal inequality \citep[Ch.~2]{hall1980martingale} give
\[
\E\max_{jn\le i\le b_j}|S_i|^q\le C\bigl\{(nhC_2)^{q/2}+nh^{q/2}C_q\bigr\}\le C\bigl(h^{q/4}+h^{(q-1)/2}\bigr).
\]
There are at most $T_0h^{-1/2}+1$ blocks, and their exit events partition $\{1\le\tau\le N_h\}$. Summing over them, with (i), $(1-c\sqrt h)^{jn}\le e^{-cj}$ and $P(\tau=0)\le\E|e_0|^q/r^q$, gives (ii).
\endgroup

\emph{(iii) Tightness and containment.} For $K>0$,
\[
P(h^{-1/4}|e_{k_T}|>K)\le P(\tau\le N_h)+\frac{\E[|e_{k_T}|^q\mathbf 1\{k_T<\tau\}]}{K^qh^{q/4}}.
\]
Here $(1-c\sqrt h)^{k_T}h^{-q/4}\le e^{-cT/\sqrt h+c\sqrt h}h^{-q/4}\to0$, so $\limsup_hP(h^{-1/4}|e_{k_T}|>K)\le C/K^q$.

\begingroup
For containment, put $L_h=k_T+\lceil M/\sqrt h\rceil$. Since $T<T_0$, $L_h\le N_h$ for all sufficiently small $h$. On $\{\tau>N_h\}$, split $\{k_T,\dots,L_h\}$ into blocks of at most $n$ increments, truncating the last endpoint at $L_h$, and apply the same stopped Gronwall and martingale bounds with threshold $Kh^{1/4}$. The survival moment at every block start is $O(h^{q/4})$ by (i), because the block starts are at least $k_T$. For each fixed $K>0$, the drift term $2\sqrt h\bar b$ is smaller than $Kh^{1/4}/(3e^{2\rho M_1})$ for all sufficiently small $h$. Each of the at most $M+2$ blocks therefore contributes at most
\[
(3e^{2\rho M_1}/(Kh^{1/4}))^q\,C\bigl(h^{q/4}+h^{(q-1)/2}\bigr)\le CK^{-q}.
\]
Adding the probability of leaving the survival event gives, for each fixed $K$ and all sufficiently small $h$,
\[
P\Bigl(\sup_{0\le\tau'\le M}|Z_{\mathsf S,h}(\tau')|>K\Bigr)
\le P(\tau\le N_h)+C_MK^{-q},
\]
where $C_M$ is independent of $h$ and $K$. Part~(ii), followed by $K\to\infty$, proves compact containment.
\endgroup

\emph{Consequences.} Theorem~\ref{th::local-ou} requires Assumption~\ref{ass:local-tracking}(i), given by Lemma~\ref{trk:lem:local}(ii), and Assumption~\ref{ass::local-ou}. Items (ii)--(v) of Assumption~\ref{ass::local-ou} are in Lemma~\ref{trk:lem:local}(iii). Item (i) holds along subsequences by Prohorov's theorem, and $\Lambda=\R$ makes the tube condition automatic. Corollary~\ref{cor:unstopped_local_ou} requires the containment just shown.
\end{proof}

\begin{lemma}[Efficiency identity in the scale model]\label{trk:lem:effid}
Let $f_\theta(y)=e^{-\theta/2}f(e^{-\theta/2}y)$ be the density of $e^{\theta/2}\varepsilon$. Suppose Assumption~\ref{trk:ass:prim}(a) holds, with the density regularity and finite scale Fisher information of Assumption~\ref{trk:ass:scale}(b), and with $G_{\mathsf S}$ positive and continuously differentiable near $\delta_{\mathsf S}$. Then
\[
 \E_F[\psi_{\mathsf S}(\varepsilon,\delta_{\mathsf S})s_F(\varepsilon)]=J_F.
\]
Hence \eqref{trk:eq:sens} holds with $D=1$, and $\operatorname{eff}_{\mathsf S}(F)=1/(I_F\kappa_{\mathsf S}(F))$. No integrable pointwise supremum of the density derivatives is required. The density regularity holds for Gaussian and variance-standardised Student-$t$ laws.
\end{lemma}

\begin{proof}
By Corollary~\ref{trk:cor:scale}(i) with $h=1$, the target under $f_\theta$ is $\theta+\delta_{\mathsf S}$. Its first-order condition is
\[
 \Phi(\theta)=\int\psi_{\mathsf S}(y,\theta+\delta_{\mathsf S})f_\theta(y)\,\ud y=0.
\]
Put $k(y)=-(f(y)+yf'(y))/2=f(y)s_F(y)$, with identities understood almost everywhere. Cauchy--Schwarz gives $\|k\|_1\le I_F^{1/2}$. The dilation operators
\[
 (U_\theta r)(y)=e^{-\theta/2}r(e^{-\theta/2}y)
\]
are strongly continuous isometries on $L^1(\mathbb R)$. Local absolute continuity and the fundamental theorem of calculus give
\[
 f_\theta-f_0=\int_0^\theta U_s k\,\ud s
 \quad\hbox{in }L^1,
 \qquad
 \frac{f_\theta-f_0}{\theta}\longrightarrow k
 \quad\hbox{in }L^1.
\]
Here the integral is oriented when $\theta<0$. In particular, $\int k\,\ud y=0$.

Writing $m(z)=zg'(z)/2$, the residual-only scaling gives
\[
 \psi_{\mathsf S}(y,a)=\frac{g(ye^{-a/2})}{G_{\mathsf S}(a)},
 \qquad
 \partial_a\psi_{\mathsf S}(y,a)
 =-\frac{m(ye^{-a/2})}{G_{\mathsf S}(a)}
 -\frac{G_{\mathsf S}'(a)g(ye^{-a/2})}{G_{\mathsf S}(a)^2}.
\]
Both expressions are uniformly bounded for $a$ near $\delta_{\mathsf S}$ because $g,m$ are bounded and $G_{\mathsf S}$ is positive and $C^1$. Split the difference quotient of $\Phi$ into a density term and a score term. The density term converges by the displayed $L^1$ differentiability and the uniform bound on the score; the score term converges by dominated convergence using the bound on its parameter derivative. Therefore
\[
 0=\Phi'(0)
 =\E_F[\partial_a\psi_{\mathsf S}(\varepsilon,\delta_{\mathsf S})]
  +\E_F[\psi_{\mathsf S}(\varepsilon,\delta_{\mathsf S})s_F(\varepsilon)].
\]
The same derivative bound and Assumption~\ref{ass:conditional_differentiation} give
\[
 \E_F[\partial_a\psi_{\mathsf S}(\varepsilon,\delta_{\mathsf S})]
 =-R_F''(\delta_{\mathsf S})=-J_F,
\]
which proves the identity.
\end{proof}

\begingroup
\begin{lemma}[M-estimation and the efficiency interpretation]\label{trk:lem:mest}
Let $Y_i=e^{x/2}\varepsilon_i$ be i.i.d.\ with $\varepsilon_i\sim F$, let $g\in\mathfrak B$, and retain the unique interior target $\delta_{\mathsf S}$ and positive local curvature of Assumption~\ref{trk:ass:scale}. Put $\lambda_0=x+\delta_{\mathsf S}$.
\begin{enumerate}[label=(\roman*)]
\item Suppose $m\ge0$ and $\mathsf P(m(w\varepsilon)>0)>0$. Then $\Psi_n(\lambda)=n^{-1}\sum_i g(Y_i e^{-\lambda/2})$ is non-increasing. Any sequence satisfying $\Psi_n(\widehat\lambda_n)=o_p(n^{-1/2})$ obeys
\[
 \sqrt n(\widehat\lambda_n-\lambda_0)\Rightarrow N(0,\kappa_{\mathsf S}(F)).
\]
This applies to the monotone CRPS and Student-$t$ log drivers under the stated identification conditions.
\item For a Gaussian or fixed-shape Student density-power working scale density, $\beta>0$, suppose $\mathsf P(Y=0)=0$. Every unrestricted $o_p(1)$-near minimiser of the empirical density-power criterion is consistent. An exact minimiser, or a near minimiser also satisfying $\Psi_n(\widehat\lambda_n)=o_p(n^{-1/2})$, has the normal limit in (i).
\item If the correctly specified true-$F$ maximum-likelihood estimator of $x$ is regular with asymptotic variance $1/I_F$, then $1/(I_F\kappa_{\mathsf S})$ is the asymptotic relative variance efficiency of $\widehat\lambda_n-\delta_{\mathsf S}(F)$ for estimating $x$. Equivalently, compare $\widehat\lambda_n$ with the true-$F$ estimator shifted by $\delta_{\mathsf S}(F)$ so that both estimate $\lambda_0$.
\end{enumerate}
\end{lemma}
\begin{proof}
(i) Differentiation gives $\Psi_n'(\lambda)=-n^{-1}\sum_i m(Y_i e^{-\lambda/2})\le0$. The population function is strictly decreasing near its zero, with derivative $-b_{\mathsf S}<0$. The law of large numbers at two fixed points on either side of $\lambda_0$ and monotonicity imply consistency of any stated approximate root. Since the functions $g(ye^{-\lambda/2})$ are bounded and uniformly Lipschitz in $\lambda$, the local estimating-equation expansion and i.i.d.\ central limit theorem give
\[
 \sqrt n(\widehat\lambda_n-\lambda_0)
 =b_{\mathsf S}^{-1}n^{-1/2}\sum_i g(Y_i e^{-\lambda_0/2})+o_p(1),
\]
with variance $\Xi_{\mathsf S}/b_{\mathsf S}^2=\kappa_{\mathsf S}$; see \citet[Theorem~5.21]{vaart1998asymptotic}.

(ii) Let $q$ be the unit-scale working density, $u=e^{\lambda/2}>0$, $C=\int q^{1+\beta}$, and $A=1+1/\beta$. Up to a term independent of $u$, the empirical risk is
\[
 R_n(u)=u^{-\beta}\left[C-\frac{A}{n}\sum_{i=1}^n q(Y_i/u)^\beta\right].
\]
The density is even, bounded, continuous and decreasing in $|y|$. As $u\downarrow0$, $\E q(Y/u)^\beta\to0$ by dominated convergence and $\mathsf P(Y=0)=0$. Choose $u_->0$ so the expectation at $u_-$ is less than $C/(2A)$. Monotonicity and the law of large numbers make $R_n(u)>0$ throughout $u\le u_-$ with probability tending to one. At the other end, $C\le q(0)^\beta<Aq(0)^\beta$. Thus $u^\beta R(u)\to C-Aq(0)^\beta<0$ as $u\to\infty$, where $R=\E R_n$. Choose a fixed $u_0$ with $R(u_0)<0$. For $u\ge u_+$,
\[
 R_n(u)\ge-[Aq(0)^\beta-C]u_+^{-\beta}.
\]
Choose $u_+>u_0>u_-$ large enough that this lower bound exceeds $R(u_0)/2$, decreasing $u_-$ if needed. Since $R_n(u_0)\to R(u_0)$, both tails are separated from the sample infimum by a fixed gap with probability tending to one. All $o_p(1)$-near minimisers then lie in the same deterministic compact interval with probability tending to one. Continuity, a bounded envelope on that interval and uniqueness of the population minimum give consistency. At an interior exact minimum the scaled score vanishes. For a consistent approximate root with the stated residual the same expansion as in (i) holds: the bounded score derivative has a uniform local law of large numbers, and its expectation at the target is $-b_{\mathsf S}<0$. This last rate condition distinguishes objective accuracy from first-order accuracy.

\end{proof}

For the Gaussian density-power normalisation, the population risk is
\[
 R_F(\delta)=c_\beta e^{-\beta\delta/2}
 \bigl[(1+\beta)^{-1/2}-(1+\beta^{-1})\E_F e^{-\beta e^{-\delta}\varepsilon^2/2}\bigr].
\]
It tends to $+\infty$ as $\delta\to-\infty$ for the continuous innovation laws considered, and to zero as $\delta\to+\infty$. The reported risks at the numerical target roots lie between $-5.42$ and $-5.56$ for the six laws. These endpoint calculations support population separation at a negative minimum, while the sign-change argument below establishes uniqueness. The sample criterion requires the two uniform tail bounds proved in Lemma~\ref{trk:lem:mest}(ii).
The following argument verifies the population identification and curvature hypotheses for the reference laws.
\endgroup

\begin{lemma}[Student-$t_6$ CRPS and density-power cells]\label{trk:lem:tcells}
Let $q_\nu$ and $F_\nu$ be the density and distribution function of the variance-standardised $t_\nu$ law with $\nu=6$, put $d_\nu=\nu-2$, and use the scalings $2e^{-\lambda/2}$ (CRPS) and $2e^{\beta\lambda/2}$ (density-power) of \ref{app::formulas}. The scaled drivers are
\[
g^{t}_{\rm CRPS}(z)=c_\nu-\frac{2(d_\nu+z^2)}{\nu-1}\,q_\nu(z),\qquad
g^{t}_{\rm DP}(z)=\beta C_{\nu,\beta}-(1+\beta)\,q_\nu(z)^\beta\,\frac{d_\nu-\nu z^2}{d_\nu+z^2},
\]
with $c_\nu$ as in \ref{app::formulas} and $C_{\nu,\beta}=\int q_\nu^{1+\beta}$.
\begin{itemize}[leftmargin=1.5em]
\item {Both drivers belong to $\mathfrak{B}$, with finite constants $M_0,M_1,M_2$ as proved below.}
\item For the CRPS driver $m(z)=z^2q_\nu(z)\ge0$; for the density-power driver $m$ changes sign.
\item Both satisfy $g(0)<0<g(\pm\infty)$.
\item {For the CRPS driver and each continuous nondegenerate law in Table~\ref{trk:tab:eff}, $\delta\mapsto\E_F[g(e^{-\delta/2}\varepsilon)]$ is strictly decreasing and has a unique zero: its derivative satisfies
\[
 \frac{\ud}{\ud\delta}\E_F[g(e^{-\delta/2}\varepsilon)]
 =-\E_F[m(e^{-\delta/2}\varepsilon)]<0,
\]
and its limits have opposite signs. \ref{pur:sec:dp} proves the corresponding uniqueness result for density power.}
\end{itemize}
Together with the stochastic and moment assumptions, these calculations verify the driver conditions for the local results. Lemma~\ref{trk:lem:mest}(i) applies to the monotone CRPS equation. For density power, \ref{pur:sec:dp} establishes population identification and restoration for the tabulated laws, and Lemma~\ref{trk:lem:mest}(ii) supplies empirical containment for the unrestricted i.i.d.\ scale problem. Global restoration also enters Proposition~\ref{trk:prop:selftune}.
\end{lemma}

\begin{proof}
\emph{CRPS.} Let $G_\nu(z)=\int_z^\infty xq_\nu(x)\,dx=(d_\nu+z^2)q_\nu(z)/(\nu-1)$; the identity follows by differentiating both sides. The CRPS of the scale family is $e^{\lambda/2}c(z)$ with $c(z)=z\{2F_\nu(z)-1\}+2G_\nu(z)-c_\nu$, as in \ref{app::formulas}. Since $c'(z)=2F_\nu(z)-1$, $\partial_\lambda\{e^{\lambda/2}c(ye^{-\lambda/2})\}=\tfrac12e^{\lambda/2}\{2G_\nu(z)-c_\nu\}$, and the scaling gives $g^t_{\rm CRPS}$. Its derivative is $2zq_\nu(z)$, so $m(z)=z^2q_\nu(z)$.

\emph{Density-power.} The score is $S=e^{-\beta\lambda/2}\{C_{\nu,\beta}-(1+\beta^{-1})q_\nu(z)^\beta\}$, the Student analogue of \ref{app::formulas}. Differentiation, with $1+zq_\nu'(z)/q_\nu(z)=(d_\nu-\nu z^2)/(d_\nu+z^2)$, gives $g^t_{\rm DP}$.

\emph{Boundedness.} Since $q_\nu(z)=O(|z|^{-\nu-1})$ and $zq'_\nu/q_\nu$ is rational and bounded, $g$, $zg'$ and $z^2g''$ are bounded for both drivers.

{\emph{Numerical checks.} {The reported finite-difference comparison of drivers and loss derivatives has error below $10^{-8}$.} The numerical search locates the targets. \ref{pur:sec:dp} establishes uniqueness on the entire log-scale line.}
\end{proof}

\subsubsection{Identification of the density-power scale projection}\label{pur:sec:dp}
Let $q$ be either working density (Gaussian or variance-standardised Student-$t_6$), $C=\int q^{1+\beta}$ and $\beta=0.15$. Its scaled density-power driver $g$ changes sign exactly once on $z>0$, from negative to positive at $z_0$. Indeed $C\le q(0)^\beta$ gives $g(0)<0$. Where the second term in $g$ is negative, its absolute value is $(1+\beta)q(z)^\beta$ times $1-z^2$ in the Gaussian case, or $(d_\nu-\nu z^2)/(d_\nu+z^2)$ in the Student case. Both factors decrease strictly there. Once the second term is nonnegative, $g$ is positive and tends to $\beta C$ at infinity.

For a Gaussian or variance-standardised Student innovation law $F$, write $p_c$ for the density of $c|\varepsilon|$, $c>0$, and $M(c)=\int_0^\infty g(z)p_c(z)\,dz$. If $c_2>c_1$, $R=p_{c_2}/p_{c_1}$ is strictly increasing on $(0,\infty)$. For a Gaussian innovation this follows from the exponential density ratio. For a Student innovation of degree $v$, the ratio is proportional to
\[
 \left(\frac{1+z^2/[(v-2)c_1^2]}{1+z^2/[(v-2)c_2^2]}\right)^{(v+1)/2}.
\]
Therefore $M(c_2)-R(z_0)M(c_1)=\int g(z)[R(z)-R(z_0)]p_{c_1}(z)\,dz>0$. Bounded convergence gives $M(0+)=g(0)<0$ and $M(+\infty)=\beta C>0$. Continuity and the strict inequality give a unique zero $c_\star$. Since the risk derivative has the opposite sign to the mean driver at candidate log scale $\delta$, $\delta_{\mathsf S}=-2\log c_\star$ uniquely minimises risk.

The log-scale score $\partial_{\log c}\log p_c(z)$ is strictly increasing in $z>0$ for both true families. Differentiation is valid locally under a Gaussian polynomial envelope, or under an integrable Student density envelope times its bounded scale score. At $c_\star$, subtract the value of that score at $z_0$ under the integral. The sign change of $g$ gives $\partial_{\log c}M(c_\star)>0$. Hence $b_{\mathsf S}=\tfrac12\partial_{\log c}M(c_\star)>0$, and $M(c_\star e^{-e/2})$ has the restoring sign for every $e\ne0$, with nonzero limits at both ends. This is a population assertion for the stated scales and laws.

\begin{lemma}[The Gaussian log score]\label{trk:lem:glog}
Let $g(z)=z^2-1$, $\E_F[\varepsilon^2]=1$, and suppose $\E_F|\varepsilon|^{2q}<\infty$ and $\E|\eta|^q<\infty$ for some $q>2$. Retain the scale, innovation and initialisation conditions of Assumption~\ref{trk:ass:prim}, replacing its bounded-driver clause by these local moment conditions. Then the conclusions of Lemmas~\ref{trk:lem:local} and~\ref{trk:lem:exit} hold with:
\begin{itemize}[leftmargin=1.5em]
\item mean field $\bar g(e)=e^{-e}-1$, frozen drift $b_{\mathsf S}=1$, $\mu=e^{-r}$ and local Lipschitz constant $e^r$ on $[-r,r]$;
\item $\Xi_{\mathsf S}=\mathrm{kurt}_F-1$;
\item constants depending on $r$ and $q$.
\end{itemize}
{In particular, under the remaining initial-condition and target-shock assumptions, the local weak conclusions of Theorem~\ref{th::local-ou} and Corollary~\ref{cor:unstopped_local_ou} apply under $t_\nu$ innovations for every $\nu>4$. They make no invariant-MSE claim: Proposition~\ref{rev:prop:stationary-tail} rules out finite invariant second moments for the pure integrator when $4<\nu\le6$.}
\end{lemma}

\begin{proof}
Here $w=1$ and $g(we^{-e/2}\varepsilon)=e^{-e}\varepsilon^2-1$, so $\bar g(e)=e^{-e}-1$, $-\bar g'(e)=e^{-e}\in[e^{-r},e^r]$ and $|\bar g''|\le e^r$ on $[-r,r]$. The proof of Lemma~\ref{trk:lem:local} goes through with these local bounds in place of $M_1$ and $M_2$:
\begin{itemize}[leftmargin=1.5em]
\item the linearisation remainder is at most $\rho e^r(M+1)K^2h^{1/4}/2$;
\item $u_{\mathsf S}(y,m+x)-u_{\mathsf S}(y,m)=\varepsilon^2(e^{-x}-1)$, whose centred conditional second moment is at most $(\mathrm{kurt}_F-1)e^{2r}x^2$;
\item $\tilde\xi=\rho(\varepsilon^2-1)-\sigma\eta$ has a finite moment of order $q>2$, which gives the Lindeberg condition.
\end{itemize}
In the proof of Lemma~\ref{trk:lem:exit}, $|g-\bar g|\le2M_0$ is replaced by $\E_{kh}|g-\bar g|^q=e^{-qe_k}\E|\varepsilon^2-1|^q\le e^{qr}\E|\varepsilon^2-1|^q$ on $\{|e_k|\le r\}$. The Gronwall step only uses $|\bar g(e)|\le e^r|e|$ before $\tau$, where $|e|\le r$. For $t_\nu$, $\E|\varepsilon|^{2q}<\infty$ for all $q<\nu/2$, so some $q>2$ exists if and only if $\nu>4$. For fixed tube radius, Lemma~\ref{trk:lem:exit}(ii) gives $O(\E|e^{(h)}_0|^q+h^{(q-2)/4})$ for $2<q<\nu/2$. The polynomial rate $O(h^{(q-2)/4})$ follows if the initial moment is of that order or smaller. An initial moment tending to zero without a rate gives only $o(1)$ from this bound.
\end{proof}
\endgroup

\section{Scoring rules, drivers, and tuning constants}\label{app::formulas}
This appendix collects the explicit proper scores, raw derivatives, predictable scalings, implemented updates, and reference-state catalogue used in Section~\ref{sec::experiments} and \ref{sec::empirical}. Throughout, $v=e^{\lambda}$ is the predictive variance, $y=\sqrt{v}\,z$, and $z=e^{-\lambda/2}y$ is variance-standardised. Scores are negatively oriented. The common identity is
\begin{equation}\label{eq::driver_convention}
  \psi_{\mathsf S}(y,\lambda)=-\partial_{\lambda}\mathsf{S}(P_{\lambda,\theta},y),
  \qquad
  u_{\mathsf S}(y,\lambda)=G_{\mathsf S}(\lambda)\psi_{\mathsf S}(y,\lambda),
\end{equation}
where the derivative holds $y$ fixed; only afterwards may one substitute $y=\sqrt v z$. The finite-step scalings are rule-specific:
\[
G_{\log}=2,\qquad
G_{\beta}(\lambda)=2e^{\beta\lambda/2},\qquad
G_{\rm CRPS}(\lambda)=2e^{-\lambda/2},\qquad
G_{\rm MMD}=2.
\]
The first three choices cancel scale homogeneity and produce residual-only updates. Fixed-data-bandwidth MMD is non-homogeneous and retains current-state dependence. Curvature scaling $G_{\mathsf S}=J_{\mathsf S}^{-1}$ is a separate oracle or working-model normalisation used only where explicitly stated. {Remark~\ref{rem:dynamic_equivalence}} preserves a path only when a positive rescaling of the score derivative is accompanied by the reciprocal rescaling of $G$; an arbitrary change in $G$ changes the path. Raw $J$ and $K$ are not comparable across arbitrary score normalisations, whereas $K_{\mathsf S}/J_{\mathsf S}^2$ is scale invariant.

\medskip\noindent\emph{Tuning constants.} The unit-variance catalogue is evaluated at the fixed shape parameters $\nu=6$ (Student degrees of freedom), $\beta=0.15$ (density-power exponent), and $\ell=2$ (MMD-RBF bandwidth), selected for a clear tail ordering at modest efficiency cost under correct specification. In the empirical application, the MMD bandwidth is selected once from each market's training sample in data units and then frozen across candidate states and out of sample (\ref{app::panel_fit_diagnostics}). The derived catalogue constants entering the closed-form drivers are
\begin{equation}\label{eq::tuning_constants}
\begin{aligned}
  c_{\beta}&=(2\pi)^{-\beta/2},
  &\kappa_{\beta}&=\beta\,c_{\beta}(1+\beta)^{-1/2},
  &D_0&=\ell^{2}+1,\\
  a_0(z)&=\ell\,D_0^{-1/2}e^{-z^{2}/(2D_0)},
  &b_{\ell}&=2\ell\,(\ell^{2}+2)^{-3/2}.
\end{aligned}
\end{equation}
At these values $c_{\beta}=0.871$, $\kappa_{\beta}=0.122$, and $b_{\ell}=0.272$.
Table~\ref{tab::catalogue} summarises the resulting reference-state curvature, innovation
variance, curvature-normalised scale, and tail classification.

\begin{table}[H]
\centering\footnotesize
\resizebox{\textwidth}{!}{\begin{tabular}{@{}llrrrll@{}}
\toprule
Rule & $g_{\mathsf S}(z)$ & $J_{\mathsf S}$ & $K_{\mathsf S}$ & $\sqrt{K_{\mathsf S}}/|J_{\mathsf S}|$ & tail & recovers \\
\midrule
Gaussian log & $z^2-1$ & 0.500 & 0.500 & 1.414 & unbounded & GAS / GARCH-type log-vol \\
density-power & $(1+\beta)c_\beta e^{-\beta z^2/2}(z^2-1)+\kappa_\beta$ & 0.357 & 0.268 & 1.451 & tail-saturating & density-power / $\beta$-divergence \\
CRPS & $\pi^{-1/2}-\sqrt{2/\pi}\,e^{-z^2/2}$ & 0.071 & 0.012 & 1.573 & tail-saturating & CRPS filter \\
MMD-RBF & $2a(z)\!\left(\tfrac{z^2}{D^2}-\tfrac1D\right)+b_\ell$ & 0.034 & 0.002 & 1.456 & tail-saturating & energy / MMD filter \\
\bottomrule
\end{tabular}
 }
\caption{The scoring-rule catalogue on the fixed Gaussian working density. $J_{\mathsf S}$ is the
curvature of the scoring risk (the sensitivity governing local mean reversion in
Proposition~\ref{prop::local_mean}); $K_{\mathsf S}$ is the local innovation variance. The column
$g_{\mathsf S}(z)$ is the scaled recursion driver of Definition~\ref{def::scoring-rule-filter},
whereas $J_{\mathsf S}$ and $K_{\mathsf S}$ are computed from the raw innovation
$\psi_{\mathsf S}=-\partial_\lambda\mathsf S$ of~\eqref{eq::distinction} (for the log score
$g_{\mathsf S}(z)=z^2-1$ while $\psi_{\mathsf S}=\tfrac12(z^2-1)$), so $K_{\mathsf S}$ is the variance
of $\psi_{\mathsf S}$, not of $g_{\mathsf S}$. The log
score satisfies $J_{\mathsf S}=K_{\mathsf S}$ (the second Bartlett identity); the bounded,
tail-saturating criteria do not. The curvature-normalised scale $\sqrt{K_{\mathsf S}}/|J_{\mathsf S}|$
is comparable across rules ($1.41$--$1.57$): after curvature scaling the local mean dynamics
align, but $K_{\mathsf S}/J_{\mathsf S}^{2}$ and the driver shapes can still differ throughout the
residual support. The scores, drivers, and the constants
$c_\beta,\kappa_\beta,a(z),D,b_\ell$ appearing in the driver column are defined in this appendix.}
\label{tab::catalogue}
\end{table}

\medskip\noindent\emph{Scores and updates} (Gaussian working density). For each criterion we give the per-observation proper score and the update $u_{\mathsf S}=G_{\mathsf S}\psi_{\mathsf S}$, together with its tail limit.
\begin{itemize}
\item \emph{Gaussian log-score:} $\mathsf{S}=\tfrac12\log(2\pi)+\tfrac12\lambda+\tfrac12 e^{-\lambda}y^{2}$, so $g_{\mathsf S}(z)=z^{2}-1$; \emph{unbounded}.
\item \emph{Student-$t$ log-score} ($\nu$): $\mathsf{S}=\log[\sqrt{\nu-2}\,\mathrm B(1/2,\nu/2)]+\tfrac12\lambda+\tfrac12(\nu+1)\log\!\big(1+\tfrac{z^{2}}{\nu-2}\big)$, so $g_{\mathsf S}(z)=\dfrac{(\nu+1)z^{2}}{(\nu-2)+z^{2}}-1$; tail limit $\nu$ (tail-saturating).
Here $\mathrm B$ is the beta function. These density constants cancel in within-family fitting at fixed shape. They enter predictive negative log scores when comparing different densities or shapes.
\item \emph{Density-power} ($\beta$): the score and residual update are
\[
\begin{aligned}
\mathsf S
&=c_{\beta}e^{-\beta\lambda/2}
  \left[(1+\beta)^{-1/2}-(1+\tfrac1\beta)e^{-\beta z^{2}/2}\right],\\
g_{\mathsf S}(z)
&=(1+\beta)c_{\beta}e^{-\beta z^{2}/2}(z^{2}-1)+\kappa_{\beta}.
\end{aligned}
\]
Its tail limit is $\kappa_{\beta}$ (tail-saturating).
\item \emph{CRPS:} $\mathsf{S}=e^{\lambda/2}\big[z(2\Phi_{\mathcal N}(z)-1)+2\phi_{\mathcal N}(z)-\pi^{-1/2}\big]$, so $g_{\mathsf S}(z)=\pi^{-1/2}-\sqrt{2/\pi}\,e^{-z^{2}/2}$; tail limit $\pi^{-1/2}$ (tail-saturating).
\item \emph{MMD-RBF} ($\ell$): with a bandwidth $\ell>0$ fixed in data units across all candidate states,
\[
\mathsf{S}=1-\frac{2\ell}{\sqrt{\ell^{2}+v}}e^{-y^{2}/\{2(\ell^{2}+v)\}}+\frac{\ell}{\sqrt{\ell^{2}+2v}}.
\]
Put $r_{\lambda}=\ell e^{-\lambda/2}$ and $D_{\lambda}=r_{\lambda}^{2}+1$. The current-state update is
\[
u_{\mathsf S}(\lambda,z)
=2\frac{r_{\lambda}}{\sqrt{D_{\lambda}}}e^{-z^{2}/(2D_{\lambda})}
\left(\frac{z^{2}}{D_{\lambda}^{2}}-\frac1{D_{\lambda}}\right)
+\frac{2r_{\lambda}}{(r_{\lambda}^{2}+2)^{3/2}},
\]
with tail limit $2r_{\lambda}(r_{\lambda}^{2}+2)^{-3/2}$ (tail-saturating). At the unit-variance reference state $\lambda=0$, this reduces to the former $a_0,D_0,b_{\ell}$ display. No scalar state factor converts fixed-bandwidth MMD into a universal residual-only update.
\end{itemize}
At a fixed mesh, only the Gaussian log-score update is unbounded; the other four converge to the finite, non-zero limits above and are tail-saturating rather than redescending (Section~\ref{section::introduction}).

\medskip\noindent\emph{High-frequency MMD scope.} For Gaussian candidate variance $he^a$, true variance $he^v$, and bandwidth $\ell$ fixed in data units, the population risk is
\[
 \mathsf R_{{\rm MMD},h}(a;v)=1+\ell(\ell^2+2he^a)^{-1/2}
 -2\ell\{\ell^2+h(e^a+e^v)\}^{-1/2}.
\]
At the target $a=v$, differentiation with respect to log variance gives
$\psi_{{\rm MMD},h}/J_{{\rm MMD},h}\to z^2-1$. If $a$ and $v$ instead denote
variance multipliers in $ha$ and $hv$, differentiation with respect to $a$
gives $v(z^2-1)$. Thus fixed bandwidth has an unbounded pointwise
curvature-normalised limit. A mesh bandwidth $\ell_h=\sqrt h\ell_0$ defines a
separate design whose uniform theory and implementation conditions fall
outside the present high-frequency result.

\medskip\noindent\emph{CRPS on the Student-$t$ working density.} Let $f_{\nu}$ and $F_{\nu}$ denote the density and distribution function of the variance-standardised variable $Z=\sqrt{(\nu-2)/\nu}\,T_{\nu}$, so that $\mathsf{Var}(Z)=1$, and put $z=y/\sqrt v$ and $d_{\nu}=\nu-2$. The CRPS formula of \citet{jordan2019evaluating} becomes
\begin{equation}\label{eq::crps_student}
  \mathrm{CRPS}_{t_\nu}(v,y)=\sqrt v\left\{z\big(2F_{\nu}(z)-1\big)+2f_{\nu}(z)\frac{d_{\nu}+z^{2}}{\nu-1}-c_{\nu}\right\},
  \qquad
  c_{\nu}=\frac{2\sqrt{\nu-2}}{\nu-1}\frac{B(\tfrac12,\nu-\tfrac12)}{B(\tfrac12,\tfrac{\nu}{2})^{2}},
\end{equation}
with $B(\cdot,\cdot)$ the beta function. The state $v=e^{\lambda}$ is predictive variance for both working densities; there is no additional factor $\nu/(\nu-2)$ in prediction or evaluation.

\begin{table}[!htbp]
\centering\footnotesize
\begin{tabular}{@{}lrrrl@{}}
\toprule
Rule & ${g^{\circ}_{\mathsf S}}(B)$ & initial $\alpha|{g^{\circ}_{\mathsf S}}(B)-{g^{\circ}_{\mathsf S}}(0)|$ & cumulative bound & tail \\
\midrule
Gaussian log & 63.000 & 3.200 & $\infty$ & unbounded \\
density-power & 0.898 & 0.106 & 21.842 & tail-saturating \\
CRPS & 4.000 & 0.283 & 13.333 & tail-saturating \\
MMD-RBF & 4.103 & 0.268 & 21.155 & tail-saturating \\
\midrule
Student log$^{\dagger}$ & 7.472 & 0.440 & 26.741 & tail-saturating \\
\bottomrule
\end{tabular}
 \caption{Impulse-response calibration for the additive recursion ($\alpha=0.05$, $\varphi=0.97$, outlier $B=8$), using the curvature-normalised driver ${g^{\circ}_{\mathsf S}}=\psi_{\mathsf S}/J_{\mathsf S}$. The initial response is $\alpha|{g^{\circ}_{\mathsf S}}(B)-{g^{\circ}_{\mathsf S}}(0)|$; the cumulative tail-envelope bound is $2\alpha M/(1-|\varphi|)$, where $M=\sup_z|{g^{\circ}_{\mathsf S}}(z)|$, and is therefore infinite for the Gaussian log rule. $^{\dagger}$The Student-$t$ log row changes the density and is shown only as a benchmark.}
\label{tab::impulse}
\end{table}

\medskip\noindent\emph{Sensitivity to the tuning constants.}
Table~\ref{tab::sensitivity} reports the tuning grid. At the
{Gaussian reference residual, it reports curvature-normalised update second moments, gross-outlier response and contaminated pseudo-true variance. For off-target Student-log rows the second moment is not a variance or an asymptotic efficiency.} The Student-log sweep changes the density and remains a
separate benchmark; the other rows retain the Gaussian working density.

\begin{table}[ht]
\centering\small
\begin{tabular}{@{}llrrr@{}}
\toprule
Rule & Constant & $K_{\mathsf S}/J_{\mathsf S}^2$ & $|{g^{\circ}_{\mathsf S}}(10)|$ & $v^\star$ \\
\midrule
Gaussian log (ref.) & --- & 2.00 & 99.00 & 4.95 \\
\midrule
\multirow{5}{*}{Student log} & $\nu=4$ & 2.59 & 5.70 & 1.78 \\
 & $\nu=5$ & 2.37 & 6.69 & 1.64 \\
 & $\nu=6$ & 2.27 & 7.66 & 1.59 \\
 & $\nu=8$ & 2.16 & 9.52 & 1.60 \\
 & $\nu=10$ & 2.11 & 11.27 & 1.65 \\
\midrule
\multirow{5}{*}{density-power} & $\beta=0.05$ & 2.01 & 9.22 & 1.70 \\
 & $\beta=0.1$ & 2.05 & 0.95 & 1.23 \\
 & $\beta=0.15$ & 2.10 & 0.25 & 1.13 \\
 & $\beta=0.2$ & 2.17 & 0.24 & 1.09 \\
 & $\beta=0.3$ & 2.34 & 0.37 & 1.07 \\
\midrule
\multirow{5}{*}{MMD-RBF} & $\ell=1$ & 2.56 & 2.00 & 1.11 \\
 & $\ell=1.5$ & 2.25 & 2.83 & 1.15 \\
 & $\ell=2$ & 2.12 & 4.00 & 1.21 \\
 & $\ell=3$ & 2.03 & 7.85 & 1.38 \\
 & $\ell=4$ & 2.01 & 15.37 & 1.59 \\
\midrule
CRPS & --- & 2.48 & 4.00 & 1.20 \\
\bottomrule
\end{tabular}
 \par
\caption{{Sensitivity at the Gaussian reference state. The $K_{\mathsf S}/J_{\mathsf S}^2$ column reports the ratio of the raw derivative's uncentred second moment to squared local curvature. For the off-target Student-$\nu$ rows this is a reference-state diagnostic, not an innovation variance or own-target efficiency. For example, at $\nu=6$, $\mathsf E g=0.100830790$ and $b=0.747923025$, so $\mathsf E g^2/b^2=2.267997961$, whereas $\operatorname{Var}(g)/b^2=2.249823040$; at the actual Gaussian-data target, $\delta=0.139128390$, the own-target $\kappa$ is $2.214859913$. The response column is $|g_{\mathsf S}^{\circ}(10)|=|g_{\mathsf S}(10)/(2J_{\mathsf S})|$. The variance-contamination target uses $\epsilon=0.05$, $\tau=80$ and $v^\star_{\log}=4.95$. Student rows change the density; other sweeps retain the Gaussian family.}}
\label{tab::sensitivity}
\end{table}

For the Student-log rows we independently evaluate the population estimating equation. In units of $v_0$, write $v$ for the candidate variance and $t$ for the variance of a Gaussian component. With $b=\sqrt{(\nu-2)v/t}$ and $\operatorname{erfcx}(z)=e^{z^2}\operatorname{erfc}(z)$ its mean update is
\[
 H_\nu(v,t)=\nu-(\nu+1)\sqrt{\pi/2}\,b\,\operatorname{erfcx}(b/\sqrt2).
\]
Indeed $H_\nu=\E[(\nu+1)Z^2/(b^2+Z^2)-1]$ for $Z\sim N(0,1)$. Solving $0.95H_\nu(v,1)+0.05H_\nu(v,80)=0$ gives $1.7843634$ at $\nu=4$ and $1.5909031$ at $\nu=6$. The former rounds to $1.78$. These are population projections in the stated static mixture design.

\subsection{\texorpdfstring{{Driver shapes}}{Driver shapes}}\label{sec::exp_geometry}
Figure~\ref{fig::driver_geometry} compares the curvature-normalised drivers $z\mapsto{g^{\circ}_{\mathsf S}}(z)=\psi_{\mathsf S}(z)/J_{\mathsf S}$. The Gaussian log driver grows quadratically; the other three are bounded and tail-saturating. CRPS approaches its plateau monotonically, whereas density-power and MMD decline from an interior peak. None of these four scale drivers is redescending. These shapes illustrate Propositions~{\ref{prop::local_mean} and~\ref{prop::outlier}}.

\subsection{Pseudo-true target under contamination}\label{sec::exp_target}
The baseline variance and the variance of the full contaminated law are different targets. Here the baseline is $v_0$, whereas the mixture variance is $v_0\{1+\epsilon(\tau-1)\}$. A projection nearer $v_0$ is therefore nearer the clean-component variance; this alone does not establish improved accuracy for the full observation variance or under a common predictive loss.

Let $Q_{\epsilon,\tau}=(1-\epsilon)\mathcal N(0,v_0)+\epsilon\mathcal N(0,\tau v_0)$ with $\epsilon=0.05$. The Gaussian-log projection equals the mixture variance, $v^\star_{\log}=v_0\{1+\epsilon(\tau-1)\}$. The Gaussian CRPS projection has a different large-contaminant limit. Its centring equation for $w=v^\star_{\rm CRPS}/v_0$ is
\[
 (1-\epsilon)\sqrt{w/(w+1)}+\epsilon\sqrt{w/(w+\tau)}=1/\sqrt2.
\]
This follows by taking the expectation of the Gaussian CRPS driver in each mixture component. For $\tau\ge1$ the left side increases strictly in $w$ and decreases in $\tau$ at fixed $w$. When $\epsilon<1-1/\sqrt2$, the root increases to $\{2(1-\epsilon)^2-1\}^{-1}$ as $\tau\to\infty$, or $1.2422$ at $\epsilon=0.05$. Figure~\ref{fig::contamination_target} also shows the numerical density-power and MMD targets over the plotted range.

\begingroup
A different target is not by itself better or worse. In the scale-equivariant examples, the projection is the latent log variance plus a constant $\delta_{\mathsf S}(F)$ (Corollary~\ref{trk:cor:scale}). That shift determines the limiting variance level, whereas the frozen tracking variance measures precision about the shifted target. The two quantities answer different questions. Section~\ref{trk:sec:numerics} compares the latter under zero update--target shock covariance. Accuracy for the latent state also includes squared target bias; Proposition~\ref{trk:prop:biasprec} gives the corresponding expansion in its bounded-driver, driftless reference model. {The Gaussian-log pure-integrator case has the separate stationary-moment obstruction of Proposition~\ref{rev:prop:stationary-tail}.}
\endgroup

{The impulse-response experiment in \ref{sec::exp_impulse} calibrates Proposition~\ref{prop::outlier}(i).}

\subsection{\texorpdfstring{{Supplementary numerical results}}{Supplementary numerical results}}\label{app:moved_numerics}

\begin{figure}[H]
\centering
\includegraphics[width=0.85\linewidth]{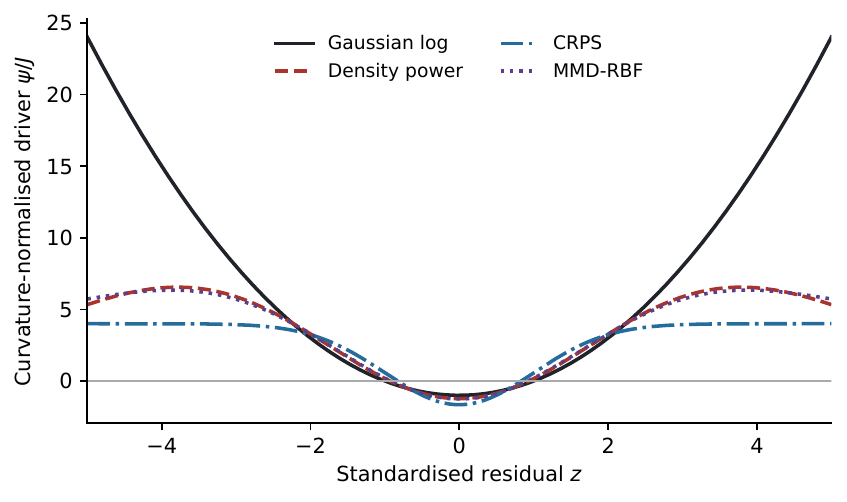}
\caption{{Curvature-normalised Gaussian-reference scale drivers at variance $v_0=1$, with density-power exponent $\beta=0.15$ and RBF bandwidth $\ell=2$: the signed response $g^{\circ}_{\mathsf S}(z)=\psi_{\mathsf S}(z)/J_{\mathsf S}$ is shown for each rule. These deterministic curves use the formulas in \appref{app::formulas}. Gaussian log grows quadratically; CRPS, density power and fixed-bandwidth MMD are bounded with nonzero tail limits at this state. Reference-curvature normalisation gives unit local mean-reversion loading at the Gaussian reference state.}}
\label{fig::driver_geometry}
\end{figure}

\begin{figure}[H]
\centering
\includegraphics[width=0.8\linewidth]{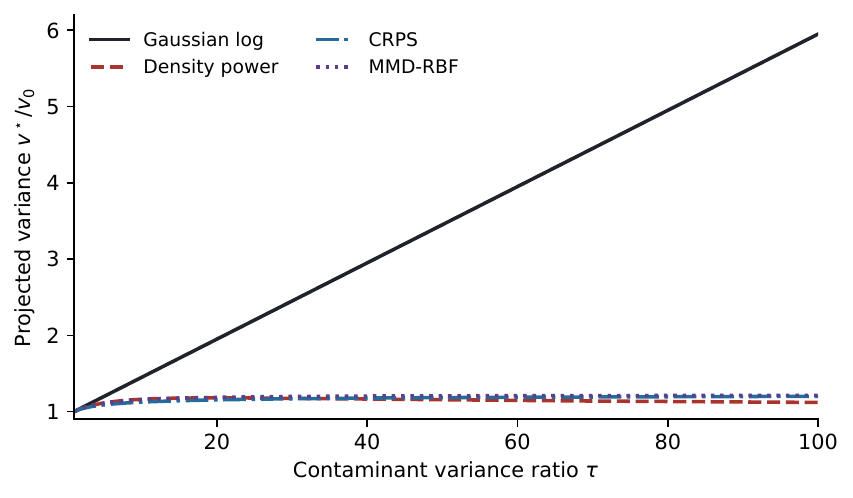}
\caption{{Static population variance projections in a $5\%$ Gaussian variance-contamination design, with baseline variance $v_0$ and contaminating variance $\tau v_0$. The Gaussian-log projection is $v^\star/v_0=1+0.05(\tau-1)$. The other curves are score-specific roots computed from closed-form Gaussian-mixture expectations, with absolute root tolerance $10^{-12}$ and relative tolerance $10^{-13}$. They concern population score-risk targets, not invariant filter distributions. The numerical root calculation alone does not establish global uniqueness.}}
\label{fig::contamination_target}
\end{figure}

\FloatBarrier

\subsection{Outlier impulse response}\label{sec::exp_impulse}
The third diagnostic uses Proposition~\ref{prop::outlier}{(i)}. In the additive recursion $\lambda_{t+1}=\omega+\varphi\lambda_t+\alpha g_{\mathsf S}(Z_t)$ with
$|\varphi|<1$, replacing $Z_t=0$ by $Z_t=B$ shifts the path by
$\lambda^{(B)}_{t+1+k}-\lambda^{(0)}_{t+1+k}=\varphi^{k}\alpha\{g_{\mathsf S}(B)-g_{\mathsf S}(0)\}$.
The Gaussian log response grows quadratically in $B$; bounded rules cap it but, being tail-saturating, converge to non-zero per-step plateaus. Under curvature scaling, at $B=8$, $\alpha=0.05$, and $\varphi=0.97$, the log rule has initial response $3.200$ and no finite cumulative tail-envelope bound, whereas the three bounded Gaussian-density rules have initial responses $0.106$--$0.283$ and finite bounds $13.333$--$21.842$. Table~\ref{tab::impulse} in \appref{app::formulas} records the full comparison, including the Student-$t$ density benchmark.

\begingroup
\subsection{Analytic driver bounds}\label{app:trk_inv_diagnostics}
Table~\ref{inv:tab:const} gives the global derivative bounds used in Proposition~\ref{inv:prop:mono}. Application to a parameter-estimation problem requires the contraction inequality uniformly over its fixed admissible parameter set, together with the theorem's population assumptions.
\endgroup

\begin{table}[ht]
\begingroup
\centering\small
\begin{tabular}{lcc}
\toprule
Driver & global $M_+$ & value \\
\midrule
Gaussian CRPS & $\sqrt{2/\pi}/e$ & 0.294 \\
Student-$t_6$ log & $7/4$ & 1.750 \\
Student-$t_6$ CRPS & $\max_z z^2q_6(z)$, attained at $z^2=8/5$ & 0.231 \\
\bottomrule
\end{tabular}
\caption{{Analytic scale-derivative maxima for the bounded monotone cells under the scalings of {\appref{app::formulas}}. Here $m(z)=zg'(z)/2$ and $0\le m(z)\le M_+$. Numerical values are rounded evaluations of the displayed global formulas. Density-power grid extrema are not used as certified global constants.}}\label{inv:tab:const}
\endgroup
\end{table}

\FloatBarrier

\section{Vector local tracking extension}\label{app::multivariate}

\subsection{Scope and vector recursion}

This appendix records a vector version of the local moving-target approximation.  It is stated
as a stopped result because local smoothness does not itself prove that the normalised tracking
error remains in the region where the expansion is valid.  An unstopped result is obtained only
on each finite window by Corollary~\ref{cor:vector-unstopped}.  No global
multivariate stability claim is made.

Let $\Lambda\subset\mathbb R^d$ be open. As in the scalar tracking section,
let the observation filtration $\mathcal F_{kh}^{\mathrm{obs},(h)}$ be generated
by the known initial filter state, any observed presample information, and the
observations through time $kh$, and
let $\mathcal F_{kh}^{\mathrm{obs},(h)}\subseteq\mathcal G_{kh}^{(h)}$.
All state, target, and observation arrays below are adapted to the analysis
filtration $\mathcal G^{(h)}$. For a differentiable proper scoring rule
$\mathsf S$, write
\[
 \psi_h(y,\lambda)=-\nabla_\lambda\mathsf S(P^{(h)}_\lambda,y),
 \qquad
 u_h(y,\lambda)=G_h(\lambda)\psi_h(y,\lambda),
\]
where $\psi_h,u_h\in\mathbb R^d$ and $G_h(\lambda)$ is a predictable $d\times d$ scaling matrix.
The predictable time index on $G_h$ and hence on $u_h$ is suppressed throughout; at step $k$
the scaling is jointly measurable with respect to
$\mathcal G_{kh}^{(h)}\otimes\mathcal B(\Lambda)$ before
$Y_{(k+1)h}$ is observed. If the initial state is known and the scaling is
jointly measurable with respect to
$\mathcal F_{kh}^{\mathrm{obs},(h)}\otimes\mathcal B(\Lambda)$, the recursion uses only observed
information. A scaling known only through $\mathcal G_{kh}^{(h)}$ instead
defines an oracle recursion. Let
\[
 R_{h,k}(\lambda)
 =\mathsf E[\mathsf S(P^{(h)}_\lambda,Y_{(k+1)h})\mid\mathcal G_{kh}^{(h)}],
\]
and let $m^{(h)}_{kh}$ be a $\mathcal G_{kh}^{(h)}$-measurable unique interior minimiser of $R_{h,k}$. This analysis target need not be observed by the recursion. Assume
differentiation may be interchanged with conditional expectation locally at this target and that
$G_h(m^{(h)}_{kh})$ is finite. Put
\[
 \bar u_{h,k}(\lambda)
 =\mathsf E[u_h(Y_{(k+1)h},\lambda)\mid\mathcal G_{kh}^{(h)}].
\]
We use a jointly measurable conditional-mean version of $\bar u_{h,k}$ whose defining identity also holds at every $\mathcal G_{kh}^{(h)}$-measurable local state.
The first-order condition then gives $\bar u_{h,k}(m^{(h)}_{kh})=0$. The filter and this
predictable pseudo-true target satisfy
\begin{align}
 \lambda^{(h)}_{(k+1)h}
 &=\lambda^{(h)}_{kh}+\rho\sqrt h\,
      u_h(Y_{(k+1)h},\lambda^{(h)}_{kh}), \label{eq:vector-filter}\\
 m^{(h)}_{(k+1)h}-m^{(h)}_{kh}
 &=\sqrt h\,\Gamma_{kh}\eta_{(k+1)h}+r^m_{h,k+1}, \label{eq:target-motion}
\end{align}
where $\rho>0$, $\eta_{(k+1)h}\in\mathbb R^q$, $\Gamma_{kh}\in\mathbb R^{d\times q}$ is
predictable, and $\mathsf E[\eta_{(k+1)h}\mid\mathcal G_{kh}^{(h)}]=0$. Define
\[
 B_{h,k}=-D_\lambda\bar u_{h,k}(m^{(h)}_{kh}),
 \qquad
 \xi^0_{h,k+1}
 =\rho u_h(Y_{(k+1)h},m^{(h)}_{kh})-\Gamma_{kh}\eta_{(k+1)h}.
\]
Fix $T>0$, set $k_T=\lfloor T/h\rfloor$, and work on the local time interval
$0\le\tau\le M$.  With $k_j=k_T+j$, define
\[
 Z_h(\tau)=h^{-1/4}
 \{\lambda^{(h)}_{(k_T+\lfloor\tau/\sqrt h\rfloor)h}
      -m^{(h)}_{(k_T+\lfloor\tau/\sqrt h\rfloor)h}\}.
\]
For $K>0$, let
\[
 \sigma_{h,K}=\inf\{\tau\in[0,M]:\|Z_h(\tau)\|\ge K\}\wedge M,
 \qquad
 Z_h^K(\tau)=Z_h(\tau\wedge\sigma_{h,K}).
\]
For exact finite-grid indexing, write
\(
 n_h(\tau)=\left\lfloor\frac{\tau}{\sqrt h}\right\rfloor
\)
, and
\(
 n_{h,K}(\tau)=\left\lfloor\frac{\tau\wedge\sigma_{h,K}}{\sqrt h}\right\rfloor
\).

\subsection{Stopped local approximation}

\begin{assumption}[Vector local array]\label{ass:vector-local-array}
For each fixed $M,K<\infty$, the following conditions hold as $h\downarrow0$.
\begin{enumerate}[leftmargin=2em,label=(\roman*)]
 \item The initial normalised error $Z_h(0)$ is $\mathcal G_{k_Th}^{(h)}$-measurable and
 $Z_h(0)\Rightarrow Z_0$. The recursion \eqref{eq:vector-filter} is well-defined and remains in
 $\Lambda$ through the local window $k_T\le k\le k_T+n_h(M)$. Moreover, for every fixed $M,K$,
 the entire local tube is contained in the parameter domain almost surely for all sufficiently
 small $h$:
 \[
  m^{(h)}_{k_jh}+h^{1/4}\{z\in\mathbb R^d:\|z\|\le K\}\subset\Lambda,
  \qquad 0\le j\le n_h(M).
 \]
 This condition is automatic when $\Lambda=\mathbb R^d$.

 \item There is a deterministic matrix $B\in\mathbb R^{d\times d}$ whose eigenvalues have
 strictly positive real parts such that
 \[
  \max_{j\le M/\sqrt h}\|B_{h,k_j}-B\|_{\rm op}\xrightarrow{p}0.
 \]

 \item The stopped mean field expansion is uniform in cumulative local time:
 \[
 \sup_{\tau\le M}\left\|
  \sum_{j<n_{h,K}(\tau)}
  h^{1/4}\bigl[
   \bar u_{h,k_j}(m^{(h)}_{k_jh}+h^{1/4}Z_h(j\sqrt h))
   +h^{1/4}B_{h,k_j}Z_h(j\sqrt h)
  \bigr]
 \right\|\xrightarrow{p}0.
 \]

 \item Put $z_{h,j}=Z_h(j\sqrt h)$ and define the centred evaluation difference
 \begin{equation*}
  \Delta_{h,j}
  =u_h(Y_{(k_j+1)h},m^{(h)}_{k_jh}+h^{1/4}z_{h,j})
     -\bar u_{h,k_j}(m^{(h)}_{k_jh}+h^{1/4}z_{h,j})
     -u_h(Y_{(k_j+1)h},m^{(h)}_{k_jh}).
 \end{equation*}
 Evaluation away from the target contributes no additional martingale at the local scale:
 \[
  \sup_{\tau\le M}\left\|
   \sum_{j<n_{h,K}(\tau)}h^{1/4}\Delta_{h,j}
  \right\|\xrightarrow{p}0.
 \]

 \item The target-motion remainder is negligible:
 \[
  h^{-1/4}\max_{n\le M/\sqrt h}
  \left\|\sum_{j<n}r^m_{h,k_j+1}\right\|\xrightarrow{p}0.
 \]

 \item The composite array is square-integrable and conditionally centred, and for every
 $\varepsilon>0$,
 \begin{align}
  \sup_{\tau\le M}\left\|
   \sum_{j<n_h(\tau)}\sqrt h\,
   \mathsf E[\xi^0_{h,k_j+1}(\xi^0_{h,k_j+1})^\top\mid\mathcal G_{k_jh}^{(h)}]
   -\tau Q
  \right\|_{\rm op}&\xrightarrow{p}0, \label{eq:vector-qv}\\
  \sum_{j<n_h(M)}\sqrt h\,
  \mathsf E\!\left[
   \|\xi^0_{h,k_j+1}\|^2
   \mathbf 1_{\{h^{1/4}\|\xi^0_{h,k_j+1}\|>\varepsilon\}}
   \mid\mathcal G_{k_jh}^{(h)}
  \right]&\xrightarrow{p}0, \label{eq:vector-lindeberg}
 \end{align}
 for a deterministic symmetric nonnegative-definite matrix $Q\in\mathbb R^{d\times d}$.

\end{enumerate}
\end{assumption}

Define $\mathcal M_h(\tau)=\sum_{j<n_h(\tau)}h^{1/4}\xi^0_{h,k_j+1}$.
The multidimensional martingale FCLT and the joint-tightness argument in the proof of
Corollary~\ref{cor:primitive_joint_fclt}, applied to linear projections, show that items~(i) and~(vi)
imply
\[
 \left(Z_h(0),\mathcal M_h\right)
 \Longrightarrow
 (Z_0,Q^{1/2}W)
 \quad\text{in }\mathbb R^d\times D([0,M],\mathbb R^d),
\]
where $W$ may be chosen as a standard $d$-dimensional Brownian motion independent of $Z_0$.

\begin{theorem}[Stopped vector local Ornstein--Uhlenbeck limit]
\label{thm:vector-stopped-ou}
Under Assumption~\ref{ass:vector-local-array}, for every $K$ at which the exit-time map of the
limiting continuous process is almost surely continuous,
\[
 Z_h^K\Longrightarrow Z(\,\cdot\wedge\sigma_K)
 \quad\text{in }D([0,M],\mathbb R^d),
\]
where
\[
 dZ(\tau)=-\rho BZ(\tau)\,d\tau+Q^{1/2}\,dW_\tau,
 \qquad Z(0)=Z_0,
 \qquad
 \sigma_K=\inf\{\tau:\|Z(\tau)\|\ge K\}\wedge M.
\]
The Brownian motion is independent of $Z_0$.  If the frozen, unstopped process is initialized in
its invariant law, its covariance $V_\infty$ is the unique symmetric nonnegative-definite solution
of
\[
 \rho BV_\infty+V_\infty\rho B^\top=Q,
 \qquad
 V_\infty=\int_0^\infty e^{-\rho Bs}Qe^{-\rho B^\top s}\,ds.
\]
This is an invariant covariance of the frozen limiting process, not a finite-$h$ stationary
covariance and not the covariance at a finite local time for arbitrary $Z_0$.
\end{theorem}

\begin{proof}
Subtract \eqref{eq:target-motion} from \eqref{eq:vector-filter}, divide by $h^{1/4}$, and stop at
$\sigma_{h,K}$.  Assumption~\ref{ass:vector-local-array}(iii)--(v) makes the cumulative nonlinear,
evaluation, and target remainders $o_p(1)$ uniformly on $[0,M]$.  Hence the stopped recursion has
the asymptotic integral form
\[
 Z_h^K(\tau)
 =Z_h(0)-\rho\sum_{j<n_{h,K}(\tau)}
       \sqrt h\,B_{h,k_j}Z_h(j\sqrt h)
   +\sum_{j<n_{h,K}(\tau)}
       h^{1/4}\xi^0_{h,k_j+1}+o_p(1).
\]

Let $\widehat Z_h$ solve the global linear integral equation with start $Z_h(0)$, driving martingale $\mathcal M_h$ and drift matrix $-\rho B$. The all-radii assumptions give the displayed remainder bound up to the filter's own exit, including the step arriving there. Gronwall's inequality with coefficient $\rho\|B\|_{\rm op}$ gives $\sup_{\tau\le\sigma_{h,K}}\|Z_h(\tau)-\widehat Z_h(\tau)\|=o_p(1)$. The linear comparison converges to the stated continuous diffusion. For each fixed $\epsilon>0$, the limiting upper bound on the probability of reaching $K$ is $\mathsf P(\sup_{\tau\le M}\|Z(\tau)\|\ge K-\epsilon)$, by the scalar proof with the norm in place of the absolute value. This tends to zero as $K$ increases. It follows that the processes converge on the full window; applying the stopping map at a continuity radius gives the stopped convergence. Positive stability of $B$ makes the covariance integral finite and gives the unique Lyapunov solution.

\end{proof}

\begin{corollary}[Removal of the stop]\label{cor:vector-unstopped}
Under Assumption~\ref{ass:vector-local-array}, $Z_h\Rightarrow Z$ in $D([0,M],\mathbb R^d)$ on each prescribed finite window.
\end{corollary}
\begin{proof}
The preceding exit bound makes the global comparison error negligible in probability. The continuous-mapping theorem for the linear equation applies also when $Q$ is singular.
\end{proof}

\subsection{Composite covariance and matrix stability}

At finite $h$, let $\Var_k$ and $\Cov_k$ denote conditional variance and
covariance given $\mathcal G_{kh}^{(h)}$. The covariance of the composite
innovation is exactly
\begin{align}
 Q_{h,k}
 &=\rho^2\Var_{k}\{u_h(Y_{(k+1)h},m^{(h)}_{kh})\}
   +\Gamma_{kh}\Var_k(\eta_{(k+1)h})\Gamma_{kh}^\top \notag\\
 &\quad-\rho\Cov_k\{u_h(Y_{(k+1)h},m^{(h)}_{kh}),\eta_{(k+1)h}\}\Gamma_{kh}^\top \notag\\
 &\quad-\rho\Gamma_{kh}
      \Cov_k\{\eta_{(k+1)h},u_h(Y_{(k+1)h},m^{(h)}_{kh})\}. \label{eq:full-vector-q}
\end{align}
Thus $Q$ in Theorem~\ref{thm:vector-stopped-ou} is the limiting covariance density of the
cumulative predictable characteristic built from \eqref{eq:full-vector-q}; pointwise convergence
of $Q_{h,k}$ is not required. Dropping the target-motion and cross-covariance terms is valid only
under additional restrictions such as a locally deterministic target.

\begin{remark}[Positive stability is not Euclidean contraction]
If, at the pseudo-true point,
\(B=G\nabla_\lambda^2R(m),\)
with $G$ and $\nabla_\lambda^2R(m)$ symmetric positive definite, then $B$ is positive stable
because it is similar to the symmetric positive-definite matrix
$G^{1/2}\nabla_\lambda^2R(m)G^{1/2}$.  This is enough for the frozen linear system to possess the
invariant covariance above.  It does not imply that
$\operatorname{sym}(G\nabla_\lambda^2R(m))$ is positive definite.  A Euclidean one-step contraction
or strong-monotonicity result requires a separate lower bound on that symmetric part (or a
specified weighted norm).  The two conditions must not be used interchangeably.
\end{remark}

\subsection{\texorpdfstring{{Matrix frontier}}{Matrix frontier}}\label{app:matrix_statement}
\begingroup
\label{trk:sec:matrix}

Consider the vector setting of \ref{app::multivariate}, frozen at $T$. Suppose that $B_{h,k_j}\to B_T$, positive stable, and that the conditional covariance of $u_h(Y,m)$ tends to $\Xi_T\succ0$. Suppose also that $\Cov(u_h(Y,m),\eta)\to\chi_T$, $\Var(\eta)\to I_q$ and $\Gamma_{kh}\to\Gamma\in\R^{d\times q}$. Write
\[
\mathcal K=\rho B_T,\qquad \Sigma_{\mathsf S}=B_T^{-1}\Xi_TB_T^{-\top},\qquad \tilde C=B_T^{-1}\chi_T .
\]
When $B_T=GJ$ and $\Xi_T=GKG^\top$ with $G$ invertible, $\Sigma_{\mathsf S}=J^{-1}KJ^{-1}$ is the sandwich covariance. For positive definite $A_1,A_2$, write $A_1\#A_2=A_1^{1/2}(A_1^{-1/2}A_2A_1^{-1/2})^{1/2}A_1^{1/2}$ for their geometric mean \citep{kubo1980means,bhatia2007positive}.

\begingroup
\begin{proposition}[Matrix frontier]\label{trk:prop:matrix}
Under the conditions above and those of Theorem~\ref{thm:vector-stopped-ou}:
\begin{enumerate}[label=(\roman*)]
\item The frozen invariant covariance $V$ solves
\begin{equation}\label{trk:eq:lyap}
 \mathcal KV+V\mathcal K^\top
 =\mathcal K\Sigma_{\mathsf S}\mathcal K^\top+\Gamma\Gamma^\top
 -\mathcal K\tilde C\Gamma^\top-\Gamma\tilde C^\top\mathcal K^\top.
\end{equation}
For fixed underlying score and target characteristics, invertible scaling affects this expression through the effective gain matrix $\mathcal K$; the sandwich $\Sigma_{\mathsf S}$ and the corresponding rescaled cross-covariance are invariant to that change of scaling.
\item If $\chi_T=0$ and $\Gamma\Gamma^\top\succ0$, then for every positive-stable $\mathcal K$,
\begin{equation}\label{trk:eq:geo}
 V\succeq\Sigma_{\mathsf S}\#\Gamma\Gamma^\top.
\end{equation}
Equality is attained at the positive-stable matrix $\mathcal K^\star=(\Sigma_{\mathsf S}\#\Gamma\Gamma^\top)\Sigma_{\mathsf S}^{-1}$. This is optimisation over all positive-stable gain matrices. If only a scalar gain with fixed preconditioner, or a restricted class of preconditioners, is admissible, the bound still holds but its equality matrix need not be feasible.
\item For a common target-motion covariance, $\Sigma_{\mathsf S}\succeq\Sigma_{\mathsf S'}$ implies $\Sigma_{\mathsf S}\#\Gamma\Gamma^\top\succeq\Sigma_{\mathsf S'}\#\Gamma\Gamma^\top$. Under the matrix analogue of Assumption~\ref{trk:ass:dtl}, with invertible $D=\partial_x\lambda^\star_{\mathsf S}$ and $\tilde I\succ0$, the identities and inequality
\[
 B_TD=\Cov(u_{\mathsf S},\tilde s),\qquad
 D^{-1}\Sigma_{\mathsf S}D^{-\top}\succeq\tilde I^{-1}
\]
hold. Under the conditions of part~(ii), write $\Gamma=D\Gamma_x$. Then the minimum over the full positive-stable gain class satisfies
\[
 D^{-1}(\Sigma_{\mathsf S}\#\Gamma\Gamma^\top)D^{-\top}
 =(D^{-1}\Sigma_{\mathsf S}D^{-\top})\#\Gamma_x\Gamma_x^\top
 \succeq\tilde I^{-1}\#\Gamma_x\Gamma_x^\top.
\]
This is the matrix information benchmark within the zero-alignment class. Identifying it with the unrestricted optimal covariance additionally requires the corresponding no-predictable-target-shock condition, not merely $\chi_T=0$ for the chosen driver; see Remark~\ref{rev:rem:leverage} and \citet{beutner2026consistency}.
\end{enumerate}
\end{proposition}
\endgroup
\proofref{app:proof_trk_prop_matrix}

{In the scalar case, \eqref{trk:eq:geo} reduces to \eqref{trk:eq:vmin} with $\gamma_{\mathsf S}=0$. The matrix calculation replaces $\kappa_{\mathsf S}$ by the sandwich covariance and preserves the efficiency ordering for a common target-motion covariance. Its scope remains a frozen, zero-alignment covariance comparison over the explicitly stated matrix-gain class.}
\endgroup

\begingroup\label{app:matrix_frontier}\endgroup

\subsection{\texorpdfstring{{Proof of Proposition~\ref{trk:prop:matrix}}}{Proof of Proposition}}\label{app:proof_trk_prop_matrix}
\begingroup
\begin{proof}
(i) By \eqref{eq:full-vector-q} with the frozen limits, $Q=\rho^2\Xi_T+\Gamma\Gamma^\top-\rho\chi_T\Gamma^\top-\rho\Gamma\chi_T^\top$. Since $\rho^2\Xi_T=\mathcal K\Sigma_{\mathsf S}\mathcal K^\top$, $\rho\chi_T=\mathcal K\tilde C$ and $\rho B_T=\mathcal K$, the Lyapunov equation of Theorem~\ref{thm:vector-stopped-ou} becomes \eqref{trk:eq:lyap}.

(ii) Put $V^\star=\Sigma_{\mathsf S}\#\Gamma\Gamma^\top$, the unique positive definite solution of $V^\star\Sigma_{\mathsf S}^{-1}V^\star=\Gamma\Gamma^\top$ \citep[Ch.~4]{bhatia2007positive}, and $\mathcal K^\star=V^\star\Sigma_{\mathsf S}^{-1}$. This gain is positive stable because it is similar to $\Sigma_{\mathsf S}^{-1/2}V^\star\Sigma_{\mathsf S}^{-1/2}\succ0$. For any positive-stable $\mathcal K$, subtracting the equation for $V^\star$ in completed-square form from \eqref{trk:eq:lyap} with $\chi_T=0$ gives the exact gain-cost identity
\[
 \mathcal K(V-V^\star)+(V-V^\star)\mathcal K^\top
 =(\mathcal K-\mathcal K^\star)\Sigma_{\mathsf S}(\mathcal K-\mathcal K^\star)^\top.
\]
Since $\mathcal K$ is positive stable, the unique Lyapunov solution is
\[
 V-V^\star=\int_0^\infty e^{-t\mathcal K}
 (\mathcal K-\mathcal K^\star)\Sigma_{\mathsf S}
 (\mathcal K-\mathcal K^\star)^\top e^{-t\mathcal K^\top}\,dt\succeq0.
\]
The right-hand side vanishes if and only if $\mathcal K=\mathcal K^\star$: the integrand is continuous and positive semidefinite, and its value at zero vanishes only for that gain since $\Sigma_{\mathsf S}\succ0$. Thus equality in \eqref{trk:eq:geo} is attained precisely at the stated positive-stable gain. The identity also describes the loss if the gain is restricted to a scalar multiple of a fixed preconditioner; attainment within that restricted class is not automatic.

(iii) Monotonicity of $\#$ is a property of operator means \citep{kubo1980means}. The identity $B_TD=\Cov(u_{\mathsf S},\tilde s)$ follows as in Proposition~\ref{trk:prop:eff}(i), with matrix derivatives. Matrix Cauchy--Schwarz, $\Xi_T\succeq\Cov(u_{\mathsf S},\tilde s)\tilde I^{-1}\Cov(\tilde s,u_{\mathsf S})=B_TD\tilde I^{-1}D^\top B_T^\top$, gives $\Sigma_{\mathsf S}\succeq D\tilde I^{-1}D^\top$. The geometric mean is congruence invariant, $M(A_1\#A_2)M^\top=(MA_1M^\top)\#(MA_2M^\top)$; take $M=D^{-1}$.
\end{proof}
\endgroup


\clearpage

\clearpage\begin{thebibliography}{99}

\bibitem[Arnold(1974)]{arnold1974stochastic}
Arnold, Ludwig. 1974. \emph{Stochastic Differential Equations: Theory and Applications}. New York: Wiley.


\bibitem[Beutner et al.(2023)]{beutner2023consistency}
Beutner, Eric, Yicong Lin, and Andr\'e Lucas. 2023. ``Consistency, Distributional Convergence, and Optimality of Score-Driven Filters.'' Tinbergen Institute Discussion Paper 2023-051/III, August 2023 version. \url{https://papers.tinbergen.nl/23051.pdf}.
\bibitem[Beutner et al.(2026)]{beutner2026consistency}
Beutner, Eric, Yicong Lin, and Andr\'e Lucas. 2026. ``Consistency, Distributional Convergence, and Optimality of Time-Varying Parameters in Score-Driven Models.'' \emph{Journal of Econometrics} 255: 1--21. Article 106218. doi:10.1016/j.jeconom.2026.106218.


\bibitem[Bhatia(2007)]{bhatia2007positive}
{Bhatia, Rajendra. 2007. \emph{Positive Definite Matrices}. Princeton: Princeton University Press.}


\bibitem[Blasques et al.(2023)]{blasques2023quasi}
Blasques, Francisco, Christian Francq, and S\'ebastien Laurent. 2023. ``Quasi Score-Driven Models.'' \emph{Journal of Econometrics} 234 (1): 251--275. doi:10.1016/j.jeconom.2021.12.005.
\bibitem[Blasques et~al.(2022)]{blasques2022maximum}
{Blasques, Francisco, Janneke van Brummelen, Siem Jan Koopman, and Andr\'e Lucas. 2022. ``Maximum Likelihood Estimation for Score-Driven Models.'' \emph{Journal of Econometrics} 227 (2): 325--346.}
\bibitem[Blasques et al.(2018)]{blasques2018feasible}
Blasques, Francisco, Paolo Gorgi, Siem Jan Koopman, and Olivier Wintenberger. 2018. ``Feasible Invertibility Conditions and Maximum Likelihood Estimation for Observation-Driven Models.'' \emph{Electronic Journal of Statistics} 12 (1): 1019--1052. doi:10.1214/18-EJS1416.


\bibitem[Blasques et al.(2015)]{blasques2015information}
Blasques, Francisco, Siem Jan Koopman, and Andr\'e Lucas. 2015. ``Information-Theoretic Optimality of Observation-Driven Time Series Models for Continuous Responses.'' \emph{Biometrika} 102 (2): 325--343. doi:10.1093/biomet/asu076.
\bibitem[Federal Reserve Board(n.d.)]{fred_dexthus}
Board of Governors of the Federal Reserve System (US). \emph{Thai Baht to U.S. Dollar Spot Exchange Rate} [DEXTHUS]. Retrieved from FRED, Federal Reserve Bank of St.~Louis. \url{https://fred.stlouisfed.org/series/DEXTHUS}.


\bibitem[Bougerol(1993)]{bougerol1993kalman}
Bougerol, Philippe. 1993. ``Kalman Filtering with Random Coefficients and Contractions.'' \emph{SIAM Journal on Control and Optimization} 31 (4): 942--959. doi:10.1137/0331041.
\bibitem[Buccheri et al.(2021)]{buccheri2021continuous}
Buccheri, Giuseppe, Fulvio Corsi, Franco Flandoli, and Giulia Livieri. 2021. ``The Continuous-Time Limit of Score-Driven Volatility Models.'' \emph{Journal of Econometrics} 221 (2): 655--675. doi:10.1016/j.jeconom.2020.07.042.


\bibitem[Chen et al.(2022)]{chen2022sa}
Chen, Zaiwei, Shancong Mou and Siva Theja Maguluri. 2022. Stationary behavior of constant stepsize SGD type algorithms: an asymptotic characterization. \emph{Proceedings of the ACM on Measurement and Analysis of Computing Systems} 6(1). \href{https://doi.org/10.1145/3508039}{doi:10.1145/3508039}.
\bibitem[Cobb(1978)]{cobb1978nile}
Cobb, George W. 1978. ``The Problem of the Nile: Conditional Solution to a Changepoint Problem.'' \emph{Biometrika} 65 (2): 243--251.
\bibitem[Corradi(2000)]{corradi2000reconsidering}
Corradi, Valentina. 2000. ``Reconsidering the Continuous-Time Limit of the GARCH(1, 1) Process.'' \emph{Journal of Econometrics} 96 (1): 145--153. doi:10.1016/S0304-4076(99)00053-6.


\bibitem[Creal et al.(2013)]{creal2013gas}
Creal, Drew, Siem Jan Koopman, and Andr\'e Lucas. 2013. ``Generalized Autoregressive Score Models with Applications.'' \emph{Journal of Applied Econometrics} 28 (5): 777--795. doi:10.1002/jae.1279.


\bibitem[Davidson(1994)]{davidson1994stochastic}
{Davidson, James. 1994. \emph{Stochastic Limit Theory: An Introduction for Econometricians}. Oxford: Oxford University Press.}
\bibitem[Dawid(1986)]{dawid2004probability}
Dawid, A. Philip. 1986. ``Probability Forecasting.'' In \emph{Encyclopedia of Statistical Sciences}, edited by Samuel Kotz, Norman L. Johnson, and Campbell B. Read, vol. 7, 210--218. New York: Wiley.
\bibitem[Dawid et al.(2016)]{dawid2016minimum}
Dawid, A. Philip, Monica Musio, and Laura Ventura. 2016. ``Minimum Scoring Rule Inference.'' \emph{Scandinavian Journal of Statistics} 43 (1): 123--138. doi:10.1111/sjos.12168.
\bibitem[de Punder(2026)]{depunder2026prada}
de Punder, Ramon F.\ A. 2026. \emph{Proper and Robust Autoregressive Derivative Adaptive Models}. Tinbergen Institute Discussion Paper 2026-022/III, version dated 20 November 2026. \url{https://papers.tinbergen.nl/26022.pdf}.
\bibitem[de Punder, Dimitriadis, and Lange(2026)]{de2024kullback}
de Punder, Ramon, Timo Dimitriadis, and Rutger-Jan Lange. 2026. ``Expected Kullback--Leibler-Based Characterizations of Score-Driven Updates.'' arXiv:2408.02391 [math.ST].


\bibitem[Dieuleveut et al.(2020)]{dieuleveut2020bridging}
Dieuleveut, Aymeric, Alain Durmus, and Francis Bach. 2020. ``Bridging the Gap between Constant Step Size Stochastic Gradient Descent and Markov Chains.'' \emph{Annals of Statistics} 48(3): 1348--1382. doi:10.1214/19-AOS1850.

\bibitem[Donker van Heel et al.(2026)]{donker2025stability}
Donker van Heel, Simon W., Rutger-Jan Lange, Bram van Os, and Dick van Dijk. 2026. ``Gradient-Based Filtering under Misspecification: Stability and Error Bounds.'' arXiv:2502.05021 [stat.ME], version 8.
\bibitem[Engle and Manganelli(2004)]{engle2004caviar}
{Engle, Robert F., and Simone Manganelli. 2004. ``CAViaR: Conditional Autoregressive Value at Risk by Regression Quantiles.'' \emph{Journal of Business \& Economic Statistics} 22 (4): 367--381.}
\bibitem[Ethier and Kurtz(1986)]{ethier1986markov}
Ethier, Stewart N., and Thomas G. Kurtz. 1986. \emph{Markov Processes: Characterization and Convergence}. New York: Wiley.
\bibitem[Ethier and Kurtz(1994)]{ethier1994convergence}
Ethier, Stewart N., and Thomas G. Kurtz. 1994. ``Convergence to Fleming--Viot Processes in the Weak Atomic Topology.'' \emph{Stochastic Processes and Their Applications} 54 (1): 1--27. doi:10.1016/0304-4149(94)00006-9.
\bibitem[Garman and Klass(1980)]{garman1980estimation}
Garman, Mark B., and Michael J. Klass. 1980. ``On the Estimation of Security Price Volatilities from Historical Data.'' \emph{Journal of Business} 53 (1): 67--78.
\bibitem[Gneiting and Raftery(2007)]{gneiting2007strictly}
Gneiting, Tilmann, and Adrian E. Raftery. 2007. ``Strictly Proper Scoring Rules, Prediction, and Estimation.'' \emph{Journal of the American Statistical Association} 102 (477): 359--378. doi:10.1198/016214506000001437.
\bibitem[Gneiting et al.(2007)]{gneiting2007probabilistic}
Gneiting, Tilmann, Fadoua Balabdaoui, and Adrian E. Raftery. 2007. ``Probabilistic Forecasts, Calibration and Sharpness.'' \emph{Journal of the Royal Statistical Society: Series B} 69 (2): 243--268.
\bibitem[Godambe(1960)]{godambe1960optimum}
{Godambe, V. P. 1960. ``An Optimum Property of Regular Maximum Likelihood Estimation.'' \emph{Annals of Mathematical Statistics} 31 (4): 1208--1211.}
\bibitem[Gorgi et al.(2024)]{gorgi2024optimality}
Gorgi, Paolo, C. S. A. Lauria, and Alessandra Luati. 2024. ``On the Optimality of Score-Driven Models.'' \emph{Biometrika} 111 (3): 865--880. doi:10.1093/biomet/asad067.
\bibitem[Hall and Heyde(1980)]{hall1980martingale}
{Hall, Peter, and Christopher C. Heyde. 1980. \emph{Martingale Limit Theory and Its Application}. New York: Academic Press.}


\bibitem[Harvey(2013)]{harvey2013dynamic}
Harvey, Andrew C. 2013. \emph{Dynamic Models for Volatility and Heavy Tails: With Applications to Financial and Economic Time Series}. Cambridge: Cambridge University Press. doi:10.1017/CBO9781139540933.
\bibitem[Holm(1979)]{holm1979}
Holm, Sture. 1979. A simple sequentially rejective multiple test procedure. \emph{Scandinavian Journal of Statistics} 6(2), 65--70. \href{https://www.jstor.org/stable/4615733}{Journal archive}.
\bibitem[Jordan et al.(2019)]{jordan2019evaluating}
Jordan, Alexander, Fabian Kr\"uger, and Sebastian Lerch. 2019. ``Evaluating Probabilistic Forecasts with scoringRules.'' \emph{Journal of Statistical Software} 90 (12): 1--37. doi:10.18637/jss.v090.i12.
\bibitem[Kalman and Bucy(1961)]{kalman1961new}
{Kalman, Rudolf E., and Richard S. Bucy. 1961. ``New Results in Linear Filtering and Prediction Theory.'' \emph{Journal of Basic Engineering} 83 (1): 95--108.}


\bibitem[Kubo and Ando(1980)]{kubo1980means}
{Kubo, Fumio, and Tsuyoshi Ando. 1980. ``Means of Positive Linear Operators.'' \emph{Mathematische Annalen} 246 (3): 205--224.}


\bibitem[Kushner(1984)]{kushner1984approximation}
Kushner, Harold J. 1984. \emph{Approximation and Weak Convergence Methods for Random Processes, with Applications to Stochastic Systems Theory}. MIT Press Series in Signal Processing, Optimization, and Control 6. Cambridge, MA: MIT Press.
\bibitem[Lange et al.(2026)]{lange2026implicit}
Lange, Rutger-Jan, Bram van Os, and Dick van Dijk. 2026. ``Implicit Score-Driven Filters for Time-Varying Parameter Models.'' \emph{Journal of Econometrics} 255: Article 106251. doi:10.1016/j.jeconom.2026.106251. Author version: \url{https://papers.tinbergen.nl/22066.pdf}.
\bibitem[Li and Wu(2026)]{liwu2026moments}
Li, Xiaoli, and Wei Biao Wu. 2026. \emph{Moment-Accurate Gaussian Mixtures for Constant-Step Stochastic Approximation}. Preprint, arXiv:2610.05595v1, 4 October 2026. \url{https://arxiv.org/abs/2610.05595}.

\bibitem[Meyn and Tweedie(2009)]{meyn2009markov}
{Meyn, Sean, and Richard L. Tweedie. 2009. \emph{Markov Chains and Stochastic Stability}. 2nd ed. Cambridge: Cambridge University Press.}
\bibitem[Nelson(1990)]{nelson1990arch}
Nelson, Daniel B. 1990. ``ARCH Models as Diffusion Approximations.'' \emph{Journal of Econometrics} 45 (1--2): 7--38. doi:10.1016/0304-4076(90)90092-8.
\bibitem[Nelson(1996)]{nelson1996filtering}
Nelson, Daniel B. 1996. ``Asymptotic Filtering Theory for Multivariate ARCH Models.'' \emph{Journal of Econometrics} 71 (1--2): 1--47. Working-paper record: \url{https://www.nber.org/papers/t0162}.
\bibitem[Nelson and Foster(1994)]{nelson1994filtering}
Nelson, Daniel B., and Dean P. Foster. 1994. ``Asymptotic Filtering Theory for Univariate ARCH Models.'' \emph{Econometrica} 62 (1): 1--41.
\bibitem[Parkinson(1980)]{parkinson1980extreme}
Parkinson, Michael. 1980. ``The Extreme Value Method for Estimating the Variance of the Rate of Return.'' \emph{Journal of Business} 53 (1): 61--65.
\bibitem[Patton(2011)]{patton2011volatility}
Patton, Andrew J. 2011. ``Volatility Forecast Comparison Using Imperfect Volatility Proxies.'' \emph{Journal of Econometrics} 160 (1): 246--256.
\bibitem[Savage(1954)]{savage1954foundations}
Savage, Leonard J. 1954. \emph{The Foundations of Statistics}. New York: Wiley.
\bibitem[Straumann and Mikosch(2006)]{straumann2006quasi}
Straumann, Daniel, and Thomas Mikosch. 2006. ``Quasi-Maximum-Likelihood Estimation in Conditionally Heteroscedastic Time Series: A Stochastic Recurrence Equations Approach.'' \emph{The Annals of Statistics} 34 (5): 2449--2495. doi:10.1214/009053606000000803.
\bibitem[Stroock and Varadhan(1979)]{stroock2007multidimensional}
Stroock, Daniel W., and S. R. S. Varadhan. 1979. \emph{Multidimensional Diffusion Processes}. Grundlehren der Mathematischen Wissenschaften 233. Berlin: Springer.
\bibitem[van der Vaart(1998)]{vaart1998asymptotic}
{van der Vaart, A. W. 1998. \emph{Asymptotic Statistics}. Cambridge: Cambridge University Press.}
\bibitem[Wintenberger(2013)]{wintenberger2013continuous}
{Wintenberger, Olivier. 2013. ``Continuous Invertibility and Stable QML Estimation of the EGARCH(1,1) Model.'' \emph{Scandinavian Journal of Statistics} 40 (4): 846--867.}
\bibitem[Wu and He(2026)]{wuhe2026continuous}
Wu, Yinhao, and Ping He. 2026. ``The Continuous-Time Limit of Quasi Score-Driven Volatility Models.'' \emph{Journal of Time Series Analysis} 47 (5): 924--938. doi:10.1111/jtsa.12848.
\bibitem[Zachary(2004)]{zachary2004veraverbeke}
{Zachary, Stan. 2004. ``A Note on Veraverbeke's Theorem.'' \emph{Queueing Systems} 46 (1--2): 9--14. \url{https://arxiv.org/abs/1410.8551}.}
\end{thebibliography}
\end{document}